\documentclass[9pt]{amsart}
\usepackage{amsmath}
\usepackage{amsxtra}
\usepackage{amscd}
\usepackage{amsthm}
\usepackage{amsfonts}
\usepackage{amssymb}
\usepackage{eucal}
\usepackage[all]{xy}
\usepackage{graphicx}

\newtheorem{cor}[subsubsection]{Corollary}
\newtheorem{lem}[subsubsection]{Lemma}
\newtheorem{prop}[subsubsection]{Proposition}
\newtheorem{thmconstr}[subsubsection]{Theorem-Construction}
\newtheorem{propconstr}[subsubsection]{Proposition-Construction}

\newtheorem{thm}[subsubsection]{Theorem}

\theoremstyle{definition}

\theoremstyle{remark}
\newtheorem{rem}[subsubsection]{Remark}

\newcommand{\thmref}[1]{Theorem~\ref{#1}}

\newcommand{\secref}[1]{Sect.~\ref{#1}}
\newcommand{\lemref}[1]{Lemma~\ref{#1}}
\newcommand{\propref}[1]{Proposition~\ref{#1}}
\newcommand{\corref}[1]{Corollary~\ref{#1}}

\numberwithin{equation}{section}

\newcommand{\nc}{\newcommand}
\nc{\renc}{\renewcommand}
\nc{\ssec}{\subsection}
\nc{\sssec}{\subsubsection}
\nc{\on}{\operatorname}

\nc\ol{\overline}
\nc\wt{\widetilde}
\nc\tboxtimes{\wt{\boxtimes}}
\nc\tstar{\wt{\star}}
\nc{\alp}{a}

\nc{\ZZ}{{\mathbb Z}}
\nc{\NN}{{\mathbb N}}
\nc{\OO}{{\mathbb O}}
\renc{\SS}{{\mathbb S}}
\nc{\DD}{{\mathbb D}}
\nc{\GG}{{\mathbb G}}

\nc{\Fq}{{\mathbb F}_q}
\nc{\Fqb}{\ol{{\mathbb F}_q}}
\nc{\Ql}{\ol{{\mathbb Q}_\ell}}
\nc{\id}{\text{id}}
\nc\X{\mathcal X}

\nc{\Hom}{\on{Hom}}
\nc{\Lie}{\on{Lie}}
\nc{\Loc}{\on{Loc}}
\nc{\Pic}{\on{Pic}}
\nc{\Bun}{\on{Bun}}
\nc{\IC}{\on{IC}}
\nc{\Aut}{\on{Aut}}
\nc{\rk}{\on{rk}}
\nc{\Sh}{\on{Sh}}
\nc{\Perv}{\on{Perv}}
\nc{\pos}{{\on{pos}}}
\nc{\Conv}{\on{Conv}}
\nc{\Sph}{\on{Sph}}
\nc{\Sym}{\on{Sym}}
\nc{\BunBb}{\overline{\Bun}_B}
\nc{\BunNb}{\overline{\Bun}_N}
\nc{\BunTb}{\overline{\Bun}_T}
\nc{\BunBbm}{\overline{\Bun}_{B^-}}
\nc{\BunBbel}{\overline{\Bun}_{B,el}}
\nc{\BunBbmel}{\overline{\Bun}_{B^-,el}}
\nc{\Buno}{\overset{o}{\Bun}}
\nc{\BunPb}{{\overline{\Bun}_P}}
\nc{\BunBM}{\Bun_{B(M)}}
\nc{\BunBMb}{\overline{\Bun}_{B(M)}}
\nc{\BunPbw}{{\widetilde{\Bun}_P}}
\nc{\BunBP}{\widetilde{\Bun}_{B,P}}
\nc{\GUb}{\overline{G/U}}
\nc{\GUPb}{\overline{G/U(P)}}

\nc{\Hhom}{\underline{\on{Hom}}}
\nc\syminfty{\on{Sym}^{\infty}}
\nc\lal{\ol{\kappa_x}}
\nc\xl{\ol{x}}
\nc\thl{\ol{\theta}}
\nc\nul{\ol{\nu}}
\nc\mul{\ol{\mu}}
\nc\Sum\Sigma
\nc{\oX}{\overset{o}{X}{}}
\nc{\hl}{\overset{\leftarrow}h{}}
\nc{\hr}{\overset{\rightarrow}h{}}
\nc{\M}{{\mathcal M}}
\nc{\N}{{\mathcal N}}
\nc{\F}{{\mathcal F}}
\nc{\D}{{\mathcal D}}
\nc{\Q}{{\mathcal Q}}
\nc{\Y}{{\mathcal Y}}
\nc{\G}{{\mathcal G}}
\nc{\E}{{\mathcal E}}
\nc{\CalC}{{\mathcal C}}
\nc\Dh{\widehat{\D}}

\nc{\C}{{\mathcal C}}
\nc{\K}{{\mathcal K}}
\renewcommand{\H}{{\mathcal H}}

\nc{\T}{{\mathcal T}}
\nc{\V}{{\mathcal V}}
\renc{\P}{{\mathcal P}}
\nc{\A}{{\mathcal A}}
\nc{\B}{{\mathcal B}}
\nc{\U}{{\mathcal U}}

\nc{\Gr}{{\on{Gr}}}

\nc{\frn}{{\check{\mathfrak u}(P)}}

\nc{\fC}{\mathfrak C}
\nc{\p}{\mathfrak p}
\nc{\q}{\mathfrak q}
\nc\f{{\mathfrak f}}

\nc{\qo}{{\mathfrak q}}
\nc{\po}{{\mathfrak p}}
\nc{\s}{{\mathfrak s}}
\nc\w{\text{w}}

\renewcommand{\mod}{{\on{-mod}}}

\nc\mathi\iota
\nc\Spec{\on{Spec}}
\nc\Mod{\on{Mod}}
\nc{\tw}{\widetilde{\mathfrak t}}
\nc{\pw}{\widetilde{\mathfrak p}}
\nc{\qw}{\widetilde{\mathfrak q}}
\nc{\jw}{\widetilde j}

\nc{\grb}{\overline{\Gr}}
\nc{\I}{\mathcal I}

\nc{\kappach}{{\check\kappa_x}}
\nc{\Lambdach}{{\check\Lambda}{}}
\nc{\much}{{\check\mu}}
\nc{\omegach}{{\check\omega}}
\nc{\nuch}{{\check\nu}}
\nc{\etach}{{\check\eta}}
\nc{\alphach}{{\checka}}
\nc{\oblvtach}{{\check\oblvta}}
\nc{\pich}{{\check\pi}}
\nc{\ch}{{\check h}}

\nc{\Hb}{\overline{\H}}

\nc{\BA}{{\mathbb{A}}}
\nc{\BC}{{\mathbb{C}}}
\nc{\BF}{{\mathbb{F}}}
\nc{\BG}{{\mathbb{G}}}
\nc{\BM}{{\mathbb{M}}}
\nc{\BO}{{\mathbb{O}}}
\nc{\BD}{{\mathbb{D}}}
\nc{\BN}{{\mathbb{N}}}
\nc{\BP}{{\mathbb{P}}}
\nc{\BQ}{{\mathbb{Q}}}
\nc{\BR}{{\mathbb{R}}}
\nc{\BZ}{{\mathbb{Z}}}
\nc{\BS}{{\mathbb{S}}}

\nc{\CA}{{\mathcal{A}}}
\nc{\CB}{{\mathcal{B}}}

\nc{\CE}{{\mathcal{E}}}
\nc{\CF}{{\mathcal{F}}}
\nc{\CG}{{\mathcal{G}}}
\nc{\CH}{{\mathcal{H}}}

\nc{\CL}{{\mathcal{L}}}
\nc{\CC}{{\mathcal{C}}}
\nc{\CM}{{\mathcal{M}}}
\nc{\CN}{{\mathcal{N}}}
\nc{\CK}{{\mathcal{K}}}
\nc{\CO}{{\mathcal{O}}}
\nc{\CP}{{\mathcal{P}}}
\nc{\CQ}{{\mathcal{Q}}}
\nc{\CR}{{\mathcal{R}}}
\nc{\CS}{{\mathcal{S}}}
\nc{\CT}{{\mathcal{T}}}
\nc{\CU}{{\mathcal{U}}}
\nc{\CV}{{\mathcal{V}}}
\nc{\CW}{{\mathcal{W}}}
\nc{\CX}{{\mathcal{X}}}
\nc{\CY}{{\mathcal{Y}}}
\nc{\CZ}{{\mathcal{Z}}}
\nc{\CI}{{\mathcal{I}}}
\nc{\CJ}{{\mathcal{J}}}

\nc{\csM}{{\check{\mathcal A}}{}}
\nc{\oM}{{\overset{\circ}{\mathcal M}}{}}
\nc{\obM}{{\overset{\circ}{\mathbf M}}{}}
\nc{\oCA}{{\overset{\circ}{\mathcal A}}{}}
\nc{\obA}{{\overset{\circ}{\mathbf A}}{}}
\nc{\ooM}{{\overset{\circ}{M}}{}}
\nc{\osM}{{\overset{\circ}{\mathsf M}}{}}
\nc{\vM}{{\overset{\bullet}{\mathcal M}}{}}
\nc{\nM}{{\underset{\bullet}{\mathcal M}}{}}
\nc{\oD}{{\overset{\circ}{\mathcal D}}{}}
\nc{\obD}{{\overset{\circ}{\mathbf D}}{}}
\nc{\oA}{{\overset{\circ}{\mathbb A}}{}}
\nc{\op}{{\overset{\bullet}{\mathbf p}}{}}
\nc{\cp}{{\overset{\circ}{\mathbf p}}{}}
\nc{\oU}{{\overset{\bullet}{\mathcal U}}{}}
\nc{\oZ}{{\overset{\circ}{\mathcal Z}}{}}
\nc{\ofZ}{{\overset{\circ}{\mathfrak Z}}{}}
\nc{\oF}{{\overset{\circ}{\fF}}}

\nc{\fa}{{\mathfrak{a}}}
\nc{\fb}{{\mathfrak{b}}}
\nc{\fd}{{\mathfrak{d}}}
\nc{\ff}{{\mathfrak{f}}}
\nc{\fg}{{\mathfrak{g}}}
\nc{\fgl}{{\mathfrak{gl}}}
\nc{\fh}{{\mathfrak{h}}}
\nc{\fj}{{\mathfrak{j}}}
\nc{\fl}{{\mathfrak{l}}}
\nc{\fm}{{\mathfrak{m}}}
\nc{\fn}{{\mathfrak{n}}}
\nc{\fu}{{\mathfrak{u}}}
\nc{\fp}{{\mathfrak{p}}}
\nc{\fr}{{\mathfrak{r}}}
\nc{\fs}{{\mathfrak{s}}}
\nc{\ft}{{\mathfrak{t}}}
\nc{\fz}{{\mathfrak{z}}}
\nc{\fsl}{{\mathfrak{sl}}}
\nc{\hsl}{{\widehat{\mathfrak{sl}}}}
\nc{\hgl}{{\widehat{\mathfrak{gl}}}}
\nc{\hg}{{\widehat{\mathfrak{g}}}}
\nc{\chg}{{\widehat{\mathfrak{g}}}{}^\vee}
\nc{\hn}{{\widehat{\mathfrak{n}}}}
\nc{\chn}{{\widehat{\mathfrak{n}}}{}^\vee}

\nc{\fA}{{\mathfrak{A}}}
\nc{\fB}{{\mathfrak{B}}}
\nc{\fD}{{\mathfrak{D}}}
\nc{\fE}{{\mathfrak{E}}}
\nc{\fF}{{\mathfrak{F}}}
\nc{\fG}{{\mathfrak{G}}}
\nc{\fK}{{\mathfrak{K}}}
\nc{\fL}{{\mathfrak{L}}}
\nc{\fM}{{\mathfrak{M}}}
\nc{\fN}{{\mathfrak{N}}}
\nc{\fP}{{\mathfrak{P}}}
\nc{\fU}{{\mathfrak{U}}}
\nc{\fV}{{\mathfrak{V}}}
\nc{\fZ}{{\mathfrak{Z}}}

\nc{\bb}{{\mathbf{b}}}
\nc{\bc}{{\mathbf{c}}}
\nc{\bd}{{\mathbf{d}}}
\nc{\bbf}{{\mathbf{f}}}
\nc{\be}{{\mathbf{e}}}
\nc{\bi}{{\mathbf{i}}}
\nc{\bj}{{\mathbf{j}}}
\nc{\bn}{{\mathbf{n}}}
\nc{\bp}{{\mathbf{p}}}
\nc{\bq}{{\mathbf{q}}}
\nc{\bu}{{\mathbf{u}}}
\nc{\bv}{{\mathbf{v}}}
\nc{\bx}{{\mathbf{x}}}
\nc{\bs}{{\mathbf{s}}}
\nc{\by}{{\mathbf{y}}}
\nc{\bw}{{\mathbf{w}}}
\nc{\bA}{{\mathbf{A}}}
\nc{\bK}{{\mathbf{K}}}
\nc{\bB}{{\mathbf{B}}}
\nc{\bC}{{\mathbf{C}}}
\nc{\bG}{{\mathbf{G}}}
\nc{\bD}{{\mathbf{D}}}
\nc{\bE}{{\mathbf{E}}}
\nc{\bH}{{\mathbf{H}}}
\nc{\bI}{{\mathbf{I}}}
\nc{\bM}{{\mathbf{M}}}
\nc{\bN}{{\mathbf{N}}}
\nc{\bV}{{\mathbf{V}}}
\nc{\bW}{{\mathbf{W}}}
\nc{\bX}{{\mathbf{X}}}
\nc{\bZ}{{\mathbf{Z}}}
\nc{\bS}{{\mathbf{S}}}

\nc{\sA}{{\mathsf{A}}}
\nc{\sB}{{\mathsf{B}}}
\nc{\sC}{{\mathsf{C}}}
\nc{\sD}{{\mathsf{D}}}
\nc{\sF}{{\mathsf{F}}}
\nc{\sK}{{\mathsf{K}}}
\nc{\sM}{{\mathsf{M}}}
\nc{\sO}{{\mathsf{O}}}
\nc{\sW}{{\mathsf{W}}}
\nc{\sQ}{{\mathsf{Q}}}
\nc{\sP}{{\mathsf{P}}}
\nc{\sZ}{{\mathsf{Z}}}
\nc{\sfp}{{\mathsf{p}}}
\nc{\sr}{{\mathsf{r}}}
\nc{\bk}{{\mathsf{k}}}
\nc{\sg}{{\mathsf{g}}}
\nc{\sff}{{\mathsf{f}}}
\nc{\sfb}{{\mathsf{b}}}
\nc{\sfc}{{\mathsf{c}}}
\nc{\sd}{{\mathsf{d}}}

\nc{\BK}{{\bar{K}}}

\nc{\tA}{{\widetilde{\mathbf{A}}}}
\nc{\tB}{{\widetilde{\mathcal{B}}}}
\nc{\tg}{{\widetilde{\mathfrak{g}}}}
\nc{\tG}{{\widetilde{G}}}
\nc{\TM}{{\widetilde{\mathbb{M}}}{}}
\nc{\tO}{{\widetilde{\mathsf{O}}}{}}
\nc{\tU}{{\widetilde{\mathfrak{U}}}{}}
\nc{\TZ}{{\tilde{Z}}}
\nc{\tx}{{\tilde{x}}}
\nc{\tbv}{{\tilde{\bv}}}
\nc{\tfP}{{\widetilde{\mathfrak{P}}}{}}
\nc{\tz}{{\tilde{\zeta}}}
\nc{\tmu}{{\tilde{\mu}}}

\nc{\urho}{\underline{\pi}}
\nc{\uB}{\underline{B}}
\nc{\uC}{{\underline{\mathbb{C}}}}
\nc{\ui}{\underline{i}}
\nc{\uj}{\underline{j}}
\nc{\ofP}{{\overline{\mathfrak{P}}}}
\nc{\oB}{{\overline{\mathcal{B}}}}
\nc{\og}{{\overline{\mathfrak{g}}}}
\nc{\oI}{{\overline{I}}}

\nc{\eps}{\varepsilon}
\nc{\hrho}{{\hat{\pi}}}

\nc{\one}{{\mathbf{1}}}
\nc{\two}{{\mathbf{t}}}

\nc{\Rep}{{\mathop{\operatorname{\rm Rep}}}}
\nc{\Tot}{{\mathop{\operatorname{\rm Tot}}}}
\nc{\Ker}{{\mathop{\operatorname{\rm Ker}}}}
\nc{\Hilb}{{\mathop{\operatorname{\rm Hilb}}}}
\nc{\End}{{\mathop{\operatorname{\rm End}}}}
\nc{\Ext}{{\mathop{\operatorname{\rm Ext}}}}
\nc{\CHom}{{\mathop{\operatorname{{\mathcal{H}}\it om}}}}
\nc{\GL}{{\mathop{\operatorname{\rm GL}}}}
\nc{\gr}{{\mathop{\operatorname{\rm gr}}}}
\nc{\Id}{{\mathop{\operatorname{\rm Id}}}}
\nc{\de}{{\mathop{\operatorname{\rm def}}}}
\nc{\length}{{\mathop{\operatorname{\rm length}}}}
\nc{\supp}{{\mathop{\operatorname{\rm supp}}}}

\nc{\Cliff}{{\mathsf{Cliff}}}
\nc{\Fl}{\on{Fl}}
\nc{\Fib}{{\mathsf{Fib}}}
\nc{\Coh}{{\mathsf{Coh}}}
\nc{\QCoh}{{\mathsf{QCoh}}}
\nc{\IndCoh}{{\mathsf{IndCoh}}}
\nc{\FCoh}{{\mathsf{FCoh}}}

\nc{\reg}{{\text{\rm reg}}}

\nc{\cplus}{{\mathbf{C}_+}}
\nc{\cminus}{{\mathbf{C}_-}}
\nc{\cthree}{{\mathbf{C}_*}}
\nc{\Qbar}{{\bar{Q}}}
\nc\Eis{\on{Eis}}
\nc\Eisb{\ol\Eis{}}
\nc\Eisr{\on{Eis}^{rat}{}}
\nc\wh{\widehat}
\nc{\Def}{\on{Def_{\check{\fb}}(E)}}
\nc{\barZ}{\overline{Z}{}}
\nc{\barbarZ}{\overline{\barZ}{}}
\nc{\barpi}{\overline\iota}
\nc{\barbarpi}{\overline\barpi}
\nc{\barpip}{\overline\iota{}^+}
\nc{\barpim}{\overline\iota{}^-}

\nc{\fq}{\mathfrak q}

\nc{\fqb}{\ol{\fq}{}}
\nc{\fpb}{\ol{\fp}{}}
\nc{\fpr}{{\fp^{rat}}{}}
\nc{\fqr}{{\fq^{rat}}{}}

\nc{\hattimes}{\wh\otimes}

\nc{\bh}{{\bar{h}}}
\nc{\bOmega}{{\overline{\Omega(\check \fn)}}}

\nc{\seq}[1]{\stackrel{#1}{\sim}}

\nc{\cT}{{\check{T}}}
\nc{\cG}{{\check{G}}}
\nc{\cM}{{\check{M}}}
\nc{\cB}{{\check{B}}}

\nc{\ct}{{\check{\mathfrak t}}}
\nc{\cg}{{\check{\fg}}}
\nc{\cb}{{\check{\fb}}}
\nc{\cn}{{\check{\fn}}}

\nc{\cLambda}{{\check\Lambda}}

\nc{\cla}{{\check\kappa_x}}
\nc{\cmu}{{\check\mu}}
\nc{\cnu}{{\check\nu}}
\nc{\ceta}{{\check\eta}}

\nc{\DefbE}{{\on{Def}_{\cB}(E_\cT)}}

\nc{\imathb}{{\ol{\imath}}}
\nc{\rlr}{\overset{\longrightarrow}{\underset{\longrightarrow}\longleftarrow}}

\nc{\KG}{K\backslash G}
\nc{\comult}{{co\text{-}mult}}
\nc{\counit}{{co\text{-}unit}}
\nc{\uHom}{{\underline{\Maps}}}
\nc{\dgSch}{\on{Sch}}
\nc{\Sch}{\on{Sch}}
\nc{\affdgSch}{\on{Sch}^{\on{aff}}}
\nc{\affSch}{{}^{\on{cl}}\!\on{Sch^{\on{aff}}}}
\nc{\Groupoids}{\on{Grpd}}
\nc{\inftygroup}{\on{Spc}}
\nc{\inftyCat}{\infty\on{-Cat}}
\nc{\StinftyCat}{\inftyCat^{\on{St}}}
\nc{\MoninftyCat}{\infty\on{-Cat}^{\on{Mon}}}
\nc{\SymMoninftyCat}{\infty\on{-Cat}^{\on{SymMon}}}
\nc{\SymMonStinftyCat}{\on{DGCat}^{\on{SymMon}}}
\nc{\MonStinftyCat}{\on{DGCat}^{\on{Mon}}}
\nc{\inftystack}{\on{Stk}}
\nc{\inftystackalg}{Stk^{1\text{-}alg}}
\nc{\inftyprestack}{\on{PreStk}}
\nc{\inftydgnearstack}{\on{NearStk}}
\nc{\inftydgstack}{\on{Stk}}
\nc{\inftydgstackalg}{DGStk^{1\text{-}alg}}
\nc{\inftydgprestack}{\on{PreStk}}
\nc{\dgindSch}{\on{indSch}}
\nc{\indSch}{{}^{\on{cl}}\!\on{indSch}}
\nc{\infSch}{\on{infSch}}
\nc{\dr}{{\on{dR}}}

\nc{\mmod}{{\on{-}\!{\mathbf{mod}}}}

\nc{\starr}{\text{\dh}}
\nc{\Spectra}{\on{Spectra}}
\nc{\Crys}{\on{Crys}}
\nc{\oblv}{{\mathbf{oblv}}}
\nc{\ind}{{\mathbf{ind}}}
\nc{\CMaps}{{\mathcal Maps}}
\nc{\Maps}{\on{Maps}}
\nc{\bMaps}{\mathbf{Maps}}
\nc{\BMaps}{\ul{\on{Maps}}}
\nc{\Grid}{\on{Grid}}
\nc{\hGrid}{\on{Grid}^{\geq\,\on{dgnl}}}
\nc{\Diag}{\on{Diag}}
\nc{\bDelta}{\mathbf{\Delta}}
\nc{\tCateg}{(\infty\on{-2)-Cat}}
\nc{\ul}{\underline}
\nc{\Seg}{\on{Seq}}
\nc{\biSeg}{\on{bi-Seq}}
\nc{\triSeg}{\on{tri-Seq}}
\nc{\quadSeg}{\on{quad-Seq}}
\nc{\nSeg}{\on{n-Seq}}
\nc{\Segm}{\on{Seg}^{\on{mkd}}}
\nc{\fLm}{\fL^{\on{mkd}}}
\nc{\inftyCatm}{\inftyCat^{\on{mkd}}}
\nc{\Blocks}{\mathbf{Blocks}}
\nc{\Snakes}{\mathbf{Snakes}}
\nc{\bifL}{\on{bi-}\!\fL}
\nc{\Sets}{\on{Sets}}
\nc{\Ran}{\on{Ran}}
\nc{\Vect}{\on{Vect}}
\nc{\Shv}{\on{Shv}}
\nc{\unn}{\mathbf{union}}

\begin{document}


\vskip1cm

\title[The Atiyah-Bott formula for the cohomology of $\Bun_G$]
{The Atiyah-Bott formula for the cohomology \\ of the moduli space of bundles on a curve}

\author{D. Gaitsgory}



\date{\today}

\maketitle

\tableofcontents

\section*{Introduction}

\ssec{What is this text?}  \label{ss:what}

The present paper is a companion of \cite{Main Text}.  
\footnote{The contents of this paper are joint work with J.~Lurie, who chose not to sign it as an author. It is made public with his
consent, but the responsibility for any deficiency or undesired outcome of this paper lies entirely with the author.}
The goal of {\it loc.cit.} is to prove the Tamagawa number formula
for function fields, namely that the volume of the automorphic space of a semi-simple simply connected group, with respect 
to the Tamagawa measure, equals $1$. The present paper gives a different approach to some of the steps in the proof.

\sssec{}

Let us recall the strategy of the proof in \cite{Main Text}. We start with a (smooth, complete and geometrically connected)
curve $X$ over a finite field $\BF_q$ with the field of fractions $K$, and let $G_K$ be a semi-simple simply connected group
over $K$. We are interested in the quantity
$$\mu_{\on{Tam}}(G_K(\BA)/G_K(K)),$$
where $\BA$ is the ring of ad\`eles of $K$, and $\mu_{\on{Tam}}$ is the Tamagawa measure on $G_K(\BA)$.

\medskip

First, we choose an integral model $G$ of $G_K$ over $X$. This is a smooth group-scheme over $X$ with connected fibers.
Let $\Bun_G$ denote the moduli space of $G$-bundles on $X$; this is an algebraic stack locally of finite type and smooth over 
$\BF_q$.
We interpret the desired equality
$$\mu_{\on{Tam}}(G_K(\BA)/G_K(K))=1$$
as the following 
\begin{equation} \label{e:Tam}
|\Bun_G(\BF_q)|=q^{\dim(\Bun_G)}
\cdot \underset{x}\Pi\, \frac{|k_x|^{\dim(G_x)}}{|G(k_x)|}.
\end{equation}

In the above formula, $|\Bun_G(\BF_q)|$ is the (infinite) sum over the set of isomorphism classes of $G$-bundles on $X$:
$$\underset{\CP}\Sigma\, \frac{1}{|\on{Aut}(\CP)|}.$$

In the right-hand side, the product is over the set of closed points of $X$; for a point $x$ we denote by $k_x$ its residue field
and by $G_x$ the fiber of $G$ at this point. 

\medskip

Note that both sides in \eqref{e:Tam} are infinite expressions, so part of the statement is that both sides are well-defined 
(i.e., are convergent). 

\sssec{}

The next step is the Grothendieck-Lefschetz trace formula that says that the left-hand side in \eqref{e:Tam} equals 
\begin{equation} \label{e:trace comp}
\on{Tr}\left(\on{Frob},\on{H}^*_c(\Bun_G)\right).
\end{equation}

The Grothendieck-Lefschetz trace formula for \emph{quasi-compact} algebraic stacks follows easily from the case of 
varieties. However, $\Bun_G$ is not quasi-compact, so a separate argument is needed to justify it. In particular, one
needs to show that the expression in \eqref{e:trace comp} makes sense as a complex number. This is non-obvious
because $\on{H}^*_c(\Bun_G)$ lives in infinitely many cohomological degrees.

\sssec{}

Assuming the Grothendieck-Lefschetz trace formula for $\Bun_G$, and using Verdier duality, we rewrite the desired equality as
\begin{equation} \label{e:num prod prev}
\on{Tr}(\on{Frob}^{-1}, \on{H}^*(\Bun_G))=\underset{x}\Pi\, \frac{|k_x|^{\dim(G_x)}}{|G(k_x)|},
\end{equation}
where the factor of $q^{\dim(\Bun_G)}$ got subsumed in the dualizing sheaf on $\Bun_G$.

\medskip

Along with the previous steps mentioned above, the equality \eqref{e:num prod prev} is established in \cite{Main Text}. Our goal in this
paper is to give a different proof of \eqref{e:num prod prev}. In fact, there will be two points of difference, one big and the other small.

\medskip

The key step in both proofs is the \emph{cohomological product formula} for $\on{H}^*(\Bun_G)$, see \eqref{e:product formula prev}.  

\sssec{}

The small point of difference is how we derive the numerical identity \eqref{e:num prod prev} from the formula for $\on{H}^*(\Bun_G)$. 

\medskip

In \cite{Main Text} this is done by considering the filtration by powers of the maximal ideal on the algebra of cochains $\on{C}^*(\Bun_G)$.
Since we live in the world of higher algebra, such a filtration is not something naive, but rather a certain canonical construction in homotopy theory. 

\medskip

In this paper we derive \eqref{e:num prod prev} by a more pedestrian method: we show that from the \emph{cohomological product formula} 
\eqref{e:product formula prev} 
one can deduce a more explicit formula for $\on{H}^*(\Bun_G)$, called the Atiyah-Bott formula, see \eqref{e:AB prev}. Now, given the 
Atiyah-Bott formula, the derivation of \eqref{e:num prod prev} is a standard manipulation with L-functions, see \secref{s:num}. 

\medskip

So, in a sense, this part of the argument is more elementary in the present paper than in \cite{Main Text}. 

\sssec{} 

The big point of difference is how we prove the cohomological product formula \eqref{e:product formula prev}. This formula
is a local-to-global result, and in both \cite{Main Text} and the present paper we deduce it from another local-to-global result,
namely, the non-abelian Poincar\'e duality, \eqref{e:non-ab Poinc prev}  (we do not reprove the non-abelian Poincar\'e duality
here). 

\medskip

In both cases, the derivation
$$\text{Non-abelian Poincar\'e duality} \,\, \Rightarrow\,\, \text{Cohomological product formula}$$
morally amounts to performing Verdier duality on the \emph{Ran space} of $X$. 

\medskip

Now, roughly speaking, the above derivation in \cite{Main Text} does not explicitly use the words ``Verdier duality". Rather, it amounts
to a giant commutative diagram of complexes of $\BZ_\ell$-modules.  

\medskip

One can regard the bulk of the work in the present paper as breaking the giant diagram from \cite{Main Text} into a series of several 
conceptually defined steps. Our primary tool is the extensive use of \emph{sheaves}. I.e., most of the isomorphisms that we need to go through 
are isomorphisms between sheaves on some spaces (rather than complexes of $\BZ_\ell$-modules). See \secref{ss:summary}, where we
list the series of isomorphisms that eventually leads to the cohomological product formula. 

\medskip

However, there is a price that one needs to pay: for some of the most non-trivial steps, sheaves on varieties are not sufficient.
And neither are sheaves on \emph{prestacks} (the latter are geometric objects generalizing varieties, see \secref{ss:prestacks}).
What we will need is the notion of $\ell$-adic sheaf on a (contravarant) functor from the category of schemes to that of $\infty$-categories;
we will call the latter gadgets \emph{lax prestacks}. 

\ssec{The Ran space and the cohomological product formula}

Thus, the main thrust of this paper is the proof of the cohomological product formula. Here we shall explain what it says. To do
so we shall need the language of sheaves on prestacks, see \secref{s:prestacks}. 

\medskip

We change the notations from \secref{ss:what}, and we take $X$ to be a smooth and complete curve over an algebraically
closed ground field $k$. 

\sssec{}

Let $X$ be our curve. Both in \cite{Main Text} and in the present paper, one of our main tools is the Ran space of $X$. This is 
a prestack that classifies \emph{finite non-empty subsets} of points of $X$, see \secref{ss:Ran} for the precise definition.

\medskip

Given the group-scheme $G$ over $X$, we construct the following sheaf, denoted $\CB$, on $\Ran$. For a collection of (distinct)
points $x_1,...,x_n$, the !-fiber of $\CB$ at the point $\{x_1,...,x_n\}\in \Ran$ is 
\begin{equation} \label{e:shape of B}
\underset{i=1,...,n}\bigotimes \on{C}^*(BG_{x_i}),
\end{equation}
where $G_{x_i}$ is the fiber of $G$ at $x_i$, and $BG_{x_i}$ is its classifying stack. Here and in the sequel, $\on{C}^*(-)$ denotes the
algebra of cochains on a given space. 

\medskip

For every $x_1,...,x_n$ as above, we have a canonically defined restriction map
$$\Bun_G\to BG_{x_1}\times...\times BG_{x_n},$$
and the corresponding pullback map on cohomology.

\medskip

Integrating over the Ran space, we obtain a map
\begin{equation} \label{e:product formula prev}
\on{C}^*_c(\Ran,\CB)\to \on{C}^*(\Bun_G)
\end{equation}

The cohomological product formula says that the map \eqref{e:AB prev} is an isomorphism (provided that the generic fiber
of $G$ is semi-simple and simply-connected). 

\sssec{}  \label{sss:Euler}

The reader might wonder why we call the isomorphism \eqref{e:product formula prev} the ``cohomological product formula". 
The reason is that, given the expression for the !-fibers of $\CB$ in \eqref{e:shape of B}, one can think of $\on{C}^*_c(\Ran,\CB)$
as a kind of Euler product
\begin{equation} \label{e:Euler}
\underset{x\in X}\bigotimes\, \on{C}^*(BG_x),
\end{equation}
\emph{understood informally}.

\medskip

Note that when we compare the geometric picture to the classical one of automorphic functions, the latter is one categorical
level lower: so a vector space in geometry corresponds to a number in the classical theory. So, the product \eqref{e:Euler}
is a geometric analog of the special value of the corresponding L-function (rather than the restricted tensor product, akin
to one in the definition of automorphic representations).  

\medskip

A possible of justification for thinking of $\on{C}^*_c(\Ran,\CB)$ as the ``product" is the following. In the course of the proof of the 
numerical product formula \eqref{e:num prod prev} in \secref{s:num}, we will show that when working over the finite field, the trace 
of $\on{Frob}^{-1}$ on $\on{H}^*_c(\Ran,\CB)$ will actually turn out to be equal to the product of
$$\underset{x}\Pi\, \on{Tr}(\on{Frob}_x^{-1},\on{H}^*(BG_x)).$$

We should emphasize, however, that the proof of this identity is specific to our situation (we prove it by first rewriting 
$\on{C}_c^*(\Ran,\CB)$ via the Atiyah-Bott formula, see below), i.e., we use some additional structure on our $\CB$. 
It would be interesting to find a general statement applicable to $\on{C}^*_c(\Ran,\CF)$ for a general (\emph{factorizable}) 
sheaf $\CF$ on $\Ran$.

\sssec{}

The formula \eqref{e:product formula prev} embodies a local-to-global principle: the cohomology of $\Bun_G$ is 
assembled in a canonical way from the cohomologies of the classifying spaces of the fibers of $G$.   However,
it is not best adapted to explicit computations (recall that we need to compute the trace of Frobenius on $\on{H}^*(\Bun_G)$). 

\medskip

In \secref{s:AB} we will derive from the formula for $\on{C}^*(\Bun_G)$, given by \eqref{e:product formula prev}, another
expression for $\on{C}^*(\Bun_G)$, known as the Atiyah-Bott formula. This is literally the formula from \cite{AB} when
the ground field is that of complex numbers and $G$ is split. 

\medskip

Namely, let us assume that $G$ is such that it is reductive over some open $X'\subset X$, and has unipotent fibers at points
of $X-X'$. To the datum of $G$ one attaches a lisse sheaf $M$ over $X'$, whose !-fiber at $x\in X'$ has a basis consisting
of a set of homogeneous generators of $\on{H}^*(BG_x)$, i.e., the \emph{exponents} of $G_x$.

\medskip

In \secref{s:AB} we will show that the cohomological product formula implies the following (non-canonical) isomorphism:

\begin{equation} \label{e:AB prev}
\on{C}^*(\Bun_G) \simeq \Sym\left(\on{C}^*(X',M)\right)
\end{equation}

It is using this isomorphism that we will be able to calculate the trace of $\on{Frob}^{-1}$ on $\on{H}^*(\Bun_G)$. 

\ssec{Non-abelian Poincar\'e duality}

We shall now try to outline the proof of the cohomological product formula, i.e., of the fact that the map
\eqref{e:product formula prev} is an isomorphism. This is what the bulk of this paper is about. 

\sssec{}

Let $\Gr_{\Ran}$ be the Ran version of the affine Grassmannian of $G$. I.e., this is the prestack that classifies the data of
$$(\CP,I,\alpha),$$
where $\CP$ is a $G$-bundle on $X$, $I\in \Ran$ and $\alpha$ is a trivialization of $\CP$ on $X-I$.

\medskip

We have an evident forgetful map 
$$\Gr_{\Ran}\to \Bun_G,\quad (\CP,I,\alpha)\mapsto \CP.$$
The main geometric ingredient in the proof of the cohomological product formula
is the fact that this map induces an isomorphism on homology. This is proved in \cite{Main Text} under the name \emph{non-abelian Poincar\'e duality}.
We do not reprove this fact here. 

\medskip

So, we need to explain how the above isomorphism on homology
implies the cohomological product formula. Note that the map $\Gr_{\Ran}\to \Bun_G$ used above,
and the maps $$\Bun_G\to BG_{x_1}\times...\times BG_{x_n},$$
used in \eqref{e:product formula prev} are of (seemingly) different nature. 

\sssec{}

Consider the forgetful map 
$$\Gr_{\Ran}\to \Ran,\quad (\CP,I,\alpha)\mapsto I.$$

Let $\CA$ be the sheaf on $\Ran$ equal to the pushforward of the dualizing sheaf under the above map. Explicitly, the !-fiber of $\CA$
at a point $\{x_1,...,x_n\}\in \Ran$ is
$$\underset{i=1,...,n}\bigotimes \on{C}_*(\Gr_{x_i}),$$
where $\Gr_{x_i}$ denotes the affine Grassmannian of $G$ at $x_i$. 

\medskip

We can reformulate the non-abelian Poincar\'e duality as the fact that the map
\begin{equation}\label{e:non-ab Poinc prev}
\on{C}_c^*(\Ran,\CA)\to \on{C}_*(\Bun_G)
\end{equation}
is an isomorphism.  

\sssec{}  \label{sss:estim}

Comparing \eqref{e:product formula prev} with \eqref{e:non-ab Poinc prev}, we obtain that we need to construct an isomorphism
\begin{equation} \label{e:global duality prev}
\on{C}_c^*(\Ran,\CB)\to \on{C}_c^*(\Ran,\CA)^\vee,
\end{equation}
that makes the diagram
$$
\CD
\on{C}_c^*(\Ran,\CB)   @>>>   \on{C}_c^*(\Ran,\CA)^\vee  \\
@V{\text{\eqref{e:product formula prev}}}VV     @A{\sim}A{\text{\eqref{e:non-ab Poinc prev}}}A   \\
\on{C}^*(\Bun_G)   @>{\sim}>> \on{C}_*(\Bun_G)^\vee 
\endCD
$$
commute. We call the desired (but eventually, actual) isomorphism \eqref{e:global duality prev} the \emph{global duality} statement. 

\ssec{Verdier duality and the procedure of taking the units out}  \label{ss:take out prev}

\sssec{}

In order to construct \eqref{e:global duality prev}, it is very tempting to suppose that the sheaf $\CB$
is the Verdier dual of $\CA$, thereby inducing a duality on their global cohomologies.

\medskip

The first question is: what is Verdier duality on a geometric object such as $\Ran$? We do define what Verdier duality means in such a context
in Part II of the paper and we prove a theorem to the effect that for a sheaf on $\Ran$ \emph{satisfying certain cohomological estimates},
the (compactly supported) cohomology of its Verdier dual will map isomorphically to the dual of its (compactly supported) cohomology.

\medskip

The problem is that our $\CA$ does \emph{not} satisfy the above cohomological estimates. In fact, it violates them so badly that its
Verdier dual equals $0$. \footnote{In fact, the Verdier dual of $\CA$ is $0$ even for $G=\{1\}$, in which case $\CA$ equals the dualizing
sheaf on $\Ran$. Moreover, the ``constants" that come from the case $G=\{1\}$ is what causes the problem for any $G$, and the way 
to get around it consists of taking these constants out, see \secref{sss:out prev} below.} 
So, the desired duality on global cohomology happens for a reason (seemingly) different from Verdier duality.

\medskip

But it turns out that trying to push through the idea of Verdier duality is not a lost case: one needs to modify $\CA$ and $\CB$. The
corresponding modification is the procedure of \emph{taking the units out}. 

\sssec{}

We shall explain how taking the units out allows to implement Verdier duality by analogy with associative algebras. 
\footnote{In fact, this is more than just an analogy, as sheaves on the Ran space and associative algebras are part of a common paradigm.
Namely, in the context of topology with $X$ being the real line $\BR$, sheaves on its Ran space, equipped with a factorization structure, 
are equivalent to associative algebras.}.

\medskip

Recall that on the category of \emph{non-unital} associative algebras (say, over a ground field $\Lambda$) there is a canonically defined contravariant 
self-functor, the Koszul duality. Namely, it sends an algebra $A$ to 
$$\on{KD}(A):=\on{ker}\left(\CHom_A(\Lambda,\Lambda)\to \CHom_\Lambda(\Lambda,\Lambda)\right),$$ where $\Lambda$ is equipped with the \emph{trivial}
(i.e., zero) action of $A$, and where $\CHom_\Lambda(\Lambda,\Lambda)$ is (obviously) isomorphic to $\Lambda$. 

\medskip

The functor $\on{KD}$ is \emph{not} an equivalence. In fact, it sends many objects to $0$. For example, if $A$ was obtained from a \emph{unital}
associative algebra $A_{\on{untl}}$ by treating it as a non-unital algebra (i.e., by forgetting that it had a unit), then $\on{KD}(A)=0$. 

\medskip

Now, suppose that we start with a unital augmented algebra $A_{\on{untl,aug}}$ and take $A_{\on{red}}$ to be its augmentation ideal. Then it is a sensible procedure
to consider $B_{\on{red}}:=\on{KD}(A_{\on{red}})$ and then create a unital augmented algebra $B_{\on{untl,aug}}$ by adjoining a unit to $B_{\on{red}}$,
while at the background we also have $A$ and $B$, obtained from $A_{\on{untl,aug}}$ and $B_{\on{untl,aug}}$, respectively, by forgetting the unit
and the augmentation.  

\medskip

This will turn out to be how our sheaves $\CA$ and $\CB$ are related to each other.

\sssec{}  

We should think of sheaves on the Ran space as analogs of non-unital associative algebras (see footnote above). 
In Part I of this paper we introduce geometric gadgets, denoted $\Ran_{\on{untl}}$ and $\Ran_{\on{untl,aug}}$, so that sheaves on 
them are analogs of unital and unital augmented associative algebras, respectively. 

\medskip

It is here that \emph{lax prestacks} come into the picture, because such are $\Ran_{\on{untl}}$ and $\Ran_{\on{untl,aug}}$. 

\medskip

The categories of sheaves on $\Ran$,  $\Ran_{\on{untl}}$ and $\Ran_{\on{untl,aug}}$ are related to each other by several functors
that mirror the corresponding functors for associative algebras: 
$$\on{OblvUnit}:\Shv(\Ran_{\on{untl}})\to \Shv(\Ran), \quad
\on{OblvAug}:\Shv(\Ran_{\on{untl,aug}})\to \Shv(\Ran_{\on{untl}}),$$
$$\on{AddUnit}:\Shv(\Ran) \to \Shv(\Ran_{\on{untl}}),$$
and a pair of adjoint functors
$$\on{AddUnit}_{\on{aug}}:\Shv(\Ran) \rightleftarrows \Shv(\Ran_{\on{untl,aug}}):\on{TakeOut},$$
while
$$\on{OblvAug}\circ \on{AddUnit}_{\on{aug}}\simeq \on{AddUnit}.$$

We will show that the sheaves $\CA$ and $\CB$ are obtained as
$$\CA\simeq \on{OblvUnit}\circ \on{OblvAug}(\CA_{\on{untl,aug}}) \text{ and } 
\CB\simeq \on{OblvUnit}\circ \on{OblvAug}(\CB_{\on{untl,aug}}),$$
respectively for canonically defined sheaves $\CA_{\on{untl,aug}}$ and $\CB_{\on{untl,aug}}$ on $\Ran_{\on{untl,aug}}$.

\medskip

We then set
$$\CA_{\on{red}}:=\on{TakeOut}(\CA_{\on{untl,aug}})  \text{ and } \CB_{\on{red}}:=\on{TakeOut}(\CB_{\on{untl,aug}}).$$

\sssec{}   \label{sss:out prev}

In our case, the !-fibers of $\CA_{\on{red}}$ and $\CB_{\on{red}}$ at a point $\{x_1,...,x_n\}\in \Ran$ are
$$\underset{i=1,...,n}\bigotimes \on{C}^{\on{red}}_*(\Gr_{x_i}) \text{ and } \underset{i=1,...,n}\bigotimes \on{C}^*_{\on{red}}(BG_{x_i}),$$
respectively.   So the effect of replacing $\CA$ be $\CA_{\on{red}}$ and $\CB$ be $\CB_{\on{red}}$ consists at the level of !-fibers
of replacing 
$$\on{C}_*(\Gr_{x_i})\rightsquigarrow \on{C}^{\on{red}}_*(\Gr_{x_i})  \text{ and }
\on{C}^*(BG_{x_i}) \rightsquigarrow \on{C}^*_{\on{red}}(BG_{x_i})$$
\emph{in each factor}.  

\sssec{}

But let us remember that we are interested in the (compactly supported) cohomology of $\CA$ and $\CB$, respectively, on $\Ran$. 
How is  (compactly supported) cohomology affected by the passage
\begin{equation} \label{e:passage}
\CA\rightsquigarrow \CA_{\on{red}} \text{ and } \CB\rightsquigarrow \CB_{\on{red}}?
\end{equation}

We will show that for any sheaf $\CF$ on the Ran space we have a canonical isomorphism
$$\on{C}^*_c(\Ran,\CF) \simeq \on{C}^*_c\left(\Ran,\on{OblvUnit}\circ \on{OblvAug} \circ \on{AddUnit}_{\on{aug}}(\CF)\right).$$

So, the replacement \eqref{e:passage} leaves the (compactly supported) cohomology on $\Ran$ unchanged (up to a direct summand 
equal to $\BQ_\ell$). 

\medskip

This is parallel to the situation in associative algebras, in which for a unital augmented algebra $A_{\on{untl,aug}}$, the Hochschild homology of $A$
(obtained from $A_{\on{untl,aug}}$ by forgetting the unit) and $A_{\on{red}}$ (obtained from $A_{\on{untl,aug}}$ by taking the augmentation
ideal) are isomorphic (up to a copy of the ground field). 

\sssec{}

Thus, we modify our attempt to deduce the duality between $\on{C}_c^*(\Ran,\CA)$ and $\on{C}_c^*(\Ran,\CB)$ from Verdier duality
on the Ran space as follows:

\medskip

We will show that the sheaves $\CA_{\on{red}}$ and $\CB_{\on{red}}$ on $\Ran$ are related by Verdier duality
\begin{equation} \label{e:local duality prev}
\CB_{\on{red}}\simeq \BD_{\Ran}(\CA_{\on{red}})
\end{equation}
and that the resulting map 
$$\on{C}_c^*(\Ran,\CB_{\on{red}})\simeq  \on{C}_c^*(\Ran,\BD_{\Ran}(\CA_{\on{red}}))\to \on{C}_c^*(\Ran,\CA_{\on{red}})^\vee$$
is an isomorphism, thereby producing \eqref{e:global duality prev}.

\medskip

The fact that 
$$\on{C}_c^*(\Ran,\BD_{\Ran}(\CA_{\on{red}}))\to \on{C}_c^*(\Ran,\CA_{\on{red}})^\vee$$
is an isomorphism follows from the fact that the sheaf $\CA_{\on{red}}$ (unlike its precursor $\CA$) \emph{does} satisfies
the cohomological estimate mentioned in \secref{sss:estim}. It is actually here that the assumption that $G$ be 
\emph{semi-simple simply connected} is used. Indeed, the isomorphism \eqref{e:local duality prev} will take place for any reductive $G$,
but the cohomological estimate would fail unless $G$ is semi-simple simply connected. 

\begin{rem}
The above phenomenon of the local duality implying the global one also has an analog for associative algebras: for a non-unital
associative algebra that lives in cohomological degrees $\leq -1$, the Hochschild homology of its Koszul dual maps isomorphically
to the dual of its Hochschild homology.
\end{rem} 

\ssec{Local duality}

\sssec{}

Above we have explained how to reduce the cohomological product formula to the isomorphism \eqref{e:local duality prev}. 
Note, however, that \eqref{e:local duality prev} should not be just \emph{some} isomorphism:
we need it to make the diagram
$$
\CD
\on{C}_c^*(\Ran,\CB_{\on{red}})     @>{\text{\eqref{e:local duality prev}}}>>  \on{C}_c^*(\Ran,\BD_{\Ran}(\CA_{\on{red}}))    \\
@V{\sim}VV   @VV{\sim}V   \\
\on{C}_c^*(\Ran,\CB)     & & \on{C}_c^*(\Ran,\CA_{\on{red}})^\vee   \\
& &   @VV{\sim}V   \\
@V{\text{\eqref{e:product formula prev}}}VV    \on{C}_c^*(\Ran,\CA)^\vee   \\
& &  @A{\sim}A{\text{\eqref{e:non-ab Poinc prev}}}A   \\
\on{C}^*(\Bun_G)   @>{\sim}>> \on{C}_*(\Bun_G)^\vee 
\endCD
$$
commute.

\medskip

In Part III of this paper we explain the general mechanism of what data on a pair of sheaves $\CA$ and $\CB$ on $\Ran$
(or, rather, their unital augmented enhancements  $\CA_{\on{untl,aug}}$ and $\CB_{\on{untl,aug}}$) is needed in order to define a 
map in one direction 
\begin{equation} \label{e:local duality map}
\CB_{\on{red}}\to \BD_{\Ran}(\CA_{\on{red}}).
\end{equation}

Not surprisingly, it turns out that in our case this data amounts to a geometric input, which is a local analog of one used in defining the 
maps \eqref{e:product formula prev} and \eqref{e:non-ab Poinc prev}. 
 
\sssec{}

Having produced the map in \eqref{e:local duality map}, we now have to show that it is an isomorphism. We call this the 
\emph{local duality} statement. By definition, it says that a certain map of sheaves on $\Ran$ is an isomorphism.

\medskip

Now, $\Ran$ is a huge space, but it has one piece that it easy to understand: we have a canonically defined map
$$X\to \Ran.$$

We show that in order to prove that the map \eqref{e:local duality map} is an isomorphism, it is sufficient to show that 
it induces an isomorphism between the respective restrictions of the two sides to $X$.  This reduction step is achieved
via the \emph{factorization} structure on the two sides. The relevant notion of factorization is developed in Part IV 
of the paper. 

\medskip

We then further explore the locality properties of the map \eqref{e:local duality map} by showing that the question of
this map being an isomorphism is enough to resolve in the case when $G$ is a constant group scheme. We also show
that we are free to replace the initial curve by any other curve, and it is sufficient to show that the map in question induces 
an isomorphism at the level of !-fibers at some/any point $\{x\}\in \Ran$.

\sssec{}

Thus, we are reduced to showing that the resulting map
\begin{equation}  \label{e:pointwise duality map}
(\CB_{\on{red}})_{\{x\}}\to  (\BD_{\Ran}(\CA_{\on{red}}))_{\{x\}}
\end{equation} 
is an isomorphism. We call this the \emph{pointwise duality} statement. 

\medskip

We give two (ideologically similar, but technically different) proofs of this assertion, in Sects. \ref{s:dagger} and \ref{s:P1}, respectively.  
In both proofs we regard the map \eqref{e:pointwise duality map} as a local analog of the map \eqref{e:non-ab Poinc prev}.

\medskip

Although the question of the map \eqref{e:pointwise duality map} being an isomorphism is a local one (it makes sense for non-complete
curves), the method we use to prove it is global: we reduce it back to the fact that the map \eqref{e:non-ab Poinc prev}
is an isomorphism for a complete curve. 

\ssec{Summary}   \label{ss:summary}

Let us summarize the steps involved in the derivation of the cohomological product formula (i.e., the isomorphism \eqref{e:product formula prev})
from non-abelian Poincar\'e duality (i.e., the isomorphism \eqref{e:non-ab Poinc prev}). 

\sssec{}

Show that $\CA$ and $\CB$ can be equipped with a unital augmented structure; consider the corresponding reduced versions
$\CA_{\on{red}}$ and $\CB_{\on{red}}$; show that the reduced versions have the same (compactly supported) cohomology on $\Ran$.

\medskip

The preparatory material for these constructions is Part I of the paper, and the construction itself is carried out in \secref{s:local duality}. 

\sssec{}

Show that $\CA_{\on{red}}$ is well-behaved for Verdier duality on the Ran space, i.e., that the map
$$\on{C}_c^*(\Ran,\BD_{\Ran}(\CA_{\on{red}}))  \to \on{C}_c^*(\Ran,\CA_{\on{red}})^\vee$$
is an isomorphism.

\medskip

The preparatory material for this is Part II of the paper.

\sssec{}

Construct a map
$$\CB_{\on{red}}\to \BD_{\Ran}(\CA_{\on{red}}).$$

\medskip

The preparatory material for the construction is Part III of the paper, and the construction itself is carried out in \secref{s:local duality}. 

\sssec{}

Show that is enough to show that the above map $\CB_{\on{red}}\to \BD_{\Ran}(\CA_{\on{red}})$ becomes
an isomorphism after restricting to $X\subset \Ran$.

\medskip

The preparatory material for this is Part IV of the paper.

\sssec{}

Reduce to the pointwise assertion that the map
$$(\CB_{\on{red}})_{\{x\}}\to  (\BD_{\Ran}(\CA_{\on{red}}))_{\{x\}}$$
is an isomorphism (for the split form of $G$, and some/any $x\in X$).

\medskip

This is done in \secref{s:pointwise}. 

\sssec{}

Prove that the map $(\CB_{\on{red}})_{\{x\}}\to  (\BD_{\Ran}(\CA_{\on{red}}))_{\{x\}}$ is an isomorphism. 

\medskip

This is done in Sects. \ref{s:dagger} and \ref{s:P1}. 

\ssec{Contents}

\sssec{}

After becoming familiar with the language of sheaves on prestacks and the Ran space (Sects. \ref{s:prestacks} and \ref{ss:Ran}),
the reader may want to start with \secref{s:global duality}, where the cohomological product formula is stated. The other sections 
in this paper essentially constitute the proof of the cohomological product formula.

\medskip

After \secref{s:global duality}, the reader may try to proceed to \secref{s:local duality}, where the most essential part of the proof is
contained. However, reading \secref{s:local duality} without any idea of what is happening in Parts I, II and III of the paper will 
probably be impossible. Other than reading this paper in order, an alternative reasonable strategy would be to read \secref{s:local duality}
stage-by-stage and refer back to the relevant parts in the previous sections (we tried to point to the part of the background
material needed for each step). 

\medskip

We shall now describe the contents of this paper section-by-section. 

\sssec{}

Part 0 of the paper is devoted to the discussion of sheaves on prestacks and lax prestacks. 

\medskip

In \secref{s:prestacks} we specify what we mean by a \emph{sheaf theory} on schemes, explain how 
each such automatically extends to give rise to the notion of sheaf on a prestack. We then discuss
the technically crucial notion of \emph{pseudo-proper} map: these are maps for which the functor of
pushforward of sheaves is well-behaved. The material in this section is necessary for the formulation
of the main results of this paper.

\medskip

In \secref{s:lax} we introduce \emph{lax prestacks} and sheaves on them. Examples of lax prestacks
are $\Ran_{\on{untl}}$ and $\Ran_{\on{untl,aug}}$, and they are needed in order to perform the
manipulation of \emph{taking the unit out}, discussed in \secref{ss:take out prev} above. The reader
needs to be familiar with lax prestacks if he wants to go in any depth into the proof of the
cohomological product formula. 

\medskip

In \secref{s:uhc} we explore the general question of when is a map of lax prestacks $f:\CY_1\to \CY_2$
is such that pullback with respect to it defines a fully faithful functor. This section may be skipped on the 
first pass and returned to when necessary. (We should say, however, that the results in this section 
are (i) generally useful and (ii) not difficult to formulate and prove, so it may be fun to read in any case.) 

\sssec{}

Part I of the paper is devoted to the discussion of various versions of the Ran space.

\medskip

In \secref{s:Ran} we introduce the usual Ran space, as well as its unital version. We construct the functor
of \emph{adding the unit} from sheaves on $\Ran$ to sheaves on $\Ran_{\on{untl}}$. 
We show that the forgetful functor from 
sheaves on $\Ran_{\on{untl}}$ to sheaves on $\Ran$ preserves compactly supported cohomology. 

\medskip

In \secref{s:aug} we introduce yet another version of the Ran space, $\Ran_{\on{untl,aug}}$. We show that
the procedure of adding units, viewed now as a functor from sheaves on $\Ran$ to sheaves on $\Ran_{\on{untl,aug}}$,
is fully faithful, with the right adjoint given by the functor of \emph{taking the units out}.

\medskip

In \secref{s:untl} we will present an alternative (but equivalent) point of view on the unital and unital augmented Ran space. 
The contents of this section will not be used elsewhere in the paper.

\sssec{}

Part II of the paper is devoted to the discussion of Verdier duality on the (usual) Ran space. 

\medskip

In \secref{s:Verdier} we introduce the functor of Verdier duality on the category of sheaves on an arbitrary prestack
(for which the diagonal morphism is pseudo-proper). We then single out a class of prestacks, called \emph{pseudo-schemes with a finitary diagonal}, 
for which the functor of Verdier duality is manageable (=can be expressed in terms of Verdier duality on schemes).  

\medskip

In \secref{s:Verdier on Ran} we specialize the discussion of Verdier duality to the case of the Ran space. First, we show
that $\Ran$ is a pseudo-scheme with a finitary diagonal, so we can explicitly describe Verdier duality on $\Ran$. We then formulate a
crucial result, \thmref{t:Verdier on Ran},  that gives sufficient conditions on a sheaf $\CF$ on $\Ran$ for the map
$$\on{C}^*_c(\Ran,\BD_{\Ran}(\CF))\to \on{C}^*_c(\Ran,\CF)^\vee$$
to be an isomorphism. 

\medskip

In \secref{s:proofs Verdier} we prove \thmref{t:Verdier on Ran} and some other results of similar nature. 

\sssec{}

Part III is devoted to the discussion of the interaction of the procedure of inserting the unit (from Part I) with 
Verdier duality on $\Ran$.

\medskip

In \secref{s:pairings aug} we discuss the following question: given a pair of sheaves $\CF_{\on{untl,aug}},\CG_{\on{untl,aug}}$ on 
$\Ran_{\on{untl,aug}}$, what kind of datum on them gives rise to a map
$$\CG_{\on{red}}\to \BD_{\Ran}(\CF_{\on{red}})?$$
In the above formula, $\CF_{\on{red}}$ and $\CG_{\on{red}}$ are obtained from $\CF_{\on{untl,aug}}$ and $\CG_{\on{untl,aug}}$,
respectively, by \emph{taking the units out}.  It turns out that the corresponding datum on $\CF_{\on{untl,aug}},\CG_{\on{untl,aug}}$ 
is a map
$$\CF_{\on{untl,aug}}\boxtimes \CG_{\on{untl,aug}}\to \omega_{\Ran_{\on{untl,aug}}\times \Ran_{\on{untl,aug}}},$$
\emph{defined over a certain open subset} of $\Ran_{\on{untl,aug}}\times \Ran_{\on{untl,aug}}$. Note that this is very different
from the usual Verdier duality on a prestack $\CY$, where the map goes to $(\on{diag}_\CY)_!(\omega_\CY)$. 

\medskip

In \secref{s:expl} we give an explicit construction of the !-fiber of the Verdier dual of a sheaf $\CF$ on the Ran space in terms
of the sheaf $\CF_{\on{untl,aug}}$ on $\Ran_{\on{untl,aug}}$, obtained from $\CF$ by inserting the unit. 

\sssec{}

Part IV is devoted to a structure that one can impose on sheaves on $\Ran$, known as \emph{factorization}. 

\medskip

In \secref{s:factorize} we introduce the notions of commutative factorization algebra and cocommutative factorization coalgebra
in sheaves on $\Ran$. We show that, when a certain cohomological estimate is satisfied, Verdier duality maps 
cocommutative factorization coalgebras to commutative factorization algebras.

\medskip

In \secref{s:fact aug} we study how the procedure of inserting the unit from Part I interacts with factorization. It turns out that
for sheaves on $\Ran_{\on{untl,aug}}$ one can introduce its own notions of commutative factorization algebra and cocommutative 
factorization coalgebra (very different in spirit from those on the usual $\Ran$). Yet, we show that the two notions of
factorization match up under the functor of inserting the unit. 

\sssec{}

Part V is the core of this paper, in which we assemble all the pieces of theory developed hitherto and 
prove the cohomological product formula.

\medskip

In \secref{s:global duality} we state the cohomological product formula. We also recall the statement of non-abelian Poincar\'e duality,
from which we deduce that the cohomological product formula follows from the fact that a certain (specified) pairing between 
$\on{C}^*_c(\Ran,\CA)$ and $\on{C}^*_c(\Ran,\CB)$ defines an isomorphism 
\begin{equation} \label{e:global duality again}
\on{C}^*_c(\Ran,\CB)\to \on{C}^*_c(\Ran,\CB)^\vee.  
\end{equation} 

\medskip

\secref{s:local duality} is where everything happens. We use the material in Parts I and III of the paper to state the 
\emph{local duality} theorem, and we use Part II of the paper to show the the local duality theorem implies the isomorphism
\eqref{e:global duality again}.

\medskip

In \secref{s:pointwise} we reduce the local duality theorem to the pointwise duality theorem, using the material from Part IV
of the paper.

\medskip

Finally, in Sects. \ref{s:dagger} and \ref{s:P1} we give two proof of the poinwtise duality theorem.

\sssec{}
 
In Part VI we apply the cohomological product formula to deduce the numerical product formula \eqref{e:num prod prev}. 

\medskip

In \secref{s:AB} we deduce from the cohomological product formula the Atiyah-Bott formula, which is an 
explicit (however, non-canonical) description of the algebra of cochains on $\Bun_G$.

\medskip

In \secref{s:num} we use the the Atiyah-Bott formula to calculate the trace of $\on{Frob}^{-1}$ on 
$\on{H}^*(\Bun_G)$, thereby proving the numerical product formula \eqref{e:num prod prev}. 

\ssec{Conventions and notation}

\sssec{Algebraic geometry}

Throughout the paper $k$ will be an algebraically closed ground field. 

\medskip

We shall denote by $\Sch$ the category of \emph{separated} $k$-schemes of finite type. In this paper we do \emph{not} use derived algebraic geometry,
so ``schemes" are classical schemes.

\medskip

We use the notation $\on{pt}$ for $\Spec(k)$. 

\sssec{Higher category theory: why?}

Higher category theory is indispensable in this paper. 

\medskip

Let is first point out the place where the use of higher categories is \emph{not} essential: we define prestacks (resp., lax prestacks) to be functors 
from the category of schemes to that of $\infty$-groupoids (resp., $\infty$-categories). This is not necessary: the 
prestacks (resp., lax prestacks) that we shall encounter in this paper take values in ordinary groupoids (resp., ordinary
categories).  \footnote{The latter is not surprising: if one stays within the realm of classical (as opposed to derived) algebraic
geometry, higher prestacks or lax prestacks are far less ubiquitous.}  

\medskip

The place where higher categories are essential is that we need the category $\Shv(S)$ of sheaves on a scheme $S$ to be regarded
as an $\infty$-category. In fact, we need more: we need the assignment that sends a scheme to the category of sheaves on
it to be a functor:
$$\Shv:(\Sch)^{\on{op}}\to \inftyCat,\quad S\mapsto \Shv(S)$$
where $\Sch$ is the (ordinary) category of schemes, perceived as an $\infty$-category, and $\inftyCat$ is 
\emph{the $\infty$-category of $\infty$-categories}. 

\medskip

The reason we need this (as opposed to have sheaves form, say, a triangulated category) is that in order to define sheaves 
on a prestack $\CY$, we will be taking the \emph{limit} of the categories $\Shv(S)$ over the index category of schemes $S$ mapping to 
$\CY$. Now, in order for this limit to be well-behaved we need it to be taken in $\inftyCat$. 

\medskip

We should say that our use of higher category theory is model-independent. I.e., one can perceive them as quasi-categories (i.e., a particular kind
simplicial sets), but any other model (e.g., complete Segal spaces) will do as well. All the statements that we ever make are 
homotopy-invariant. 

\sssec{Higher category theory: glossary}   \label{sss:hom th conv}

For the most part, our use of higher category theory is limited to the material covered in \cite{Lu1}.  Here are some basic notions and
pieces of notation that the reader should know:

\medskip

We let $\inftygroup$ denote the $\infty$-category of $\infty$-groupoid. 

\medskip

An $\infty$-category $\bC$ is an $\infty$-groupoid of all $1$-morphisms in $\bC$ are isomorphisms. For every $\bC$ there exists
a universal $\infty$-groupoid that receives a map from $\bC$; we call it the enveloping groupoid of $\bC$ and denote it by $\bC_{\on{str}}$;
it is obtained from $\bC$ by inverting all $1$-morphisms. We shall say that $\bC$ is \emph{contractible} or that it has a \emph{trivial homotopy type} 
if $\bC_{\on{str}}$ is isomorphic, as an $\infty$-groupoid, to $\{*\}$ (the one-point category). 

\medskip

For an $\infty$-category $\bC$ and objects $\bc_1,\bc_2\in \bC$ we write $\Maps_\bC(\bc_1,\bc_2)$ for the space ($\infty$-groupoid)
of maps from $\bc_1$ to $\bc_2$.

\medskip

For a functor $F:\bC'\to \bC$ and an object $\bc\in \bC$ we denote by $\bC'_{\bc/}$ (resp., $\bC'_{/\bc}$) the corresponding \emph{under-category}
(resp., \emph{over-category}). By definition, this is the category of pairs $(\bc'\in \bC',\bc\to F(\bc'))$ (resp., $(\bc'\in \bC',F(\bc')\to \bc)$).   We let $\bC'_\bC$
to be the \emph{fiber} of $\bC'$ over $\bc$, i.e., the category of pairs $(\bc'\in \bC',\bc\simeq F(\bc')$, where $\simeq$ means a (specified)
isomorphism. 

\medskip

Let $F:\bC\to \bD$ be a functor and let $\alpha:\bc_1\to \bc_2$ be a $1$-morphism in $\bC$. We shall say that $\alpha$ is coCartesian if for 
any $\bc\in \bC$ the map
$$\Maps_\bC(\bc_2,\bc)\to \Maps_\bD(F(\bc_2),F(\bc))\underset{\Maps_\bD(F(\bc_1),F(\bc))}\times \Maps(\bc_1,\bc)$$
is an isomorphism (of $\infty$-groupoids). A $1$-morphism is said to be Cartesian is it is coCartesian for the functor
$F^{\on{op}}:\bC^{\on{op}}\to \bD^{\on{op}}$. A functor $F$ is said to be a \emph{locally coCartesian fibration} (resp., \emph{locally Cartesian fibration})
if for every $1$-morphism $\bd_1\to \bd_2$ in $\bD$ there exists a coCartesian (resp., Cartesian) $1$-morphism $\alpha:\bc_1\to \bc_2$
equipped with a datum of commutativity for the diagram
$$
\CD
F(\bc_1)   @>{F(\alpha)}>> F(\bc_2)  \\
@V{\sim}VV   @VV{\sim}V   \\
\bd_1  @>>>  \bd_2. 
\endCD
$$
 A functor $F$ is said to be a \emph{coCartesian fibration} (resp., \emph{Cartesian fibration}) if it is 
a \emph{locally coCartesian fibration} (resp., \emph{locally Cartesian fibration}) \emph{and} the composition of two 
coCartesian (resp., Cartesian) $1$-morphisms is coCartesian (resp., Cartesian). A functor $F$ is said to be a \emph{coCartesian fibration in groupoids} 
(resp., \emph{Cartesian fibration in groupoids}) if, in addition, its fibers are $\infty$-groupoids. We refer the reader to \cite[Sect. 2.4]{Lu1}
for more details. 

\medskip

If $\bI$ is an index $\infty$-category and $\Phi:\bI\to \bC$ is a functor, where $\bC$ is another $\infty$-category, we can talk about its colimit or limit in $\bC$, 
denoted,
$$\underset{i\in \bI}{\on{colim}}\, \Phi(i) \text{ and } \underset{i\in \bI}{\on{lim}}\, \Phi(i),$$
respectively, which, if they exist, are defined uniquely up to a canonical isomorphism (in proper language, up to a contractible set of choices).
We shall say that $\bC$ is \emph{cocomplete} if the colimits of all functors from all index categories $\bI$ 
(satisfying a certain cardinality condition that we ignore) exist. The reader is referred to \cite[Sect. 1.2.13]{Lu1} for more details
and references for the actual definition. 

\medskip

Let $F:\bC\to \bD$ be a functor. Then for another $\infty$-category $\bE$, precomposition with $F$ defines a functor
$$\on{Funct}(\bD,\bE) \to \on{Funct}(\bC,\bE).$$
The (partially defined left (resp., right) adjoint of this functor is called the functor of left (resp., right) Kan extension. The 
left (resp., right) Kan extension of a given functor $\Phi:\bC\to \bE$ can be calculated explicitly on objects:
$$\on{LKE}(\Phi)(\bd)=\underset{\bc\in \bC_{/\bd}}{\on{colim}}\, \Phi(\bc)  \text{ and }
\on{RKE}(\Phi)(\bd)=\underset{\bc\in \bC_{\bd/}}{\on{lim}}\, \Phi(\bc).$$
If $F$ is a coCartesian (resp., Cartesian) fibration, then in the above colimit (resp., limit) the index category can be replaced
by the fiber $\bC_\bd$. See \cite[Sect. 4.3]{Lu1} for more details.

\sssec{Higher algebra} 

We shall need the following notions from higher algebra, developed in \cite{Lu2}:  stable $\infty$-category (Sect .1), symmetric monoidal 
$\infty$-category (Definition 2.0.0.7), and commutative algebra in a symmetric monoidal $\infty$-category (Sect. 2.1.3). 

\sssec{A lemma on limits vs colimits}  \label{sss:limits and colimits}

In two places in the text we will encounter the following situation.  Let $\StinftyCat\subset \inftyCat$ be the non-full subcategory,
where we restrict the objects to be cocomplete stable categories and $1$-morphisms to be colimit-preserving. First,
we note that by \cite[Proposition 5.5.3.13]{Lu1}, the above inclusion preserves limits.

\medskip

Let 
$$\Phi:\bI\to \StinftyCat,\quad i\mapsto \bC_i,\quad (i_1\overset{\alpha}\to i_2)\mapsto \Phi_\alpha\in \on{Funct}(\bC_{i_1},\bC_{i_2})$$
be a functor. Denote 
$$\bC:=\underset{i\in \bI}{\on{colim}}\, \bC_i,$$
where the \emph{colimit is taken in $\StinftyCat$}. For every $i\in \bI$, we have a tautologically defined functor $\on{ins}_i:\bC_i\to \bC$.

\medskip

By the Adjoint Functor Theorem (see \cite[Corollary 5.5.2.9(i)]{Lu1}), for every arrow $i_1\overset{\alpha}\to i_2$ in $\bI$, the corresponding functor $\Phi_\alpha$
admits a right adjoint $\Psi_\alpha$ (which is not necessarily continuous). Consider the functor 
$$\Psi:\bI^{\on{op}}\to \inftyCat,$$
which is the same as $\Phi$ on objects, but which sends 
$$(i_1\overset{\alpha}\to i_2) \mapsto \Psi_\alpha.$$

The right adjoints to the functors $\on{ins}_i$, denoted $\on{ev}_i$, define a functor
\begin{equation} \label{e:from colimit to limit}
\bC\to \underset{i\in \bI^{\on{op}}}{\on{lim}}\, \bC_i,
\end{equation}
where in the right-hand side we are taking the \emph{limit of the functor $\Psi$ in $\inftyCat$}.

\medskip

Suppose now that in the above situation the functors $\Psi_\alpha$ are also colimit preserving. Then it follows
that the functors $\on{ev}_i$ are also colimit preserving.  In this case it is easy to show that for every $\bc$ the 
map
$$\underset{i\in \bI}{\on{colim}}\, \on{ins}_i\circ \on{ev}_i(\bc)\to \bc$$
is an isomorphism.

\ssec{Acknowledgments}

The author would like to thank J.~Lurie for undertaking this project together (but, in fact, so much more).  

\medskip

In addition to the people listed in the acknowledgments in \cite{Main Text}, the author would like to thank S.~Raskin for 
the suggestion to consider direct images for sheaves on lax prestacks.

\medskip

The author is very grateful to Q.~Ho and S.~Raskin for reading the text and catching a number of mistakes. 

\medskip

The author is supported by NSF grant DMS-1063470.

\newpage 

\centerline{\bf Part 0: Preliminaries on prestacks, lax prestacks and sheaves}

\bigskip

\section{Sheaves and prestacks}  \label{s:prestacks}

The goal of this paper is to prove a certain isomorphism of vector spaces (the \emph{Atiyah-Bott formula}, \eqref{e:AB prev}).
One of these vector spaces is defined by an explicit elementary procedure and another as a cohomology of something. 
So neither side explicitly mentions sheaves (that is to say, non-constant sheaves).  However, our way of proving the desired
isomorphism heavily uses sheaves--they already appear in the main step of the proof, namely, the \emph{cohomological product formula}, 
\eqref{e:AB prev}.

\medskip

We are used to considering sheaves on schemes. However, even to state the product formula, we need a slightly more general
set-up: namely, we will need to consider sheaves on \emph{prestacks}. 

\medskip

The goal of the present section it to define what prestacks are and what we mean by a sheaf on a prestack. 
These notions are necessary for understanding the core of this paper, namely Part V. 

\ssec{Sheaves on schemes}  \label{ss:sheaves}

In this subsection we explain what we mean by a \emph{theory of sheaves on schemes}. The discussion here is aimed
at the technically-minded reader. So, the reader who is willing to understand sheaves on schemes intuitively can
skip this subsection and proceed to \secref{ss:prestacks}. 

\sssec{}

In this paper we assume being given a \emph{theory of sheaves}, which is a functor
$$\Shv^!:(\Sch)^{\on{op}}\to \inftyCat.$$

\medskip

I.e., to a scheme $S$ we assign an $\infty$-category $\Shv(S)$ and to a morphism $f:S_1\to S_2$
we assign a pullback functor
$$f^!:\Shv(S_2)\to \Shv(S_1).$$

These functors are endowed with a homotopy-coherent system of compatibilities for compositions of morphisms. 

\begin{rem}
We should explain why we are using the !-pullback rather than the *-pullback. The reason is that the prestacks that 
we will consider are ``ind-objects" (colimits of schemes under closed embeddings), and we want to think of a sheaf on such 
a prestack as a colimit of sheaves on schemes that comprise it. This dictates the choice of the !-pullback over the *-pullback, 
it being the \emph{right adjoint} of the direct image functor,  see \secref{sss:limits and colimits}.
\end{rem}

\sssec{Technical assumptions}  \label{sss:assump cont}

We assume the following additional \emph{properties} of the functor $\Shv^!$: 

\medskip

\noindent(i) It takes values in the full subcategory of $\inftyCat$ formed by cocomplete stable categories.

\medskip

\noindent(ii) For every morphism $f:S_1\to S_2$ in $\Sch$, the corresponding functor $f^!:\Shv(S_2)\to \Shv(S_1)$
commutes with colimits.

\medskip

In other words, $\Shv^!$ (uniquely) factors through a non-full subcategory $\StinftyCat\subset\inftyCat$, where
we restrict the objects to be cocomplete stable categories and $1$-morphisms to be colimit-preserving.

\sssec{Symmetric monoidal structure}  \label{sss:assump mon}

Recall that the category $\StinftyCat$ has a natural symmetric monoidal structure given by tensor product. We 
shall assume that our functor $\Shv^!$ is endowed with the following additional pieces of structure:

\medskip

\noindent(a) $\Shv^!$ is endowed with a \emph{right-lax symmetric monoidal structure}, where $\Sch$
is considered as equipped with the Cartesian symmetric monoidal structure. 

\medskip

The datum in (a) stipulates the existence of a functor
$$\Shv(S_1)\otimes \Shv(S_2)\to \Shv(S_1\times S_2),\quad \CF_1,\CF_2\mapsto \CF_1\boxtimes \CF_2,$$
equipped with a homotopy-coherent system of compatibilities.

\medskip

In particular, for any $S$, pullback with respect to the diagonal morphism $S\overset{\on{diag}_S}\longrightarrow S\times S$
defines on $\Shv(S)$ a structure of symmetric monoidal category
$$\CF_1,\CF_2\mapsto \CF_1\overset{!}\otimes \CF_2:=\on{diag}_S^!(\CF_1\boxtimes \CF_2).$$

We shall refer to this structure as the \emph{pointwise} symmetric monoidal structure. 

\medskip 

The second additional piece of structure on the functor $\Shv^!$ is:

\medskip

\noindent(b) The symmetric monoidal category $\Shv(\on{pt})$ is identified with $\Lambda\mod$, where 
$\Lambda$ is our (fixed) commutative ring of coefficients.

\sssec{}

Here are some examples of sheaf theories:

\medskip

\noindent(1) When the ground field is $\BC$ and an arbitrary ring $\Lambda$, we can take $\Shv(S)$ to be the ind-completion of
the category of constructible sheaves on $S$ with $\Lambda$-coefficients.

\medskip

\noindent(2) For any ground field $k$ and $\Lambda=\BZ/\ell^n\BZ, \BZ_\ell, \BQ_\ell$ (where $\ell$ is assumed to be invertible in $k$),
and we take $\Shv(S)$ to be the ind-completion of the category of constructible \'etale sheaves on $S$ with $\Lambda$-coefficients, 
see \cite[Sect. 4]{Main Text}. 

\medskip

\noindent(3) When the ground field has characteristic $0$, we take $\Shv(S)$ to be the category of holonomic D-modules on $S$. 

\medskip

\noindent(4) When the ground field has characteristic $0$, we take $\Shv(S)$ to be the category of all D-modules on $S$. 

\sssec{}  \label{sss:constr}

In what follows, we shall refer to Examples (1)-(3) above as the \emph{context of constructible sheaves}. Its main feature is that
for any morphism $f:S_1\to S_2$, the functor $f^!$ admits a left adjoint, denoted $f_!$.  
(We note that this property fails in general in example (4), unless the morphism $f$ is proper.)

\medskip

In particular, we obtain that in the constructible context, the functor $f^!$ commutes with \emph{limits}. 

\medskip

In order for a sheaf theory to work well in applications, one should impose several more conditions (e.g., proper base change
and the K\"unneth formula). \footnote{We emphasize that whatever these additional conditions are, they are \emph{properties}
(i.e., certain maps should be isomorphisms), but \emph{not} additional pieces of structure.}
Rather than listing these conditions, we shall assume that our sheaf theory is one of the examples
(1)-(4) listed above.  

\medskip

Note that in each of these examples, for a scheme $S$, the category $\Shv(S)$ possesses a t-structure. In fact, in the constructible
context there are two t-structures: the usual and the perverse one. In the context of D-modules, there is the usual D-module t-structure. 

\ssec{Prestacks}  \label{ss:prestacks}

A prestack is an arbitrary contravariant functor from the category of schemes to that of
groupoids. In this way prestacks are geometric objects that generalize algebraic stacks
or ind-schemes. When working over the ground field of complex numbers, in many cases
the behavior of a prestack is well approximated by that of the underlying topological space 
(for example, see Remark \ref{r:homology} below). 

\sssec{}

By definition, a prestack is an arbitrary functor
$$(\Sch)^{\on{op}}\to \inftygroup.$$

We let $\on{PreStk}$ denote the $\infty$-category of prestacks. 

\sssec{}

Yoneda embedding defines a fully-faithful functor
$$\Sch\hookrightarrow \on{PreStk}.$$

\ssec{Sheaves on prestacks}  \label{ss:sheaves prestacks}

Sheaves on prestacks will be defined in a very straightforward way, see below. Of course, we will not be able
to say much about the category of sheaves on a general prestack. But fortunately, the very formal properties
that one gets for free from the definition will suffice for our purposes. 

\sssec{}

For a prestack $\CY$, the category $\Shv^!(\CY)$ is defined by
$$\Shv^!(\CY):=\underset{S\in \Sch_{/\CY}}{\on{lim}}\, \Shv(S),$$
where for a morphism $f:S_1\to S_2$ in $\Sch_{/\CY}$ the corresponding transition functor
$\Shv(S_2)\to \Shv(S_1)$ is $f^!$. 

\sssec{}

In other words, an object $\CF\in \Shv^!(\CY)$ is an assignment of

\begin{itemize}

\item $S\in \Sch,y\in \CY(S)\,\,\mapsto \,\,\CF_{S,y}\in \Shv^!(S)$,

\item $(S'\overset{f}\to S)\in \Sch,y\in \CY(S_2)\,\, \mapsto \,\, (\CF_{S',f(y)}\simeq f^!(\CF_{S,y}))\in \Shv^!(S')$,
\end{itemize}
satisfying a homotopy-coherent system of compatibilities. 

\sssec{}

A bit more functorially, one can view the assignment $\CY\mapsto \Shv^!(\CY)$ as the \emph{right Kan extension} 
of the functor
$$\Shv^!:(\Sch)^{\on{op}}\to \inftyCat$$
along the Yoneda embedding $$(\Sch)^{\on{op}}\hookrightarrow (\on{PreStk})^{\on{op}}.$$

\medskip

In particular, for a morphism $f:\CY_1\to \CY_2$ between prestacks, we have a tautologically defined functor
$$f^!:\Shv^!(\CY_2)\to \Shv^!(\CY_1).$$

\medskip

\noindent Notation: sometimes, for $\CF\in \Shv^!(\CY_2)$ we shall use a shorthand notation 
$$\CF|_{\CY_1}:=f^!(\CF).$$

\medskip

For a prestack $\CY$ we let $\omega_\CY\in \Shv^!(\CY)$ be the \emph{dualizing sheaf}, i.e., the !-pullback 
of $\Lambda\in \Shv(\on{pt})$. 

\sssec{}

The additional assumptions on the functor $\Shv^!$ of Sects. \ref{sss:assump cont} and \ref{sss:assump mon}
imply that the resulting functor
$$\Shv^!:(\on{PreStk})^{\on{op}}\to \inftyCat$$ also factors via a functor to $\StinftyCat$ and as such is endowed
with a symmetric monoidal structure.  In particular, every $\Shv^!(\CY)$ acquires a symmetric monoidal structure,
and a symmetric monoidal functor from $\Lambda\mod$. 

\medskip

\begin{lem} \label{l:constr limits prestack}
In the context of constructible sheaves \footnote{See \secref{sss:constr} for what we mean by that.}, 
the functor $f^!$ commutes with \emph{limits}.
\end{lem}

\begin{proof}

Since the functor of !-pullback commutes with limits for morphisms between \emph{schemes},
for a prestack $\CY$, limits in $\Shv^!(\CY)$ are computed value-wise, i.e., for a 
family $a\mapsto \CF^a$ and $(S,y)\in \Sch_{/\CY}$, we have 
$$(\underset{a}{\on{lim}}\, \CF^a)_{S,y}\simeq \underset{a}{\on{lim}}\,  (\CF^a_{S,y}).$$

\medskip

This makes the assertion of the lemma manifest.

\end{proof}

\ssec{Direct images?}  

We now have the theory of sheaves on prestacks, but the only functoriality so far is the !-pullback. One can wonder:
what about the other functors, such as the *-pullback, *-pushforward or !-pushforward? The answer is that we will 
not even attempt to define them, but only grab whatever naturally comes our way.

\sssec{}

For a morphism $f:\CY_1\to \CY_2$ we can consider the \emph{partially defined} left adjoint of $f^!$
$$f_!:\Shv^!(\CY_1)\to \Shv^!(\CY_2).$$

We have:

\begin{cor}   \label{c:!-pushforward prestack}
In the context of constructible sheaves, the functor $f_!$ is always defined.
\end{cor}

\begin{proof}
Follows from the Adjoint Functor Theorem (see \cite[Corollary 5.5.2.9(ii)]{Lu1}) and \lemref{l:constr limits prestack}.
\end{proof}

\sssec{}  \label{sss:! dir im}

In particular, taking $\CY_1=\CY$ and $\CY_2=\on{pt}$ we obtain a (partially defined) functor
$$\CF\mapsto \on{C}^*_c(\CY,\CF), \quad \Shv^!(\CY)\to \Lambda\mod,$$
left adjoint to the pullback functor (the latter being the same as $-\otimes \omega_\CY$). 

\medskip

By \corref{c:!-pushforward prestack}, the functor $\on{C}^*_c(\CY,-)$ 
is always defined in the context of constructible sheaves. In general, we have: 
  
\begin{lem} \label{l:express cohomology prestack}
For $\CF\in \Shv^!(\CY)$ we have
$$\on{C}_c^*(\CY,\CF)\simeq \underset{(S,y)\in \Sch_{/\CY}}{\on{colim}}\, \on{C}_c^*(S,\CF_{S,y}),$$
whenever the right-hand side is defined.
\end{lem} 

\begin{cor}
In the context of D-modules, the functor $\on{C}_c^*(\CY,-)$ is defined on any $\CF\in \Shv^!(\CY)$ 
for which for every $(S,y)$, the corresponding object $\CF_{S,y}\in \Shv(S)$ is holonomic.
\end{cor}

In particular, in any context, the functor $\on{C}_c^*(\CY,-)$ is defined on $\omega_\CY\in \Shv^!(\CY)$. We shall use
the following notation: 
$$\on{C}_c^*(\CY,\omega_\CY)=:\on{C}_*(\CY),$$
and we shall refer to $\on{C}_*(\CY)$ as the \emph{homology} of $\CY$. In addition, we denote:
$$\on{Fib}\left((p_\CY)_!\circ (p_\CY)^!(\Lambda)\to \Lambda\right)=\on{Fib}(\on{C}_*(\CY)\to \on{C}_*(\on{pt}))=:\on{C}_*^{\on{red}}(\CY);$$
this is the reduced homology of $\CY$. 

\begin{rem} \label{r:homology}
When working over the ground field of complex numbers, the 
functor that associates to a scheme the underlying analytic space gives rise, by means of
\emph{left Kan extension}, to a functor from the category of prestacks to the $\infty$-category
of topological spaces. Under this functor, the \emph{homology of a prestack} introduced above 
is isomorphic to the homology of the corresponding topological space with coefficients in $\Lambda$.
\end{rem}

\sssec{}

For later use we also note the following:

\begin{lem}  \label{l:base change}
Let $f:\CY_1\to \CY_2$ be a schematic open embedding (i.e., the base change of $f$ by any scheme yields an open embedding).
Then the functor $f_*:\Shv^!(\CY_1)\to \Shv^!(\CY_2)$, \emph{right adjoint} to $f^!$, is defined. Moreover, for a Cartesian diagram 
of prestacks
$$
\CD
\CY'_1  @>{g_1}>>  \CY_1  \\
@V{f'}VV    @VV{f}V   \\
\CY'_2  @>{g_2}>>  \CY_2 
\endCD
$$
the natural transformation 
$$g_2^!\circ f_*\to f'_*\circ g_1^!,$$
that comes by adjunction from the isomorphism $$f'{}^!\circ g_2^!\simeq g_1^!\circ f^!,$$
is an isomorphism.
\end{lem}

\begin{proof}
Follows formally from the case when $\CY_2$ (resp., $\CY_2$ and $\CY'_2$) are schemes, in which case it is the usual base change isomorphism.
\end{proof} 

\ssec{Pseudo-properness}  \label{ss:pseudo-properness prestacks}

As we remarked above, for a morphism of prestacks $f:\CY_1\to \CY_2$ and $\CF\in \Shv^!(\CY_1)$ we often have a 
well-defined object
$$f_!(\CF)\in \Shv^!(\CY_2).$$

However, in most cases this object is incalculable: we do not have an algorithm to say what the value of $f_!(\CF)$ is
on $S\overset{y}\to \CY$. 

\medskip

But there is one class of morphisms where $f_!$ is given by a much more explicit procedure: these are maps that we
call \emph{pseudo-proper}, and which are ubiquitous in this paper.

\sssec{}  

Let $f:\CY\to S$ be a map in $\on{PreStk}$, where $S\in \Sch$. We shall say that $\CY$ is \emph{pseudo-proper}
over $S$ if $\CY$ can be written as a colimit of objects representable by schemes proper over $S$.

\medskip

We have:

\begin{prop} \label{p:pseudo-proper}
For a pseudo-proper $f:\CY\to S$, the functor $f_!$, left adjoint to $f^!$, is defined. Furthermore,
for a map of schemes $g_S:S'\to S$ and $\CY':=\CY\underset{S}\times S'$, the 
natural transformation
$$f'_!\circ g_Y^!\to g_S^!\circ f_!,$$
arising from the Cartesian diagram
\begin{equation} \label{e:proper base change}
\CD
\CY' @>{g_Y}>>  \CY \\
@V{f'}VV    @VV{f}V  \\
S' @>{g_S}>>  S,
\endCD
\end{equation}
and the identification 
$$g_Y^!\circ f^!\simeq f'{}^!\circ g_S^!,$$
is an isomorphism.
\end{prop}

\begin{proof}

Write 
$$\CY=\underset{a}{\on{colim}}\, Z_a,$$
where $Z_a$ are schemes proper over $S$. By definition, we have:
$$\Shv^!(\CY)=\underset{a}{\on{lim}}\, \Shv(Z_a),$$
where the limit is formed using the functors 
$$f^!_{a_1,a_2}:\Shv(Z_{a_2})\to \Shv(Z_{a_1}) \text{ for } f_{a_1,a_2}:Z_{a_1}\to Z_{a_2}.$$

Now, by \secref{sss:limits and colimits}, we have:
$$\Shv^!(\CY)=\underset{a}{\on{colim}}\, \Shv(Z_a),$$
where the colimit is taken in the category of cocomplete categories and colimit-preserving functors. In the formation of the limit 
we use the functors
$$(f_{a_1,a_2})_!:\Shv(Z_{a_2})\to \Shv(Z_{a_1}) \text{ for } f_{a_1,a_2}:Z_{a_1}\to Z_{a_2}.$$

\medskip

Now, the sought-for left adjoint, viewed as a functor
$$\underset{a}{\on{colim}}\, \Shv(Z_a)\to \Shv(S),$$
is given by the compatible family of functors
$$(f_a)_!:\Shv(Z_a)\to  \Shv(S) \text{ for } f_a:Z_a\to S.$$

The commutation of the diagram \eqref{e:proper base change} follows by base change,
as the maps $f_a$ are proper. 
\footnote{Note that in the context of constructible $\ell$-adic sheaves, the existence of $f_!$ does not require
the maps $f_a$ to be proper, as the functors $(f_a)_!$ are always defined (but this would not be 
the case in the context of arbitrary, i.e., not necessarily holonomic D-modules). However, the properness 
of the maps $f_a$ is needed for the commutation of \eqref{e:proper base change}.}

\end{proof} 

\sssec{}    \label{sss:pseudo proper}

We shall say that a map $f:\CY_1\to \CY_2$ between prestacks is pseudo-proper if its base change by any scheme yields
a prestack over that scheme. 

\medskip

We shall make a repeated use of the following assertion:

\begin{cor} \label{c:pseudo-proper}
If $f:\CY_1\to \CY_2$ is pseudo-proper, then the functor $f_!$, left adjoint to $f^!$, is defined. Furthermore,
for a map $g_2:\CY_2'\to \CY_2$ and $\CY'_1:=\CY_1\underset{\CY_2}\times \CY'_2$, the 
natural transformation
$$f'_!\circ g_1^!\to g_2^!\circ f_!,$$
arising from the Cartesian diagram
$$
\CD
\CY'_1 @>{g_1}>>  \CY_1 \\
@V{f'}VV    @VV{f}V  \\
\CY'_2 @>{g_2}>>  \CY_2,
\endCD
$$
and the identification 
$$g_1^!\circ f^!\simeq f'{}^!\circ g_2^!,$$
is an isomorphism. 
\end{cor}. 

\begin{proof}

Follows formally from \propref{p:pseudo-proper}.

\end{proof}

Note that the base change property in \corref{c:pseudo-proper} in particular gives an expression to the value of $f_!(\CF)$ on
any $S\overset{y}\to \CY_2$: form the pullback 
$$\CY:=S\underset{\CY_2}\times \CY_1$$
and calculate the !-pushforward of $\CF|_{\CY}$ to $S$. 

\section{Lax prestacks and sheaves}  \label{s:lax}

The contents of this section are needed for Parts I and III of the paper, which in turn will be used in \secref{s:local duality}
(the reduction of the cohomological product formula to a local statement). 

\medskip

It turns that even such general gadgets as prestacks will not be sufficient for a crucial manipulation that we will need to
perform for the proof of the product formula. We will need to generalize them in a yet different direction, 
which in a sense is non-geometric.

\medskip

We will need to consider \emph{lax prestacks}, where those are functors that take values not in \emph{groupoids},
but rather in \emph{categories}. I.e., we will allow non-invertible morphisms between points. When 
dealing with lax prestacks, the conventional geometric intuition had better be abandoned for reasons
of safety. 

\medskip

The idea of working with sheaves on lax prestacks was pioneered by S.~Raskin, and to the best of our knowledge
was first used in his PhD thesis (so far, unpublished). 

\ssec{Lax prestacks}

In this subsection we define what we mean by a \emph{lax prestack}. 

\sssec{}

We let $\on{LaxPreStk}$ denote the category of functors\footnote{Technically we need to consider functors that are 
\emph{accessible} (see \cite[Sect. 5.4.2]{Lu1} for what this means) and that take values in essentially small categories.
That said, from now on we shall ignore set-theoretical issues of this sort.}
$$(\Sch)^{\on{op}}\to \inftyCat.$$

In other words, a prestack is an assignment for any $S\in \Sch$ of a category $\CY(S)$ and for $S'\overset{f}\longrightarrow S$
of a functor 
$$y\mapsto f(y),\quad \CY(S)\to \CY(S'),$$
plus a homotopy-coherent system of compatibilities for compositions.

\medskip

Alternatively, we can think of $\on{LaxPreStk}$ as the category of Cartesian fibrations over $\Sch$.
For $\CY\in \on{LaxPreStk}$ we let $\Sch_{/\CY}$ denote the corresponding category, fibered over
$\Sch$.

\medskip

For $\CY\in \on{LaxPreStk}$ let $p_\CY:\CY\to \on{pt}$ denote the tautological map. 

\sssec{}

The category of lax prestacks contains as a full subcategory the category $\on{PreStk}$ of prestacks: it consists
of those functors that take values
in $\inftygroup\subset \inftyCat$. 

\sssec{}  \label{sss:op}

For $\CY\in  \on{LaxPreStk}$ we let $\CY^{\on{op}}$ be the lax prestack obtained by passing to fiberwise
opposite categories in
$$\Sch_{/\CY}\to \Sch.$$

\medskip

Evidently, if $\CY\in \on{PreStk}$, then $\CY^{\on{op}}\simeq \CY$. 

\sssec{}  

To a lax prestack $\CY$ we attach a prestack $\CY_{\on{str}}$ defined by 
$$\CY_{\on{str}}(S)=\underset{y\in \CY(S)}{\on{colim}}\, \{*\}.$$

I.e., $\CY_{\on{str}}(S)$ is the groupoid, universal with respect to the property of
receiving a functor from $\CY(S)$; explicitly, it is obtained by inverting all the morphisms 
in $\CY(S)$.

\medskip

We have a map $\CY\to \CY_{\on{str}}$, which is universal with respect to maps from $\CY$ to
prestacks. 

\medskip

We have: 
$$\CY_{\on{str}}\simeq (\CY^{\on{op}})_{\on{str}}.$$

\ssec{Sheaves on lax prestacks}

By now we are used to working with sheaves on prestacks:
a sheaf on a prestack is a compatible family of sheaves on schemes that map to it. When working over $\BC$,
sheaves on a prestack are closely related to sheaves on the underlying topological space. 

\medskip

Sheaves on a lax prestack are a very different animal. 

\sssec{}

Let $\Shv^!_{\Sch}$ denote the Cartesian fibration over $\Sch$ that corresponds to the functor
$$\Shv^!:(\Sch)^{\on{op}}\to \inftyCat,\quad S\mapsto \Shv^!(S),\,\, (S_1\overset{f}\to S_2)\mapsto f^!.$$

For $\CY\in \on{LaxPreStk}$ we let $\Shv^!(\CY)$ denote the category of functors
$$\Sch_{/\CY}\to \Shv^!_{\Sch}$$
that take Cartesian arrows to Cartesian arrows.

\medskip

I.e., an object $\CF\in \Shv^!(\CY)$ is an assignment of

\begin{itemize}

\item $S\in \Sch,y\in \CY(S)\,\,\mapsto \,\,\CF_{S,y}\in \Shv^!(S)$,

\item $(y_1\to y_2)\in \CY(S)\,\, \mapsto \,\,(\CF_{S,y_1}\to \CF_{S,y_2})\in \Shv^!(S)$,

\item $(S'\overset{f}\to S)\in \Sch,y\in \CY(S)\,\, \mapsto \,\, (\CF_{S',f(y)}\simeq f^!(\CF_{S,y}))\in \Shv^!(S')$,
\end{itemize}
satisfying a homotopy-coherent system of compatibilities. 

\sssec{}

The assignment
$$\CY\mapsto \Shv^!(\CY)$$
extends to a functor
$$(\on{LaxPreStk})^{\on{op}}\to \inftyCat, \quad (\CY_1\overset{f}\to \CY_2)\mapsto f^!.$$

\sssec{}

For a morphism $f:\CY_1\to \CY_2$, the functor 
$$f^!: \Shv^!(\CY_2)\to \Shv^!(\CY_1)$$
commutes with colimits, by construction.

\medskip

In addition, we have the following assertion, proved in the same way as \lemref{l:constr limits prestack}:

\begin{lem} \label{l:constr limits}
If we are working in the context of constructible sheaves \footnote{Here and in the sequel, we include here the
case of holonomic D-modules.}, the functor $f^!$ commutes with \emph{limits}.
\end{lem}

\sssec{}

Let $\CY_1$ and $\CY_2$ be two lax prestacks, and let
$$f:\CY_1\rightleftarrows \CY_2:g$$
be maps equipped with natural transformations
$$f\circ g\to \on{id}_{\CY_2} \text{ and} \on{id}_{\CY_1}\to g\circ f$$
making them into an adjoint pair.

\medskip

The following observation will be useful:

\begin{lem}  \label{l:geom adj}
Under the above circumstances, the functors 
$$g^!:\Shv^!(\CY_2)\rightleftarrows \Shv^!(\CY_1):f^!$$
form an adjoint pair.
\end{lem}  

\sssec{}  \label{sss:str}

For a lax prestack $\CY$, we let $\Shv^!_{\on{str}}(\CY)\subset \Shv^!(\CY)$ denote the full subcategory of
objects for which the arrows $\CF_{S,y_1}\to \CF_{S,y_2}$ are isomorphisms for any $S\in \Sch$ and 
$(y_1\to y_2)\in \CY(S)$.

\medskip

Pullback along $\CY\to \CY_{\on{str}}$ defines an equivalence
$$ \Shv^!(\CY_{\on{str}})\to  \Shv^!_{\on{str}}(\CY).$$

In particular, if $\CZ$ is a \emph{prestack} equipped with a map $g:\CY\to \CZ$, then for any $\CF\in \Shv^!(\CZ)$,
the object $\CF|_{\CY}$ belongs to $\Shv^!_{\on{str}}(\CY)$.

\medskip

In particular, $\omega_\CY\in \Shv^!_{\on{str}}(\CY)$.  

\ssec{Existence of left adjoints}

We now want to address the issue of direct images for morphisms between lax prestacks. Note that
this is a whole zoo: there is the *- vs !- image. But also, at the purely categorical level with no geometry,
there are the left Kan extension vs the right Kan extensions. \footnote{As was mentioned earlier, it was
the idea of Sam Raskin that one ought to be brave and consider direct images for morphisms between
lax prestacks when they seem to be the natural thing to do.} 

\medskip

In this subsection we will set for ourselves a limited goal: we will investigate some instances when the functor 
$$f^!:\Shv^!(\CY_2)\to \Shv^!(\CY_1)$$
admits a left adjoint. 

\sssec{}

We will denote by $f_!$ the (partially defined) functor, left adjoint to $f^!$. As in \corref{c:!-pushforward prestack},
the functor $f_!$ is always defined in the context of constructible sheaves. 

\sssec{}

When $\CY_2=\on{pt}$
and $\CY_1=\CY$, we shall denote the corresponding (partially defined) 
functor $\Shv^!(\CY)\to \Lambda\mod$ by 
$$\CF\mapsto \on{C}_c^*(\CY,\CF).$$

\medskip

As was already mentioned, the functor $\on{C}_c^*(\CY,-)$ is always defined
in the context of $\ell$-adic sheaves. As in \secref{sss:! dir im}, we have:

\begin{lem} \label{l:express cohomology}
For $\CF\in \Shv^!(\CY)$ we have
$$\on{C}_c^*(\CY,\CF)\simeq \underset{(S,y)\in \Sch_{/\CY}}{\on{colim}}\, \on{C}_c^*(S,\CF_{S,y}),$$
whenever the right-hand side is defined.
\end{lem}
 
In particular, in the context of D-modules, the functor $\on{C}_c^*(\CY,-)$ is defined on any 
$\CF\in \Shv^!(\CY)$ for which for every $(S,y)$, the corresponding object $\CF_{S,y}\in \Shv(S)$ is holonomic. 
In particular, it is always defined for the object $\omega_\CY$. 

\medskip

For $\CF=\omega_\CY$ we shall also use the notations
$$\on{C}_c^*(\CY,\omega_\CY)=:\on{C}_*(\CY) \text{ and }
\on{Fib}\left((p_\CY)_!\circ (p_\CY)^!(\Lambda)\to \Lambda\right)=\on{Fib}(\on{C}_*(\CY)\to \on{C}_*(\on{pt}))=:\on{C}_*^{\on{red}}(\CY).$$

We will refer to $\on{C}_*(\CY)$ (reps., $\on{C}_*^{\on{red}}(\CY)$) as the \emph{homology} (resp., \emph{reduced homology}) of $\CY$.  

\sssec{Pseudo-properness for lax prestacks}

We will now discuss a generalization of \corref{c:pseudo-proper} for maps in $\on{LaxPreStk}$. 

\medskip

Let $f:\CY_1\to \CY_2$ be a map in $\on{LaxPreStk}$. 
We shall say that $f$ is pseudo-proper if: 
 
\medskip
 
\begin{itemize}

\item For every $S\in \Sch$, the functor $\CY_1(S)\to \CY_2(S)$ is a coCartesian fibration in groupoids;


\item For every $(S\in \Sch,y_2\in \CY_2(S))$, the fiber product $S\underset{\CY_2}\times \CY_1$
is a pseudo-proper \emph{prestack} over $S$.

\end{itemize} 

As a formal corollary of \propref{p:pseudo-proper} we obtain:

\begin{cor} \label{c:lax homology}
If $f:\CY_1\to \CY_2$ is pseudo-proper, then the functor $f_!$, left adjoint to $f^!$, is defined. Furthermore,
for a map $g_2:\CY_2'\to \CY_2$ and $\CY'_1:=\CY_1\underset{\CY_2}\times \CY'_2$, the 
natural transformation
$$f'_!\circ g_1^!\to g_2^!\circ f_!,$$
arising from the Cartesian diagram
$$
\CD
\CY'_1 @>{g_1}>>  \CY_1 \\
@V{f'}VV    @VV{f}V  \\
\CY'_2 @>{g_2}>>  \CY_2,
\endCD
$$
and the identification 
$$g_1^!\circ f^!\simeq f'{}^!\circ g_2^!,$$
is an isomorphism.
\end{cor}. 

\ssec{Right adjoints}

The contents of subsection will be needed for the proof of \propref{p:out}. It can be skipped on the first pass.

\medskip

We now turn to the discussion of the right adjoint to the functor $f^!$ (note that this right adjoint will be a pretty weird functor
even for a morphism between schemes). 

\medskip

First, we know that the right adjoint to $f^!$ is always defined,
because $f^!$ commutes with colimits (by the Adjoint Functor Theorem, \cite[Corollary 5.5.2.9(i)]{Lu1}). 

\medskip

So, the question we want to address is that of a more explicit description of $(f^!)^R$.  

\sssec{}

For a general morphism $f:\CY_1\to \CY_2$ it is difficult to describe the functor $(f^!)^R$ in any explicit way. 
Yet, it is not completely ``wild": 

\medskip

Let 
$$
\CD
\CY'_1    @>{g_1}>>  \CY_1  \\
@V{f'}VV   @VV{f}V   \\
\CY'_2    @>{g_2}>>  \CY_2
\endCD
$$
be a Cartesian diagram of lax prestacks. The isomorphism
$$f'{}^!\circ g_2^!\simeq g_1^!\circ f^!$$
gives rise to a natural transformation:
\begin{equation} \label{e:strange base change}
g_2^!\circ (f^!)^R\to (f'{}^!)^R\circ g_1^!
\end{equation}.

\medskip

We claim:

\begin{lem}  \label{l:right adjoint base change}
Suppose that $g_2$ is pseudo-proper. Then the natural transformation \eqref{e:strange base change} is an isomorphism.
\end{lem}

\begin{proof}

By passing to left adjoints, we obtain a natural transformation
$$(g_1)_!\circ f'{}^!\to f^!\circ (g_2)_!,$$
which is an isomorphism by \corref{c:lax homology}.

\end{proof} 

\sssec{}

We shall now describe a situation in which we can get a handle on the functor $(f^!)^R$. 

\medskip

Let $\CF$ be an object of $\Shv^!(\CY_1)$.  
Fix $S\in \Sch$, $y_2\in \CY_2(S)$, and a map of schemes $g_S:S'\to S$. Let us denote $y'_2=g_S(y_2)\in \CY_2(S')$.
Pullback defines a functor
$$g_Y:\CY_1(S)\underset{\CY_2(S)}\times \{y_2\}\to \CY_1(S')\underset{\CY_2(S')}\times \{y_2'\}.$$

\medskip

The datum of $\CF$ defines functors
$$\CF(y_2):\CY_1(S)\underset{\CY_2(S)}\times \{y_2\}\to \Shv^!(S) \text{ and }
\CF(y'_2):\CY_1(S')\underset{\CY_2(S')}\times \{y'_2\}\to \Shv^!(S'),$$
so that the diagram
$$
\CD
 \CF(y_2):\CY_1(S)\underset{\CY_2(S)}\times \{y_2\}   @>{\CF(y_2)}>>   \Shv^!(S)  \\
@V{g_Y}VV    @VV{(g_S)^!}V    \\
 \CF(y'_2):\CY_1(S')\underset{\CY_2(S')}\times \{y'_2\}   @>{\CF(y'_2)}>>   \Shv^!(S')
\endCD
$$
commutes. 

\medskip

Hence, we obtain the following maps in $\Shv^!(S')$:
\begin{multline} \label{e:comp map}
(g_S)^! \left(\underset{\CY_1(S)\underset{\CY_2(S)}\times \{y_2\}}{\on{lim}}\, \CF(y_2)\right)\to
\underset{\CY_1(S)\underset{\CY_2(S)}\times \{y_2\}}{\on{lim}}\, (g_S)^!(\CF(y_2))\simeq \\
\simeq \underset{\CY_1(S)\underset{\CY_2(S)}\times \{y_2\}}{\on{lim}}\, \CF(y'_2)\circ g_Y
\leftarrow 
\underset{\CY_1(S')\underset{\CY_2(S')}\times \{y'_2\}}{\on{lim}}\, \CF(y'_2).
\end{multline}

Note that in the context of constructible sheaves, the map $\to$ in \eqref{e:comp map} is automatically
an isomorphism.

\sssec{}

Let us now assume that:

\begin{itemize}

\item For every $S\in \Sch$, the functor $\CY_1(S)\to \CY_2(S)$ is a Cartesian fibration;

\item For any $g:S'\to S$ the pullback functor $\CY_1(S)\to \CY_1(S')$ maps arrows that are $f(S)$-Cartesian 
to arrows that are $f(S')$-Cartesian.

\end{itemize} 

Together these conditions can be combined into saying that the functor $\Sch_{/\CY_1}\to \Sch_{/\CY_2}$ is
a Cartesian fibration. 

\begin{lem} \label{l:right adjoint}
Assume that for a given $\CF$ and any $(S'\overset{g_S}\to S, y_2\in \CY_2(S))$ as above, both maps
in \eqref{e:comp map} are isomorphisms. Then the natural map
$$(f^!)^R(\CF)_{S,y_2}\to \underset{\CY_1(S)\underset{\CY_2(S)}\times \{y_2\}}{\on{lim}}\, \CF(y_2)$$
is an isomorphism.
\end{lem}

\ssec{The notion of universally homological contractibility over a prestack} \hfill     \label{ss:uhc prestack}

\medskip

\noindent Let $F:\bC\to \bD$ be a functor between $\infty$-categories, where $\bD$ is an $\infty$-groupoid. In this case
it is easy to give a criterion when the functor
$$\on{Funct}(\bD,\bE)\to \on{Funct}(\bC,\bE),$$
given by precomposition with $F$, is fully faithful for any $\bE$. Namely, this happens if and only if all the fibers of $F$ 
are contractible (i.e., have a trivial homotopy type). 

\medskip

In this subsection we will study a generalization of this with the geometry mixed-in: i.e., when instead of  
$\infty$-categories we have lax prestacks. 

\sssec{}   \label{sss:contr over prestack} 

Let $f:\CY_1\to \CY_2$ be a map of lax prestacks, where $\CY_2\in \on{PreStk}$. We shall say that $f$
is \emph{universally homologically contractible} if for any $S\in \Sch_{/\CY_2}$, for the induced map
$$f_S:S\underset{\CY_2}\times \CY_1\to S,$$
the functor
$$(f_S)^!:\Shv^!(S)\to \Shv^!(S\underset{\CY_2}\times \CY_1)$$
is fully faithful.

\sssec{}

The following is easy:

\begin{lem}  \label{l:contr contr}
Let $f:\CY_1\to \CY_2$ be a map, where $\CY_2$ is a prestack. Suppose that for any $S\in \Sch$, the functor 
$\CY_1(S)\to \CY_2(S)$ has contractible fibers. Then the functor $f$ is universally homologically contractible.
\end{lem} 

\sssec{}

The following results easily from the definitions:

\begin{lem}  \label{l:ff from uhc}
Let $f:\CY_1\to \CY_2$ be a map of lax prestacks, where $\CY_2\in \on{PreStk}$. Suppose that $f$
is universally homologically contractible. Then the functor
$$f^!:\Shv^!(\CY_2)\to \Shv^!(\CY_1)$$
is fully faithful.
\end{lem} 

\sssec{}

Assume for a moment that $\CY_2=\on{pt}$. Then we shall say that $\CY=\CY_1$ is universally homologically
contractible if its map to $\on{pt}$ is. 

\begin{lem}  \label{l:trivial homology}
A lax prestack is universally homologically contractible if and only if the trace map
$$\on{C}_*(\CY)\to \Lambda$$
is an isomorphism. 
\end{lem}

\begin{proof}

Follows from the fact that for a pair of schemes $S$ and $S$' and $\CF_1,\CF_2\in \Shv(S)$, the map
$$\Maps_{\Shv(S)}(\CF_1\otimes \on{C}_*(S'),\CF_2)\to 
\Maps_{\Shv(S\times S')}(\CF_1\boxtimes \omega_{S'},\CF_2\boxtimes \omega_{S'})$$
is an isomorphism.

\end{proof}

\section{Universally homologically contractible maps and related notions}   \label{s:uhc} 

In this section we address the following question: under what conditions is a morphism between
\emph{lax prestacks} \footnote{In \secref{ss:uhc prestack} this question was addressed in the case when the target is a prestack.}
such the corresponding pullback functor is fully faithful?  When does it induce an isomorphism
on homology of any sheaf? These questions arise in many situations that come up in practice.

\medskip

The material in this section could be skipped on the first pass, and returned to when necessary. 

\ssec{Contractible functors (a digression)} \label{ss:contr funct}

Let $F:\bC\to \bD$ be a functor between $\infty$-categories.  In this section we will give a criterion for when
the functor
$$\on{Funct}(\bD,\bE)\to \on{Funct}(\bC,\bE),$$
given by precomposition with $F$ is a fully faithful embedding, for any $\infty$-category $\bE$.

\sssec{}

We shall say that a functor $F:\bC\to \bD$ is \emph{contractible} if for any arrow $d\overset{\alpha}\to d'$ in $\bD$, the category
$\on{Factor}_F(\alpha)$ of its factorizations as
$$d\to F(c)\to d',\quad c\in \bC$$
is contractible (i.e., has a trivial homotopy type), see \secref{sss:hom th conv} for what this means. 

\sssec{}

Suppose for a moment that $\bD$ is a groupoid. Then the following conditions are equivalent: (i) $F$ is contractible; (ii) $F$ is left or right cofinal;
(iii) $F$ has contractible fibers. 

\medskip

For a general $\bD$ we have:

\begin{lem}  \label{l:cart contr}
Assume that $F:\bC\to \bD$ is a Cartesian or coCartesian fibration. Then $F$ is contractible
if and only if it has contractible fibers.
\end{lem}

\sssec{}

We now claim:

\begin{prop} \label{p:contr crit}
A functor $F:\bC\to \bD$ is contractible if and only 
if for any $\infty$-category $\bE$, the functor
$$\on{Funct}(\bD,\bE)\to \on{Funct}(\bC,\bE),$$
given by precomposition with $F$ is a fully faithful embedding.
\end{prop} 

\begin{proof}

It is easy to see that the fully-faithfulness condition appearing in the statement of the proposition holds
if and only if it holds for $\bE=\inftygroup$, i.e., if and only if the functor
$$\on{Funct}(\bD,\inftygroup)\to \on{Funct}(\bC,\inftygroup),$$
is fully faithful.

\medskip

Furthermore, the latter is equivalent to the fact that for any $d\in \bD$ and $\Phi:\bD\to \inftygroup$
the map
$$\Maps_{\on{Funct}(\bD,\inftygroup)}(\wt{h}_{d},\Phi)\to \Maps_{\on{Funct}(\bC,\inftygroup)}(\wt{h}_{d}\circ F,\Phi\circ F)$$
should be isomorphism, where $\wt{h}_d$ is the covariant Yoneda functor corresponding to the object $d\in \bD$. 

\medskip

I.e., $F$ is contractible if and only if the counit of the adjunction 
$$\on{LKE}_F(\wt{h}_{d}\circ F)\to \wt{h}_{d}$$
is an isomorphism. 

\medskip

Now, we calculate the value of $\on{LKE}_F(\wt{h}_{d}\circ F)$ on $d'\in \bD$ as
$$\underset{c\in \bC,c\to d'}{\on{colim}}\, \wt{h}_{d}(F(c)).$$

Thus, the fiber of $\on{LKE}_F(\wt{h}_{d}\circ F)$ over a given point $d\to d'$ of $\wt{h}_d(d')$
is the homotopy type of the category $\on{Factor}_F(\alpha)$, as required. 

\end{proof} 

\ssec{Value-wise contractibility}

We shall now begin to explore how to transport the notion of contractible functor to the
context of lax prestacks, i.e., where we have category theory coupled with geometry. 

\medskip

But before we do that, we will define a notion which is an overkill (but which is often
useful in practice). 

\sssec{}

Let $f:\CY_1\to \CY_2$ be a map of lax prestacks. We shall say that $\CY$ is 
\emph{value-wise contractible} 
if the corresponding functor $$\CY_1(S)\to \CY_2(S)$$
is contractible. 

\sssec{}

From \lemref{l:cart contr} we obtain: 

\begin{cor} 
Let $f:\CY_1\to \CY_2$ be a map of lax prestacks. Suppose that for any $S\in \Sch$, the functor
$\CY_1(S)\to \CY_2(S)$ is a Cartesian or coCartesian fibration with contractible fibers. Then $f$
is value-wise contractible.
\end{cor}

\sssec{}

From \propref{p:contr crit} we obtain: 

\begin{cor}  \label{c:ptw contr}
Let $f:\CY_1\to \CY_2$ be value-wise contractible. Then the functor
$$f^!:\Shv^!(\CY_1)\to \Shv^!(\CY_2)$$
is fully faithful. In particular:

\smallskip

\noindent{\em(i)} For any $\CF\in \Shv^!(\CY)$, the map
$$\on{C}_*(\CY_1,f^!(\CF))\to \on{C}_*(\CY_2,\CF)$$
is an isomorphism whenever either side is defined.

\smallskip

\noindent{\em(ii)} The map $\on{C}_*(\CY_1)\to \on{C}_*(\CY_2)$ is an isomorphism.

\end{cor}

\ssec{Some weaker value-wise notions}

In \secref{ss:contr funct} we have an explicit criterion for when a functor $F:\bC\to \bD$ is such that for any $\bE$ and $\Phi_1,\Phi_2:\bD\to \bE$,
the map
\begin{equation} \label{e:nat trans}
\Maps_{\on{Funct}(\bD,\bE)}(\Phi_1,\Phi_2)\to \Maps_{\on{Funct}(\bC,\bE)}(\Phi_1\circ F,\Phi_2\circ F)
\end{equation}
is an isomorphism.

\medskip

One can relax this condition as follows. One can ask that \eqref{e:nat trans} be an isomorphism when $\Phi_2$ (resp., $\Phi_1$)
maps all arrows in $\bD$ to isomorphisms in $\bE$ (i.e., when it factors through the maximal sub-groupoid in $\bE$).  One can relax
it even further by requiring that both $\Phi_1$ and $\Phi_2$ have this property.

\medskip

The latter (i.e., the weakest) condition is equivalent to $F$ inducing an equivalence of homotopy types. The former condition
is equivalent to $F$ being left (resp., right) cofinal.  \footnote{The traditional, but equivalent, definition of cofinality uses functors $\Phi_2$
(resp., $\Phi_1$) that take a \emph{constant} value in $\bE$. This is tautologically equivalent to the map
$\underset{\bC}{\on{colim}}\, \Phi\circ F\to \underset{\bD}{\on{colim}}\, \Phi$ (resp., $\underset{\bD}{\on{lim}}\, \Phi\to 
\underset{\bC}{\on{lim}}\, \Phi\circ F$) being an isomorphism for any $\bE$ and $\Phi:\bD\to \bE$.} The left 
(resp., right) cofinality condition can be rewritten as saying that
for every object $\bd\in \bD$ the category $\bC_{\bd/}$ (resp., $\bC_{/\bd}$) should be contractible. 

\medskip

We will now consider the analogous value-wise notions for maps between lax prestacks. 

\sssec{}

We shall say that a map of lax prestacks $f:\CY_1\to \CY_2$ is \emph{value-wise left cofinal} if for every $S\in \Sch$,
the functor 
$$\CY_1(S)\to \CY_2(S)$$
is left cofinal.

\begin{lem}  
Suppose that $f:\CY_1\to \CY_2$ is value-wise left cofinal. Then for any $\CF\in \Shv^!(\CY_2)$, 
and $\CG\in \Shv^!_{\on{str}}(\CY_2)$, the map
$$\Maps_{\Shv(\CY_2)}(\CF,\CG)\to \Maps_{\Shv(\CY_1)}(f^!(\CF),f^!(\CG))$$
is an isomorphism.  In particular:

\smallskip

\noindent{\em(i)} For any $\CF\in \Shv^!(\CY)$, the map
$$\on{C}_*(\CY_1,f^!(\CF))\to \on{C}_*(\CY_2,\CF)$$
is an isomorphism whenever either side is defined.

\smallskip

\noindent{\em(ii)} The map $\on{C}_*(\CY_1)\to \on{C}_*(\CY_2)$ is an isomorphism.

\end{lem}

\sssec{}

Let $f:\CY_1\to \CY_2$ be a map of lax prestacks. We shall say that $f$
is a \emph{value-wise homotopy-type equivalence} if for any $S\in \Sch$, the functor
$$\CY_1(S)\to \CY_2(S)$$ induces an equivalence of homotopy types. 

\medskip

We have:

\begin{lem} \label{l:rel hom type}
Suppose that $f:\CY_1\to \CY_2$ is a value-wise homotopy-type equivalence. Then
for any $\CF,\CG\in \Shv^!_{\on{str}}(\CY_2)$, the induced map
$$\Maps_{\Shv(\CY_2)}(\CF,\CG)\to \Maps_{\Shv(\CY_1)}(f^!(\CF),f^!(\CG))$$
is an is an isomorphism.  In particular,

\smallskip

\noindent{\em(i)} For any $\CF\in \Shv^!_{\on{str}}(\CY_2)$, the map
$$\on{C}^*_c(\CY_1,f^!(\CF))\to \on{C}^*_c(\CY_2,\CF)$$
is an isomorphism whenever either side is defined. 

\smallskip

\noindent{\em(ii)} The map
$\on{C}_*(\CY_1)\to \on{C}_*(\CY_2)$ is an isomorphism.
\end{lem} 

\ssec{Universal homological contractibility for lax prestacks}  \label{ss:uhc lax}

We now come to the less naive notion of homological contractibility for a map between lax prestacks. 

\sssec{}

Let $f:\CY_1\to \CY_2$ be a map of lax prestacks. For $S\in \Sch$ and a map $\alpha:y'_2\to y''_2$ in $\CY_2(S)$,
let $\on{Factor}_f(\alpha)$ denote the following lax prestack over $S$:

\medskip

For $\wt{S}\in \Sch_{/S}$, the category $\on{Factor}_f(\alpha)(\wt{S})$ is that of factorizations of $\alpha|_{\wt{S}}$ as
$$y'_2|_{\wt{S}}\to f(y_1)\to y''_2|_{\wt{S}},\quad y_1\in \CY_1(\wt{S}).$$

\medskip

We shall say that $f$ is \emph{universally homologically contractible} if for any $(S,\alpha)$ as above, the map
$$\on{Factor}_f(\alpha)\to S$$
is universally homologically contractible.

\sssec{}

Note that if $\CY_2$ is a prestack, the two notions of universal homological contractibility (one defined above and
another in \secref{sss:contr over prestack}) coincide.

\sssec{}

We have: 

\begin{lem}
If $f:\CY_1\to \CY_2$ is value-wise contractible then it is universally homologically contractible.
\end{lem}

\begin{proof}
Follows from \lemref{l:contr contr}.
\end{proof}

\sssec{}

The following assertion is parallel to \lemref{l:cart contr}:

\begin{prop}  \label{p:Cart uhc}
Let $f:\CY_1\to \CY_2$ be such that for any $S\in \Sch$, the functor $\CY_1(S)\to \CY_2(S)$ 
is a Cartesian or coCartesian fibration. Then $f$ is universally homologically contractible
if and only if it has a universally homologically contractible fiber over any $S\in \Sch_{/\CY_2}$.
\end{prop} 

\begin{proof}
We will give a proof for Cartesian fibrations; the case of coCartesian fibrations is similar. For a given
$(S,\alpha)$ consider the prestack $\on{Factor}'_f(\alpha)$
over $S$, that attaches to $\wt{S}\in \Sch_{/S}$ the category of
$$y'_2|_{\wt{S}}\overset{\sim}\longrightarrow f(y'_1) \to f(y_1) \to y''_2|_{\wt{S}}, \quad (y'_1\to y_1)\in \CY_1(\wt{S}).$$

We have a natural forgetful map
$$\on{Factor}'_f(\alpha)\to \on{Factor}_f(\alpha),$$
and we claim that it is value-wise contractible. Indeed, for a given $\wt{S}$, the functor
$$\on{Factor}'_f(\alpha)(\wt{S})\to \on{Factor}_f(\alpha)(\wt{S})$$
is a coCartesian fibration, and the assumption that
$$\CY_1(S)\to \CY_2(S)$$
is a Cartesian fibration implies that it has contractible fibers.

\medskip

Hence, by \corref{c:ptw contr}, the pullback functor
$$\Shv(S)\to \Shv^!(\on{Factor}_f(\alpha))$$ 
is fully faithful if and only if the functor
$$\Shv(S)\to \Shv^!(\on{Factor}'_f(\alpha))$$ 
is fully faithful.

\medskip

Consider now the forgetful map
$$\on{Factor}'_f(\alpha)\to S\underset{y'_2,\CY_2}\times \CY_1,\quad 
(y'_2|_{\wt{S}}\overset{\sim}\longrightarrow f(y'_1) \to f(y_1) \to y''_2|_{\wt{S}})\mapsto y'_1.$$
We claim that it is value-wise contractible.  Indeed, for a given $\wt{S}$, the corresponding map
$$\on{Factor}'_f(\alpha)(\wt{S})\to (S\underset{y'_2,\CY_2}\times \CY_1)(\wt{S})$$
is a Cartesian fibration with contractible fibers (each fiber has an initial point). 

\medskip

Hence, the pullback functor
$$\Shv(S)\to \Shv^!(\on{Factor}'_f(\alpha))$$ 
is fully faithful if and only if the functor
$$\Shv(S)\to \Shv^!(S\underset{y'_2,\CY_2}\times \CY_1)$$
is fully faithful, as required.

\end{proof}

\begin{rem}
Note, however, that the notion of universally homologically contractibility over a lax prestack (unlike a usual prestack)
is \emph{not} stable under base change. 
\end{rem} 

\sssec{}

We have:

\begin{prop}  \label{p:uhc general}
If $f:\CY_1\to \CY_2$ is universally homologically contractible, then the functor
$$f^!:\Shv^!(\CY_1)\to \Shv^!(\CY_2)$$
is fully faithful. 
\end{prop}

Note that \propref{p:uhc general} generalizes \corref{c:ptw contr}. 

\begin{proof}[Sketch of the proof]

Given $\CF,\CG\in \Shv^!(\CY_2)$, let us construct the map
$$\Maps_{\Shv^!(\CY_1)}(f^!(\CF),f^!(\CG))\to \Maps_{\Shv^!(\CY_2)}(\CF,\CG),$$
inverse to the map given by the functor $f^!$.

\medskip

Given a map
\begin{equation} \label{e:map to start}
f^!(\CF)\to f^!(\CG), 
\end{equation}
to specify the corresponding map $\CF\to \CG$, we need to give for every 
$S\in \Sch$ and every arrow $\alpha:y'_2\to y''_2$ in $\CY_2(S)$ a map
\begin{equation} \label{e:map to end}
\CF_{S,y'_2}\to \CG_{S,y''_2}
\end{equation} 
in $\Shv(S)$.

\medskip

Consider the lax prestack $\on{Factor}_f(\alpha)$, and note that it is endowed with the following maps
$$\on{ev}',\on{ev}'':\on{Factor}_f(\alpha)\to \CY_2 \text{ and } \on{ev}^m:\on{Factor}_f(\alpha)\to \CY_1,$$
that send an $\wt{S}$-point of $\on{Factor}_f(\alpha)$, given by 
$$y'_2|_{\wt{S}}\to f(y_1)\to y''_2|_{\wt{S}}$$
to
$$y'_2|_{\wt{S}},\,\, y''_2|_{\wt{S}} \text{ and } y_1,$$
respectively.

\medskip

Consider the following four objects in $\Shv^!(\on{Factor}_f(\alpha))$:
$$(\on{ev}')^!(\CF),\,\, (\on{ev}^m)^!\circ f^!(\CF),\,\,  (\on{ev}^m)^!\circ f^!(\CG),\,\, (\on{ev}'')^!(\CG).$$

Note that 
$$(\on{ev}')^!(\CF)\simeq \CF_{S,y'_2}|_{\on{Factor}_f(\alpha)} \text{ and } 
(\on{ev}'')^!(\CG)\simeq \CG_{S,y''_2}|_{\on{Factor}_f(\alpha)},$$
and that we have the natural maps
$$(\on{ev}')^!(\CF)\to (\on{ev}^m)^!\circ f^!(\CF) \text{ and } (\on{ev}^m)^!\circ f^!(\CG)\to (\on{ev}'')^!(\CG).$$

Composing, we obtain a map
$$\CF_{S,y'_2}|_{\on{Factor}_f(\alpha)} \simeq (\on{ev}')^!(\CF)\to (\on{ev}^m)^!\circ f^!(\CF) \to
(\on{ev}^m)^!\circ f^!(\CG)\to (\on{ev}'')^!(\CG)\simeq \CG_{S,y''_2}|_{\on{Factor}_f(\alpha)},$$
where the third arrow is induced by \eqref{e:map to start}.

\medskip

Thus, we have obtained a map $\CF_{S,y'_2}|_{\on{Factor}_f(\alpha)}\to \CG_{S,y''_2}|_{\on{Factor}_f(\alpha)}$ 
in $\Shv^!(\on{Factor}_f(\alpha))$.

\medskip

Now, the assumption that the projection $\on{Factor}_f(\alpha)\to S$ is universally
homologically contractible implies that the latter map comes from a uniquely defined 
map \eqref{e:map to end}. 

\end{proof}

\begin{cor}  \label{c:lax contr}
Let $f:\CY_1\to \CY_2$ be universally homologically contractible. Then: 

\smallskip

\noindent{\em(i)} For any $\CF\in \Shv^!(\CY)$, the map
$$\on{C}_*(\CY_1,f^!(\CF))\to \on{C}_*(\CY_2,\CF)$$
is an isomorphism whenever either side is defined.

\smallskip

\noindent{\em(ii)} The map $\on{C}_*(\CY_1)\to \on{C}_*(\CY_2)$ is an isomorphism.

\end{cor}

\ssec{The notion of universal homological left cofinality}  \label{ss:l cofinal}

In this subsection we will replace the notion of value-wise left cofinality by a less naive one. 

\sssec{}

Let $f:\CY_1\to \CY_2$ be a map of lax prestacks. For $S\in \Sch$ and $y_2\in \CY_2(S)$, consider the lax prestack $(\CY_1)_{y_2/}$ that assigns to
$\wt{S}\to S$ the category of  
$$y_1\in \CY_1(\wt{S}), \,\, y_2|_{\wt{S}}\to f(y_1).$$

\medskip

We shall say that $f$ is \emph{universally homologically left cofinal} if for all $(S,y_2)$ as above, the lax prestack
$(\CY_1)_{y_2/}$ be universally homologically contractible over $S$. 

\sssec{}

Note that if $\CY_2$ is a \emph{prestack}  (as opposed to a lax prestack), then the notions of universally homological contractibility
and universally homological left cofinality coincide. 

\sssec{}

We have:

\begin{prop}
Let $f:\CY_1\to \CY_2$ be universally homologically contractible. Then it is universally homologically left cofinal. 
\end{prop}

\begin{proof}

Fix $S\in \Sch$ and $y_2\in \CY_2(S)$. Consider the following lax prestacks $(\CY_2)_{y_2/}$ and $(\CY_1,\CY_2)_{y_2/}$
over $S$. The lax prestack $(\CY_2)_{y_2/}$
attaches to $\wt{S}\to S$ the category $\CY_2(\wt{S})_{y_2|_{\wt{S}}/}$. The lax prestack $(\CY_1,\CY_2)_{y_2/}$ attaches to 
$\wt{S}\to S$ the category of
$$y'_2\in \CY_2(\wt{S}),\,\,y_1\in \CY_1(\wt{S}), \,\,y_2|_{\wt{S}}\to f(y_1)\to y'_2.$$

We have naturally defined morphisms over $S$
$$(\CY_2)_{y_2/} \leftarrow (\CY_1,\CY_2)_{y_2/}\to (\CY_1)_{y_2/}.$$

The map $(\CY_1,\CY_2)_{y_2/}\to (\CY_1)_{y_2/}$ is value-wise contractible, and hence universally homologically contractible.
Therefore, to prove the proposition, it suffices to show that the map $(\CY_1,\CY_2)_{y_2/}\to S$ is universally homologically contractible. 

\medskip

The map $(\CY_2)_{y_2/} \to S$ is value-wise contractible, and hence is universally homologically contractible. Therefore, it remains
to show that the map $(\CY_1,\CY_2)_{y_2/}\to (\CY_2)_{y_2/} $ is universally homologically contractible. 

\medskip

We note that the map $(\CY_1,\CY_2)_{y_2/}\to (\CY_2)_{y_2/} $ is a value-wise Cartesian fibration. Hence, by \propref{p:Cart uhc}, it suffices 
to show that its fibers are universally homologically contractible. However, the latter follows from the assumption
that $f$ is universally homologically contractible.

\end{proof}

\sssec{}

We observe:

\begin{lem}
Let $f$ be value-wise left cofinal. Then it is universally homologically left cofinal.
\end{lem} 

\begin{proof}

If $f$ is value-wise left cofinal, then for any $S\in \Sch$ and $y_2\in \CY_2(S)$, the map 
$$(\CY_1)_{y_2/}\to S$$
has value-wise contractible fibers, and hence is universally homologically contractible, by \lemref{l:contr contr}. 

\end{proof}

\begin{lem}
Assume that $f$ is such that for any $S\in \Sch$, the map
$$\CY_1(S)\to \CY_2(S)$$
is a Cartesian fibration. Then $f$ is universally homologically left cofinal if and only if all of its fibers are universally homologically contractible. 
\end{lem}

\begin{proof}

If $f$ is a value-wise Cartesian fibration, then 
for any $S\in \Sch$ and $y_2\in \CY_2(S)$, the inclusion
$$S\underset{y_2,\CY_2}\times \CY_1\to (\CY_1)_{y_2/}$$
is a value-wise homotopy-type equivalence. Hence, 
the assertion of the lemma follows from \lemref{l:rel hom type}.

\end{proof} 

\sssec{}

We will need the following: 

\begin{prop} \label{p:uhlc}
Let $f:\CY_1\to \CY_2$ be universally homologically left cofinal and let $g:\CY'_2\to \CY_2$ be a value-wise
coCartesian fibration. Then the base-changed map
$$f':\CY'_1:=\CY'_2\underset{\CY_2}\times \CY_1\to \CY'_2$$
is also universally homologically left cofinal.
\end{prop}

\begin{proof} 

For $S\in \Sch$ and $y'_2\in \CY'_2(S)$ we have a tautological map of lax prestacks over $S$:
$$(\CY'_1)_{y'_2/}\to (\CY_1)_{g(y'_2)/}.$$

By \lemref{l:rel hom type}, it suffices to show that the above map is a value-wise homotopy-type
equivalence. However, the latter follows from the fact that for any $\wt{S}\to S$, the resulting functor
$$(\CY'_1)_{y'_2/}(\wt{S})\to  (\CY_1)_{g(y'_2)/}(\wt{S})$$
admits a left adjoint.

\end{proof} 

\sssec{}

Finally, we claim:

\begin{prop}  \label{p:l cofinal ff}
Let $f:\CY_1\to \CY_2$ be universally homologically left cofinal. Then for any 
$\CF\in \Shv^!(\CY_2)$ and $\CG\in \Shv^!_{\on{str}}(\CY_2)$, the map
$$\Maps_{\Shv^!(\CY_2)}(\CF,\CG)\to 
\Maps_{\Shv^!(\CY_1)}(f^!(\CF),f^!(\CG))$$
is an isomorphism. 
\end{prop}

\begin{proof}[Sketch of proof]

Given $\CF\in \Shv^!(\CY_2)$  and $\CG\in \Shv^!_{\on{str}}(\CY_2)$, 
let us construct the map
$$\Maps_{\Shv^!(\CY_1)}(f^!(\CF),f^!(\CG))\to \Maps_{\Shv^!(\CY_2)}(\CF,\CG),$$
inverse to the map given by the functor $f^!$.

\medskip

Given a map
\begin{equation} \label{e:map to start new}
f^!(\CF)\to f^!(\CG), 
\end{equation}
to specify the corresponding map $\CF\to \CG$, we need to give for every 
$S\in \Sch$ and every $y_2\in \CY_2(S)$ a map
\begin{equation} \label{e:map to end new}
\CF_{S,y_2}\to \CG_{S,y_2}
\end{equation} 
in $\Shv(S)$.

\medskip

Consider the prestack $(\CY_1)_{y_2/}$, and consider the evaluation map
$$\on{ev}:(\CY_1)_{y_2/}\to \CY_1,$$
that sends an $\wt{S}$-point of $(\CY_1)_{y_2/}$, given by 
$y_2|_{\wt{S}}\to y_1$, to $y_1\in \CY_1(\wt{S})$. 

\medskip

Note that we have a canonical map
$$\CF_{S,y_2}|_{(\CY_1)_{y_2/}}\to \on{ev}^!\circ f^!(\CF)$$
and an isomorphism
$$(\CG_{S,y_2})|_{(\CY_1)_{y_2/}}\simeq \on{ev}^!\circ f^!(\CG).$$

Composing with \eqref{e:map to start new}, we obtain a map
$$\CF_{S,y_2}|_{(\CY_1)_{y_2/}}\to \on{ev}^!\circ f^!(\CF) \to \on{ev}^!\circ f^!(\CG)\simeq (\CG_{S,y_2})|_{(\CY_1)_{y_2/}}.$$

Using the assumption on the morphism $(\CY_1)_{y_2/}\to S$, we obtain that the latter map comes from a
uniquely defined map in \eqref{e:map to end new}.

\end{proof}

\begin{cor}  \label{c:l cofinal ff}
Let $f$ be universally homologically left cofinal. Then:

\smallskip

\noindent{\em(i)} For $\CF\in \Shv^!(\CY_2)$, the map
$$\on{C}_*(\CY_1,f^!(\CF))\to \on{C}_*(\CY_2,\CF)$$
is an isomorphism, whenever either side is defined.

\smallskip

\noindent{\em(ii)} The map $\on{C}_*(\CY_1)\to \on{C}_*(\CY_2)$ is an isomorphism.

\end{cor} 

\newpage 

\centerline{\bf Part I: Various incarnations of the Ran space}

\section{The non-unital and unital versions of the Ran space}  \label{s:Ran}

Let $X$ be a separated scheme (in most applications so far one takes $X$ to be a curve). The Ran space of $X$
is a geometric object that incarnates the idea that it classifies finite collections of points of $X$, but where the
cardinality of the collection is not fixed (in this it has an advantage over symmetric powers of $X$ that count
points with multiplicities).

\medskip

When taken literally, the Ran space is a \emph{prestack}, and when working over the ground field
of complex numbers it corresponds to
a reasonably behaved topological space. However, for our purposes we will also need the \emph{unital} and
\emph{unital augmented} versions of the Ran space, and those will be \emph{lax prestacks}. 

\medskip

The contents of \secref{ss:Ran} are necessary for the statement of the cohomological product formula \thmref{t:product formula}, i.e.,
the core of this paper. As a prerequisite for \secref{ss:Ran} one only needs \secref{s:prestacks}. 
The rest of the present section is needed for the material in Part III of the paper, which in turn
is used in \secref{s:local duality} (the reduction of the cohomological product formula to a local duality statement). 

\ssec{The usual (=non-unital) Ran space}   \label{ss:Ran}

In this subsection we will introduce the usual Ran space; this is the same object as one that was originally studied in \cite{BD}. 

\sssec{}

Let $X$ be a (separated) scheme. We let $\Ran$ denote the following object of $\on{PreStk}$:

\medskip

\noindent For $S\in \Sch$ we let $\Ran(S)$ be the (discrete) groupoid of finite non-empty subsets of $\Maps(S,X)$. 

\medskip

We shall refer to $\Ran$ as the \emph{non-unital} version of the Ran space of $X$.




\sssec{}

The following give a description of $\Ran$ as colimit of schemes:

\begin{prop}   \label{p:express Ran}
The prestack  $\Ran$ as canonically isomorphic to 
$$\underset{\CI\in (\on{Fin}^s)^{\on{op}}}{\on{colim}}\, X^\CI,$$
where $\on{Fin}^s$ is the category of finite non-empty sets and surjective maps.
\end{prop}

\begin{rem}
In what follows we shall denote by straight letters $I,J,K$ finite \emph{subsets} of $\Maps(S,X)$
(for a given $S$) and by script characters $\CI,\CJ,\CK$ \emph{abstract} finite sets. 
\end{rem} 

Note that \propref{p:express Ran} gives an explicit description of the category $\Shv^!(\Ran)$ in terms
of categories of sheaves on schemes:
$$\Shv^!(\Ran)\simeq \underset{\CI\in \on{Fin}^s}{\on{lim}}\, \Shv(X^\CI).$$

\begin{proof}[Proof of \propref{p:express Ran}]

This follows from the fact that, given a set $A$, the set of its finite non-empty subsets is canonically isomorphic to
$\underset{\CI\in (\on{Fin}^s)^{\on{op}}}{\on{colim}}\, \Maps(\CI,A)$. 

Indeed, the above colimit is the disjoint union over finite subsets $I\subset A$ of 
$$\underset{\CI\in (\on{Fin}^s)^{\on{op}}}{\on{colim}}\, \Maps(\CI,A)_{\text{image equals }I}.$$

Each such expression is the same as
$$\underset{\CI\in (\on{Fin}^s)^{\on{op}},\CI\twoheadrightarrow I}{\on{colim}}\,\{*\},$$
whereas the latter index category has the final point, given by $\CI=I$. 

\end{proof} 

In what follows, for a finite set $\CI$ we shall denote by $\on{ins}_\CI$ the corresponding map
$$X^\CI\to \Ran.$$

Also, for future use we introduce the notation $$\overset{\circ}X{}^\CI\subset X^\CI$$ for the open locus, 
whose $S$-points are $\CI$-tuples of elements of $\Maps(S,X)$ that have \emph{pairwise non-intersecting images}. 
I.e., $k$-points of $\overset{\circ}X{}^\CI$ are $\CJ$-tuples of pairwise distinct $k$-points of $X$. 

\sssec{}

From \propref{p:pseudo-proper}, we obtain: 

\begin{cor}  \label{c:chiral homology}
Assume that $X$ is: \hfill

\smallskip

\noindent--proper if we are working in the context of D-mdoules, and

\smallskip

\noindent--arbitrary for the context of constructible sheaves.

\smallskip

\noindent Then the functor $\on{C}_c^*(\Ran,-):\Shv^!(\Ran)\to \Lambda\mod$ is defined.
\end{cor}

The functor $\on{C}_c^*(\Ran,-)$ from \corref{c:chiral homology} is the functor of (non-unital) chiral homology. 

\sssec{}

A basic feature of the Ran space is the following theorem, due to Beilinson and Drinfeld
(see \cite[Theorem 2.4.5]{Main Text} for a proof):

\begin{thm}  \label{t:Ran contr}
Suppose $X$ is connected. Then $\Ran$ is universally homologically contractible. 
\end{thm}

According to \lemref{l:trivial homology}, an equivalent way to state this theorem is that the trace map
$$\on{C}_*(\Ran)\to \Lambda$$
is an isomorphism.

\ssec{The unital version of the Ran space}   \label{ss:untl Ran}

As is explained in \cite[Sect. 5.5]{Lu2}, in the topological setting for $X=\BR^1$, \emph{non-unital} associative algebras give rise to objects
of $\Shv^!(\Ran)$. Namely, for an associative algebra $A$ and a subset $I\subset \BR$, the !-fiber of the corresponding sheaf $\CA$
at $I$ will be 
$$\CA_I:=\underset{i\in I}\otimes\, A.$$ 
The unital Ran space for $\BR^1$ will be the natural recipient of a functor from 
\emph{unital} associative algebras: the unit in $A$ allows to map 
$$\underset{i\in I_1}\otimes\, A\to \underset{i\in I_2}\otimes\, A$$
each time we have an inclusion $I_1\subseteq I_2$, so we have a map 
$\CA_{I_1}\to \CA_{I_2}$. 

\medskip

In this subsection we introduce the unital version of the Ran space. The difference between it and the usual Ran space
is that now we will be able to account for the fact that one finite subset of $X$ is contained in another by means of maps
$\CF_{S,I_1}\to \CF_{S,I_2}$ for $I_1\subseteq I_2\subset \Maps(S,X)$. 

\sssec{}

We let $\Ran_{\on{untl}}$ denote the following object of $\on{LaxPreStk}$: 

\medskip

\noindent For 
$S\in \Sch$ we let $\Ran_{\on{untl}}(S)$ be the (ordinary) category whose objects are
finite non-empty subsets of $\Maps(S,X)$, and where the morphisms are given by
inclusion of finite subsets. 

\medskip

We shall refer to $\Ran_{\on{untl}}$ as the \emph{unital} version of the Ran space of $X$. 

\begin{rem}
One can introduce a version of $\Ran_{\on{untl}}$, where one also allows the empty set. In the topological context
this is a good idea if we want to have a closer contact with associative algebras. 
\end{rem} 

\begin{rem}
In the case of the usual Ran space, we had its explicit expression as a colimit of schemes, given by \propref{p:express Ran}.
A similar description is possible also for $\Ran_{\on{untl}}$, see \secref{ss:untl expl}.
\end{rem}

\sssec{}

We have the tautological map 
\begin{equation} \label{e:Ran to Ran untl}
\phi:\Ran\to \Ran_{\on{untl}}.
\end{equation} 
Denote by 
$$\on{OblvUnit}:=\phi^!:
\Shv^!(\Ran_{\on{untl}})\to \Shv^!(\Ran)$$
the corresponding pullback functor. 

\begin{rem}
In terms of the analogy with associative algebras the functor $\on{OblvUnit}$ corresponds to the obvious forgetful functor
from the category of unital algebras to that of non-unital ones. 
\end{rem}

\sssec{}

One can show that the functor 
$$\on{C}^*_c(\Ran_{\on{untl}},-):\Shv^!(\Ran_{\on{untl}})\to \Lambda\mod,$$ 
is well-defined (in the context of D-modules, under the additional assumption that $X$ be proper). 
When $X$ is connected, this will be done in the course of the proof of \thmref{t:unital chiral homology}.  
This is the functor of unital chiral homology. 

\medskip 

\begin{thm} \label{t:unital chiral homology}
Assume that $X$ is connected. Then the functor $\on{C}^*_c(\Ran_{\on{untl}},-)$ is well-defined and the natural
transformation
$$\on{C}^*_c(\Ran,-)\circ \on{OblvUnit} \to \on{C}^*_c(\Ran_{\on{untl}},-)$$
is an isomorphism. 
\end{thm} 

The proof is given in \secref{ss:1st proof chiral}. 

\ssec{Adding the unit}    \label{ss:add unit}

Continuing to draw on the analogy from topology, there is a naturally defined functor from the category non-unital algebras
to that of unital algebras, given by adjoining the unit. If $A$ is a non-unital algebra and $\CA\in \Shv^!(\Ran)$ the corresponding
object, then the object of $\Shv^!(\Ran_{\on{unital}})$, corresponding to the algebra $\on{AddUnit}(A)$, is given as follows
$$\on{AddUnit}(\CA_I)=\underset{J\subseteq I}\oplus\, \CA_J.$$

\medskip

The corresponding functor for sheaves on the Ran space will be introduced in this subsection. I.e., we will construct a functor
$$\on{AddUnit}: \Shv^!(\Ran)\to \Shv^!(\Ran_{\on{untl}}).$$
Ultimately, we will show (see \secref{ss:add unit simple}) that the functor $\on{AddUnit}$ is the left adjoint of 
$\on{OblvUnit}$.

\medskip

The construction of this functor was explained to us by S.~Raskin. To the best of our knowledge, it was the first usage of
pushforward for a morphism between lax prestacks.

\sssec{}  \label{sss:xi and psi}

Consider the following object of $\on{LaxPreStk}$, denoted $\Ran^\to$:

\medskip

\noindent For $S\in \Sch$ we let $\Ran^\to(S)$ be the (ordinary) category whose objects are
pairs $(J\subseteq I)$ of finite subsets of $\Maps(S,X)$ with $J\neq \emptyset$. Morphisms from
$(J\subseteq I)$ to $(J_1\subseteq I_1)$ are inclusions $I\subseteq I_1$ such that $J=J_1$.

\medskip

We have the maps
$$\psi:\Ran^\to \to \Ran_{\on{untl}}, \quad (J\subseteq I)\mapsto I,$$
$$\xi:\Ran^\to  \to \Ran, \quad (J\subseteq I)\mapsto J$$
$$\upsilon:\Ran\to \Ran^\to,\quad J\mapsto (J\subseteq J).$$

We claim:

\begin{lem}  \label{l:psi pseudo-proper}
The map $\psi$ is pseudo-proper.
\end{lem}

\begin{proof}

For $S\in \Sch$ and $I$ an object of $\Ran_{\on{untl}}(S)$, the fiber of $\Ran^\to(S)$ over it is the groupoid
of non-empty subsets $J\subseteq I$. For a morphism $I_1\subseteq I_2$ in $\Ran_{\on{untl}}(S)$, 
we have a map between the corresponding fibers, given by
$$(J\subseteq I_1)\mapsto (J\subseteq I_2).$$
This shows that $\Ran^\to(S)\to \Ran_{\on{untl}}(S)$ is a coCartesian fibration in groupoids. 

\medskip

Let us now show that for a given $I\subset \Maps(S,X)$, the fiber product
\begin{equation} \label{e:arr fib}
S\underset{\Ran_{\on{untl}}}\times \Ran^\to
\end{equation}
is pseudo-proper over $S$. 

\medskip

Consider the category whose objects are the data of $\CJ\to \CJ'\twoheadleftarrow I$, where $\CJ$ is a finite non-empty set.
Morphisms in this category are commutative diagrams that induce surjections on the $\CJ$'s. Now, the fiber product \eqref{e:arr fib}
equals
$$\underset{\CJ\to \CJ'\twoheadleftarrow I}{\on{colim}}\, X^{\CJ'}\underset{X^I}\times S,$$
which is proved similarly to \propref{p:express Ran}. 

\end{proof} 

\sssec{}

Thus, it follows from \corref{c:lax homology} that the functor 
$$\psi_!: \Shv^!(\Ran^\to)\to \Shv^!(\Ran_{\on{untl}}),$$
left adjoint to $\psi^!$, is defined.

\medskip

We define the functor 
$$\on{AddUnit}:\Shv^!(\Ran)\to \Shv^!(\Ran_{\on{untl}})$$ to be the composition $\psi_!\circ \xi^!$.

\medskip

\noindent {\it Warning}: We emphasize that in our definition the unital Ran space does not allow the empty set. Related to this 
is the fact that $\on{AddUnit}(0)=0$. 

\begin{rem} 
In \secref{sss:add unit expl} we will give an explicit description of the functor $\on{AddUnit}$ as a colimit of 
direct image functors for schemes.
\end{rem} 

\sssec{}

The functor $\on{AddUnit}$ can be described explicitly as follows:

\medskip

For $S\in \Sch$ fix an object $I\subset \Ran_{\on{untl}}(S)$ corresponding to a finite set $I$ of maps $S\to X$
with \emph{pairwise non-intersecting images}. I.e., the map $S\to \Ran_{\on{untl}}$ factors as
$$S\to \overset{\circ}X{}^I\to X^I \overset{\on{ins}_I}\longrightarrow \Ran \overset{\phi} \longrightarrow \Ran_{\on{untl}}.$$

\begin{prop} \label{p:inserting the unit} 
For $\CF\in \Shv^!(\Ran)$, the corresponding object $\on{AddUnit}(\CF)_{S,I}\in \Shv^!(S)$ is given by
$$\underset{\emptyset\neq J\subseteq I}\oplus\, \CF_{S,J},$$
where we regard $J$ as an $S$-point of $\Ran$ via $J\subseteq I \subset \Maps(S,X)$.
\end{prop} 

\begin{proof}

Follows from \corref{c:lax homology}. 

\end{proof}

\ssec{Another interpretation}  \label{ss:add unit simple}

\sssec{}

Consider the morphism
$$\phi:\Ran\to \Ran_{\on{untl}}.$$

This morphism is \emph{not} pseudo-proper.  However, we claim:

\begin{prop}  \label{p:add unit simple}
The functor
$$\phi_!: \Shv(\Ran)\to  \Shv(\Ran_{\on{untl}})$$
exists and identifies canonically with $\on{AddUnit}$.
\end{prop}

This proposition can be reformulated as follows:

\begin{cor}
The functor $\on{AddUnit}$ is the left adjoint of $\on{ObvlUnit}$.
\end{cor}

\sssec{}

For the proof of \propref{p:add unit simple} we note:

\begin{lem}
The (iso)morphism
$$\on{Id}_{\on{Shv}^!(\Ran)}\overset{\sim}\to (\xi\circ \upsilon)^!\simeq \upsilon^!\circ \xi^!$$
is unit of an adjunction. 
\end{lem}

\begin{proof}
Follows from the fact that the (iso)morphism
$$\xi\circ \upsilon\overset{\sim}\to \on{id}$$
defines the counit of an adjunction when evaluated on any $S\in \Sch$.
\end{proof} 

\begin{cor}
The functor $\upsilon_!:\Shv^!(\Ran)\to \Shv^!(\Ran^\to)$, left adjoint to $\upsilon^!$,
is well-defined and identifies with $\xi^!$.
\end{cor} 

\sssec{Proof of \propref{p:add unit simple}}

We have:

$$\on{AddUnit}=\psi_!\circ \xi^!\simeq \psi_!\circ \upsilon_!\simeq (\psi\circ \upsilon)_!\simeq \phi_!.$$

\qed

\sssec{}

Consider the unit of the adjunction
$$\on{Id}_{\Shv^!(\Ran)}\to \on{OblvUnit}\circ \on{AddUnit}.$$

\medskip

It follows from the construction that this natural transformation can be described as the composition
\begin{equation} \label{e:id to add}
\on{Id}_{\Shv^!(\Ran)}\simeq \upsilon^!\circ \xi^!\to 
\upsilon^! \circ \psi^! \circ \psi_! \circ \xi^!=
\phi^!\circ \psi_! \circ \xi^!=\on{OblvUnit}\circ \on{AddUnit}.
\end{equation} 

\ssec{Comparing chiral homology}

Assume that $X$ is proper (or arbitrary if we are working in the context of constructible sheaves). 

\sssec{}

Note that by interpreting $\on{AddUnit}$ as $\phi_!$, we obtain a tautological isomorphism 
\begin{equation} \label{e:nat trans 0}
\on{C}^*_c(\Ran_{\on{untl}},-) \circ \on{AddUnit} \simeq \on{C}^*_c(\Ran,-).
\end{equation} 

\begin{cor} \label{c:id to add}
Assume that $X$ is connected. Then for $\CF\in \Shv^!(\Ran)$, the map
$$\on{C}^*_c(\Ran,\CF) \to \on{C}^*_c(\Ran,\on{OblvUnit}\circ \on{AddUnit}(\CF)),$$
induced by \eqref{e:id to add}, is an isomorphism.
\end{cor} 

\begin{proof} 

We have a commutative diagram
$$
\CD
\on{C}^*_c(\Ran,\CF)  @>{\sim}>> \on{C}^*_c\left(\Ran_{\on{untl}},\on{AddUnit}(\CF)\right)  \\
@VVV   @VV{\on{id}}V   \\
\on{C}^*_c(\Ran,\on{OblvUnit}\circ \on{AddUnit}(\CF))  @>>>  \on{C}^*_c\left(\Ran_{\on{untl}},\on{AddUnit}(\CF)\right),
\endCD
$$
where the bottom horizontal arrow is an isomorphism by  \thmref{t:unital chiral homology}. This implies that
the left vertical map is an isomorphism as well.

\end{proof}

\ssec{Proof of \thmref{t:unital chiral homology}}  \label{ss:1st proof chiral}

\sssec{}

We will prove a stronger assertion:

\begin{thm}  \label{t:Ran left cofinal}
Suppose that $X$ is connected. Then the map $\phi:\Ran\to \Ran_{\on{untl}}$ is universally homologically left cofinal.
\end{thm} 

See \secref{ss:l cofinal} for the notion of universal homological left cofinality. This theorem implies \thmref{t:unital chiral homology}
by \corref{c:l cofinal ff}.

\sssec{}

The proof of \thmref{t:Ran left cofinal} is essentially given in \cite{Main Text} in the guise of the proof of Proposition 2.5.23 from {\it loc. cit.} 
Let us repeat it for completeness. 

\medskip

Given an $S$-point $I$ of $\Ran_{\on{untl}}$, consider the corresponding lax prestack (which turns out to actually be a prestack)  
$(\Ran)_{I/}$ over $S$, see \secref{ss:l cofinal} for the notation. For $S'\in \Sch$ the corresponding category (actually, a groupoid) 
$(\Ran)_{I/}(S')$ consists of 
$$(S'\to S, I'\subset \Maps(S',X) \text{ such that } I|_{S'}\subseteq I').$$

Note that this groupoid is a retract of $\Maps(S',S\times \Ran)$. Indeed, the inverse map sends $I'\mapsto I'\cup  I|_{S'}$. 

\medskip

We need to show that $(\Ran)_{I/}$ is universally homologically contractible over $S$. However, the property of being 
universally homologically contractible survives retracts. Hence, it suffices to show that $S\times \Ran$ is 
universally homologically contractible over $S$, which follows from \thmref{t:Ran contr}.

\sssec{}

We shall now prove another assertion in the spirit of \thmref{t:Ran left cofinal}, 
to be used later (concretely, in the proof of \thmref{t:pointwise duality}). 

\medskip

Fix a point $x\in X$, and consider the maps
\begin{equation} \label{e:add point}
\unn_x:\Ran\to \Ran,\quad I\mapsto I\cup x.
\end{equation}
and
\begin{equation} \label{e:add point unital}
\unn_{x,\on{untl}}:\Ran_{\on{untl}}\to \Ran_{\on{untl}},\quad I\mapsto I\cup x.
\end{equation}

We will prove:

\begin{prop}  \label{p:ins point}
The map \eqref{e:add point unital} is universally homologically left cofinal.
\end{prop}

Combining with \corref{c:l cofinal ff} we obtain:

\begin{cor} \hfill

\smallskip

\noindent{\em(a)} For $\CF\in \Shv^!(\Ran_{\on{untl}})$, the map
$$\on{C}_c^*(\Ran_{\on{untl}},\unn_{x,\on{untl}}^!(\CF))\to \on{C}_c^*(\Ran_{\on{untl}},\CF)$$
is an isomorphism.

\smallskip

\noindent{\em(b)} If $X$ is connected, for $\CF\in \Shv^!(\Ran)$ that lies in the essential image of the functor $\on{OblvUnit}$,
the map
$$\on{C}_c^*(\Ran,\unn_x^!(\CF))\to \on{C}_c^*(\Ran,\CF)$$
is an isomorphism.

\end{cor}

\begin{rem}
In the theory of chiral algebras, the assertion of the above corollary is known under the motto ``inserting of the vacuum does not change
chiral homology".
\end{rem}

\begin{proof}[Proof of \propref{p:ins point}]

Given an $S$-point $I$ of $\Ran_{\on{untl}}$, consider the corresponding lax prestack
$(\Ran_{\on{untl}})_{I/}$. We claim that the map $(\Ran_{\on{untl}})_{I/}\to S$ is
value-wise contractible, which would imply the required assertion by \lemref{l:contr contr}.

\medskip

For $S'\in \Sch_{/S}$ the category of lifts of $I$ to an $S'$-point of $(\Ran_{\on{untl}})_{I/}$ is that of 
$$(I'\subset \Maps(S',X) \text{ such that } I|_{S'}\subseteq I'\cup x).$$
We wish to show that this category is contractible.

\medskip

The above category contains a left cofinal full subcategory consisting of $I'$ for which $I|_{S'}\subset I'$. Now, the 
latter subcategory is contractible because it has an initial point, namely, one with $I'=I|_{S'}$.

\end{proof} 

\section{The augmented version of the Ran space and taking the units out}  \label{s:aug}

The material in this section is needed for the material in Part III of the paper, which in turn
is used in \secref{s:local duality} (the reduction of the cohomological product formula to a local statement). 

\medskip

In this section we will introduce yet one more version of the Ran space--the augmented one. Continuing the
analogy with associative algebras from \secref{ss:untl Ran}, in the topological context and when $X=\BR^1$, the augmented Ran space 
will be the natural recipient of the functor from \emph{unital augmented} associative algebras. 

\medskip

Note, however, that the category of unital augmented associative algebras is naturally equivalent to the category
of non-unital associative algebras. Therefore, it is natural to expect a parallel phenomenon for sheaves on the corresponding
Ran spaces: this will be reflected by \thmref{t:main}.  

\ssec{The augmented Ran space}

To motivate the definition of the unital augmented Ran space we shall once again appeal to the analogy with associative
algebras. For an associative algebra $A$ consider the corresponding $\CA\in \Shv^!(\Ran)$ so that for a finite subset
$I$ of points of $\BR$ we have
$$\CA_I=\underset{i\in I}\otimes\, A.$$

As was mentioned above, if $A$ is unital, whenever $I_1\subseteq I_2$ we have a map $\CA_{I_1}\to \CA_{I_2}$. Assume
now that $A$ is augmented, and let $A^+$ denote its augmentation ideal. Then, for a finite set $I$ and its (possibly empty)
subset $K$, we set
$$\CA_{K\subseteq I}:=\left(\underset{i\in K}\otimes\, A/A^+\right) \bigotimes \left(\underset{i\in I-K}\otimes\, A\right),$$
where $A/A^+$ is the ground ring, so tensoring by it is the identity functor.

\medskip
  
Now, whenever we have an inclusion $I_1\subseteq I_2$ such that $K_1\subseteq K_2$, we have a map
$$\CA_{K_1\subseteq I_1}\to \CA_{K_2\subseteq I_2}.$$

This is the kind of structure that will be encoded by sheaves on the augmented Ran space. 

\sssec{}

We define the unital augmented version of the Ran space, denoted $\Ran_{\on{untl,aug}}\in \on{LaxPreStk}$, as follows:

\medskip

\noindent For $S\in \Sch$ we let $\Ran_{\on{untl,aug}}(S)$ be the (ordinary) category whose objects are
pairs $(K\subseteq I)$ of finite subsets of $\Maps(S,X)$ with $I\neq \emptyset$. Morphisms from
$(K\subseteq I)$ to $(K_1\subseteq I_1)$ are inclusions $I\subseteq I_1$ such that $K\subseteq K_1$.

\sssec{}

We have the tautological map
$$\iota:\Ran_{\on{untl}}\to \Ran_{\on{untl,aug}},\quad  I\mapsto (\emptyset\subset I)$$
and its left inverse
$$\pi:\Ran_{\on{untl,aug}}\to \Ran_{\on{untl}},\quad (K\subseteq I)\mapsto I.$$

We let 
$$\on{OblvAug}:=\iota^!:\Shv^!(\Ran_{\on{untl,aug}})\to \Shv^!(\Ran_{\on{untl}})$$
denote the corresponding pullback functor. 

\begin{rem}
In terms of the analogy with associative algebras, the functor $\on{OblvAug}$ corresponds to the forgetful
functor from the category of unital augmented algebras to that of just unital algebras. 
\end{rem} 

\ssec{Adding the unit and augmentation}  \label{ss:unit and aug}

We now claim that the functor 
$$\on{AddUnit}: \Shv^!(\Ran)\to \Shv^!(\Ran_{\on{untl}}),$$
constructed in \secref{ss:add unit}, canonically factors as
$$\Shv^!(\Ran) \overset{\on{AddUnit}_{\on{aug}}}\longrightarrow \Shv^!(\Ran_{\on{untl,aug}})\overset{\on{OblvAug}}\longrightarrow 
\Shv^!(\Ran_{\on{untl}}).$$

\medskip

Note that this is in line with the situation with associative algebras: the algebra obtained by adjoining a unit to a non-unital algebra
is naturally augmented. 

\sssec{}

Consider the object $\Ran^\to_{\on{aug}}\in \on{LaxPreStk}$, defined as follows:

\medskip

\noindent For $S\in \Sch$ we let $\Ran^\to_{\on{aug}}(S)$ be the (ordinary) category whose objects are triples 
$$(K\subseteq I \supseteq J),\quad K\cap J\neq \emptyset$$ of finite subsets of $\Maps(S,X)$. Morphisms from 
$(K\subseteq I\supseteq J)$ to $(K_1\subseteq I_1\supseteq J_1)$ are inclusions $I\subseteq I_1$, such that
$K\subseteq K_1$ and $J=J_1$. 

\medskip

We have the natural maps
$$\psi_{\on{aug}}:\Ran^\to_{\on{aug}}\to \Ran_{\on{untl,aug}}, \quad (K\subseteq I\supseteq J)\mapsto (K\subseteq I)$$
and
$$\xi_{\on{aug}}:\Ran^\to_{\on{aug}}\to \Ran, \quad (K\subseteq I\supseteq J)\mapsto J.$$

\begin{lem}  \label{l:psi pseudo-proper aug}
The map $\psi_{\on{aug}}$ is pseudo-proper. 
\end{lem}

\begin{proof} 

The proof is similar to that of \lemref{l:psi pseudo-proper} except for the expression for the
fiber product 
$$S\underset{\Ran_{\on{untl,aug}}}\times \Ran^\to_{\on{aug}}$$
for a given $S$-point $K\subseteq I$ of $\Ran_{\on{untl,aug}}$. 

\medskip

The fiber product in question is 
$$\underset{\CJ\to \CJ'\twoheadleftarrow I}{\on{colim}}\, X^{J'}\underset{X^I}\times S,$$
where the colimit is taken over the full subcategory of that appearing in the proof
of \lemref{l:psi pseudo-proper}, consisting of those $\CJ\to \CJ'\twoheadleftarrow I$,
for which the images of $\CJ$ and $K$ in $\CJ'$ have a non-empty intersection. 

\end{proof} 

\sssec{}

We claim that there is a canonically defined natural transformation

\begin{equation} \label{e:nat trans 1}
(\psi_{\on{aug}})_!\circ (\xi_{\on{aug}})^! \to \pi^! \circ \on{AddUnit}.
\end{equation}
as functors 
$$\Shv^!(\Ran)\to \Shv^!(\Ran_{\on{untl,aug}}).$$

\medskip

Indeed, it comes by adjunction using the commutative diagram
$$
\CD
\Ran^\to_{\on{aug}}  @>{\pi^\to}>>  \Ran^\to @>{\xi}>>  \Ran \\
@V{\psi_{\on{aug}}}VV    @VV{\psi}V   \\
\Ran_{\on{untl,aug}}   @>{\pi}>>  \Ran_{\on{untl}},
\endCD
$$
where 
$$\pi^\to(K\subseteq I\supseteq J)=(J\subseteq I),$$
and using the fact that
$$\xi_{\on{aug}}=\xi\circ \pi^\to.$$

\sssec{}

We define the functor
$$\on{AddUnit}_{\on{aug}}:\Shv^!(\Ran)\to \Shv^!(\Ran_{\on{untl,aug}})$$ as the cofiber of the natural transformation
\eqref{e:nat trans 1}.

\sssec{}

The functor $\on{AddUnit}_{\on{aug}}$ can be described explicitly as follows:

\medskip

For $S\in \Sch$ fix an object $I\in \Ran_{\on{untl}}(S)$ corresponding to a finite set $I$ 
of maps $S\to X$ with \emph{pairwise non-intersecting images}. Let $K\subseteq I$ be a subset, 
regarded as an object of $\Ran_{\on{untl,aug}}(S)$. 

\begin{prop} \label{p:inserting the aug} 
For $\CF\in \Shv^!(\Ran)$, the object $\on{AddUnit}_{\on{aug}}(\CF)_{S,K\subset I}\in \Shv^!(S)$ is given by
$$\underset{\emptyset\neq J\subseteq (I-K)}\oplus\, \CF_{S,J},$$ where $J$ is regarded as an $S$-point of $\Ran$
via $J\subseteq I\subset \Maps(S,X)$. 
\end{prop} 

\begin{proof}

Follows from \corref{c:lax homology}. 

\end{proof}

\sssec{}

By construction, the functor $\on{AddUnit}_{\on{aug}}$ comes equipped with a natural transformation
$$\pi^! \circ \on{AddUnit}\to \on{AddUnit}_{\on{aug}},\quad \Shv^!(\Ran)\to \Shv^!(\Ran_{\on{untl,aug}}).$$

Applying the functor
$$\on{OblvAug}: \Shv^!(\Ran_{\on{untl,aug}})\to \Shv^!(\Ran_{\on{untl}})$$
we obtain a natural transformation

\begin{equation} \label{e:nat trans 2}
\on{AddUnit}\to \on{OblvAug}\circ \on{AddUnit}_{\on{aug}}.
\end{equation}

\begin{lem} \label{l:inserting the aug} 
The natural transformation \eqref{e:nat trans 2} is an isomorphism.
\end{lem}

\begin{proof}
Follows from \propref{p:inserting the aug}, using the next lemma.
\end{proof}

\begin{lem} \label{l:testing}
In order to check whether a map between two objects in the category $\Shv^!(\Ran)$ or $\Shv^!(\Ran_{\on{untl}})$ \emph{(}resp., $\Shv^!(\Ran_{\on{untl,aug}})$\emph{)}
is an isomorphism, it is enough to check that it induces on $S$-points corresponding to finite subsets of $I\subset \Maps(S,X)$
\emph{(}resp., $K\subset I\subset \Maps(S,X)$\emph{)} such that the maps $S\to X$ that comprise $I$ have 
\emph{pairwise non-intersecting images}.
\end{lem} 

\begin{proof}

Use the diagonal stratification of $X^I$.

\end{proof} 

\ssec{The functor of taking the unit out}   \label{ss:out}

As was mentioned in the beginning of this chapter, we expect the procedure of inserting the unit, viewed as taking values
in the unital augmented category, to be invertible. We will realize this in the present subsection.

\medskip

Namely, we will now consider the functor
$$\on{TakeOut}:\Shv^!(\Ran_{\on{untl,aug}})\to \Shv^!(\Ran),$$
right adjoint to $\on{AddUnit}_{\on{aug}}$ (this right adjoint exists because $\on{AddUnit}_{\on{aug}}$
is colimit-preserving).  We will show that the functor $\on{TakeOut}$ is the left inverse of the functor 
$\on{AddUnit}_{\on{aug}}$. 

\sssec{}

The functor $\on{TakeOut}$ can be explicitly described as follows. 

\begin{prop} \label{p:out}
For $\wt\CF\in \Shv^!(\Ran_{\on{untl,aug}})$, an object $S\in \Sch$ and an object $J\in \Ran(S)$, the object
$$\on{TakeOut}(\wt\CF)_{S,J}\in \Shv^!(S)$$ is given by
$$\on{Fib}\left(\wt\CF_{S,\emptyset \subset J}\to \underset{\emptyset\neq K\subseteq J}{\on{lim}}\, \wt\CF_{S,K\subseteq J}\right).$$
\end{prop} 

\begin{rem}
Although we will not need this, note that \propref{p:out} implies that the functor $\on{TakeOut}$ is colimit-preserving.
\end{rem} 

\begin{proof}

We first show:

\begin{lem} \label{l:pre out}
The functor $\Shv^!(\Ran_{\on{untl,aug}})\to \Shv^!(\Ran)$, right adjoint to the functor
$\pi^! \circ \on{AddUnit}$, is given by 
$$\on{OblvUnit}\circ \on{OblvAug}=\phi^!\circ \iota^!.$$
\end{lem} 

\begin{proof} 
By \corref{c:lax homology}, the functor $\pi^! \circ \on{AddUnit}$ is given by pull-push along the following
diagram
$$
\CD
\Ran_{\on{untl,aug}}\underset{\Ran_{\on{untl}}}\times \Ran^\to @>{\pi'}>>  \Ran^\to  @>{\xi}>> \Ran \\
@V{\wt\psi}VV   @VV{\psi}V   \\
\Ran_{\on{untl,aug}}   @>{\pi}>>  \Ran_{\on{untl}}.
\endCD
$$

Hence, its right adjoint is given by
$$((\xi\circ \pi')^!)^R\circ \wt\psi^!.$$

Note there is a canonically defined map of lax prestacks
$$\upsilon':\Ran\to \Ran_{\on{untl,aug}}\underset{\Ran_{\on{untl}}}\times \Ran^\to,\quad J\mapsto (\emptyset \subset J, J\subseteq J),$$
such that for every $S\in \Sch$, the functors $(\upsilon'(S),(\xi\circ \pi')(S))$ form an adjoint pair. Hence, by 
\lemref{l:geom adj}, the functor $(\upsilon')^!$
identifies with the right adjoint of $(\xi\circ \pi')^!$. 

\medskip

Therefore, the right adjoint to $\pi^! \circ \on{AddUnit}$ is given by
$(\wt\psi\circ \upsilon')^!$, while $\wt\psi\circ \upsilon'=\iota\circ \phi$. 

\end{proof}

We will now show that the right adjoint to the functor $(\psi_{\on{aug}})_!\circ (\xi_{\on{aug}})^!$ is calculated by the formula
\begin{equation} \label{e:out}
(\wt\CF\in \Shv^!(\Ran_{\on{untl,aug}}), S\in \Sch, J\in \Ran(S))\mapsto 
\left(\underset{\emptyset\neq K\subseteq J}{\on{lim}}\, \wt\CF_{S,K\subseteq J}\right)\in \Shv^!(S).
\end{equation}

\medskip

We shall do so by applying \lemref{l:right adjoint} to the morphism $\xi_{\on{aug}}:\Ran^\to_{\on{aug}}  \to \Ran$.
For a given  $(S,J)$ as above, note that the limit appearing 
in \lemref{l:right adjoint} is 
$$\underset{K\subseteq I}{\on{lim}}\, \wt\CF_{S,K\subseteq I},$$
where the limit is taken over the category of pairs of finite subsets $K\subseteq I$ of $\Maps(S,X)$, such that $I$ contains $J$
and $K\cap J\neq \emptyset$.  

\medskip

However, right cofinal in this category is the full subcategory consisting of those $I$, for which the 
inclusion $J\subseteq I$ is an equality. Hence, the limit in question maps isomorphically to that in 
formula \eqref{e:out}.

\medskip

To finish the proof we need to show that the maps \eqref{e:comp map} are isomorphisms in our case. The fact that the 
map $\to$ is an isomorphism is automatic, because the index category of $\emptyset \neq K\subseteq J$ is finite.
To show that the map $\leftarrow$ is an isomorphism, we argue as follows:

\medskip

For a given $g:S'\to S$ and the corresponding set $J'=g(J)\subset \Maps(S',X)$, we need to show that the limit over 
$\emptyset \neq K'\subseteq J'$ of a certain functor with values in $\Shv(S')$ maps isomorphically to the limit
of the restriction of this functor to the category $\emptyset \neq K\subseteq J$, where the map between the index
categories is given by $K\mapsto g(K)$. 

\medskip

However, we claim that the above map of index categories is right cofinal (which would imply that the maps of the
limits is an isomorphism). Namely, it is easy to see that this map is a Cartesian fibration, and each fiber is contractible
(it has a final object). 

\end{proof}

\begin{rem}  \label{r:out}

The description of the functor $\on{TakeOut}$ given by \propref{p:out} admits the following reformulation (we will not
need it in the sequel, except for the optional \secref{ss:take out expl}):

\medskip

Let $\Ran^\leftarrow$ denote the lax prestack, whose category of $S$-points is that of $\CK\subseteq \CJ\subset \Maps(S,X)$,
where both $\CK$ and $\CJ$ are non-empty, and morphisms $(\CK\subseteq \CJ)\to (\CK'\subseteq \CJ')$ are
inclusions $\CK\subseteq \CK'$ with $\CJ=\CJ'$. 

\medskip

Let $\xi'_{\on{aug}}$ and $\psi'_{\on{aug}}$ denote the maps from $\Ran^\leftarrow$ to $\Ran$ and $\Ran_{\on{untl,aug}}$ that send 
$$(\CK\subseteq \CJ)\mapsto \CJ \text{ and } (\CK\subseteq \CJ)\mapsto (\CK\subseteq \CJ),$$
respectively. 

\medskip

Then we obtain functor $\on{TakeOut}$ is canonically isomorphic to
$$\on{Fib}\left(\on{OblvUnit}\circ \on{OblvAug}\to ((\xi'_{\on{aug}})^!)^R\circ (\psi'_{\on{aug}})^!\right).$$
\end{rem} 

\ssec{Fully faithfulness of the functor $\on{AddUnit}_{\on{aug}}$}

In this subsection we will establish the following crucial result: the functor $\on{AddUnit}_{\on{aug}}$ is fully faithful.

\sssec{} \label{sss:disj images}

For $S\in \Sch$ and $I_1,I_2\subset \Maps(S,X)$, we shall say that the $S$-points $I_1$ and $I_2$ of $\Ran$ have
\emph{disjoint images} if for every $i_1\in I_1$ and every $i_2\in I_2$ the resulting two maps $S\rightrightarrows X$ have
non-intersecting images. 

\sssec{}

We now claim:

\begin{thm} \label{t:main}
The functor 
$$\on{AddUnit}_{\on{aug}}: \Shv^!(\Ran)\to \Shv^!(\Ran_{\on{untl,aug}})$$
is fully faithful. Its essential image consists of those objects $\wt\CF\in \Shv^!(\Ran_{\on{untl,aug}})$ that
satisfy the following two conditions: 

\smallskip

\begin{itemize}

\item$(*)$ For every $S\in \Sch$ and $I\in \Ran_{\on{untl}}(S)$, for $K=I\subseteq I$, 
the object $\wt\CF_{S,I\subseteq I}$ is zero. 

\smallskip

\item$(**)$ For $S\in \Sch$ and $L\in \Ran(S)$ and $(K\subseteq I)\in \Ran_{\on{untl,aug}}(S)$, 
such that $L$ and $I$ \emph{have disjoint images}, the map
$$(K\subseteq I)\to (K\cup L\subseteq I\cup L)$$ in $\Ran_{\on{untl,aug}}(S)$ gives rise to an \emph{isomorphism}
$\wt\CF_{S,K\subseteq I}\to \wt\CF_{S,K\cup L\subseteq I\cup L}$. 

\end{itemize}

\end{thm}

\begin{proof}

To prove that the functor $\on{AddUnit}_{\on{aug}}$ is fully faithful, we have to check that the unit of the adjunction
$$\CF\to \on{TakeOut} \circ \on{AddUnit}_{\on{aug}}(\CF),\quad \CF\in \Shv^!(\Ran)$$
is an isomorphism. By \lemref{l:testing}, it is enough to check that the above map induces an isomorphism
at $S$-valued points $I$ that correspond to finite subsets of $\Maps(S,X)$ with pairwise non-intersecting images. 

\medskip

By Propositions \ref{p:inserting the aug} and \ref{p:out}, the value of the right-hand side at $I$
is the fiber of the map 
\begin{equation} \label{e:kernel}
\underset{\emptyset \neq J\subseteq I}\oplus\, \CF_{S,J} \to 
\underset{\emptyset \neq K\subseteq I}{\on{lim}}\,\, \underset{\emptyset \neq J\subseteq I-K}\oplus\, \CF_{S,J},
\end{equation} 
defined as follows: 

\medskip

For a given $\emptyset \neq K\subseteq I$, the corresponding map
$$\underset{\emptyset \neq J\subseteq I}\oplus\, \CF_{S,J}\to \underset{\emptyset \neq J\subseteq I-K}\oplus\, \CF_{S,J}$$
is the projection onto those direct summands for which $J\cap K=\emptyset$.

\medskip

In terms of this description, the unit of the adjunction is given by the map 
$$\CF_{S,I}\to \underset{\emptyset \neq J\subseteq I}\oplus\, \CF_{S,J}$$
corresponding to the direct summand $J=I$.

\medskip

Now, the limit in the right-hand side of \eqref{e:kernel}
can be computed as
$$\underset{\emptyset \neq J\subseteq I}\oplus\, \,\underset{\emptyset \neq K\subseteq I-J}{\on{lim}}\, \CF_{S,J},$$
which is isomorphic to 
$$\underset{\emptyset \neq J\subsetneq I}\oplus\, \CF_{S,J},$$
because for each $J\neq I$ the category of indices $\emptyset \neq K\subseteq I-J$ is contractible.  Under this identification, the map
\eqref{e:kernel} is the projection on the summands with $J\neq I$. 

\medskip

This makes the isomorphism assertion manifest.

\medskip

It is clear from \propref{p:inserting the aug} and \lemref{l:testing}
that any object in the essential image of $\on{AddUnit}_{\on{aug}}$ satisfies
conditions (*) and (**). To finish the proof, we need to show that if $\wt\CF$ is an object of 
$\wt\CF\in \Shv^!(\Ran_{\on{untl,aug}})$ that satisfies conditions (*) and (**) such that
$\on{TakeOut}(\wt\CF)=0$, then $\wt\CF=0$.

\medskip

By \lemref{l:testing}, it is enough to show that $\wt\CF_{S,K\subseteq I}=0$ for 
$I$ corresponding to a finite subset of $\Maps(S,X)$ consisting of maps with pairwise
non-intersecting images.  By induction on the cardinality
of $I$ and conditions (*) and (**), we can assume that $\wt\CF_{S,K'\subseteq I'}=0$,
whenever $|I'|-|K'|<|I|$. In particular, it remains to prove that $\wt\CF_{S,\emptyset  \subset I}=0$
we can assume that $K=\emptyset$. 

\medskip

By \propref{p:out}, we have:
$$\on{TakeOut}(\wt\CF)_{S,I}=
\on{Fib}\left(\wt\CF_{S,\emptyset \subset I}\to \underset{\emptyset\neq K\subseteq I}{\on{lim}}\, \wt\CF_{S,K\subseteq I}\right).$$
However, by the induction hypothesis, all the terms in the above limit vanish. Hence,
$$0=\on{TakeOut}(\wt\CF)_{S,I}\simeq \wt\CF_{S,\emptyset \subset I}.$$

\end{proof} 

\sssec{}

By \lemref{l:pre out}, we have a natural transformation of functors
\begin{equation} \label{e:nat trans 3}
\on{TakeOut}\to \on{OblvUnit}\circ \on{OblvAug},\quad \Shv^!(\Ran_{\on{untl,aug}})\to \Shv^!(\Ran).
\end{equation}

We claim: 

\begin{cor}  \label{c:main}
Assume that $X$ is connected\footnote{And proper, if we are working in the context of D-modules.}.  
Let $\CG$ be an object of $\Shv^!(\Ran_{\on{untl,aug}})$ satisfying conditions $(*)$ and $(**)$. 
Then the natural transformation 
\eqref{e:nat trans 3} induces an isomorphism
$$\on{C}^*_c(\Ran,-)\circ \on{TakeOut}(\CG)\to \on{C}^*_c(\Ran,-)\circ \on{OblvUnit}\circ \on{OblvAug}(\CG).$$
\end{cor}

\begin{proof}

Note that if we precompose the natural transformation \eqref{e:nat trans 3} with the functor $\on{AddUnit}_{\on{aug}}$ 
we obtain a natural transformation
$$\on{Id}_{\Shv^!(\Ran)}\overset{\sim}\longrightarrow
\on{TakeOut}\circ \on{AddUnit}_{\on{aug}}\to \on{OblvUnit}\circ \on{OblvAug}\circ \on{AddUnit}_{\on{aug}}\simeq
\on{OblvUnit}\circ \on{AddUnit},$$
which equals the natural transformation \eqref{e:id to add}. 

\medskip

Hence, the assertion of the corollary follows from \corref{c:id to add}. 

\end{proof} 

\section{An alternative point of view on sheaves on the unital Ran space}   \label{s:untl}

The contents of this section will not be used in the sequel; it is added for the sake for completeness.

\medskip

Recall the presentation of $\Ran$ as a colimit of schemes, given by \propref{p:express Ran}. So, a datum of
a sheaf on $\Ran$ is amounts to an assignment
$$(\CI\in \on{Fin}^s) \rightsquigarrow \CF_\CI\in \Shv(X^\CI),$$
equipped with a homotopy-compatible system of isomorphisms 
$$(\CI\overset{\alpha}\twoheadrightarrow \CJ) \rightsquigarrow
(\on{diag}_\alpha)^!(\CF_\CI)\simeq \CF_\CJ,$$
where $\on{diag}_\alpha$ denotes the map $X^\CJ\to X^\CI$, induced by $\alpha$. 

\medskip

In this section we will give a similar description of the categories $\Shv^!(\Ran_{\on{untl}})$, $\Shv^!(\Ran_{\on{untl,aug}})$
and the functors $\on{AddUnit}$, $\on{AddUnit}_{\on{aug}}$ and $\on{TakeOut}$. 

\ssec{The case of the usual Ran space} 

First, we will give a different interpretation of the presentation of $\Ran$ as a colimit, given by \propref{p:express Ran}.

\sssec{}

For a set $A$ the functor
$$\CI\mapsto A^\CI,\quad (\on{Fin}^s)^{\on{op}}\to \on{Sets}$$
defines a Cartesian fibration in groupoids $A^{\on{Fin}^s}\to \on{Fin}^s$.

\medskip

Let $X^{\on{Fin}^s}$ denote the lax prestack that assigns to $S\in \Sch$ the category
$$\Maps(S,X)^{\on{Fin}^s}.$$

\medskip

In other words, the category $\Sch_{/X^{\on{Fin}^s}}$ is a Cartesian fibration in groupoids over $\Sch\times \on{Fin}^s$
with the fiber over $S\times \CI$ being $\Maps(S,X)^\CI$. 

\medskip

Note that we can think of an object $\CF\in \Shv^!(X^{\on{Fin}^s})$ as an assignment 
$$(\CI\in \on{Fin}^s)  \rightsquigarrow \CF_\CI\in \Shv(X^\CI),$$
equipped with a homotopy-compatible system of \emph{morphisms} 
\begin{equation} \label{e:IJ surj}
(\CI\overset{\alpha}\twoheadrightarrow \CJ) \rightsquigarrow
(\on{diag}_\alpha)^!(\CF_\CI)\to \CF_\CJ.
\end{equation} 

\sssec{}

Let $A^{\on{Fin}^s,\sim}$ denote the following the \emph{localization} of $A^{\on{Fin}^s}$: we invert
all arrows that are Cartesian over $\on{Fin}^s$. Since the fibers of $A^{\on{Fin}^s}$ over
objects of $\on{Fin}^s$ are groupoids, we obtain that $A^{\on{Fin}^s,\sim}$ is itself a groupoid. 

\medskip

Let $X^{\on{Fin}^s,\sim}$ denote the corresponding prestack.  

\medskip

The category $\Shv^!(X^{\on{Fin}^s,\sim})$ is a full subcategory of $\Shv^!(X^{\on{Fin}^s})$ consisting of those objects 
for which the maps \eqref{e:IJ surj} are \emph{isomorphisms}. 

\sssec{}

Note now that we have a canonical identification: 
$$X^{\on{Fin}^s,\sim}\simeq \underset{\CI\in \on{Fin}^s}{\on{colim}}\, X^\CI.$$

Combining with \propref{p:express Ran}, we obtain an identification
\begin{equation} \label{e:Ran as loc}
X^{\on{Fin}^s,\sim}\simeq \Ran.
\end{equation} 

\medskip

In the rest of this section we will develop an analog of the interpretation of $\Ran$, given by \eqref{e:Ran as loc}, to
the case of $\Ran_{\on{untl}}$ and $\Ran_{\on{untl,aug}}$. 

\ssec{The case of the unital Ran space}   \label{ss:untl expl}

\sssec{}

Let $\on{Fin}$ denote the category of finite non-empty sets (and all maps). For a set $A$ the functor
$$\CI\mapsto A^\CI,\quad (\on{Fin})^{\on{op}}\to \on{Sets}$$
defines a Cartesian fibration in groupoids $A^{\on{Fin}}\to \on{Fin}$.

\medskip

Let $X^{\on{Fin}}$ denote the lax prestack that assigns to $S\in \Sch$ the category
$$\Maps(S,X)^{\on{Fin}}.$$

\medskip

In other words, the category $\Sch_{/X^{\on{Fin}}}$ is a Cartesian fibration in groupoids over $\Sch\times \on{Fin}$
with the fiber over $S\times \CI$ being $\Maps(S,X)^\CI$. 

\medskip

By definition, an object $\CF\in \Shv^!(X^{\on{Fin}})$ is an assignment 
$$(\CI\in \on{Fin})  \rightsquigarrow \CF_\CI\in \Shv(X^\CI),$$
equipped with a homotopy-compatible system of \emph{morphisms} 
\begin{equation} \label{e:IJ}
(\CI\overset{\alpha}\longrightarrow \CJ) \rightsquigarrow
(\on{diag}_\alpha)^!(\CF_\CI)\to \CF_\CJ.
\end{equation} 

\sssec{}

For a set $A$, let $A^{\on{Fin},\sim}$ denote the following the \emph{localization} of $A^{\on{Fin}}$: we invert
all \emph{Cartesian} arrows that lie above that maps in $\on{Fin}$ given by 
\emph{surjective} maps of finite sets.   

\medskip

Let $X^{\on{Fin},\sim}$ denote the corresponding lax prestack.  

\medskip

The category $\Shv^!(X^{\on{Fin},\sim})$ is a full subcategory of $\Shv^!(X^{\on{Fin}})$ consisting of those objects 
for which the maps \eqref{e:IJ} are \emph{isomorphisms for those $\alpha$ that are surjective}. 

\sssec{}

We claim:
\begin{prop}  \label{p:express Ran untl}
There exists a canonical isomorphism
$$X^{\on{Fin},\sim}\simeq \Ran_{\on{untl}}.$$
\end{prop}

\begin{proof}

We need to show that for a set $A$, the category $A^{\on{Fin},\sim}$ is (canonically) equivalent to the category
$\on{Subsets}_f(A)$ of finite non-empty subsets of $A$.

\medskip

We have a naturally defined functor 
\begin{equation} \label{e:A fiber}
A^{\on{Fin}}\to \on{Subsets}_f(A)
\end{equation} 
that assigns to an object $(\CI\to A)\in A^{\on{Fin}}$ its image in $A$. We claim that this functor defines an equivalence from
$A^{\on{Fin},\sim}$ to  $\on{Subsets}_f(A)$.

\medskip

First, it is easy to see that \eqref{e:A fiber} is a Cartesian fibration,
and the arrows in $A^{\on{Fin}}$ that are Cartesian over \emph{surjective} maps in $\on{Fin}^{\on{op}}$ get sent to isomorphisms
in $\on{Subsets}_f(A)$. Hence, it remains to show that that the localization of every fiber of \eqref{e:A fiber} with respect to the
same class of morphisms is contractible. 

\medskip

For a subset $B\subset A$, the fiber in question is the category of pairs $(\CI,\CI\overset{f}\twoheadrightarrow B)$. We localize it with
respect to the morphisms that are surjective on the $\CI$'s. However, the above category has a final object, namely $\CI=B$, and the
canonical map from any object to it belongs to the class of morphisms with respect to which we localize. The assertion follows now
from the next general lemma:

\begin{lem}
Let $\bC$ be a category with a final object $\bc_f$. Then any localization of $\bC$ with respect to a class of morphisms that 
contains the canonical morphisms $\bc\to \bc_f$ for all $\bc\in \bC$, is contractible.
\end{lem}

\end{proof} 

\sssec{}   \label{sss:add unit expl}

Thus, from \propref{p:express Ran untl} we obtain a fully faithful embedding 
$$\Shv^!(\Ran_{\on{untl}})\hookrightarrow \Shv^!(X^{\on{Fin}}).$$

\medskip

Let us now describe the functor
\begin{equation} \label{e:add unit expl}
\Shv^!(\Ran) \overset{\on{AddUnit}}\longrightarrow \Shv^!(\Ran_{\on{untl}})\hookrightarrow \Shv^!(X^{\on{Fin}}).
\end{equation} 

As we shall see in \secref{sss:sheaf on Ran as a colimit}, for every finite non-empty set $\CJ$, the morphism $\on{ins}_\CJ:X^\CJ\to \Ran$
is pseudo-proper, and for any $\CF\in \Shv^!(\Ran)$, the canonical map 
$$\underset{\CI\in \on{Fin}^s}{\on{colim}}\, (\on{ins}_\CI)_!\circ (\on{ins}_\CI)^!(\CF)\to \CF$$
is an isomorphism. 

\medskip

So, it would suffice to describe the composition of \eqref{e:add unit expl} with the functor $(\on{ins}_\CJ)_!$ for every $\CJ$. 
What we will actually do is write explicitly the composed functor
\begin{equation} \label{e:IJ eval}
\Shv(X^\CJ) \overset{(\on{ins}_\CJ)_!}\longrightarrow \Shv^!(\Ran) \overset{\on{AddUnit}}\longrightarrow \Shv^!(\Ran_{\on{untl}})\hookrightarrow 
\Shv^!(X^{\on{Fin}})\to \Shv(X^\CI)
\end{equation}
the last arrow is the evaluation functor corresponding to a finite non-empty set $\CI$.

\medskip

Since the maps $\psi:\Ran^\to\to \Ran_{\on{untl}}$ and $\on{ins}_\CJ$ are pseudo-proper,
the composition \eqref{e:IJ eval} is computed algorithmically via base change: it is given by pull-push along 
\begin{equation} \label{e:pull push IJ}
\CD
(X^\CI\times X^\CJ)^\sim_{\supseteq}:=X^\CI\underset{\Ran_{\on{untl}}}\times \Ran^\to \underset{\Ran}\times X^\CJ  @>>>  X^\CJ \\
@VVV   \\
X^\CI.
\endCD
\end{equation}

So, it remains to describe explicitly the above prestack $(X^\CI\times X^\CJ)^\sim_{\supseteq}$.
As in \propref{p:express Ran} one shows that it is given as
$$\underset{\CI\twoheadrightarrow \CK\leftarrow \CJ}{\on{colim}}\, X^\CK,$$
where the colimit is taken over the category of finite non-empty sets $\CK$, equipped with a surjection $\CI\twoheadrightarrow \CK$ and
a map $\CJ\to \CK$.

\begin{rem}   \label{r:replace by scheme 1}
Let $(X^\CI\times X^\CJ)_{\supseteq}$ be the closed reduced subscheme of $X^\CI\times X^\CJ$, whose $k$-points are
those pairs of an $\CI$-tuple and a $\CJ$-tuple of $k$-points of $X$, for which the $\CJ$-tuple is contained in the $\CI$-tuple 
as a subset of $X(k)$. In other words, $(X^\CI\times X^\CJ)_{\supseteq}$ is the reduced subscheme underlying 
the union of images of the maps $X^\CK\to X^\CI\times X^\CJ$ for all $\CI\twoheadrightarrow \CK\leftarrow \CJ$. 

\medskip

It follows from \lemref{l:replace lego}(d) that pullback along the morphism 
$$(X^\CI\times X^\CJ)^\sim_{\supseteq}\to (X^\CI\times X^\CJ)_{\supseteq}$$
is an equivalence of categories. Hence, as in \corref{c:push-pull Ran}, 
when computing the functor \eqref{e:IJ eval} as pull-push, we can replace the diagram \eqref{e:pull push IJ} that
involves the prestack $(X^\CI\times X^\CJ)^\sim_{\supseteq}$ by the diagram
$$
\CD
(X^\CI\times X^\CJ)_{\supseteq} @>>>  X^\CJ \\
@VVV   \\
X^\CI,
\endCD
$$
which only involves schemes.

\end{rem}

\ssec{The case of the unital augmented Ran space}

In this subsection we shall describe the category $\Shv^!(\Ran_{\on{untl,aug}})$ in terms similar to those in \secref{ss:untl expl}.

\sssec{}

Let $\on{Fin}_{\on{aug}}$ denote the category of pairs $(\CK\subseteq \CI)$, where $\CI$ is a finite non-empty set and
$\CK$ is its (possibly empty) subset; morphisms in the category are defined naturally.

\medskip

For a set $A$ we define the Cartesian fibration $A^{\on{Fin}_{\on{aug}}}\to \on{Fin}_{\on{aug}}$ to correspond to the functor
$$(\on{Fin}_{\on{aug}})^{\on{op}}\to \on{Sets},\quad (\CK\subseteq \CI)\mapsto A^\CI.$$
 
\medskip

Let $A^{\on{Fin}_{\on{aug}},\sim}$ be the localization of $A^{\on{Fin}_{\on{aug}}}$, where we invert the Cartesian arrows that lie over
the arrows $(\CK_1\subseteq \CI_1)\to (\CK_2\subseteq \CI_2)$ for which both $\CK_1\to \CK_2$ and $\CI_1\to \CI_2$ are
surjective.

\sssec{}  \label{sss:untl Ran via schemes}

Let 
$$X^{\on{Fin}_{\on{aug}}} \text{ and } X^{\on{Fin}_{\on{aug}},\sim}$$
denote the corresponding lax prestacks. 

\medskip

In other words, the category $\Sch_{/X^{\on{Fin}_{\on{aug}}}}$ is a Cartesian fibration over $\Sch\times \on{Fin}_{\on{aug}}$,
with the fiber over $S\times (\CK\subseteq \CI)$ being $\Maps(S,X)^\CI$. 

\medskip

As in \propref{p:express Ran untl} one shows that the natural map
$$X^{\on{Fin}_{\on{aug}},\sim}\to \Ran_{\on{untl,aug}}$$
is an isomorphism.

\sssec{}

The category $\Shv^!(X^{\on{Fin}_{\on{aug}}})$ can be (tautologically) described as follows. It consists of assignments
\begin{equation} \label{e:on aug expl}
(\CK\subseteq \CI) \rightsquigarrow \CF_{\CK\subseteq \CI}\in \Shv(X^\CI),
\end{equation} 
equipped with a homotopy-compatible set of morphisms
\begin{equation} \label{e:IK}
\CF_{\CK_1\subseteq \CI_1}|_{X^{\CI_2}}\to \CF_{\CK_2\subseteq \CI_2}
\end{equation}
for every arrow $(\CK_1\subseteq \CI_1)\to (\CK_2\subseteq \CI_2)$ in $\on{Fin}_{\on{aug}}$.

\medskip

The category $\Shv^!(X^{\on{Fin}_{\on{aug}},\sim})$ is a full subcategory of $\Shv^!(X^{\on{Fin}_{\on{aug}}})$ that consists
of those objects for which the maps \eqref{e:IK} are isomorphisms for those maps $(\CK_1\subseteq \CI_1)\to (\CK_2\subseteq \CI_2)$ 
for which both $\CK_1\to \CK_2$ and $\CI_1\to \CI_2$ are surjective. 

\sssec{}

Let us now describe the functor $\on{AddUnit}_{\on{aug}}$. As in \secref{sss:add unit expl}, we shall describe the corresponding functor
$$\Shv(X^\CJ) \overset{(\on{ins}_\CJ)_!}\longrightarrow \Shv^!(\Ran) \overset{\on{AddUnit}_{\on{aug}}}\longrightarrow \Shv^!(\Ran_{\on{untl,aug}})\hookrightarrow 
\Shv^!(X^{\on{Fin}_{\on{aug}}})\to \Shv(X^\CI)$$
for $X^\CI\to X^{\on{Fin}_{\on{aug}}}$, corresponding to any given $(\CK\subseteq \CI)\in \on{Fin}_{\on{aug}}$. 

\medskip

As in \secref{sss:add unit expl}, this amounts to a description of the fiber product
$$(X^\CI\times X^\CJ)^\sim_{\cap\neq\emptyset}:=
X^\CI\underset{\Ran_{\on{untl,aug}}}\times \Ran^\to_{\on{aug}} \underset{\Ran}\times X^\CJ$$
as a prestack over $X^\CI\times X^\CJ$.  This prestack is described as the colimit
$$\underset{\CI\supseteq \CK \supseteq \CK' \twoheadrightarrow \CL \twoheadleftarrow \CJ'\subseteq \CJ}{\on{colim}}\, 
X^\CI\underset{X^{\CK'}}\times X^\CL\underset{X^{\CJ'}}\times X^\CJ,$$
where the colimit is taken over the category of diagrams  
$$\CI\supseteq \CK \supseteq \CK' \twoheadrightarrow \CL \twoheadleftarrow \CJ'\subseteq \CJ,\quad \CL\neq \emptyset,$$
and where the morphisms are surjections on the $\CL$'s.

\begin{rem}
The above index category is quite complicated. This is the reason that for some applications, it is better to use the definition
of the functor $\on{AddUnit}_{\on{aug}}$ as it is given in \secref{ss:unit and aug} rather than try to express it in terms of schemes
as in \secref{sss:untl Ran via schemes}.
\footnote{Exhibiting a lax prestack as a \emph{lax} colimit of schemes is analogous to choosing coordinates on a manifold. This is convenient
for some purposes, but is a nuisance for others.} 
\end{rem} 

\begin{rem}
As in Remark \ref{r:replace by scheme 1}, we can replace the prestack $(X^\CI\times X^\CJ)^\sim_{\cap\neq\emptyset}$ by the
corresponding closed subscheme
$$(X^\CI\times X^\CJ)_{\cap\neq\emptyset}\subset (X^\CI\times X^\CJ).$$
\end{rem} 

\ssec{A description of the functor $\on{TakeOut}$}  \label{ss:take out expl}

In this subsection we will describe the functor $\on{TakeOut}$ in terms of the presentation of $\Ran_{\on{untl,aug}}$
as $X^{\on{Fin}_{\on{aug}},\sim}$.

\sssec{}

Specifically, we would like to describe the composition
\begin{equation} \label{e:out expl}
\Shv^!(\Ran_{\on{untl,aug}}) \overset{(\psi_{\on{aug}})^!}\longrightarrow \Shv^!(\Ran^\to_{\on{aug}})  
\overset{((\xi_{\on{aug}})^!)^R}\longrightarrow \Shv^!(\Ran) 
\overset{\on{ins}_\CJ^!}\longrightarrow \Shv(X^\CJ).
\end{equation} 

On the one hand, \propref{p:out} readily provides the description of this functor. Namely, it sends an object $\CF\in \Shv^!(\Ran_{\on{untl,aug}})$,
thought of as in \eqref{e:on aug expl}, to 
$$\underset{\emptyset\neq \CK\subseteq \CJ}{\on{lim}}\, \CF_{\CK\subseteq \CJ}\in \Shv^!(X^\CJ).$$

\medskip

Just for fun, we will now rederive this formula by a different method. Namely, by Remark \ref{r:out}, the functor \eqref{e:out expl} identifies with
the composition
$$\Shv^!(\Ran_{\on{untl,aug}})\overset{(\psi'_{\on{aug}})^!}\longrightarrow \Shv^!(\Ran^\leftarrow) 
\overset{((\xi'_{\on{aug}})^!)^R}\longrightarrow  \Shv^!(\Ran) \overset{\on{ins}_\CJ^!}\longrightarrow \Shv(X^\CJ),$$
where $\Ran^\leftarrow$ is as in Remark \ref{r:out}.

\medskip

By \lemref{l:right adjoint base change}, the above functor is given by !-pull and (right adjoint of !-pullback)-push along the diagram
$$
\CD
X^\CJ\underset{\Ran}\times \Ran^\leftarrow @>{\psi'_\CJ}>> \Ran_{\on{untl,aug}}  \\
@V{\xi'_\CJ}VV  \\  
X^\CJ.
\endCD
$$

We claim:

\begin{prop}  \label{p:take out expl}
The functor
$$((\xi'_\CJ)^!)^R\circ (\psi'_\CJ)^!:\Shv^!(\Ran_{\on{untl,aug}})\to \Shv(X^\CJ)$$
is canonically isomorphic to the functor
$$\CF\mapsto \underset{\emptyset\neq \CK\subseteq \CJ}{\on{lim}}\, \CF_{\CK\subseteq \CJ}.$$
\end{prop} 

The rest of this subsection is devoted to the proof of \propref{p:take out expl}.

\sssec{}

Let $\on{Fin}^{\leftarrow}$ denote the category, whose objects are pairs of non-empty finite sets $\CK\subseteq \CL$, 
and whose morphisms are \emph{surjections} $\CL_1\twoheadrightarrow \CL_2$ that map $\CK_1\to \CK_2$
(not necessarily surjectively).

\medskip

For a set $A$, let $A^{\on{Fin}^{\leftarrow}}$ denote the Cartesian fibration in groupoids over $\on{Fin}^{\leftarrow}$
that assigns to $\CK\subseteq \CL$ the set $A^\CL$. 

\medskip

Let $A^{\on{Fin}^{\leftarrow},\sim}$ denote the localization of $A^{\on{Fin}^{\leftarrow}}$, where we invert Cartesian arrows lying
over those morphisms in $\on{Fin}^{\leftarrow}$, for which the map $\CK_1\to \CK_2$ is also surjective. 

\medskip

Let $X^{\on{Fin}^{\leftarrow}}$ and $X^{\on{Fin}^{\leftarrow},\sim}$ denote the corresponding lax prestacks. As in \propref{p:express Ran untl}
one shows that $X^{\on{Fin}^{\leftarrow},\sim}$ is canonically isomorphic to $\Ran^\leftarrow$. 

\medskip

The composed map
$$X^{\on{Fin}^{\leftarrow}}\to X^{\on{Fin}^{\leftarrow},\sim}\simeq \Ran^\leftarrow\to \Ran_{\on{untl,aug}}$$
factors as
$$X^{\on{Fin}^{\leftarrow}} \to X^{\on{Fin}_{\on{aug}}} \to X^{\on{Fin}_{\on{aug}},\sim} \simeq \Ran_{\on{untl,aug}}.$$

\medskip

Since the maps 
$$X^{\on{Fin}_{\on{aug}}}\to \Ran_{\on{untl,aug}} \text{ and } X^{\on{Fin}^{\leftarrow}}\to \Ran^\leftarrow$$
are \emph{localizations} when evaluated on every $S\in \Sch$, functor 
$((\xi'_\CJ)^!)^R\circ (\psi'_\CJ)^!$ is canonically isomorphic to  
!-pull and (right adjoint of !-pullback)-push along the diagram
\begin{equation} \label{e:IL}
\CD
X^\CJ\underset{\Ran}\times  X^{\on{Fin}^{\leftarrow}}  @>>>  X^{\on{Fin}_{\on{aug}}}  @>>> \Ran_{\on{untl,aug}} \\
@VVV   \\
X^\CJ.
\endCD
\end{equation}

\sssec{}

The lax prestack $X^\CJ\underset{\Ran}\times  X^{\on{Fin}^{\leftarrow}}$ attaches to $S\in \Sch$ the 
Cartesian fibration in groupoids over $\on{Fin}_{\on{aug}}$ with the fiber over $\CK\subseteq \CL$ being
$$\Maps(S,X^\CJ\underset{\Ran}\times X^\CL)\simeq \underset{\CL\twoheadrightarrow \CM \twoheadleftarrow \CJ}{\on{colim}}\, \Maps(S,X)^\CM.$$

\medskip

Let $\CQ$ denote the category of 
\begin{equation} \label{e:q}
\CK\subseteq \CL \overset{\alpha}\twoheadrightarrow \CM \overset{\beta}\twoheadleftarrow \CJ,
\end{equation} 
where the morphisms are commutative diagrams that induce surjective maps between the $\CL$'s. For a set $A$,
let $A^\CQ$ denote the Cartesian fibration in groupoids over $\CQ$, with the fiber over an object \eqref{e:q}
being $A^\CM$.

\medskip

Let $X^\CQ$ be the corresponding lax prestack. We have a canonically defined map
$$X^\CQ\to X^\CJ\underset{\Ran}\times  X^{\on{Fin}^{\leftarrow}},$$
which is a localization when evaluated on every $S\in \Sch$ (we invert the Cartesian arrows that lie over those maps in 
$\CQ$ that induce isomorphisms on the $\CK$'s and the $\CL$'s).

\medskip

Hence, we obtain that when computing the !-pull and (right adjoint of !-pullback)-push along the diagram \eqref{e:IL}, we can replace this diagram by
\begin{equation} \label{e:qq}
\CD 
X^\CQ  @>>>  \Ran_{\on{untl,aug}} \\
@VVV    \\
X^\CJ.
\endCD
\end{equation} 

\sssec{}

Note now that right-cofinal in the category $\CQ$ there is a full subcategory $\CQ'$ consisting of those objects \eqref{e:q} for which the map
$\beta$ is an isomorphism. 

\medskip

Hence, we can further replace the diagram \eqref{e:qq} by the diagram
\begin{equation} \label{e:qq'}
\CD 
X^{\CQ'}  @>>>  \Ran_{\on{untl,aug}} \\
@VVV    \\
X^\CJ.
\endCD
\end{equation} 

Note, however, that the lax prestack $X^{\CQ'}$ identifies with
$$X^\CJ\times \CQ'.$$

\medskip

Hence, we rewrite the functor in the proposition as
$$\CF\mapsto \underset{\CK\subseteq \CL \overset{\alpha}\twoheadrightarrow \CJ}{\on{lim}}\, \on{diag}_\alpha^!(\CF_{\CK\subseteq \CL}).$$

\sssec{}

Consider now the category $\CQ''$ of finite non-empty subsets of $\CJ$ (i.e., this is the category appearing in the statement of the proposition).
We have a canonically defined functor 
\begin{equation} \label{e:prime and prime}
\CQ'\to \CQ'',\quad (\CK\subseteq \CL \overset{\alpha}\twoheadrightarrow \CJ)\mapsto (\alpha(\CK)\subset \CJ).
\end{equation}

The functor
$$\CQ'\to \Shv(X^\CJ),\quad (\CK\subseteq \CL \overset{\alpha}\twoheadrightarrow \CJ)\mapsto \on{diag}_\alpha^!(\CF_{\CK\subseteq \CL})$$
factors through the above functor \eqref{e:prime and prime} and the functor 
$$\CQ''\to \Shv(X^\CJ),\quad (\CK\subseteq \CJ)\mapsto \CF_{\CK\subseteq \CJ}.$$

Hence, it remains to show that the functor \eqref{e:prime and prime} is right cofinal. However, the latter is clear as it is a Cartesian fibration,
with contractible fibers (each fiber has a final object). 

\newpage 

\centerline{\bf Part II: Verdier duality on the Ran space}

\section{Verdier duality on prestacks}  \label{s:Verdier}

The functor of Verdier duality (for schemes, or topological spaces) is not
something mysterious, see \secref{sss:Verdier defn} below. 

\medskip

Our proof of the \emph{cohomological product formula} \eqref{e:product formula prev}, is based on considering the operation
of Verdier duality on sheaves on the Ran space. The trouble is that the Ran space is not a scheme, and here
the difference between schemes and prestacks (even the nice one such as $\Ran$) will be significant. 

\medskip

In this section we will define what we mean by the Verdier duality functor on a prestack and describe it
rather explicitly for a class of prestacks (called pseudo-schemes with a finitary diagonal), for which it is reasonable to 
expect such a description.  

\medskip

The material in Sects. \ref{ss:Verdier} and \ref{ss:Verdier pushforward} is needed for the statements
of the results in \secref{s:local duality}. The material in the rest of this section is needed for the proofs
of the theorems stated in \secref{s:Verdier on Ran}. As a prerequisite for the present section, as well as 
the rest of Part II, one only needs \secref{s:prestacks} (i.e., lax prestacks will not appear). 

\ssec{Verdier duality}  \label{ss:Verdier}

In this subsection we introduce the functor of Verdier duality in the context of prestacks. 

\sssec{}

Let $\CY$ be a prestack such that the diagonal map 
$$\on{diag}_\CY:\CY\to \CY\times \CY$$  
is pseudo-proper (see \secref{sss:pseudo proper} for what this means).  In particular, the functor
$$(\on{diag}_\CY)_!:\Shv^!(\CY) \to \Shv^!(\CY\times \CY)$$
is defined and satisfies base change.

\medskip

Given two objects $\CF,\CG\in \Shv^!(\CY)$, by a \emph{pairing} between them we shall mean a map
$$\CF\boxtimes \CG\to (\on{diag}_\CY)_!(\omega_\CY).$$

\begin{rem}
The pseudo-properness assumption on $\on{diag}_\CY$ does not merely appear for technical reasons, i.e.,
to ensure the existence of $(\on{diag}_\CY)_!(\omega_\CY)$. More importantly, in actual Verdier duality,
we want to use $(\on{diag}_\CY)_*(\omega_\CY)$ rather than $(\on{diag}_\CY)_!(\omega_\CY)$, and 
while we do not necessarily know what $(\on{diag}_\CY)_*(\omega_\CY)$ is supposed to be for an 
arbitrary prestack $\CY$, the pseudo-properness is supposed to imply that it is isomorphic to 
$(\on{diag}_\CY)_!(\omega_\CY)$.
\end{rem} 

\sssec{}  \label{sss:Verdier defn}

For a given $\CG\in \Shv^!(\CY)$, we define its Verdier dual $\BD_{\CY}(\CG)\in \Shv^!(\CY)$ to represent the functor
that sends $\CF\in \Shv^!(\CY)$ to the space of pairings between $\CF$ and $\CG$. The above contravariant
functor is representable, because it sends colimits in $\Shv^!(\CY)$ to limits in $\inftygroup$. 

\medskip

In particular, for a given $\CF$, we have a canonically defined pairing
\begin{equation} \label{e:canonical pairing}
\CF\boxtimes \BD_\CY(\CF)\to  (\on{diag}_\CY)_!(\omega_\CY).
\end{equation} 

\ssec{Interaction of Verdier duality with pushforwards}    \label{ss:Verdier pushforward}

The basic property of Verdier duality on schemes is  \emph{commutation with direct images} for proper morphisms. 
In this subsection we shall start exploring what can be said about the extension of this property for maps
between prestacks. 

\sssec{} \label{sss:pseudo-proper pairing}

Let $f:\CY_1\to \CY_2$ be a map which is itself pseudo-proper, so that the functor $f_!$, left adjoint to $f^!$
is defined. If $\CF,\CG$ are objects of $\Shv^!(\CY_1)$, then a datum of a pairing between them gives rise
to a pairing between $f_!(\CF)$ and $f_!(\CG)$ via
$$f_!(\CF)\boxtimes f_!(\CG) \overset{\text{pseudo-properness}} {\underset{\sim}\longrightarrow}
(f\times f)_!(\CF\boxtimes \CG) \to (f\times f)_!\circ (\on{diag}_{\CY_1})_!(\omega_{\CY_1}) \to (\on{diag}_{\CY_2})_!(\omega_{\CY_2}).$$

In particular, if $\CY$ is pseudo-proper, a datum of a pairing between $\CF$ and $\CG$ gives rise to a pairing
\begin{equation} \label{e:pairing on homology}
\on{C}^*_c(\CY,\CF)\otimes \on{C}^*_c(\CY,\CG)\to \Lambda.
\end{equation}

\sssec{}

We obtain that a pseudo-proper map $f:\CY_1\to \CY_2$, the construction from \secref{sss:pseudo-proper pairing} gives
rise to a canonically defined map
\begin{equation} \label{e:duality and direct image}
f_!(\BD_{\CY_1}(\CF))\to \BD_{\CY_2}(f_!(\CF)).
\end{equation} 

We have the following fundamental fact:

\begin{lem}  \label{l:Verdier duality sch}
Let us be working in the context of constructible sheaves. 
Let $f:Y_1\to Y_2$ be a proper map between schemes. Then the map \eqref{e:duality and direct image} 
is an isomorphism.
\end{lem}

\begin{proof}
The assertion is well-known if $\CF$ is compact. In the general case, it follows from the fact that 
\emph{in the context of constructible sheaves}, the functor $f_!$ commutes with limits: indeed, since $f$ is proper,
we have $f_!=f_*$, while $f_*$ admits a left adjoint, namely, $f^*$. 
\end{proof} 

\sssec{}

In particular, we obtain that if $\CY$ is a pseudo-proper prestack and $\CF\in \Shv^!(\CY)$, the pairings \eqref{e:canonical pairing}
and \eqref{e:pairing on homology} give rise to a pairing:
$$\on{C}^*_c(\CY,\CF)\otimes \on{C}^*_c(\CY,\BD_\CY(\CF))\to \Lambda,$$
and thus to a map
\begin{equation} \label{e:duality and cohomology}
\on{C}^*_c(\CY,\BD_\CY(\CF)) \to \left(\on{C}^*_c(\CY,\CF)\right)^\vee,
\end{equation}
(which is the map \eqref{e:duality and direct image} for $\CY_2=\on{pt}$).

\medskip

If $\CY=Y$ is a proper scheme, then \lemref{l:Verdier duality sch} says that the map \eqref{e:duality and cohomology}
is an isomorphism. 

\begin{rem}
Note that the map \eqref{e:duality and cohomology} is not an isomorphism for a general pseudo-proper prestack. For example,
take $\CY$ to be the disjoint union of infinitely many copies of $\on{pt}$.
\end{rem} 

\ssec{Interaction of Verdier duality with other functors}

The material in the rest of this section is of technical nature and can be skipped on the first pass.

\medskip

In this subsection we will start exploring how Verdier duality interacts with the operations of external product and
pullback under an \'etale morphism. 

\sssec{Products}

Let $\CY_1$ and $\CY_2$ be two prestacks and $\CF_i\in \Shv^!(\CY_i)$. We have a tautologically defined map
\begin{equation} \label{e:product and Verdier}
\BD_{\CY_1}(\CF_1)\boxtimes \BD_{\CY_2}(\CF_2)\to \BD_{\CY_1\times \CY_2}(\CF_1\boxtimes \CF_2).
\end{equation}

The following in well-known:

\begin{lem} \label{l:product and Verdier sch compact}
Assume that the ring of coefficients $\Lambda$ has a bounded Tor-dimension. 
Let $\CY_i=Y_i$ be schemes and $\CF_i\in \Shv(Y_i)$ be bounded above with compact cohomology sheaves.
Then the map \eqref{e:product and Verdier} is an isomorphism.
\end{lem}

\sssec{}  \label{sss:etale and dual}

Let now $f:\CY_1\to \CY_2$ be an \'etale map between prestacks. We claim that for $\CF\in \Shv^!(\CY_2)$, there exists
a canonically defined map
\begin{equation} \label{e:etale Verdier}
f^!\circ \BD_{\CY_2}(\CF)\to \BD_{\CY_1}\circ f^!(\CF).
\end{equation} 

Indeed, to specify such a map is equivalent to specifying a map
\begin{equation} \label{e:construct etale Verdier}
f^!\circ \BD_{\CY_2}(\CF)\boxtimes f^!(\CF)\to (\on{diag}_{\CY_1})_!(\omega_{\CY_1}).
\end{equation} 

The latter map is constructed as follows. The fact that $f$ is \'etale implies that the diagonal map
$$\CY_1\to \CY_1\underset{\CY_2}\times \CY_1=(\CY_1\times \CY_1)\underset{\CY_2\times \CY_2}\times \CY_2$$
is an embedding of a connected component.

\medskip

In particular, we have a canonically defined map
$$(f\times f)^!\circ (\on{diag}_{\CY_2})_!(\omega_{\CY_2})\to (\on{diag}_{\CY_1})_!(\omega_{\CY_1}).$$

Now, the map in \eqref{e:construct etale Verdier} is the composition
$$f^!\circ \BD_{\CY_2}(\CF)\boxtimes f^!(\CF)=(f\times f)^!(\BD_{\CY_2}(\CF)\boxtimes \CF)\to
(f\times f)^!\circ (\on{diag}_{\CY_2})_!(\omega_{\CY_2})\to (\on{diag}_{\CY_1})_!(\omega_{\CY_1}),$$
where the second arrow is the pullback by means of $f\times f$ of the tautological map
$$ \BD_{\CY_2}(\CF)\boxtimes \CF\to (\on{diag}_{\CY_2})_!(\omega_{\CY_2}).$$

\medskip

We have:

\begin{lem} \label{l:etale Verdier sch}
Let us be working in the context of constructible sheaves, and let 
$\CY_1=Y_1$ and $\CY_2=Y_2$ be schemes. Then the map \eqref{e:etale Verdier} is an isomorphism.
\end{lem}

\begin{proof}
Same as that of \lemref{l:Verdier duality sch}. 
\end{proof} 

\ssec{Digression: properties of prestacks}

In this subsection we will introduce several notions that will help describe the category of sheaves on a prestack
and the Verdier duality functor.

\medskip

The key notion is that of \emph{pseudo-scheme}. In a sense, pseudo-schemes are ``as good as schemes", as
far as sheaves on them are concerned. 

\sssec{} \label{sss:pseudo-scheme}

Let $\CY$ be a prestack. We shall say that $\CY$ is a \emph{pseudo-scheme} if $\CY$ can be written as a colimit
\begin{equation} \label{e:prestack as a colimit}
\CY=\underset{a\in A}{\on{colim}}\, Z_a,\quad Z_a\in \Sch,
\end{equation} 
where the transition maps $f_{a_1,a_2}:Z_{a_1}\to Z_{a_2}$ are proper.\footnote{Pseudo-schemes are different from ind-schemes
in two respects: for the latter one requires that the transition maps $f_{a_1,a_2}$ be \emph{closed embeddings}, and (which is
also crucial) that the index category $A$ be \emph{filtered}.} 

\medskip

For $a\in A$, we let $\on{ins}_a:Z_a\to \CY$ denote the corresponding map. We claim:

\begin{prop} \label{p:ins is pseudo}
In the above notations, the maps $\on{ins}_a$ are pseudo-proper. 
\end{prop}

As an immediate corollary of this proposition, we obtain that the diagonal morphism of a pseudo-scheme
is pseudo-proper (we use here the fact that all our schemes are assumed separated). 

\begin{proof}

Let $\Sch_{\on{proper}}$ be the non-full subcategory of $\Sch$, where we restrict $1$-morphisms to be proper.
Let $\on{PreStk}_{\on{proper}}$ be the category of functors \footnote{The proof given below works for the category
$\Sch$ replaced by any $\infty$-category $\bC$ and $\on{PreStk}_{\on{proper}}$ replaced by any non-full subcategory
$\bC'$ of $\bC$ that has the same objects, but which is obtained by restricting $1$-morphisms to certain isomorphism
classes.}  
$$(\Sch_{\on{proper}})^{\on{op}}\to \inftygroup.$$

Restriction and left Kan extension define an adjoint pair of functors
$$\on{LKE}:\on{PreStk}_{\on{proper}}\rightleftarrows \on{PreStk}:\on{Res}.$$

Note that for $\CY\in \on{PreStk}_{\on{proper}}$, written as a colimit $\underset{a\in A}{\on{colim}}\, Z_a$, the value 
of $\on{LKE}(\CY)$ is the same colimit, but taken in $\on{PreStk}$. Hence, we obtain that pseudo-schemes 
are exactly those objects of $\on{PreStk}$ that lie in the essential image of the functor $\on{LKE}$. 

\medskip

To prove the proposition, it suffices to show that for $\CY$ as above and a pair of indices $a,b$, the fiber product
$Z_a\underset{\CY}\times Z_b$ is a pseudo-scheme and that its map to $Z_b$ comes from a morphism
in $\on{PreStk}_{\on{proper}}$.  For that it suffices to show that the functor $\on{LKE}$
commutes with finite limits. This follows from the next lemma:

\begin{lem}  \label{l:LKE exact}
Let $F:\bC'\to \bC$ be a functor that commutes with finite limits. Then so does the functor 
$$\on{LKE}:\on{Funct}((\bC')^{\on{op}},\inftygroup)\to \on{Funct}(\bC^{\on{op}},\inftygroup).$$
\end{lem}

\end{proof}

\begin{proof}[Proof of \lemref{l:LKE exact}]

It is enough to show that for any $\bc\in \bC$, the functor 
$$\Phi\mapsto \on{LKE}(\Phi)(\bc), \quad \on{Funct}((\bC')^{\on{op}},\inftygroup)\to \inftygroup$$
is pro-representable. However, it is pro-represented by the object of $\on{Pro}(\bC')$ (thought of as
an object of $\on{Pro}(\on{Funct}((\bC')^{\on{op}},\inftygroup))$ via the pro-extension of the Yoneda
embedding), corresponding to the functor
$$\bc'\mapsto \Maps_{\bC}(\bc,F(\bc')).$$

\end{proof}

\begin{rem}
The proof of \propref{p:ins is pseudo}
shows that any morphism between pseudo-schemes that comes from a morphism in $\on{PreStk}_{\on{proper}}$
is pseudo-proper. 

\medskip

More generally, for $\CY_1,\CY_2\in \on{PreStk}_{\on{proper}}$, the map
$$\Maps_{\on{PreStk}_{\on{proper}}}(\CY_1,\CY_2)\to \Maps_{\on{PreStk}}(\CY_1,\CY_2)$$ is
a monomorphism in $\inftygroup$, i.e., the embedding of a union of connected components, and its image consists
of those maps of prestacks $\CY_1\to \CY_2$ that are pseudo-proper. In particular, any \emph{isomorphism}
$\CY_1\to \CY_2$ as prestacks comes from an isomorphism in $\on{PreStk}_{\on{proper}}$. 

\medskip

In addition, if $\CY$ is a pseudo-scheme presented as in \eqref{e:prestack as a colimit}, and $S$ is a scheme equipped
with a \emph{pseudo-proper} map $S\to \CY$, then for any index $i$ such that the above map
$$S\to Z_i\to \CY$$
factors as $S\to Z_i\to \CY$, the resulting map $S\to Z_i$ is proper (see end of proof of \corref{c:finite duality}). 

\end{rem} 

\sssec{}   \label{sss:ppty pseudo}

Since the functor $\Shv^!$ takes colimits to limits, we have:
$$\Shv^!(\CY)\simeq \underset{a\in A}{\on{lim}}\, \Shv(Z_a),$$
where the transition functors in the formation of the limit are $f_{a_1,a_2}^!$. The corresponding functors
$$\Shv^!(\CY)\to \Shv(Z_a)$$
are given by $(\on{ins}_a)^!$.

\medskip

In addition, as was already noted in the proof of  \propref{p:pseudo-proper}, we also have: 
$$\Shv^!(\CY)\simeq \underset{a\in A}{\on{colim}}\, \Shv(Z_a),$$
where the transition functors are $(f_{a_1,a_2})_!$, and the colimit is taken in $\infty$-category
of cocomplete categories and colimit-preserving functors. The corresponding functors
$$\Shv(Z_a)\to \Shv^!(\CY)$$
are given by $(\on{ins}_a)_!$.

\medskip

Furthermore, as the functors $\on{ins}_a^!$ are colimit-preserving, for $\CF\in \Shv^!(\CY)$ we have a canonical
isomorphism:
\begin{equation} \label{e:object as colim}
\underset{a\in A}{\on{colim}}\,  (\on{ins}_a)_!\circ (\on{ins}_a)^!(\CF)\simeq \CF
\end{equation} 
(more precisely, the tautological map $\to$ is an isomorphism). 

\sssec{}

We will now add several adjectives to the notion of pseudo-scheme.

\medskip

We shall say that $\CY$ is a \emph{finitary pseudo-scheme} if it admits a presentation as in 
\eqref{e:prestack as a colimit} with the category $A$ being finite\footnote{We call an $(\infty,1)$-category \emph{finite} if
it can be obtained by a finite iteration of taking push-outs from the point category and the category $0\to 1$.}. 

\medskip

We recall that $\CY$ is said to be \emph{pseudo-proper} over a given scheme $S$, if it admits a presentation as in 
\eqref{e:prestack as a colimit} with the schemes $Z_\alpha$ being proper over $S$.

\medskip

We shall say that $\CY$ is a \emph{finitary pseudo-proper} over $S$ if it admits a presentation as in 
\eqref{e:prestack as a colimit} with the category $A$ being finite \emph{and} with the schemes $Z_\alpha$ being proper over $S$.

\medskip

The above notions have evident relative versions. So, we obtain the notion of a morphism $f:\CY_1\to \CY_2$ between
prestacks to be: (i) pseudo-schematic; (ii) finitary pseudo-schematic; (iii) pseudo-proper; (iv) finitary pseudo-proper.

\medskip

Recall that \corref{c:pseudo-proper} implies that for a pseudo-proper map  $f:\CY_1\to \CY_2$, the functor $f^!$
admits a left adjoint $f_!$. The following results from the proof of \propref{p:pseudo-proper}:

\begin{lem} \label{l:lax homology finite}
Let us be working in the context of constructible sheaves. 
Let $f:\CY_1\to \CY_2$ be finitary pseudo-proper. Then the functor $f_!$ commutes with \emph{limits}.
\end{lem} 

\sssec{}   \label{sss:strong ps-sch}

As was shown in \propref{p:ins is pseudo}, if $\CY$ is a pseudo-scheme presented as in \eqref{e:prestack as a colimit}, then the maps 
$\on{ins}_a:Z_a\to \CY$ and hence the diagonal map $\on{diag}_\CY:\CY\to \CY\times \CY$ are pseudo-proper. 

\medskip

The following is easy:

\begin{lem}
For a pseudo-scheme presented as in \eqref{e:prestack as a colimit}, the following conditions are equivalent:

\smallskip

\noindent{\em(i)} For every $a\in A$, the map $\on{ins}_a:Z_a\to \CY$ is \emph{finitary} pseudo-proper.

\smallskip

\noindent{\em(ii)} The diagonal map $\on{diag}_\CY:\CY\to \CY\times \CY$ is \emph{finitary} pseudo-proper.

\end{lem}

We shall say that pseudo-scheme has a \emph{finitary diagonal} if it satisfies the equivalent conditions of the above lemma. 

\medskip 

It will turn out that pseudo-schemes with a finitary diagonal are well adapted to the operation of Verdier duality.



\sssec{}  

We proceed with several more observations. 

\medskip

\begin{lem}  \label{l:replace lego}
Let $f:\CY_1\to \CY_2$ be a pseudo-proper map between prestacks. 

\smallskip

\noindent{\em(a)} If $f$ is \emph{injective} (i.e., the map of $\infty$-groupoids $\CY_1(S)\to \CY_2(S)$
is a \emph{monomorphism} for any $S\in \Sch$), then the functor $f_!$ is fully faithful. 

\smallskip

\noindent{\em(b)} If $f$ is \emph{finitary pseudo-proper}, 
there exists a uniquely defined reduced closed sub-prestack $\CY'_2\subset \CY_2$
(to be thought of as the reduced scheme-theoretic image of $f$), 
such that the map $f|_{(\CY_1)_{\on{red}}}$ factors through a map $f':(\CY_1)_{\on{red}}\to \CY'_2$,
the latter being surjective on $k$-points. 

\smallskip

\noindent{\em(c)} If $f$ is finitary pseudo-proper and surjective on $k$-points, then the functor
$f^!$ is conservative.

\smallskip

\noindent{\em(d)} If $f$ is injective, finitary pseudo-proper and surjective on $k$-points, then the functors
$(f_!,f^!)$ are mutually inverse equivalences.  

\end{lem} 

\begin{proof}

Point (i) follows from base change, since the assumption that $f$ be injective is equivalent to the fact that
the diagonal map 
$$\CY_1\to \CY_1\underset{\CY_2}\times \CY_1$$
be an isomorphism.

\medskip

For point (ii) we can assume that $\CY_2=S\in \Sch$. If $\CY_1$ is the (finite) colimit of the schemes
$Z_a$, we take $\CY'_2=:S'$ to be the union of the images of the schemes $Z_a$ in $S$, equipped with
the reduced structure.

\medskip

For point (iii) we can again take $\CY_2=S\in \Sch$. Then the assumption implies that $S$ can be written
as a union of locally closed subschemes $S_\alpha$, for each of which there exists a scheme $S'_\alpha$,
equipped with a map $S'_\alpha\to \CY_1$, such that the composed map 
$$S'_\alpha\to \CY_1\to S$$
factors through a map $S'_\alpha\to S_\alpha$, the latter being surjective on $k$-points. Now, the assertion
follows from the fact that pullback along a map of schemes that is surjective on $k$-points is a conservative
functor on $\Shv(-)$. 

\medskip

Point (iv) is a combination of (i) and (iii).

\end{proof} 

\ssec{Consequences for Verdier duality}

In this subsection we will assume that our prestacks are \emph{pseudo-schemes with a finitary diagonal} (see \secref{sss:strong ps-sch} for
what this means).  We will show that the Verdier duality functor on such prestacks can be described explicitly in terms of 
Verdier duality on schemes. 

\medskip

In addition, we will show that Verdier duality commutes with direct images under finitary pseudo-proper maps. 
 
\sssec{}

Let $\CY$ be written as in \eqref{e:prestack as a colimit}. We claim: 

\begin{prop} \label{p:on fin-proper diag}
In the above notations, for every $a\in A$, the natural transformation
$$(\on{ins}_a)_!\circ \BD_{Z_a}\to \BD_\CY\circ (\on{ins}_a)_!$$
of \eqref{e:duality and direct image} is an isomorphism.
\end{prop}

Note that this proposition gives an expression for what the Verdier duality functor \emph{actually does} 
on some particular objects of $\Shv^!(\CY)$. 

\begin{proof}

Let $\CF$ be an object of $\Shv(Z_a)$, and let $\CG$ be an object of $\Shv^!(\CY)$. It suffices to show that
for any index $b$, the map
\begin{multline*}
\Maps_{\Shv(Z_b)}(\on{ins}_b^!(\CG),\on{ins}_b^!\circ (\on{ins}_a)_!\circ \BD_{Z_a}(\CF))\to \\
\to \Maps_{\Shv^!(\CY\times Z_b)}\left((\on{ins}_a)_!(\CF) \boxtimes \on{ins}_b^!(\CG),
(\on{Graph}_{\on{ins}_b})_!(\omega_{Z_b})\right)
\end{multline*}
is an isomorphism.  We will prove more generally that for any $\CH\in \Shv(Z_b)$, the map
\begin{multline}  \label{e:a&b}
\Maps_{\Shv(Z_b)}(\CH,\on{ins}_b^!\circ (\on{ins}_a)_!\circ \BD_{Z_a}(\CF))\to \\
\to \Maps_{\Shv^!(\CY\times Z_b)}\left((\on{ins}_a)_!(\CF)\boxtimes \CH,
(\on{Graph}_{\on{ins}_b})_!(\omega_{Z_b})\right)
\end{multline}
is an isomorphism.  

\medskip

Consider the prestack $\wt{Z}_{a,b}:=Z_a\underset{\CY}\times Z_b$; let $\wt{q}_a$, $\wt{q}_b$ and $\wt{q}_{a,b}$ 
denote its maps to $Z_a$, $Z_b$ and $Z_a\times Z_b$, respectively.  By base change (i.e., \corref{c:pseudo-proper}, due to the
fact that $\on{ins}_a$ is pseudo-proper), we can rewrite the left-hand side
in \eqref{e:a&b} as
\begin{equation} \label{e:a&b lhs}
\Maps_{\Shv(Z_b)}(\CH,(\wt{q}_b)_!\circ \wt{q}_a^!\circ \BD_{Z_a}(\CF))
\end{equation}
and the right-hand side as 
\begin{equation} \label{e:a&b rhs}
\Maps_{\Shv^!(Z_a\times Z_b)}\left(\CF\boxtimes \CH,
(\wt{q}_{a,b})_!(\omega_{\wt{Z}_{a,b}})\right),
\end{equation}
respectively. 

\medskip

The map from \eqref{e:a&b lhs} to \eqref{e:a&b rhs} comes from the map in $\Shv(Z_a\times Z_b)$
$$\CF \boxtimes \left((\wt{q}_b)_!\circ \wt{q}_a^!\circ \BD_{Z_a}(\CF)\right) \to (\wt{q}_{a,b})_!(\omega_{\wt{Z}_{a,b}}),$$
which is obtained by applying the functor $(\on{id}\times \wt{q}_b)_!$ to the map in $\Shv(Z_a\times \wt{Z}_{a,b})$
$$\CF \boxtimes \left(\wt{q}_a^!\circ \BD_{Z_a}(\CF)\right) \to (\on{Graph}_{\wt{q}_a})_!(\omega_{\wt{Z}_{a,b}}),$$
and which is in turned obtained by applying $(\on{id}\times \wt{q}_a)^!$ to the map
$$\CF\boxtimes \BD_{Z_a}(\CF) \to (\on{diag}_{Z_a})_!(\omega_{Z_a})$$
of \eqref{e:canonical pairing}.

\medskip

Write
$$\wt{Z}_{a,b}=\underset{i}{\on{colim}}\, Z^i_{a,b},$$
where the index $i$ runs over a finite category and where $Z^i_{a,b}$ are schemes proper over $Z_a$ and $Z_b$. Let
$q_a^i,q_b^i,q_{a,b}^i$ denote the maps from $Z^i_{a,b}$ to $Z_a$, $Z_b$ and $Z_a\times Z_b$, respectively.

\medskip

We have:
$$(\wt{q}_b)_!\circ \wt{q}_a^!\circ \BD_{Z_a}(\CF)\simeq \underset{i}{\on{colim}}\, 
(q^i_b)_! \circ (q^i_a)^! \circ \BD_{Z_a}(\CF)$$
and 
$$(\wt{q}_{a,b})_!(\omega_{\wt{Z}_{a,b}})\simeq \underset{i}{\on{colim}}\, 
(q^i_{a,b})_!(\omega_{Z^i_{a,b}}).$$

The map from \eqref{e:a&b lhs} to \eqref{e:a&b rhs} fits into a commutative diagram
$$
\CD
\Maps_{\Shv(Z_b)}(\CH,(\wt{q}_b)_!\circ \wt{q}_a^!\circ \BD_{Z_a}(\CF))   @>>>  
\Maps_{\Shv^!(Z_a\times Z_b)}\left(\CF\boxtimes \CH, (\wt{q}_{a,b})_!(\omega_{\wt{Z}_{a,b}})\right)  \\
@A{\sim}AA    @AA{\sim}A   \\
\Maps_{\Shv(Z_b)}\left(\CH, \underset{i}{\on{colim}}\, (q^i_b)_! \circ (q^i_a)^! \circ \BD_{Z_a}(\CF)\right) & & 
\Maps_{\Shv^!(Z_a\times Z_b)}\left(\CF\boxtimes \CH, \underset{i}{\on{colim}}\, (q^i_{a,b})_!(\omega_{Z^i_{a,b}})\right)  \\
@AAA    @AAA    \\
\underset{i}{\on{colim}}\, \Maps_{\Shv(Z_b)}\left(\CH, (q^i_b)_! \circ (q^i_a)^! \circ \BD_{Z_a}(\CF)\right) @>>>
\underset{i}{\on{colim}}\, 
\Maps_{\Shv^!(Z_a\times Z_b)}\left(\CF\boxtimes \CH, (q^i_{a,b})_!(\omega_{Z^i_{a,b}})\right),
\endCD
$$
where the bottom horizontal arrow comes from maps
\begin{equation} \label{e:map for i}
\Maps_{\Shv(Z_b)}\left(\CH, (q^i_b)_! \circ (q^i_a)^! \circ \BD_{Z_a}(\CF)\right)\to 
\Maps_{\Shv^!(Z_a\times Z_b)}\left(\CF\boxtimes \CH, (q^i_{a,b})_!(\omega_{Z^i_{a,b}})\right),
\end{equation}
 defined in the same way as the map from \eqref{e:a&b lhs} to \eqref{e:a&b rhs}.

\medskip

Now, in the above diagram the vertical maps are isomorphisms because the index category was finite. The bottom
horizontal arrow is an isomorphism because the maps \eqref{e:map for i} are isomorphisms, due to the properness
of $q^i_b$.  

\medskip

Hence, the top horizontal arrow is an isomorphism, as required.

\end{proof} 

\sssec{}

As a consequence of \propref{p:on fin-proper diag}, we obtain the following expression for the Verdier duality functor on
the prestack $\CY$ in terms of the schemes $Z_a$:

\begin{propconstr} \label{p:expl Verdier}
In the above notations, for $\CF\in \Shv^!(\CY)$ 
we have a canonical isomorphism
$$\BD_{\CY}(\CF)\simeq \underset{a\in A^{\on{op}}}{\on{lim}}\, (\on{ins}_a)_!\circ \BD_{Z_a}\circ (\on{ins}_a)^!(\CF).$$
\end{propconstr}

\begin{proof}

By \eqref{e:object as colim}, we have:
$$\CF\simeq  \underset{a\in A}{\on{colim}}\,  (\on{ins}_a)_!\circ (\on{ins}_a)^!(\CF).$$

\medskip

The functor of Verdier duality maps colimits into limits. Hence,
$$\BD_{\CY}(\CF)\simeq \underset{a\in A}{\on{lim}}\,  \BD_{\CY}\circ (\on{ins}_a)_!\circ \on{ins}_a^!(\CF).$$

Now, the assertion of the proposition follows from \propref{p:on fin-proper diag}.

\end{proof} 

\sssec{}

The next assertion shows that, when dealing with maps between pseudo-schemes, each with a finitary diagonal, Verdier duality commutes
with direct images under finitary pseudo-proper maps: 

\begin{cor} \label{c:finite duality}
Let us be working in the context of constructible sheaves. 
Let $f:\CY_1\to \CY_2$ be a finitary pseudo-proper map between pseudo-schemes, each having a finitary diagonal. 
Then the natural transformation
$$f_!\circ \BD_{\CY_2}\to \BD_{\CY_1}\circ f_!$$
of \eqref{e:duality and direct image} is an isomorphism.
\end{cor}

\begin{proof}

Fix presentations
$$\CY_1= \underset{a_1\in A_1}{\on{colim}}\, Z_{1,a_1},\quad Z_{1,a_1}\in \Sch$$
and
$$\CY_2=\underset{a_2\in A_2}{\on{colim}}\, Z_{2,a_2},\quad Z_{2,a_2}\in \Sch$$
as in \eqref{e:prestack as a colimit}.

\medskip

For $\CF\in \Shv^!(\CY_1)$ we have
$$\CF\simeq  \underset{a_1\in A_1}{\on{colim}}\,  (\on{ins}_{a_1})_!\circ (\on{ins}_{a_1})^!(\CF),$$
and hence
$$f_!(\CF) \simeq \underset{a_1\in A_1}{\on{colim}}\,  (f\circ \on{ins}_{a_1})_!\circ (\on{ins}_{a_1})^!(\CF).$$

By \propref{p:expl Verdier}, we have
$$\BD_{\CY_1}(\CF)\simeq \underset{a_1\in A_1}{\on{lim}}\, (\on{ins}_{a_1})_!\circ \BD_{Z_{a_1}}\circ (\on{ins}_{a_1})^!(\CF).$$

By \lemref{l:lax homology finite}, we have
$$f_!\circ \BD_{\CY_1}(\CF)\simeq \underset{a_1\in A_1}{\on{lim}}\, (f\circ \on{ins}_{a_1})_!\circ \BD_{Z_{a_1}}\circ (\on{ins}_{a_1})^!(\CF),$$
and since the functor $\BD_{\CY_2}$ takes colimits to limits, we also have:
$$\BD_{\CY_2}\circ f_!(\CF)\simeq \underset{a_1\in A_1}{\on{lim}}\,  \BD_{\CY_2}\circ (f\circ \on{ins}_{a_1})_!\circ (\on{ins}_{a_1})^!(\CF).$$

\medskip

Thus, we obtain that it suffices to show that for a given index $a_1\in A_1$ and $\CF_1\in \Shv(Z_{a_1})$, the map
$$(f\circ \on{ins}_{a_1})_!\circ \BD_{Z_{a_1}}(\CF_1)\to \BD_{\CY_2}\circ (f\circ \on{ins}_{a_1})_!(\CF_1)$$
(which is the map \eqref{e:duality and direct image} for the morphism $f\circ \on{ins}_{a_1}:Z_{a_1}\to \CY_2$)
is an isomorphism.

\medskip

Let $a_2\in A_2$ be some index so that the morphism $f\circ \on{ins}_{a_1}$ factors as 
$$Z_{a_1}\overset{g}\longrightarrow Z_{a_2} \overset{\on{ins}_{a_2}}\longrightarrow \CY_2.$$

By \propref{p:on fin-proper diag} applied to the morphism $\on{ins}_{a_2}$, it suffices to show that the 
map $g:Z_{a_1}\to Z_{a_2}$ is proper. 

\medskip

Write 
$$Z_{a_2}\underset{\CY_2}\times \CY_1\simeq \underset{b}{\on{colim}}\, W_b,\quad W_b\in \Sch$$
where the schemes $W_b$ are proper over $Z_{a_2}$. Let $b$ be an index such that the map 
$$Z_{a_1}\to Z_{a_2}\underset{\CY_2}\times \CY_1$$
factors through a map $h:Z_{a_1}\to W_b$.  It suffices to show that the map $h$ is proper. 

\medskip

Write
$$W_b\underset{\CY_1}\times Z_{a_1}\simeq \underset{c}{\on{colim}}\, V_c,\quad V_c\in \Sch,$$
where the schemes $V_c$ are proper over $W_b$. Let $c$ be an index such that the map
$$Z_{a_1}\to W_b\underset{\CY_1}\times Z_{a_1}$$ factors through a map $Z_{a_1}\to V_c$.

\medskip

It suffices to show that the latter map $Z_{a_1}\to V_c$ is proper. But this is so because it admits
a left inverse, namely,
$$V_c\to W_b\underset{\CY_1}\times Z_{a_1}\to Z_{a_1}.$$

\end{proof}

\sssec{}

Finally, we have the following assertion:

\begin{lem} \label{l:etale and dual}
Let us be working in the context of constructible sheaves. Let $f:\CY_1\to \CY_2$ be an \'etale map. 
Then the natural transformation \eqref{e:etale Verdier} is an isomorphism.
\end{lem}

\begin{proof}

Fix a presentation 
$$\CY_2=\underset{a\in A}{\on{colim}}\, Z_{2,a},\quad Z_{2,a}\in \Sch$$
as in \eqref{e:prestack as a colimit}, and set
$$Z_{1,a}:=\CY_1\underset{\CY_2}\times Z_{2,a},$$
so that
$$\CY_1\simeq \underset{a\in A}{\on{colim}}\, Z_{1,a}.$$

Let 
$$\on{ins}_{1,a}:Z_{1,a}\to \CY_1 \text{ and } \on{ins}_{2,a}:Z_{2,a}\to \CY_2$$
denote the resulting maps. Let $f_a$ denote the map $Z_{1,a}\to Z_{2,a}$; by 
assumption these maps are \'etale.

\medskip

For $\CF\in \Shv^!(\CY_2)$, we have
$$\BD_{\CY_2}(\CF)\simeq \underset{a}{\on{lim}}\, (\on{ins}_{2,a})_!\circ \BD_{Z_{2,a}}\circ (\on{ins}_{2,a})^!(\CF)$$
and
$$\BD_{\CY_1}\circ f^!(\CF) \simeq 
\underset{a}{\on{lim}}\, (\on{ins}_{1,a})_!\circ \BD_{Z_{1,a}}\circ (\on{ins}_{1,a})^!\circ f^!(\CF).$$

In terms of these identification, the map \eqref{e:etale Verdier} equals the composition
\begin{multline}  \label{e:comp etale and Verdier}
f^! \left( \underset{a}{\on{lim}}\, (\on{ins}_{2,a})_!\circ \BD_{Z_{2,a}}\circ (\on{ins}_{2,a})^!(\CF)\right) \to
\underset{a}{\on{lim}}\, f^! \circ (\on{ins}_{2,a})_!\circ \BD_{Z_{2,a}}\circ (\on{ins}_{2,a})^!(\CF) \simeq \\
\simeq \underset{a}{\on{lim}}\,  (\on{ins}_{1,a})_!\circ f_a^! \circ \BD_{Z_{2,a}}\circ (\on{ins}_{2,a})^!(\CF) \to  
 \underset{a}{\on{lim}}\,  (\on{ins}_{1,a})_!\circ \BD_{Z_{1,a}}\circ f_a^! \circ (\on{ins}_{2,a})^!(\CF) = \\
=\underset{a}{\on{lim}}\, (\on{ins}_{1,a})_!\circ \BD_{Z_{1,a}}\circ (\on{ins}_{1,a})^!\circ f^!(\CF),
\end{multline} 
where the third arrow is the natural transformation \eqref{e:etale Verdier} for the morphism $f_a$,
and hence is an isomorphism by \lemref{l:etale Verdier sch}. 

\medskip

Hence, it remains to show that the first arrow in \eqref{e:comp etale and Verdier} is an isomorphism.
However, this follows from the fact that in the context of constructible sheaves, the functor of pullback
commutes with limits, see \lemref{l:constr limits}.

\end{proof}

\section{The case of the Ran space, statements of Verdier duality results}  \label{s:Verdier on Ran}

As was mentioned in the preamble to \secref{s:Verdier}, for the proof of the cohomological product formula, we will need to perform
Verdier duality on the Ran space. More precisely, we will need to show that for some specific object in $\Shv^!(\Ran)$,
the compactly supported cohomology of its Verdier dual maps isomorphically to the dual of its own compactly supported 
cohomology. Now, such an isomorphism would fail in general (not surprisingly so, because  $\Ran$ is ``infinite", which
technically means \emph{not finitary}). Our task in this section we will be to give sufficient conditions on an object of
$\Shv^!(\Ran)$ so that the isomorphism in question does hold. 

\medskip

The main results of this section are stated in Sects. \ref{ss:Verdier on Ran} and \ref{ss:open variant}; the prerequisites 
for these subsections are \secref{s:prestacks} and Sects. \ref{ss:Verdier} and \ref{ss:Verdier pushforward}. The material in
the rest of the present section is needed for the proofs of the results stated in Sects. \ref{ss:Verdier on Ran} and \ref{ss:open variant};
in order to be able to understand it, the reader needs to be familiar with the entirety of \secref{s:Verdier}. 

\ssec{The category of sheaves on the Ran space}  \label{ss:shv on Ran}

In this subsection we will give a more explicit description of the category of sheaves on the Ran space. Namely, for
a finite non-empty set $\CI$, we have a pair of adjoint functors
$$(\on{ins}_\CI)_!:\Shv^!(X^\CI)\rightleftarrows \Shv^!(\Ran):(\on{ins}_\CI)^!.$$

We will give a more explicit description of the composed functor $(\on{ins}_\CJ)^!\circ (\on{ins}_\CI)_!$ for a pair
of finite sets $\CI$ and $\CJ$.

\medskip

The material in this subsection is needed for the proof of Theorems \ref{t:Verdier on Ran}, 
\ref{t:Verdier Ran open} and \ref{t:products}. But it is not necessary for the statement of these theorems,
and so can be skipped on the first pass. 

\sssec{}  \label{sss:sheaf on Ran as a colimit}

Recall that we can represent $\Ran$ as 
$$\underset{\CI\in (\on{Fin}^s)^{\on{op}}}{\on{colim}}\, X^{\CI},$$
where $\on{Fin}^s$ is the category of finite non-empty sets and surjective maps and that
$\on{ins}_\CI$ denotes the resulting map $X^\CI\to \Ran$.

\medskip

Thus, we obtain that $\Ran$ is a pseudo-scheme (see \secref{sss:pseudo-scheme} for what this means).  In particular,
by \propref{p:ins is pseudo}, the morphisms $\on{ins}_\CI$ are pseudo-proper (for an alternative proof of this result in the
case of $\Ran$, see \secref{sss:diag Ran} below). By \secref{sss:ppty pseudo}, the category $\Shv^!(\Ran)$ can be described as 
$$\Shv^!(\Ran)\simeq \underset{\CI\in \on{Fin}^s}{\on{lim}}\, \Shv^!(X^\CI),$$
and also
$$\Shv^!(\Ran)\simeq \underset{\CI\in \on{Fin}^s}{\on{colim}}\, \Shv(X^\CI).$$

In particular, for $\CF\in \Shv^!(\Ran)$, we have a canonical isomorphism
\begin{equation} \label{e:sheaf on Ran as a colimit}
\underset{\CI\in \on{Fin}^s}{\on{colim}}\, (\on{ins}_\CI)_!\circ (\on{ins}_\CI)^!(\CF)\to \CF
\end{equation} 
(more precisely, the tautological map $\to$ is an isomorphism). 

\sssec{}  \label{sss:diag Ran}

We claim that $\Ran$ is a pseudo-scheme with a finitary diagonal. Indeed, for a pair of
finite sets $\CI$ and $\CJ$, 
the fiber product
$$\wt{X}{}^{\CI,\CJ}:=X^\CI\underset{\Ran}\times X^\CJ\simeq (X^\CI\times X^\CJ)\underset{\Ran\times \Ran}\times \Ran$$
can be described as follows:
\begin{equation} \label{e:prod over Ran}
\wt{X}{}^{\CI,\CJ}\simeq \underset{\CI\twoheadrightarrow \CK \twoheadleftarrow \CJ}{\on{colim}}\, X^\CK.
\end{equation}

\medskip

Let $\wt{q}_\CI$ and $\wt{q}_\CJ$ denote the maps from $\wt{X}{}^{\CI,\CJ}$ to $X^\CI$ and $X_\CJ$, respectively. 
By \corref{c:pseudo-proper}
we have
$$(\on{ins}_\CJ)^!\circ (\on{ins}_\CI)_!\simeq (\wt{q}_\CJ)_!\circ (\wt{q}_\CI)^!,\quad \Shv(X^\CI)\to \Shv(X^\CJ).$$

\medskip

However, in order to have a better understanding of the category $\Shv^!(\Ran)$, one should describe the above
functors more explicitly, as we shall presently do.  

\sssec{}  \label{sss:coincide}

Let  $X^{\CI,\CJ}\subset X^\CI\times X^\CJ$ be the reduced closed subscheme, whose set of $k$-points corresponds to 
those pairs of an $\CI$-tuple and a $\CJ$-tuple of $k$-points of $X$, for which the corresponding \emph{subsets} of $X(k)$ 
coincide. 

\medskip

In other words, $X^{\CI,\CJ}$ is the reduced subscheme underlying the union of the closed subsets
$$\on{Graph}_{\on{diag}_{\alpha,\beta}},$$
where the union runs over the set of isomorphism classes of surjections
$$\CI\overset{\alpha}\twoheadrightarrow \CK \overset{\beta}\twoheadleftarrow \CJ,$$
and $\on{diag}_{\alpha,\beta}$ denotes the resulting map $X^\CK\to X^\CI\times X^\CJ$. 

\medskip

We have an evident map 
$$g:\wt{X}{}^{\CI,\CJ}\to X^{\CI,\CJ}.$$

\begin{lem} \label{l:glue}
The functors $(g_!,g^!)$ are mutually inverse equivalences of categories.
\end{lem} 

\begin{proof}
Follows from \lemref{l:replace lego}(d).
\end{proof} 

Let 
$$X^\CI \overset{q_\CI} \longleftarrow X^{\CI,\CJ} \overset{q_\CJ}\longrightarrow X^\CJ$$
denote the two maps. We obtain:

\begin{cor} \label{c:push-pull Ran}
The functor 
$$(\on{ins}_\CJ)^!\circ (\on{ins}_\CI)_!\simeq (\wt{q}_\CJ)_!\circ (\wt{q}_\CI)^!$$ is canonically isomorphic to 
$(q_\CJ)_!\circ (q_\CI)^!$.
\end{cor}

\ssec{Verdier duality on the Ran space}  \label{ss:Verdier on Ran}

In thus subsection we will formulate two main results pertaining to Verdier duality on the Ran space, Theorems
\ref{t:Verdier on Ran} and \ref{t:products}. \thmref{t:Verdier on Ran} (or rather its generalization 
\thmref{t:Verdier Ran open}) will be used in \secref{s:local duality} to deduce the cohomological product formula from a
\emph{local duality} statement.  \thmref{t:products} will be used in \secref{s:pointwise} to establish a crucial factorization
property. 

\medskip

For the rest of this section we will be working in the context of constructible sheaves. 

 
\sssec{}

We will apply the discussion of \secref{ss:Verdier} to the case $\CY=\Ran$, i.e., the Ran space of $X$, where
$X$ is a (separated) scheme.  

\medskip

Thus, we obtain: 

\begin{itemize}

\item For two objects $\CF,\CG\in \Shv^!(\Ran)$, there is a notion of pairing between map, by which we mean a map 
$$\CF\boxtimes \CG\to  (\on{diag}_{\Ran})_!(\omega_{\Ran});$$

\item Given a pairing, we obtain a pairing between $\CF$ and $\CG$, 
$$\on{C}^*_c(\Ran,\CF) \otimes \on{C}^*_c(\Ran,\CG)\to k$$
(under the additional assumption that $X$ be proper if we are in the context of arbitrary D-modules).

\item For $\CF\in \Shv^!(\Ran)$, we have a well-defined object $\BD_{\Ran}(\CF)\in \Shv^!(\Ran)$;

\smallskip

\item There is a canonically defined pairing
$$\on{C}^*_c(\Ran,\CF) \otimes \on{C}^*_c(\Ran, \BD_{\Ran}(\CF))\to k.$$ 

\end{itemize} 

In addition, from \propref{p:expl Verdier} and \secref{sss:diag Ran}, we obtain: 

\begin{cor} \label{c:Verdier on Ran}
For a given $\CF\in \Shv^!(\Ran)$, there exists a canonical isomorphism
$$\BD_{\Ran}(\CF)\simeq \underset{\CI\in \on{Fin}^s}{\on{lim}}\, (\on{ins}_\CI)_!\circ \BD_{X^\CI}\circ (\on{ins}_\CI)^!(\CF).$$
\end{cor}

\sssec{}

The functor $\BD_{\Ran}$ is not in general very well-behaved. For example, one can show that for $X$ smooth of positive dimension, we have:
$$\BD_{\Ran}(\omega_{\Ran})=0.$$

Therefore, it is not true that (for $X$ proper) the functor of compactly supported cohomology
commutes with Verdier duality. We will now formulate a connectivity assumption that ensures that
the corresponding isomorphism holds for a given object. 

\medskip

Namely, the following result gives a sufficient condition for that map \eqref{e:duality and cohomology} to be an isomorphism
(i.e., commutation of Verdier duality with the functor of compactly supported cohomology): 

\begin{thm} \label{t:Verdier on Ran}
Let $X\in \Sch$ be a proper smooth curve.
Let $\CF\in \Shv^!(\Ran)$ have the property that for every integer $k\geq 0$ there exists an integer $n_k\geq 0$, such that 
the object $\on{ins}_\CI^!(\CF)|_{\overset{\circ}X{}^\CI}$ is concentrated in \emph{perverse}\footnote{In the context of D-modules,
this is the standard t-structure on the category of D-modules.} 
cohomological degrees 
$\leq -k-|\CI|$ whenever $|\CI|>n_k$.
Then the map \eqref{e:duality and cohomology} is an isomorphism.
\end{thm} 

\sssec{}

The next theorem gives a sufficient condition for the map \eqref{e:product and Verdier} to be an isomorphism. 
It will play a crucial role in the discussion of the behavior of \emph{factorization algebras} under Verdier duality,
see \secref{ss:factor and Verdier}. 

\begin{thm} \label{t:products}
Let $X$ be a smooth curve, and let the ring of coefficients $\Lambda$ be of finite cohomological dimension. 
Let $\CF_1,\CF_2\in \Shv^!(\Ran)$ have the property appearing in \thmref{t:Verdier on Ran}. 
Assume, in addition, that every $\CI$, the object $\on{ins}_\CI^!(\CF_i)\in \Shv(X^\CI)$ is bounded above 
and has compact cohomology sheaves. Then the map
$$\BD_{\Ran}(\CF_1)\boxtimes \BD_{\Ran}(\CF_2)\to \BD_{\Ran\times \Ran}(\CF_1\boxtimes \CF_2)$$
of \eqref{e:product and Verdier} is an isomorphism.
\end{thm} 

\sssec{}

The proofs of the above two theorems rely on considering truncated Ran spaces $\Ran^{\leq n}$ and 
a crucial stabilization assertion, namely \propref{p:stabilization}.

\ssec{A variant}  \label{ss:open variant}

In this subsection we will state a version of 
\thmref{t:Verdier on Ran} that involves non-proper curves. This theorem will be used in the derivation of
the cohomological product formula from the local duality statement in \secref{s:local duality}.

\sssec{}

Let $X'\overset{\bj}\hookrightarrow X$ be an open subset, and let
$$\Ran'\overset{\bj_{\Ran}}\hookrightarrow \Ran$$
be the corresponding open sub-prestack. 

\medskip

Note that the pullback functor $\bj_{\Ran}^!:\Shv^!(\Ran)\to \Shv^!(\Ran')$ admits both a left and a right adjoints,
to be denoted $(\bj_{\Ran})_!$ and $(\bj_{\Ran})_*$, respectively,

\medskip

Explicitly, for $\CF'\in \Shv^!(\Ran')$, we have:
\begin{equation} \label{e:bj!}
(\bj_{\Ran})_!(\CF')=\underset{\CI\in \on{Fin}^s}{\on{colim}}\,  (\on{ins}_\CI)_!  \circ (\bj_\CI)_! \circ (\on{ins}'_\CI)^!(\CF'),
\end{equation} 
where 
$$\on{ins}'_\CI:X'{}^\CI\to \Ran' \text{ and } \bj_\CI:X'{}^\CI\to X^\CI$$
denote the corresponding maps.

\medskip

For $\CF'\in \Shv^!(\Ran')$, the object $(\bj_{\Ran})_*(\CF')$ is uniquely characterized by the property that
\begin{equation} \label{e:descr j*}
(\on{ins}_\CI)^!  \circ (\bj_{\Ran})_*(\CF')=(\bj_\CI)_* \circ (\on{ins}'_\CI)^!(\CF').
\end{equation} 

In fact, the above formula for $(\bj_{\Ran})_*(\CF')$ implies that
$$(\bj_{\Ran})_*(\CF')=\underset{\CI\in \on{Fin}^s}{\on{colim}}\,  (\on{ins}_\CI)_!  \circ (\bj_\CI)_* \circ (\on{ins}'_\CI)^!(\CF').$$

\sssec{}

The above description of $(\bj_{\Ran})_*$ implies that the counit of the adjunction
$$(\bj_{\Ran})^!\circ (\bj_{\Ran})_*\to \on{Id}_{ \Shv^!(\Ran')}$$
is an isomorphism, so the functor $(\bj_{\Ran})_*$ is fully faithful.

\medskip

As a formal consequence, we obtain:

\begin{cor}
The functor $(\bj_{\Ran})_!$ is fully faithful.
\end{cor} 

Finally, we have: 

\begin{lem} \label{l:j}
For $\CF'\in \Shv^!(\Ran')$, we have a canonical isomorphism
$$\BD_{\Ran}\circ (\bj_{\Ran})_!(\CF')\simeq (\bj_{\Ran})_*\circ \BD_{\Ran'}(\CF').$$
\end{lem}

\begin{proof}
Let $\CG$ be an object of $\Shv^!(\Ran)$. We have
\begin{multline*}
\Maps_{\Shv^!(\Ran)}(\CG,(\bj_{\Ran})_*\circ \BD_{\Ran}(\CF'))\simeq 
\Maps_{\Shv^!(\Ran')}((\bj_{\Ran})^!(\CG),\BD_{\Ran}(\CF'))= \\
=\Maps_{\Shv^!(\Ran'\times \Ran')}\left((\bj_{\Ran})^!(\CG)\boxtimes \CF',(\on{diag}_{\Ran'})_!(\omega_{\Ran'})\right)
\end{multline*} 
and
\begin{multline} \label{e:duality and open}
\Maps_{\Shv^!(\Ran)}(\CG,\BD_{\Ran}\circ (\bj_{\Ran})_!(\CF'))=\\
=\Maps_{\Shv^!(\Ran\times \Ran)}\left(\CG\boxtimes (\bj_{\Ran})_!(\CF'),(\on{diag}_{\Ran})_!(\omega_{\Ran})\right).
\end{multline}

Now, the description of the functor $(\bj_{\Ran})_!$ implies that the canonically defined map
$$(\on{id}_{\Ran}\times \bj_{\Ran})_!(\CG\boxtimes \CF')\to \CG\boxtimes (\bj_{\Ran})_!(\CF')$$
is an isomorphism. Hence, the expression in \eqref{e:duality and open} can be rewritten as
\begin{multline*}
\Maps_{\Shv^!(\Ran\times \Ran')}\left(\CG\boxtimes \CF',(\on{id}_{\Ran}\times \bj_{\Ran})^!(\on{diag}_{\Ran'})_!(\omega_{\Ran'})\right)\simeq \\
\simeq \Maps_{\Shv^!(\Ran\times \Ran')}
\left(\CG\boxtimes \CF',(\on{Graph}_{\bj_{\Ran}})_!(\omega_{\Ran'})\right)\simeq \\
\simeq \Maps_{\Shv^!(\Ran\times \Ran')}\left(\CG\boxtimes \CF', (\bj_{\Ran}\times \on{id}_{\Ran'})_*\circ 
(\on{diag}_{\Ran'})_!(\omega_{\Ran'})\right)\simeq \\
\simeq \Maps_{\Shv^!(\Ran'\times \Ran')}\left((\bj_{\Ran})^!(\CG)\boxtimes \CF',(\on{diag}_{\Ran'})_!(\omega_{\Ran'})\right),
\end{multline*} 
as desired.

\end{proof} 

\sssec{}

We are now going to state a variant of \thmref{t:Verdier on Ran} that involves $\Shv^!(\Ran')$:

\begin{thm} \label{t:Verdier Ran open} Let $X$ be a proper smooth curve and $X'\subset X$
an open subscheme. Let $\CF'\in \Shv^!(\Ran')$ have the property as in \thmref{t:Verdier on Ran}, 
for the curve $X'$. Then for $\CF:=(\bj_{\Ran})_!(\CF')$, the map \eqref{e:duality and cohomology} 
is an isomorphism.
\end{thm} 

This theorem will be proved in \secref{ss:on open}. 

\ssec{The truncated Ran space}

In this subsection we will perform the first step towards the proofs of the results stated above. Namely, we will
reduce these theorems to statements of the sort that some cohomology stabilizes in the limit. 
The idea is the following: 

\medskip

The reason for the non-commutation of the functor of compactly supported cohomology on the Ran space
with Verdier duality is that $\Ran$ is not finitary. In this subsection we will introduce a truncated version of the Ran
space, denoted $\Ran^{\leq n}$, which will be a finitary pseudo-scheme.   The idea of $\Ran^{\leq n}$ is very simple:
whereas $\Ran$ parameterizes finite non-empty collections of points of $X$, its sub-prestack $\Ran^{\leq n}$ parameterizes
those collections that have cardinality $\leq n$. 

\medskip

The stabilization referred to above 
says that for a given range of cohomological degrees we can replace all of $\Ran$ by $\Ran^{\leq n}$. 

\sssec{}

For an integer $n$, let $\Ran^{\leq n}$ be the following prestack: 
for $S\in \Sch$ we let $\Ran^{\leq n}(S)$ be the (discrete) groupoid of finite non-empty sets of $\Maps(S,X)$
of cardinality $\leq n$. Thus, 
$$\Ran^{\leq n}\simeq \underset{\CI\in (\on{Fin}^{s,\leq n})^{\on{op}}}{\on{colim}}\, X^\CI,$$
where $\on{Fin}^{s,\leq n}\subset \on{Fin}^s$ is the full subcategory consisting of sets of cardinality
$\leq n$. In particular, if $X$ is proper, then the prestack $\Ran^{\leq n}$ is finitary pseudo-proper. 
Moreover, the discussion in \secref{sss:diag Ran} shows that $\Ran^{\leq n}$ is a pseudo-scheme with a finitary diagonal. 

\sssec{}

Let $\on{ins}^{\leq n}$ denote the tautological map $\Ran^{\leq n}\hookrightarrow \Ran$. It is easy to
see as in \secref{sss:diag Ran} that the map $\on{ins}^{\leq n}$ is finitary pseudo-proper. In particular,
we have an adjoint pair
$$(\on{ins}^{\leq n})_!:\Shv^!(\Ran^{\leq n})\rightleftarrows \Shv^!(\Ran):(\on{ins}^{\leq n})^!.$$

\medskip

We have
$$\Ran \simeq \underset{n}{\on{colim}}\, \Ran^{\leq n},$$
where the maps $\on{ins}^{\leq n}$ are pseudo-proper. 
Hence, by \secref{sss:limits and colimits}, we have
\begin{equation} \label{e:F via trunc}
\CF\simeq \underset{n}{\on{colim}}\,  (\on{ins}^{\leq n})_!\circ (\on{ins}^{\leq n})^!(\CF).
\end{equation} 

In particular, we have:
\begin{equation} \label{e:on Ran via n}
\on{C}^*_c(\Ran,\CF)\simeq \underset{n}{\on{colim}}\, \on{C}^*_c\left(\Ran^{\leq n},(\on{ins}^{\leq n})^!(\CF)\right).
\end{equation} 

\sssec{}

The following assertion, proved in \secref{ss:proof of Verdier one}, 
is one stabilization statement that goes into the proof of \thmref{t:Verdier on Ran}: 

\begin{prop} \label{p:Verdier on Ran 1}
Let $X$ be a smooth curve. Let $\CF\in \Shv^!(\Ran)$ be such that there exists an integer $m\geq 0$ such that the object 
$\on{ins}_\CI^!(\CF)|_{\overset{\circ}X{}^\CI}$ is concentrated in \emph{perverse} cohomological degrees $\leq -3$ whenever $|\CI|> m$.
Then the map
$$\on{C}^*_c\left(\Ran^{\leq n},(\on{ins}^{\leq n})^!(\CF)\right)\to \on{C}^*_c(\Ran,\CF)$$
induces an isomorphism in cohomological degrees $\geq 0$ for $n\geq m$.
\end{prop}

\begin{rem}
Note that the requirement on the cohomological degree in \propref{p:Verdier on Ran 1} is weaker than that in
\thmref{t:Verdier on Ran}. We need the bound on the cohomological degrees to be $\leq -C_0$ rather than $-|\CI|-C_1$, where $C_0$ and
$C_1$ are constants.  
\end{rem}

\sssec{Duality via the truncated Ran space}

Since the map $\on{ins}^{\leq n}$ is finitary pseudo-proper, the functor $(\on{ins}^{\leq n})_!$ intertwines Verdier duality 
on $\Ran^{\leq n}$ with Verdier duality on $\Ran$, by \corref{c:finite duality}. Hence, as in \propref{p:expl Verdier}, we obtain
\begin{equation} \label{e:Verdier on Ran}
\BD_{\Ran}(\CF)\simeq \underset{n}{\on{lim}}\, (\on{ins}^{\leq n})_!\circ \BD_{\Ran^{\leq n}}\circ (\on{ins}^{\leq n})^!(\CF).
\end{equation} 

\medskip

Observe also that by \propref{p:expl Verdier}, applied to $\Ran^{\leq n}$ and \corref{c:finite duality}, we obtain:

\begin{cor} \label{c:truncated duality}
For $\CF\in \Shv^!(\Ran)$ we have a canonical isomorphism
$$(\on{ins}^{\leq n})_!\circ \BD_{\Ran^{\leq n}}\circ (\on{ins}^{\leq n})^!(\CF)\simeq 
\underset{\CI\in \on{Fin}^{s,\leq n}}{\on{lim}} \, (\on{ins}_\CI)_!\circ \BD_{X^\CI}\circ (\on{ins}_\CI)^!(\CF).$$
\end{cor}

\sssec{}

In \secref{ss:stab 1} and \ref{ss:stab 2} we will prove the following crucial ingredient in the proof of
Theorems \ref{t:Verdier on Ran} and \ref{t:products}:

\begin{prop} \label{p:stabilization}
Let $X$ be a smooth curve, and let $\CF\in \Shv^!(\Ran)$ be such that there exists an integer $m\geq 0$ such that the object 
$\on{ins}_\CI^!(\CF)|_{\overset{\circ}X{}^\CI}$ is concentrated in \emph{perverse} cohomological degrees $\leq -|\CI|-2$ whenever $|\CI|>m$.

\smallskip

\noindent{\em(i)} Assume that $X$ is proper. Then the map
$$ \on{C}^*_c\left(\Ran,\BD_{\Ran}(\CF)\right) \to 
\on{C}^*_c\left(\Ran^{\leq n},\BD_{\Ran^{\leq n}}\circ (\on{ins}^{\leq n})^!(\CF)\right)$$
induces an isomorphism in cohomological degrees $\leq 0$ for $n\geq m$.

\smallskip

\noindent{\em(ii)} For any $\CJ$, the map
$$\on{ins}_\CJ^!\circ \BD_{\Ran}(\CF) \to \on{ins}_\CJ^!\circ (\on{ins}^{\leq n})_!\circ \BD_{\Ran^{\leq n}}\circ (\on{ins}^{\leq n})^!(\CF)$$
induces an isomorphism in perverse cohomological degrees $\leq 0$ for $n\geq m$.

\end{prop} 

\sssec{Proof of \thmref{t:Verdier on Ran}}   \label{sss:proof of Verdier on Ran}

Let us assume \propref{p:Verdier on Ran 1} and \propref{p:stabilization} and deduce \thmref{t:Verdier on Ran}. The idea
is to reduce to the situation  where we can apply \corref{c:finite duality} (we cannot apply this proposition to $\Ran$ because it is
not finitary).

\medskip

We need to prove that for every cohomological degree $k\geq 0$, the map
$$\left(\on{C}^*_c\left(\Ran,\BD_\CY(\CF)\right)\right)^{\leq k}\to \left(\on{C}^*_c(\Ran,\CF)^\vee\right)^{\leq k}$$
is an isomorphism. With no restriction of generality, we can take $k=0$ (if an object $\CF\in \Shv^!(\Ran)$ satisfies the assumption of
\thmref{t:Verdier on Ran}, then so does any of its shifts). 

\medskip

We have a commutative diagram
$$
\CD
\left(\on{C}^*_c\left(\Ran,\BD_\CY(\CF)\right)\right)^{\leq 0}   @>>> \left(\on{C}^*_c(\Ran,\CF)^\vee\right)^{\leq 0} \\
@VVV   @VVV   \\
\left(\on{C}^*_c\left(\Ran^{\leq n},\BD_{\Ran^{\leq n}}\circ (\on{ins}^{\leq n})^!(\CF)\right)\right)^{\leq 0}   @>>>
\left(\on{C}^*_c\left(\Ran^{\leq n},(\on{ins}^{\leq n})^!(\CF)\right)^\vee\right)^{\leq 0}. 
\endCD
$$

Note that the bottom horizontal arrow in the diagram is an isomorphism for any $n$ because the map 
$$\on{C}^*_c\left(\Ran^{\leq n},\BD_{\Ran^{\leq n}}\circ (\on{ins}^{\leq n})^!(\CF)\right)   \to
\on{C}^*_c\left(\Ran^{\leq n},(\on{ins}^{\leq n})^!(\CF)\right)^\vee$$
is (by \corref{c:finite duality}). 

\medskip 

Hence, it remains to show that the vertical maps are isomorphisms for some $n$. Now, the left vertical map
is an isomorphism for $n\gg 0$ by \propref{p:stabilization}(i).  For the right vertical map, it suffices to show that the map
$$\left(\on{C}^*_c\left(\Ran^{\leq n},(\on{ins}^{\leq n})^!(\CF)\right)\right)^{\geq 0}\to \left(\on{C}^*_c(\Ran,\CF)\right)^{\geq 0}$$ is
an isomorphism for $n\gg 0$, and that is given by \propref{p:Verdier on Ran 1}.

\section{Proofs of the stabilization and Verdier duality results}   \label{s:proofs Verdier}

\ssec{Proof of \propref{p:Verdier on Ran 1}}   \label{ss:proof of Verdier one} 

\sssec{}

Consider the $n$-th Cartesian power $X^n$ and the \emph{prestack-theoretic quotient} $X^n/\Sigma_n$. 
Note that the map $\on{ins}_n:X^n\to \Ran$ canonically factors through a map
$$\on{ins}_n/\Sigma_n:X^n/\Sigma_n\to \Ran.$$

\medskip

Let $\overset{\circ}X{}^n\overset{j_n}\hookrightarrow X^n$ denote the complement in $X^n$ to the diagonal divisor, and consider
the corresponding open embedding
$$j_n/\Sigma_n:\overset{\circ}X{}^n/\Sigma_n\hookrightarrow X^n/\Sigma_n.$$

Note that the symmetric power 
$\overset{\circ}X{}^{(n)}$ is the \'etale sheafification of $\overset{\circ}X{}^n/\Sigma_n$, so the pullback functor
\begin{equation} \label{e:symmetric}
\Shv(\overset{\circ}X{}^{(n)})\to \Shv^!(\overset{\circ}X{}^n/\Sigma_n)
\end{equation} 
is an equivalence of categories. 

\sssec{}

We will use the following assertion: 

\begin{lem} \label{l:Chevalley-Cousin}
For $\CF\in \Shv^!(\Ran)$, the object 
$$\on{coFib}\left((\on{ins}^{\leq n-1})_! \circ (\on{ins}^{\leq n-1})^!(\CF)\to 
((\on{ins}^{\leq n})_! \circ (\on{ins}^{\leq n})^!(\CF)\right) \in \Lambda\mod$$ 
is canonically isomorphic to 
$$(\on{ins}_n/\Sigma_n)_! \circ (j_n/\Sigma_n)_*\circ (j_n/\Sigma_n)^!\circ (\on{ins}_n/\Sigma_n)^!(\CF).$$ 
\end{lem} 

\begin{proof}

The map 
$$\on{ins}^{n-1,n}:\Ran^{\leq n-1}\to \Ran^{\leq n}$$
is finitary pseudo-proper and injective. Let 
$$(\Ran^{\leq n}-\Ran^{\leq n-1})\subset \Ran^{\leq n}$$
denote the open sub-prestack equal the complement of its scheme-theoretic image (see \lemref{l:replace lego}(b) for
what this means).  Now, the assertion of the lemma follows from \lemref{l:replace lego}(d) and the fact that the map 
$\on{ins}_n:X^n\to \Ran^{\leq n}$ defines an isomorphism 
$$\overset{\circ}X{}^n/\Sigma_n\to (\Ran^{\leq n}-\Ran^{\leq n-1}).$$

\end{proof}

\begin{cor} \label{c:Chevalley-Cousin}
For $\CF\in \Shv^!(\Ran)$, the object 
$$\on{coFib}\left(\on{C}^*_c\left(\Ran^{\leq n-1},(\on{ins}^{\leq n-1})^!(\CF)\right)\to 
\on{C}^*_c\left(\Ran^{\leq n},(\on{ins}^{\leq n})^!(\CF)\right)\right) \in \Lambda\mod$$ 
is canonically isomorphic to 
$$\on{C}^*\left(\overset{\circ}X{}^n,(j_n)^!\circ (\on{ins}_n)^!(\CF)\right)_{\Sigma_n}.$$
\end{cor} 

\begin{rem}
Note that in \corref{c:Chevalley-Cousin}, we are using the functor $\on{C}^*(\overset{\circ}X{}^n,-)$, 
and not $\on{C}^*_c(\overset{\circ}X{}^n,-)$. 
\end{rem}

\sssec{}

We now claim:

\begin{lem} \label{l:estimate}
Let $\CF\in \Shv^!(\Ran)$ and $m\in \BN$ have the property that $\on{ins}_\CI^!(\CF)|_{\overset{\circ}X{}^\CI}$ is concentrated in 
perverse cohomological degrees $\leq -3$ whenever $|\CI|>m$. Then for any $m \leq  n_1\leq n_2$, the map
$$\on{C}^*_c\left(\Ran^{\leq n_1}, (\on{ins}^{\leq n_1})^!(\CF)\right) \to 
\on{C}^*_c\left(\Ran^{\leq n_2}, (\on{ins}^{\leq n_2})^!(\CF)\right)$$
induces an isomorphism in cohomological degrees $\geq 0$. 
\end{lem}

\begin{proof}

By \corref{c:Chevalley-Cousin}, it suffices to show that for $n>m$, the object 
$$\on{C}^*\left(\overset{\circ}X{}^n,(j_n)^!\circ (\on{ins}_n)^!(\CF)\right)_{\Sigma_n}$$
lives in cohomological degrees $<-1$ (i.e., $\leq -2$). For that, it 
suffices to show that the object $\on{C}^*(\overset{\circ}X{}^n,(j_n)^!\circ (\on{ins}_n)^!(\CF))$
lives in cohomological degrees $\leq -2$. 

\medskip 

By assumption, $j_n^!\circ \on{ins}_n^!(\CF)$ is concentrated in perverse cohomological degrees $\leq -3$. 
Thus, it remains to show that the variety $\overset{\circ}X{}^n$ has the property that the functor
$\on{C}^*(\overset{\circ}X{}^n,-)$ has cohomological dimension bounded on the right by $1$.
This follows from the fact that the projection on, say, the first coordinate
$$\overset{\circ}X{}^n\to X$$
is an affine morphism. 

\end{proof}

\sssec{}

Finally, we note that \lemref{l:estimate} readily implies \propref{p:Verdier on Ran 1}: 
take $n_1=n$ and pass to the colimit with respect to $n_2$.

\ssec{Proof of \propref{p:stabilization}(i)}  \label{ss:stab 1}

The proof of \propref{p:stabilization}(i) is more subtle than that
of \propref{p:Verdier on Ran 1}, because it is not true that the functor $\on{C}^*_c(\Ran,-)$ commutes
with limits. So, the crux of the argument will be to commute a limit with a colimit. 

\sssec{}

By \eqref{e:on Ran via n}, it suffices to show that for all $k\geq 0$, the map
\begin{multline} \label{e:dual estimate limit}
\on{C}^*_c\left(\Ran^{\leq k},(\on{ins}^{\leq k})^! \circ \BD_{\Ran}(\CF)\right) 
\to \on{C}^*_c\left(\Ran^{\leq k},(\on{ins}^{\leq k})^! \circ (\on{ins}^{\leq n})_! \circ  \BD_{\Ran^{\leq n}}\circ (\on{ins}^{\leq n})^!(\CF)\right)
\end{multline}
induces an isomorphism in cohomological degrees $\leq 0$ whenever $n\geq m$. 

\medskip

Note that the functor $(\on{ins}^{\leq k})^!$ commutes with limits (because it admits a left adjoint). In addition,
the functor $ \on{C}^*_c(\Ran^{\leq k},-)$ \emph{does} commutes with limits 
by \lemref{l:lax homology finite}. 

\medskip

Therefore, in order to prove that \eqref{e:dual estimate limit} induces an isomorphism in cohomological degrees $\leq 0$,
it suffices to show that under the assumptions of the proposition, for $m\leq n_1\leq n_2$, the map
\begin{multline*} 
\on{C}^*_c\left(\Ran^{\leq k},(\on{ins}^{\leq k})^! \circ (\on{ins}^{\leq n_2})_! \circ  \BD_{\Ran^{\leq n_2}}\circ (\on{ins}^{\leq n_2})^!(\CF)\right)
\to \\
\to 
\on{C}^*_c\left(\Ran^{\leq k},
(\on{ins}^{\leq k})^! \circ (\on{ins}^{\leq n_1})_! \circ  \BD_{\Ran^{\leq n_1}}\circ (\on{ins}^{\leq n_1})^!(\CF)\right)
\end{multline*}
induces an isomorphism in cohomological degrees $\leq 0$. 

\medskip

Thus, it remains to show:

\begin{lem}  \label{l:estimate dual}
Under the assumptions of \propref{p:stabilization}, for $n> m$ and any $k$ the map
\begin{multline} \label{e:dual estimate}
\on{C}^*_c\left(\Ran^{\leq k},
(\on{ins}^{\leq k})^! \circ (\on{ins}^{\leq n})_! \circ  \BD_{\Ran^{\leq n}}\circ (\on{ins}^{\leq n})^!(\CF)\right)\to \\
\to 
\on{C}^*_c\left(\Ran^{\leq k}, (\on{ins}^{\leq k})^! \circ (\on{ins}^{\leq n-1})_! \circ  \BD_{\Ran^{\leq n-1}}\circ (\on{ins}^{\leq n-1})^!(\CF)\right)
\end{multline}
induces an isomorphism in cohomological degrees $\leq 0$. 
\end{lem} 

\begin{rem}
Note that the fact that the map 
$$\on{C}^*_c\left(\Ran,
(\on{ins}^{\leq n})_! \circ  \BD_{\Ran^{\leq n}}\circ (\on{ins}^{\leq n})^!(\CF)\right)\to
\on{C}^*_c\left(\Ran,(\on{ins}^{\leq n-1})_! \circ  \BD_{\Ran^{\leq n-1}}\circ (\on{ins}^{\leq n-1})^!(\CF)\right)$$
induces an isomorphism in cohomological degrees $\leq 0$ follows by duality from \lemref{l:estimate}.
In particular, this statement needs a weaker assumption on the cohomological degrees in which 
$\on{ins}_\CI^!(\CF)|_{\overset{\circ}X{}^\CI}$ lives. 
\end{rem} 

\sssec{}  \label{sss:cousin step}

The rest of this subsection is devoted to the proof of \lemref{l:estimate dual}.

\medskip

By \corref{c:Chevalley-Cousin}, it suffices to show that for any $k'\leq k$, the object 
\begin{multline*}
\biggl(\on{Fib}\biggl(\on{C}^*\left(\overset{\circ}X{}^{k'},(j_{k'})^!\circ (\on{ins}_{k'})^!\circ 
(\on{ins}^{\leq n})_! \circ  \BD_{\Ran^{\leq n}}\circ (\on{ins}^{\leq n})^!(\CF)\right)\to \\
\to \on{C}^*\left(\overset{\circ}X{}^{k'},(j_{k'})^!\circ (\on{ins}_{k'})^!\circ 
(\on{ins}^{\leq n-1})_! \circ  \BD_{\Ran^{\leq n-1}}\circ (\on{ins}^{\leq n-1})^!(\CF)\right)\biggr)\biggr)_{\Sigma_{k'}} 
\end{multline*}
is concentrated in cohomological degrees $>1$ (i.e., $\geq 2$). Since $k$ was arbitrary,
we can take $k'=k$. 

\medskip

Note that for $\CG\in \Shv^!(\overset{\circ}X{}^k/\Sigma_k)$, the norm map
$$\on{C}^*(\overset{\circ}X{}^k,\CG)_{\Sigma_k} \to
\on{C}^*(\overset{\circ}X{}^k,\CG)^{\Sigma_k}$$
is an isomorphism.

\medskip

Hence, it suffices to show that the object
\begin{multline} \label{e:estimate dual bis}
\on{Fib}\biggl(\on{C}^*\left(\overset{\circ}X{}^k,j_k^!\circ (\on{ins}_k)^!\circ 
(\on{ins}^{\leq n})_! \circ  \BD_{\Ran^{\leq n}}\circ (\on{ins}^{\leq n})^!(\CF)\right) \to \\
\on{C}^*\left(\overset{\circ}X{}^k, (j_k)^!\circ (\on{ins}_k)^!\circ 
(\on{ins}^{\leq n-1})_! \circ  \BD_{\Ran^{\leq n-1}}\circ (\on{ins}^{\leq n-1})^!(\CF)\right)\biggr) 
\end{multline}
is concentrated in cohomological degrees $\geq 2$. 

\sssec{}  \label{sss:final steps 1}

By Lemmas \ref{l:Chevalley-Cousin} and \ref{l:lax homology finite}, the object \eqref{e:estimate dual bis} identifies with
$$\left(\on{C}^*\left(\overset{\circ}X{}^k, j_k^!\circ (\on{ins}_k)^!\circ 
(\on{ins}_n)_! \circ \BD_{X^n}\circ (j_n)_* \circ (j_n)^!\circ (\on{ins}_n)^! (\CF)\right)\right){}^{\Sigma_n}.$$

Thus, it suffices that the object
$$\on{C}^*\left(\overset{\circ}X{}^k, j_k^!\circ (\on{ins}_k)^!\circ 
(\on{ins}_n)_! \circ \BD_{X^n}\circ (j_n)_* \circ (j_n)^!\circ (\on{ins}_n)^! (\CF)\right)$$
is concentrated in cohomological degrees $\geq 2$. 

\medskip

By assumption, the object $(j_n)^!\circ (\on{ins}_n)^!(\CF)\in \Shv(X^n)$
is concentrated in perverse cohomological degrees $\leq -n-2$. Hence,  $(j_n)_* \circ (j_n)^!\circ (\on{ins}_n)^!(\CF)$
lives in perverse cohomological degrees $\leq -n-2$, as the morphism $(j_n)_*$ is affine. 

\medskip

We will show that for any $\CG\in \Shv(X^n)$ that lives in perverse
cohomological degrees $\leq -n-2$, the object 
\begin{equation} \label{e:estimate dual bis bis}  
\on{C}^*\left(\overset{\circ}X{}^k, j_k^!\circ (\on{ins}_k)^!\circ 
(\on{ins}_n)_! \circ \BD_{X^n}(\CG)\right)
\end{equation} 
is concentrated in cohomological degrees $\geq 2$. 

\sssec{}

Let $X^{k,n}\subset X^k\times X^n$ be the closed subset, corresponding to the condition
that the $k$-tuple and the $n$-tuple coincide set-theoretically (see \secref{sss:coincide}). 
Let $q_k$ and $q_n$ denote the two projections
$$X^k \overset{q_k}\longleftarrow X^{k,n}\overset{q_n}\longrightarrow X^n.$$

By \corref{c:push-pull Ran}, the functor 
$$(\on{ins}_k)^!\circ (\on{ins}_n)_!:\Shv^!(X^n)\to \Shv^!(X^m)$$
is canonically isomorphic to 
$$(q_k)_!\circ (q_n)^!,$$
where $q_k$ and $q_n$ are the two projections
$$X^k \overset{q_k}\longleftarrow X^{k,n}\overset{q_n}\longrightarrow X^n.$$

\medskip

Hence, we obtain that the expression in \eqref{e:estimate dual bis bis} identifies canonically with
\begin{equation}\label{e:estimate dual bis bis bis}  
\on{C}^*\left(\overset{\circ}X{}^k,j_k^!\circ (q_k)_!\circ (q_n)^! \circ \BD_{X^n}(\CG)\right),
\end{equation}
and the latter expression only involves schemes (rather than prestacks).

\sssec{} \label{sss:final steps 2}

Note that both maps $q_k$ and $q_n$ are finite. Let $\overset{\circ}X{}^{k,n}$ denote the
preimage of $\overset{\circ}X{}^k$ under $q_k$. Let $\overset{\circ}q{}^n$ denote the resulting 
(quasi-finite) map $\overset{\circ}X{}^{k,n}\to X^n$. 

\medskip

By base change (and using the fact that $q_k$ is proper), we rewrite the expression  
in \eqref{e:estimate dual bis bis bis} as 
$$\on{C}^*\left(X^n,\left((\overset{\circ}q_n)_*(\omega_{\overset{\circ}X^{k,n}})\overset{!}\otimes \BD_{X^n}(\CG)\right)\right),$$
and, due to the compactness of the object $(\overset{\circ}q_n)_*(\omega_{\overset{\circ}X^{k,n}})$, further as 
$$\CHom\left(\CG, (\overset{\circ}q_n)_*(\omega_{\overset{\circ}X^{k,n}})\right),$$
where $\CHom(-,-)\in \Lambda\mod$ denotes the complex of maps between given two objects.

\medskip

Due to the assumption on $\CG$, it remains to show that $(\overset{\circ}q_n)_*(\omega_{\overset{\circ}X^{k,n}})$
is concentrated in perverse cohomological degrees $\geq -n$. The latter is immediate from the fact that $\overset{\circ}q_n$
is quasi-finite, while $\dim(X^n)=n$. 

\ssec{Proof of \thmref{t:Verdier Ran open}}  \label{ss:on open}

In this subsection we will be working in the context of constructible sheaves. 

\sssec{}

Let 
$$\on{ins}'{}^{\leq n}:\Ran'{}^{\leq n}\to \Ran' \text{ and } \bj_{\Ran^{\leq n}}:\Ran'{}^{\leq n}\to \Ran^{\leq n}$$
denote the corresponding morphisms. We have a tautological isomorphism
$$(\on{ins}^{\leq n})_!\circ (\bj_{\Ran^{\leq n}})_!\simeq (\bj_{\Ran})_!\circ (\on{ins}'{}^{\leq n})_!.$$

In addition, it is easy to see that the natural transformation 
$$(\on{ins}^{\leq n})_!\circ (\bj_{\Ran^{\leq n}})_*\to (\bj_{\Ran})_*\circ (\on{ins}'{}^{\leq n})_!$$
is also an isomorphism. 

\medskip

As in \lemref{l:j}, for $\CG'\in \Shv^!(\Ran'{}^{\leq n})$, we have a canonical isomorphism
$$\BD_{\Ran^{\leq n}}\circ (\bj_{\Ran^{\leq n}})_!(\CG')\simeq 
(\bj_{\Ran^{\leq n}})_*\circ \BD_{\Ran'{}^{\leq n}}(\CG').$$

\sssec{}

By \propref{p:Verdier on Ran 1}, applied to the curve $X'$, as in \secref{sss:proof of Verdier on Ran}, in order to prove
\thmref{t:Verdier Ran open}, it suffices to establish the following variant of \propref{p:stabilization}(i):

\begin{prop}  \label{p:stabilization open}
Let $\CF\in \Shv^!(\Ran')$ be such that there exists an integer $m\geq 0$ such that the object 
$(\on{ins}'_\CI)^!(\CF')|_{\overset{\circ}X'{}^\CI}$ is concentrated in \emph{perverse} cohomological degrees $\leq -|\CI|-2$ whenever $|\CI|>m$.
Then the map
$$ \on{C}^*_c\left(\Ran,(\bj_{\Ran})_*\circ \BD_{\Ran'}(\CF')\right) \to 
\on{C}^*_c\left(\Ran^{\leq n},(\bj_{\Ran^{\leq n}})_*\circ 
\BD_{\Ran'{}^{\leq n}}\circ (\on{ins}'{}^{\leq n})^!(\CF')\right)$$
induces an isomorphism in cohomological degrees $\leq 0$ for $n\geq m$.
\end{prop} 

This, in turn, reduces to the following variant of \lemref{l:estimate dual}:

\begin{lem}  \label{l:estimate dual open}
Under the assumptions of \propref{p:stabilization open}, for $n> m$ and any $k$ the map
\begin{multline*} 
\on{C}^*_c\left(\Ran^{\leq k},(\bj_{\Ran^{\leq k}})_*\circ 
(\on{ins}'{}^{\leq k})^! \circ (\on{ins}'{}^{\leq n})_! \circ  \BD_{\Ran'{}^{\leq n}}\circ (\on{ins}'{}^{\leq n})^!(\CF')\right)\to \\
\to 
\on{C}^*_c\left(\Ran^{\leq k}, (\bj_{\Ran^{\leq k}})_*\circ 
(\on{ins}'{}^{\leq k})^! \circ (\on{ins}'{}^{\leq n-1})_! \circ  \BD_{\Ran'{}^{\leq n-1}}\circ (\on{ins}'{}^{\leq n-1})^!(\CF')\right)
\end{multline*}
induces an isomorphism in cohomological degrees $\leq 0$. 
\end{lem} 

\sssec{}

Note that for any $m$, the diagram
$$
\CD
\Shv(X^{m}/\Sigma_{m})   @<{(\bj_{X^{m}/\Sigma_{m}})_*}<<   \Shv(X'{}^{m}/\Sigma_{m})  \\
@V{(\on{ins}'_{m})_!}VV    @VV{(\on{ins}_{m})_!}V    \\
\Shv(\Ran)  @<{(\bj_{\Ran})_*}<< \Shv(\Ran')  
\endCD
$$
commutes. 

\medskip

As in \secref{sss:cousin step}, this implies that in order to prove \lemref{l:estimate dual open}, it suffices to show that the object
\begin{multline*} 
\on{Fib}\biggl(\on{C}^*\left(\overset{\circ}X{}'{}^k,(j'_k)^!\circ (\on{ins}'_k)^!\circ 
(\on{ins}'{}^{\leq n})_! \circ  \BD_{\Ran'{}^{\leq n}}\circ (\on{ins}'{}^{\leq n})^!(\CF')\right) \to \\
\on{C}^*\left(\overset{\circ}X{}'{}^k, (j'_k)^!\circ (\on{ins}'_k)^!\circ 
(\on{ins}'{}^{\leq n-1})_! \circ  \BD_{\Ran'{}^{\leq n-1}}\circ (\on{ins}'{}^{\leq n-1})^!(\CF')\right)\biggr) 
\end{multline*}
is concentrated in cohomological degrees $>1$. 

\medskip

However, this was proved in Sects. \ref{sss:final steps 1}-\ref{sss:final steps 2}. 
 
\ssec{Proof of \propref{p:stabilization}(ii)}  \label{ss:stab 2}

The proof was essentially given in the process of proof of \propref{p:stabilization}(i):

\sssec{}

The functor $\on{ins}_k^!$ commutes with limits, so we need to show that the map
$$\underset{n'}{\on{lim}}\, \on{ins}_k^!\circ (\on{ins}^{\leq n'})_! \circ \BD_{\Ran^{\leq n'}}\circ 
(\on{ins}^{\leq n'})^! (\CF) \to  \on{ins}_k^!\circ (\on{ins}^{\leq n})_! \circ \BD_{\Ran^{\leq n}}\circ 
(\on{ins}^{\leq n})^! (\CF)$$
induces an isomorphism in degrees $\leq 0$ for $n\geq m$. 

\medskip

For that, it suffices to show that for $n'\geq n\geq m$, the map
$$\on{ins}_k^!\circ (\on{ins}^{\leq n'})_! \circ \BD_{\Ran^{\leq n'}}\circ 
(\on{ins}^{\leq n'})^! (\CF) \to  \on{ins}_k^!\circ (\on{ins}^{\leq n})_! \circ \BD_{\Ran^{\leq n}}\circ 
(\on{ins}^{\leq n})^! (\CF)$$
induces an isomorphism in cohomological degrees $\leq 0$. 

\medskip

The latter is equivalent to showing that
\begin{equation} \label{e:restr of dual}
\on{ins}_k^! \left( \on{Fib} \left((\on{ins}^{\leq n})_! \circ \BD_{\Ran^{\leq n}}\circ 
(\on{ins}^{\leq n})^! (\CF)\to \on{ins}^{\leq n-1}_! \circ \BD_{\Ran^{\leq n-1}}\circ 
(\on{ins}^{\leq n-1})^! (\CF)\right)\right)
\end{equation} 
is concentrated in cohomological degrees $>1$ (i.e., $\geq 2$) for $n>m$. 

\sssec{}

By Lemmas \ref{l:Chevalley-Cousin} and \ref{l:lax homology finite}, the object \eqref{e:restr of dual}, identifies with
$$\left(\on{ins}_k^! \circ  (\on{ins}_n)_! \circ \BD_{X^n} \circ (j_n)_* \circ j_n^!\circ \on{ins}_n^! (\CF)\right)^{\Sigma_n}.$$

So, it is enough to show that the object 
$$\on{ins}_k^! \circ (\on{ins}_n)_! \circ \BD_{X^n} \circ (j_n)_* \circ j_n^!\circ \on{ins}_n^! (\CF)$$
is concentrated in perverse cohomological degrees $\geq 2$. 

\medskip

The object $(j_n)_* \circ j_n^!\circ \on{ins}_n^! (\CF)$ is concentrated in perverse cohomological degrees $\leq -n-2$. We will show
that for any $\CG\in \Shv(X^n)$ which lives in such degrees, the object  
$$\on{ins}_k^! \circ (\on{ins}_n)_! \circ \BD_{X^n}(\CG)$$
lives in perverse degrees $\geq 2$.

\sssec{}

By \corref{c:push-pull Ran}, we have
$$\on{ins}_k^!\circ  (\on{ins}^{\leq n})_! \circ \BD_{X^n}(\CG) \simeq 
(q_k)_!\circ q_n^! \circ  \BD_{X^n}(\CG).$$

Since the morphisms $q_n$ and $q_k$ are both finite, the functor $(q_k)_!\circ q_n^!$ is right t-exact. 

\medskip

Hence, it suffices to show that the object $\BD_{X^n}(\CG)$ lives in perverse degrees $\geq 2$. However, this follows
from the fact that $\CG$ lives in perverse degrees 
$\leq -n-2$, while $\dim(X^n)=n$.  \footnote{Note that when working in the context of constructible sheaves, the object
$\BD_{X^n}(\CG)$ actually lives in degrees $\geq 2+n$.}

\ssec{Proof of \thmref{t:products}}

In this subsection we will be referring to the perverse t-structure on $\Shv(Z)$ for $Z\in \Sch$ (which the standard
t-structure if we work in the context of D-modules). 

\sssec{}

We need to show that for every $m_1,m_2$ and every cohomological degree $k\geq 0$, the map in \eqref{e:product and Verdier}
$$(\on{ins}_{m_1}\times \on{ins}_{m_2})^!\left((\BD_{\Ran}(\CF_1)\boxtimes \BD_{\Ran}(\CF_2)\right)\to
(\on{ins}_{m_1}\times \on{ins}_{m_2})^!\circ  \BD_{\Ran\times \Ran}(\CF_1\boxtimes \CF_2)$$
induces an isomorphism in cohomological degrees $\leq k$. With no restriction of generality, we can assume that $k=0$. 

\sssec{}

First, we claim that the object
$$\on{ins}_{m_i}^!\circ \BD_{\Ran}(\CF_i)\in \Shv(X^{m_i})$$
(for $i=1,2$) is bounded below.

\medskip

To prove this, by \propref{p:stabilization}(ii), it suffices to show that the object
$$\on{ins}_{m_i}^!\circ (\on{ins}^{\leq n})_!\circ \BD_{\Ran^{\leq n}}\circ (\on{ins}^{\leq n})^!(\CF_i)$$
is bounded below for some $n$ sufficiently large. 

\medskip

By \corref{c:truncated duality}, we have:
$$(\on{ins}^{\leq n})_!\circ \BD_{\Ran^{\leq n}}\circ (\on{ins}^{\leq n})^!(\CF_i)\simeq 
\underset{\CI\in \on{Fin}^{s\leq n}}{\on{lim}}\,  (\on{ins}_\CI)_! \circ \BD_{X^\CI}\circ (\on{ins}_\CI)^!(\CF_i).$$

Since the above limit is finite, it suffices to show that for each $\CI$, the object
$$\on{ins}_{m_i}^!\circ (\on{ins}_\CI)_! \circ \BD_{X^\CI}\circ (\on{ins}_\CI)^!(\CF_i)$$
(for $i=1,2$) is bounded below.

\medskip

By assumption, $(\on{ins}_\CI)^!(\CF_i)$ is bounded above.
Hence, $\BD_{X^\CI}\circ (\on{ins}_\CI)^!(\CF_i)$ is bounded below. Finally, the 
functor $\on{ins}_{m_i}^!\circ (\on{ins}^{\leq n})_!$ has a finite cohomological dimension by 
\corref{c:push-pull Ran}. 

\sssec{}

Combining the facts that:

\smallskip

\noindent(a) The ring of coefficients $\Lambda$ has a finite cohomological dimension;
 
\smallskip

\noindent(b) $\on{ins}_{m_i}^!\circ \BD_{\Ran}(\CF_i)$ are bounded below; 

\smallskip

\noindent(c) \propref{p:stabilization}(ii),

\smallskip

we obtain that there exists an integer $n\gg 0$, such that the map
\begin{multline*}
(\on{ins}_{m_1}\times \on{ins}_{m_2})^!\left((\BD_{\Ran}(\CF_1)\boxtimes \BD_{\Ran}(\CF_2)\right)\to  \\ \to
(\on{ins}_{m_1}\times \on{ins}_{m_2})^!\circ (\on{ins}^{\leq n}\times \on{ins}^{\leq n})_!
\left(\left(\BD_{\Ran^{\leq n}}\circ (\on{ins}^{\leq n})^! (\CF_1)\right)
\boxtimes \left(\BD_{\Ran^{\leq n}}\circ (\on{ins}^{\leq n})^!(\CF_2)\right)\right)
 \end{multline*} 
induces an isomorphism in degrees $\leq 0$. 

\medskip

By an analog of \propref{p:stabilization}(ii) for $\Ran\times \Ran$, we can choose $n$ large enough so that the map
\begin{multline*} 
(\on{ins}_{m_1}\times \on{ins}_{m_2})^!\circ \BD_{\Ran\times \Ran}(\CF_1\boxtimes \CF_2)\to \\
(\on{ins}_{m_1}\times \on{ins}_{m_2})^!\circ (\on{ins}^{\leq n}\times \on{ins}^{\leq n})_!\circ \BD_{\Ran^{\leq n}\times \Ran^{\leq n}}
\circ (\on{ins}^{\leq n}\times \on{ins}^{\leq n})^! (\CF_1\boxtimes \CF_2)
\end{multline*} 
induces an isomorphism in degrees $\leq 0$.

\medskip

Hence, it suffices to show that under the compactness assumption of the theorem, the map
\begin{multline}  \label{e:finite product}
(\on{ins}_{m_1}\times \on{ins}_{m_2})^! \circ (\on{ins}^{\leq n}\times \on{ins}^{\leq n})_! 
\left(\left(\BD_{\Ran^{\leq n}}\circ (\on{ins}^{\leq n})^!(\CF_1)\right)\boxtimes \left(\BD_{\Ran^{\leq n}}\circ (\on{ins}^{\leq n})^!(\CF_2)\right)
\right) \\
\to (\on{ins}_{m_1}\times \on{ins}_{m_2})^!\circ (\on{ins}^{\leq n}\times \on{ins}^{\leq n})_!\circ 
\BD_{\Ran^{\leq n}\times \Ran^{\leq n}}
\circ (\on{ins}^{\leq n}\times \on{ins}^{\leq n})^! (\CF_1\boxtimes \CF_2)
\end{multline} 
induces an isomorphism in cohomological degrees $\leq 0$. We will show that the map \eqref{e:finite product} is an isomorphism
(in all degrees). In fact, we will show that the map
\begin{multline*} 
(\on{ins}^{\leq n}\times \on{ins}^{\leq n})_! 
\left(\left(\BD_{\Ran^{\leq n}}\circ (\on{ins}^{\leq n})^!(\CF_1)\right)\boxtimes \left(\BD_{\Ran^{\leq n}}\circ (\on{ins}^{\leq n})^!(\CF_2)\right)
\right)\to \\
(\on{ins}^{\leq n}\times \on{ins}^{\leq n})_!\circ 
\BD_{\Ran^{\leq n}\times \Ran^{\leq n}}
\circ (\on{ins}^{\leq n}\times \on{ins}^{\leq n})^! (\CF_1\boxtimes \CF_2)
\end{multline*} 
is an isomorphism.

\sssec{}

By \corref{c:truncated duality}, we have: 
\begin{multline*} 
(\on{ins}^{\leq n}\times \on{ins}^{\leq n})_! 
\left(\left(\BD_{\Ran^{\leq n}}\circ (\on{ins}^{\leq n})^!(\CF_1)\right)\boxtimes \left(\BD_{\Ran^{\leq n}}\circ (\on{ins}^{\leq n})^!(\CF_2)\right)
\right) \simeq \\
\simeq \underset{\CI_1,\CI_2\in \on{Fin}^{s,\leq n}}{\on{lim}}\,
(\on{ins}_{\CI_1}\times \on{ins}_{\CI_2})_!\circ 
\left(\left(\BD_{X^{\CI_1}}\circ \on{ins}^!_{\CI_1}(\CF_1)\right)\boxtimes \left(\BD_{X^{\CI_2}}\circ \on{ins}^!_{\CI_2}(\CF_2)\right)\right),
\end{multline*}
and
\begin{multline*} 
(\on{ins}^{\leq n}\times \on{ins}^{\leq n})_!\circ 
\BD_{\Ran^{\leq n}\times \Ran^{\leq n}}
\circ (\on{ins}^{\leq n}\times \on{ins}^{\leq n})^! (\CF_1\boxtimes \CF_2) \simeq \\
\simeq \underset{\CI_1,\CI_2\in \on{Fin}^{s,\leq n}}{\on{lim}}\,
(\on{ins}_{\CI_1}\times \on{ins}_{\CI_2})_!\circ \BD_{X^{\CI_1}\times X^{\CI_2}}\circ (\on{ins}_{\CI_1}\times \on{ins}_{\CI_2})^!
(\CF_1\boxtimes \CF_2).
\end{multline*}

Hence, it is sufficient to show that for every $\CI_1,\CI_2$, the map
$$\left(\BD_{X^{\CI_1}}\circ \on{ins}^!_{\CI_1}(\CF_1)\right)\boxtimes \left(\BD_{X^{\CI_2}}\circ \on{ins}^!_{\CI_2}(\CF_2)\right)\to
\BD_{X^{\CI_1}\times X^{\CI_2}}\circ (\on{ins}_{\CI_1}\times \on{ins}_{\CI_2})^!
(\CF_1\boxtimes \CF_2)$$
is an isomorphism. 

\medskip

By assumption, $\on{ins}^!_{\CI_1}(\CF_1)$ and $\on{ins}^!_{\CI_2}(\CF_2)$ are bounded above with compact cohomologies. 
Hence, the required isomorphism follows from \lemref{l:product and Verdier sch compact}.

\newpage 

\centerline{\bf Part III: Duality via sheaves on the augmented Ran space}

\section{Pairings on the augmented vs. usual Ran space}  \label{s:pairings aug}

In this section we will perform a manipulation crucial for our derivation of the cohomological product formula 
\eqref{e:product formula prev} from non-abelian Poincar\'e duality \eqref{e:non-ab Poinc prev}; this will
be done in \secref{s:local duality}.  It is for this manipulation that we need the unital augmented version
of the Ran space, developed in Part I. 

\medskip

Say we start with two objects $\CA,\CB\in \Shv^!(\Ran)$ and we want to construct an isomorphism
\begin{equation} \label{e:isom to show prev}
\on{C}^*_c(\Ran,\CB)\to \on{C}^*_c(\Ran,\CA)^\vee.
\end{equation} 
A natural first attempt would be to identify $\CB$ with the Verdier dual of $\CA$. 
However, the object $\CA$ that we have 
in mind will be such that it will grossly violate the assumption of \thmref{t:Verdier on Ran}. So, it is unreasonable to expect that
the (compactly supported) cohomology of the Verdier dual of $\CA$ has anything to do with the dual of the 
(compactly supported) cohomology of $\CA$. In fact, in our case the Verdier dual of $\CA$ will
be zero. 

\medskip

Instead, we will find that $\CA$ is naturally of the form $$\CA\simeq \on{OblvUnit}\circ \on{OblvAug}(\CA_{\on{untl,aug}})$$ 
for $\CA_{\on{untl,aug}}\in \Shv^!(\Ran_{\on{untl,aug}})$, and similarly for $\CB$. Moreover, the objects 
$\CA_{\on{untl,aug}},\CB_{\on{untl,aug}}\in  \Shv^!(\Ran_{\on{untl,aug}})$ satisfy the assumption of \thmref{t:main}. Set
$$\CA_{\on{red}}:=\on{TakeOut}(\CA_{\on{untl,aug}}) \text{ and } \CB_{\on{red}}:=\on{TakeOut}(\CB_{\on{untl,aug}}).$$

By \corref{c:main}, we have the isomorphisms
$$\on{C}^*_c(\Ran,\CA_{\on{red}})\simeq \on{C}^*_c(\Ran,\CA) \text{ and }
\on{C}^*_c(\Ran,\CB_{\on{red}})\simeq \on{C}^*_c(\Ran,\CB).$$

\medskip

The reason for replacing $\CA$ 
by $\CA_{\on{red}}$ is that the latter object does satisfy the assumption of \thmref{t:Verdier on Ran}. So,
to establish the desired isomorphism \eqref{e:isom to show prev}, it now suffices to construct an isomorphism
$$\BD_{\Ran}(\CA_{\on{red}})\simeq \CB_{\on{red}}.$$
To do so we will first need to construct a pairing 
\begin{equation} \label{e:isom to construct prev}
\CA_{\on{red}}\boxtimes \CB_{\on{red}}\to (\on{diag}_{\Ran})_!(\omega_{\Ran}).
\end{equation} 

What we do in the present is section is explain what kind of data on $\CA$ and $\CB$
is needed to construct a pairing \eqref{e:isom to construct prev}. This will be addressed by \thmref{t:aug pairings}. 

\ssec{The notion of pairing for augmented sheaves}  \label{ss:notion of pairing}

\sssec{}

Let 
$$(\Ran_{\on{untl,aug}}\times \Ran_{\on{untl,aug}})_{\on{\on{sub,disj}}}$$ denote the following lax prestack: for $S\in \Sch$,
the category $(\Ran_{\on{untl,aug}}\times \Ran_{\on{untl,aug}})_{\on{\on{sub,disj}}}(S)$ is a full subcategory of
$(\Ran_{\on{untl,aug}}\times \Ran_{\on{untl,aug}})(S)$ that corresponds to quadruples
$$(K_1\subseteq I_1), (K_2\subseteq I_2)$$ for which $K_1$ and $I_2$ have disjoint images \emph{ and } $K_2$ and $I_1$ 
have have \emph{disjoint images} (see \secref{sss:disj images} what this means).

\medskip

Let $\wt\CF,\wt\CG$ be two objects of $\Shv^!(\Ran_{\on{untl,aug}})$. By a \emph{pairing} between them we shall mean a map
in $\Shv^!((\Ran_{\on{untl,aug}}\times \Ran_{\on{untl,aug}})_{\on{\on{sub,disj}}})$
\begin{equation} \label{e:aug pairing}
\wt\CF\boxtimes \wt\CG|_{(\Ran_{\on{untl,aug}}\times \Ran_{\on{untl,aug}})_{\on{\on{sub,disj}}}}\to 
\omega_{(\Ran_{\on{untl,aug}}\times \Ran_{\on{untl,aug}})_{\on{\on{sub,disj}}}}.
\end{equation} 

\sssec{}

Our goal of this section is to prove the following: 

\begin{thmconstr} \label{t:aug pairings}  \hfill

\smallskip

\noindent{\em(i)} For $\CF,\CG\in \Shv^!(\Ran)$, a datum of a pairing $\CF\boxtimes \CG\to (\on{diag}_{\Ran})_!(\omega_{\Ran})$
gives rise to a pairing
$$\on{AddUnit}_{\on{aug}}(\CF)\boxtimes \on{AddUnit}_{\on{aug}}(\CG)|_{(\Ran_{\on{untl,aug}}\times \Ran_{\on{untl,aug}})_{\on{\on{sub,disj}}}}\to 
\omega_{(\Ran_{\on{untl,aug}}\times \Ran_{\on{untl,aug}})_{\on{\on{sub,disj}}}}.$$

\smallskip

\noindent{\em(ii)} For $\wt\CF,\wt\CG\in \Shv^!(\Ran_{\on{untl,aug}})$, the datum of a pairing \eqref{e:aug pairing} gives rise to
a pairing
\begin{equation} \label{e:take out pairing}
\on{TakeOut}(\wt\CF)\boxtimes \on{TakeOut}(\wt\CG)\to (\on{diag}_{\Ran})_!(\omega_{\Ran}).
\end{equation} 

\smallskip

\noindent{\em(iii)} The constructions in \emph{(i)} and \emph{(ii)} are compatible under adjunction maps
$$\CF\to \on{TakeOut}\circ \on{AddUnit}_{\on{aug}}(\CF),\,\, \CG\to \on{TakeOut}\circ \on{AddUnit}_{\on{aug}}(\CG)$$
and
$$\on{AddUnit}_{\on{aug}}\circ \on{TakeOut}(\wt\CF)\to \wt\CF,\,\, 
\on{AddUnit}_{\on{aug}}\circ \on{TakeOut}(\wt\CG)\to \wt\CG.$$

\smallskip

\noindent{\em(iv)} Given a pairing \eqref{e:aug pairing} and the corresponding pairing \eqref{e:take out pairing}, 
the diagram
$$
\CD
\on{OblvUnit}\circ \on{OblvAug} (\wt\CF)\boxtimes 
\on{OblvUnit}\circ \on{OblvAug} (\wt\CG) @>>>  \omega_{\Ran\times \Ran}  \\
@A{\text{\eqref{e:nat trans 3}}\boxtimes \text{\eqref{e:nat trans 3}}}AA   @AAA   \\
\on{TakeOut}(\wt\CF) \boxtimes \on{TakeOut}(\wt\CG) @>>> (\on{diag}_{\Ran})_!(\omega_{\Ran})
\endCD
$$
commutes.

\end{thmconstr} 

Combining this theorem with \thmref{t:main}, we obtain:

\begin{cor}  \label{c:pairing on reduced}
There exists a canonical bijection between pairings $\CF\boxtimes \CG\to (\on{diag}_{\Ran})_!(\omega_{\Ran})$
and pairings \eqref{e:aug pairing} for $\wt\CF:=\on{AddUnit}_{\on{aug}}(\CF)$ and 
$\wt\CG:=\on{AddUnit}_{\on{aug}}(\CG)$.
\end{cor}

From point (iv) of the theorem we obtain:

\begin{cor} \label{c:remove unit and pairing}
Given a pairing $\CF\boxtimes \CG\to (\on{diag}_{\Ran})_!(\omega_{\Ran})$, for
the corresponding pairing 
$$\wt\CF\boxtimes \wt\CG|_{(\Ran_{\on{untl,aug}}\times \Ran_{\on{untl,aug}})_{\on{\on{sub,disj}}}}\to 
\omega_{(\Ran_{\on{untl,aug}}\times \Ran_{\on{untl,aug}})_{\on{\on{sub,disj}}}}$$
with 
$$\wt\CF:=\on{AddUnit}_{\on{aug}}(\CF),\,\,\, \wt\CG:=\on{AddUnit}_{\on{aug}}(\CG),$$
the diagram
$$
\CD
\on{C}^*_c\left(\Ran,\CF'\right) \otimes 
\on{C}^*_c\left(\Ran,\CG'\right)  @>>>
\on{C}^*_c\left(\Ran\times \Ran,\omega_{\Ran\times \Ran}\right)  @>>>  \Lambda  \\
@A\text{\eqref{e:id to add}}\boxtimes \text{\eqref{e:id to add}}AA   & & @AA{\on{id}}A   \\
\on{C}^*_c\left(\Ran,\CF\right) \otimes 
\on{C}^*_c\left(\Ran,\CG\right)   @>>>   \on{C}^*_c\left(\Ran,\omega_{\Ran}\right) @>>>  \Lambda
\endCD
$$
commutes, where
$$\CF':=\on{OblvUnit}\circ \on{OblvAug} (\wt\CF),\,\, \CG':=\on{OblvUnit}\circ \on{OblvAug} (\wt\CG).$$
\end{cor} 

\begin{rem}
Note that if $X$ is connected, then in \corref{c:remove unit and pairing} the left vertical arrow is an isomorphism due to \corref{c:main}
and the two horizontal arrows on the right are also isomorphisms, due to \thmref{t:Ran contr}.  
\end{rem} 

\ssec{Proof of \thmref{t:aug pairings}(i)}

\sssec{}

Recall the notations from \secref{ss:unit and aug}. Recall that the functor $\on{AddUnit}_{\on{aug}}$ is defined as 
$$\on{coFib}\left((\psi_{\on{aug}})_!\circ (\xi_{\on{aug}})^! \to \pi^! \circ \on{AddUnit}\right).$$

We shall first construct a map
\begin{equation} \label{e:aug pairing prime}
\pi^! \circ \on{AddUnit}(\CF)\boxtimes \pi^! \circ \on{AddUnit}(\CG)\to \omega_{\Ran_{\on{untl,aug}}\times \Ran_{\on{untl,aug}}},
\end{equation}
and show that its restriction to $(\Ran_{\on{untl,aug}}\times \Ran_{\on{untl,aug}})_{\on{\on{sub,disj}}}$ is equipped with:

\begin{itemize}

\item(a) a null-homotopy when precomposed with 
$$(\psi_{\on{aug}})_!\circ (\xi_{\on{aug}})^! (\CF) \boxtimes
\pi^! \circ \on{AddUnit}(\CG)\to \pi^! \circ \on{AddUnit}(\CF)\boxtimes \pi^! \circ \on{AddUnit}(\CG);$$

\item(b) a null-homotopy when precomposed with 
$$\pi^! \circ \on{AddUnit}(\CF)\boxtimes (\psi_{\on{aug}})_!\circ (\xi_{\on{aug}})^! (\CG)\to 
\pi^! \circ \on{AddUnit}(\CF)\boxtimes \pi^! \circ \on{AddUnit}(\CG);$$
 
\item(c) a datum of compatibility of the two null-homotopies when precomposed with
$$(\psi_{\on{aug}})_!\circ (\xi_{\on{aug}})^! (\CF) \boxtimes (\psi_{\on{aug}})_!\circ (\xi_{\on{aug}})^! (\CG)\to
\pi^! \circ \on{AddUnit}(\CF)\boxtimes \pi^! \circ \on{AddUnit}(\CG).$$

\end{itemize}

\sssec{}

Recall that
$$ \on{AddUnit}=\psi_!\circ \xi^!,$$
where the morphisms $\psi$ and $\xi$ were introduced in \secref{sss:xi and psi}.

\medskip

The map \eqref{e:aug pairing prime} is the pullback by means of $\pi$ of a map
\begin{equation} \label{e:unital pairing}
 \on{AddUnit}(\CF)\boxtimes \on{AddUnit}(\CG)\to \omega_{\Ran_{\on{untl}}\times \Ran_{\on{untl}}}.
\end{equation}

In its turn, the map \eqref{e:unital pairing} comes by the $(\psi\times \psi)_!,(\psi\times \psi)^!$
adjunction from a map
\begin{equation} \label{e:arrow pairing}
(\xi\times \xi)^!(\CF\boxtimes \CG) \to \omega_{\Ran^\to\times \Ran^\to}.
 \end{equation}
 
Finally, the map \eqref{e:arrow pairing} is the pullback by means of $\xi\times \xi$ of the map 
\begin{equation} \label{e:diag pairing}
\CF\boxtimes \CG \to (\on{diag}_{\Ran})_!(\omega_{\Ran})\to \omega_{\Ran\times \Ran}.
\end{equation} 

\sssec{}

Let us now calculate the composition
\begin{multline*}
(\psi_{\on{aug}})_!\circ (\xi_{\on{aug}})^! (\CF) \boxtimes
\pi^! \circ \on{AddUnit}(\CG)\to \pi^! \circ \on{AddUnit}(\CF)\boxtimes \pi^! \circ \on{AddUnit}(\CG)\to \\
\to \omega_{\Ran_{\on{untl,aug}}\times \Ran_{\on{untl,aug}}}.
\end{multline*}

It comes by pullback by means of $\on{id}_{\Ran_{\on{untl,aug}}}\times \pi$ of a map
$$(\psi_{\on{aug}})_!\circ (\xi_{\on{aug}})^! (\CF) \boxtimes  \psi_!\circ \xi^!(\CG)\to  
\omega_{\Ran_{\on{untl,aug}}\times \Ran_{\on{untl}}},$$
which in turn comes by means of the
$(\psi_{\on{aug}}\times \psi)_!,(\psi_{\on{aug}}\times \psi)^!$
adjunction from a map
$$(\xi_{\on{aug}})^! (\CF) \boxtimes \xi^!(\CG) \to \omega_{\Ran^\to_{\on{aug}}\times \Ran^\to},$$
where the latter is the pullback of \eqref{e:diag pairing} by means of the map $\xi_{\on{aug}}\times \xi$. 

\sssec{}

The open sub-prestack
$$(\Ran_{\on{untl,aug}}\times \Ran_{\on{untl,aug}})_{\on{\on{sub,disj}}}\subset \Ran_{\on{untl,aug}}\times 
\Ran_{\on{untl,aug}}$$
is contained in the preimage by means of $\on{id}_{\Ran_{\on{untl,aug}}}\times \pi$ of the open sub-prestack
$$(\Ran_{\on{untl,aug}}\times \Ran_{\on{untl}})_{\on{disj}}\subset \Ran_{\on{untl,aug}}\times \Ran_{\on{untl}},$$
defined by the condition that $K_1$ and $I_2$ have disjoint images. 

\medskip

The required null-homotopy is provided by the fact that the fiber product
$$(\Ran_{\on{untl,aug}}\times \Ran_{\on{untl}})_{\on{disj}}
\underset{\Ran_{\on{untl,aug}}\times \Ran_{\on{untl}}}\times (\Ran^\to_{\on{aug}}\times \Ran^\to)
\underset{\Ran\times \Ran}\times \Ran$$
is empty: indeed, it classifies the data of 
$$(K_1\subseteq I_1\supseteq J_1,\,J_2\subseteq I_2),$$
where $K_1\cap J_1\neq \emptyset$, $J_1=J_2$, while $K_1$ and $I_2$ have disjoint images. 

\sssec{}

The null-homotopy in point (b) is constructed similarly, and the compatibility of the two null-homotopies in (c) follows
from the construction. 

\ssec{Proof of \thmref{t:aug pairings}(ii)}

\sssec{Preparations-I}  \label{sss:opens}

Recall that the map $\on{diag}_{\Ran}$ is finitary pseudo-proper.  Hence, the complement of its scheme-theoretic image, denoted 
$$(\Ran\times \Ran)_{\neq},$$
is a well-defined open sub-prestack of $\Ran\times \Ran$ (see \lemref{l:replace lego}(b)).
. 

\medskip

Let $(\Ran\times \Ran)_{\subseteq}$ (reps., $(\Ran\times \Ran)_{\supseteq}$) be the prestack that assigns to $S\in \Sch$
the sub-groupoid of $\Maps(S,\Ran\times \Ran)$ consisting of those $I_1,I_2$ for which $I_1\subseteq I_2$
(resp., $I_1\supseteq I_2$).

\medskip

\begin{lem}
The map $(\Ran\times \Ran)_{\subseteq}\to \Ran\times \Ran$ is finitary pseudo-proper.
\end{lem}

\begin{proof}
For a pair of finite sets $\CI_1$ and $\CI_2$, the fiber product
$$(X^{\CI_1}\times X^{\CI_2})\underset{\Ran\times \Ran}\times (\Ran\times \Ran)_{\subseteq}$$ 
equals the colimit of $X^\CK$, taken over the category, whose objects are
$$\CI_1\to \CK\twoheadleftarrow \CI_2,$$
and morphisms are the surjections of the $\CK$'s.
\end{proof} 

Similarly, the map $(\Ran\times \Ran)_{\supseteq}\to \Ran\times \Ran$ is finitary pseudo-proper.
Let $$(\Ran\times \Ran)_{\not\subseteq}\subset \Ran\times \Ran \supset (\Ran\times \Ran)_{\not\supseteq}$$
be the complementary open sub-prestacks. 

\medskip

We have
$$(\Ran\times \Ran)_{\subseteq}\underset{\Ran\times \Ran}\times (\Ran\times \Ran)_{\supseteq}=\Ran,$$
where $\Ran$ maps to $\Ran\times \Ran$ diagonally. Hence the map
$$(\Ran\times \Ran)_{\not\subseteq}\cup (\Ran\times \Ran)_{\not\supseteq}\to (\Ran\times \Ran)_{\neq}$$
is a Zariski cover. Set
$$(\Ran\times \Ran)_{\not\subseteq,\not\supseteq}:=(\Ran\times \Ran)_{\not\subseteq}\cap (\Ran\times \Ran)_{\not\supseteq}
\subset \Ran\times \Ran.$$

\sssec{Plan of the proof}

Let us be given a pairing
\begin{equation} \label{e:aug pairing again}
\wt\CF\boxtimes \wt\CG|_{(\Ran_{\on{untl,aug}}\times \Ran_{\on{untl,aug}})_{\on{\on{sub,disj}}}}\to 
\omega_{(\Ran_{\on{untl,aug}}\times \Ran_{\on{untl,aug}})_{\on{\on{sub,disj}}}},
\end{equation}
and we wish to construct a pairing
\begin{equation} \label{e:recover pairing}
\on{TakeOut}(\wt\CF)\boxtimes \on{TakeOut}(\wt\CG)\to (\on{diag}_{\Ran})_!(\omega_{\Ran}).
\end{equation}

\medskip

Recall that
$$\on{TakeOut}(\wt\CF)=\on{Fib}\left(\phi^!\circ \iota^!(\wt\CF) \to (\xi^!_{\on{aug}})^R\circ \psi_{\on{aug}}^!(\wt\CF)\right),$$
$$\on{TakeOut}(\wt\CG)=\on{Fib}\left(\phi^!\circ \iota^!(\wt\CG) \to (\xi^!_{\on{aug}})^R\circ \psi_{\on{aug}}^!(\wt\CG)\right).$$

\medskip 

Thus, by \secref{sss:opens}, in order to construct \eqref{e:recover pairing}, it would be sufficient to supply 
the following:

\begin{itemize}

\item(a) a map 
\begin{equation} \label{e:main pairing}
\phi^!\circ \iota^!(\wt\CF) \boxtimes \phi^!\circ \iota^!(\wt\CG) \to \omega_{\Ran\times \Ran};
\end{equation} 

\item(b) a map
\begin{equation} \label{e:pairing subset}
(\xi^!_{\on{aug}})^R\circ \psi_{\on{aug}}^!(\wt\CF) \boxtimes \phi^!\circ \iota^!(\wt\CG) \to 
\omega_{(\Ran\times \Ran)_{\not\subseteq}},
\end{equation} 

\item(b') a datum of factoring of the restriction of the map \eqref{e:main pairing} to $(\Ran\times \Ran)_{\not\subseteq}$ via the map 
\eqref{e:pairing subset}.

\smallskip

\item(c) a map 
\begin{equation} \label{e:pairing supset}
\phi^!\circ \iota^!(\wt\CF) \boxtimes (\xi^!_{\on{aug}})^R\circ \psi_{\on{aug}}^!(\wt\CG)\to \omega_{(\Ran\times \Ran)_{\not\supseteq}};
\end{equation} 

\item(c') a datum of factoring of the restriction of the map \eqref{e:main pairing} to $(\Ran\times \Ran)_{\not\supseteq}$ via the map 
\eqref{e:pairing supset}.

\smallskip

\item(d) a map
\begin{equation} \label{e:pairing supset subset}
(\xi^!_{\on{aug}})^R\circ \psi_{\on{aug}}^!(\wt\CF) \boxtimes  (\xi^!_{\on{aug}})^R\circ \psi_{\on{aug}}^!(\wt\CG) \to
\omega_{(\Ran\times \Ran)_{\not\subseteq,\not\supseteq}},
\end{equation} 

\item(d') data of factoring of the restriction of the maps \eqref{e:pairing supset} and \eqref{e:pairing subset} to
$(\Ran\times \Ran)_{\not\subseteq,\not\supseteq}$ via the map \eqref{e:pairing supset subset}. 

\medskip

\item(d'') data of compatibility of the factorizations in (b') and (c') with that in (d').

\end{itemize}

\medskip

The map \eqref{e:main pairing} is obtained from the map \eqref{e:aug pairing again} by pullback by means
of the map $\iota\circ \phi$; we note that the image of this map lands in 
$(\Ran_{\on{untl,aug}}\times \Ran_{\on{untl,aug}})_{\on{\on{sub,disj}}}$.

\medskip 

The construction of the maps \eqref{e:pairing supset}, \eqref{e:pairing subset} and \eqref{e:pairing supset subset} is
based on the material explained in the next subsection. 

\sssec{Preparations-II} 

Consider now the following lax prestacks:
$$(\Ran\times \Ran)^\sim_{\not\subseteq},\,\, (\Ran\times \Ran)^\sim_{\not\supseteq},\,\, 
(\Ran\times \Ran)^\sim_{\not\subseteq,\not\supseteq}:$$

For $S\in \Sch$, the category $(\Ran\times \Ran)^\sim_{\not\subseteq}(S)$ is that of triples
$$(K_1\subseteq J_1,J_2),$$
where $K_1,J_1,J_2$ are finite non-empty subsets of $\Maps(S,X)$, and where we require that
$K_1$ and $J_2$ have disjoint images. As morphisms we allow inclusions of the $K_1$'s and isomorphisms of 
$J_1$'s and $J_2$'s. 

\medskip

Let $\kappa_{\subseteq}$ denote the forgetful map
$$(\Ran\times \Ran)^\sim_{\not\subseteq}\to (\Ran\times \Ran)_{\not\subseteq}.$$

\medskip

The category $(\Ran\times \Ran)^\sim_{\not\supseteq}(S)$ is that of triples
$$(J_1,K_2\subseteq J_2),$$
where $K_2,J_1,J_2$ are finite non-empty subsets of $\Maps(S,X)$, and where we require that
$K_2$ and $J_1$ have disjoint images. As morphisms we allow inclusions of the $K_2$'s and isomorphisms of 
$J_1$'s and $J_2$'s. 

\medskip

Let $\kappa_{\supseteq}$ denote the forgetful map
$$(\Ran\times \Ran)^\sim_{\not\supseteq}\to (\Ran\times \Ran)_{\not\supseteq}.$$

\medskip

The category $(\Ran\times \Ran)^\sim_{\not\subseteq,\not\supseteq}(S)$ is that of triples
$$(K_1\subseteq J_1,K_2\subseteq J_2),$$
where $K_1,K_2,J_1,J_2$ are finite non-empty subsets of $\Maps(S,X)$, and where we require that
$K_2$ and $J_1$ have disjoint images and $K_1$ and $J_2$ have disjoint images.
As morphisms we allow inclusions of the $K_1$'s and $K_2$'s and isomorphisms of 
$J_1$'s and $J_2$'s. 

\medskip

Let $\kappa_{\subseteq,\supseteq}$ denote the forgetful map
$$(\Ran\times \Ran)^\sim_{\not\subseteq,\not\supseteq}\to (\Ran\times \Ran)_{\not\subseteq,\not\supseteq}.$$

We will need the following assertion, which will be proved in \secref{ss:descr of omega}

\begin{prop}  \label{p:descr of omega}   
The maps
$$\kappa_{\subseteq}: (\Ran\times \Ran)^\sim_{\not\subseteq}\to (\Ran\times \Ran)_{\not\subseteq},$$
$$\kappa_{\supseteq}: (\Ran\times \Ran)^\sim_{\not\supseteq}\to (\Ran\times \Ran)_{\not\supseteq},$$
$$\kappa_{\subseteq,\supseteq}: (\Ran\times \Ran)^\sim_{\not\subseteq,\not\supseteq}\to 
(\Ran\times \Ran)_{\not\subseteq,\not\supseteq}$$
are universally homologically contractible\footnote{See \secref{ss:uhc prestack} for what this means.}.
\end{prop}

\sssec{Constructions of the maps} 

In order to construct the map \eqref{e:pairing supset}, by \propref{p:descr of omega} and \lemref{l:ff from uhc}, given $S\in \Sch$ and an 
$S$-point $(J_1,J_2)$ of $(\Ran\times \Ran)_{\not\subseteq}$, we need to construct a map
$$\left((\xi^!_{\on{aug}})^R\circ \psi_{\on{aug}}^!(\wt\CF) \boxtimes \phi^!\circ \iota^!(\wt\CG)\right)_{S,(J_1,J_2)}\to
\underset{\emptyset\neq K_1\subseteq J_1,K_1 \text{ is disjoint from }J_2}{\on{lim}}\, \omega_S.$$
 
 \medskip
 
We will now use the description of the object
$$(\xi^!_{\on{aug}})^R\circ \psi_{\on{aug}}^!(\wt\CF),$$
given by \propref{p:out}. Namely, we have
$$(\xi^!_{\on{aug}})^R\circ \psi_{\on{aug}}^!(\wt\CF)\simeq 
\underset{\emptyset\neq K_1\subseteq J_1}{\on{lim}}\, \wt\CF_{S,K_1\subseteq J_1}.$$

\medskip

Thus, we need to construct a compatible family of maps, for every $K_1\subseteq J_1$ such that $K_1$
and $J_2$ have disjoint images, of maps
\begin{equation} \label{e:pairing finally}
\wt\CF_{S,K_1\subseteq J_1}\overset{!}\otimes \wt\CG_{S,\emptyset \subset J_2}\to \omega_S.
\end{equation}

Note, however, that the quadruple $$(K_1\subseteq J_1),(\emptyset \subset J_2)$$ is an $S$-point of 
$(\Ran_{\on{untl,aug}}\times \Ran_{\on{untl,aug}})_{\on{\on{sub,disj}}}$, and the map \eqref{e:pairing finally}
results from \eqref{e:aug pairing again}.

\medskip

The maps in (c) and (d) are constructed similarly, and the compatibilities in (b'), (c') and (d') follow 
from the construction.  

\medskip

This completes the construction for \thmref{t:aug pairings}(ii).

\sssec{}

The compatibilities stated in \thmref{t:aug pairings}(iii) and \thmref{t:aug pairings}(iv) 
follows by unwinding the constructions. 

\ssec{Proof of \propref{p:descr of omega}}   \label{ss:descr of omega}

We will prove the assertion concerning $\kappa_{\subseteq}$,
as the other two cases are similar.

\sssec{}

It suffices to show that for any pair of finite sets $(\CI_1,\CI_2)$, the induced map
$$(X^{\CI_1}\times X^{\CI_2}) \underset{\Ran\times \Ran}\times (\Ran\times \Ran)^\sim_{\not\subseteq}\to
(X^{\CI_1}\times X^{\CI_2}) \underset{\Ran\times \Ran}\times (\Ran\times \Ran)_{\not\subseteq}$$
is universally homologically contractible.

\medskip 

For a \emph{non-empty} subset $\CK_1\subset \CI_1$, let 
$$(X^{\CI_1}\times X^{\CI_2})_{\CK_1\cap \CI_2=\emptyset} \subset (X^{\CI_1}\times X^{\CI_2})$$
be the open subscheme, corresponding to the condition that for every $k_1\in \CK_1$ and $i_2\in \CI_2$
the corresponding maps $S\rightrightarrows X$ have non-intersecting images. We have:
$$(X^{\CI_1}\times X^{\CI_2})_{\CK'_1\cap \CI_2=\emptyset}\cap (X^{\CI_1}\times X^{\CI_2})_{\CK''_1\cap \CI_2=\emptyset}=
(X^{\CI_1}\times X^{\CI_2})_{(\CK'_1\cup \CK''_1)\cap \CI_2=\emptyset},$$
and 
$$\underset{\emptyset\neq \CK_1}\cup\, (X^{\CI_1}\times X^{\CI_2})_{\CK_1\cap \CI_2=\emptyset}=
(X^{\CI_1}\times X^{\CI_2}) \underset{\Ran\times \Ran}\times (\Ran\times \Ran)_{\not\subseteq}.$$

\medskip

Hence, it suffices to show that for any $\CK_1$, the map
\begin{equation} \label{e:separate points}
(X^{\CI_1}\times X^{\CI_2})_{\CK_1\cap \CI_2=\emptyset} \underset{(\Ran\times \Ran)_{\not\subseteq}}\times
(\Ran\times \Ran)^\sim_{\not\subseteq}\to (X^{\CI_1}\times X^{\CI_2})_{\CK_1\cap \CI_2=\emptyset}
\end{equation}
is universally homologically contractible.

\sssec{}

We claim that the map \eqref{e:separate points} is value-wise contractible (see \lemref{l:contr contr}). 
Namely, we claim that for a given
$S$-point of $(X^{\CI_1}\times X^{\CI_2})_{\CK_1\cap \CI_2=\emptyset}$, the category of its lifts to an $S$-point of 
\begin{equation} \label{e:separate points above}
(X^{\CI_1}\times X^{\CI_2})_{\CK_1\cap \CI_2=\emptyset} \underset{(\Ran\times \Ran)_{\not\subseteq}}\times
(\Ran\times \Ran)^\sim_{\not\subseteq}
\end{equation}
is contractible. 

\medskip

For a given $S$-point of $(X^{\CI_1}\times X^{\CI_2})_{\CK_1\cap \CI_2=\emptyset}$ let 
$I_1,I_2\subset \Maps(S,X)$ be the images of the maps
$$\CI_1\to \Maps(S,X) \text{ and } \CI_2\to \Maps(S,X),$$
respectively. Let $K_1\subset I_1$ be the image of $\CK_1\subseteq \CI_1$.

\medskip

The category of lifts of our given $S$-point of $(X^{\CI_1}\times X^{\CI_2})_{\CK_1\cap \CI_2=\emptyset}$ to an
$S$-point of \eqref{e:separate points above} is that of finite non-empty subsets $K'_1\subseteq I_1$
such that $K'_1$ and $I'_2$ have disjoint images. 

\medskip

Note that left cofinal in this category is the full subcategory that consists of
those $K'_1$ that contain $K_1$.  Indeed, the left adjoint to the embedding is given by $K'_1\mapsto K'_1\cup K_1$.

\medskip

Hence, it is sufficient to show that the above subcategory is contractible. But it contains an initial object, namely,
one with $K'_1=K_1$. 

\section{An explicit expression for the Verdier dual}   \label{s:expl}

The contents of this section will be needed for one of the two proofs of the \emph{pointwise duality} statement,
\thmref{t:pointwise duality}. This section may be skipped on the first pass. 

\ssec{The formula}

In this subsection we will be working in the context of constructible sheaves. 

\sssec{}

Let $\CF$ be an object of $\Shv^!(\Ran)$, and consider the object $\CG:=\BD_{\Ran}(\CF)\in \Shv^!(\Ran)$.

\medskip

Fix a point $x\in X$ and consider the corresponding map $\on{pt}\to \Ran$, which we symbolically denote by $\{x\}$. 
Consider the object
$$\CG_{\{x\}}\in \Lambda\mod.$$

The goal of this subsection is to give an explicit description of $\CG_{\{x\}}$ in terms of $\wt\CF:=\on{AddUnit}_{\on{aug}}(\CF)$.

\sssec{}  \label{sss:not cont x}

Consider the following lax prestack, denoted $(\Ran_{\on{untl,aug}})_{x\notin}$. For $S\in \Sch$, the category
$(\Ran_{\on{untl,aug}})_{x\notin}(S)$ is the full subcategory of $\Ran_{\on{untl,aug}}(S)$, consisting of pairs
$K\subseteq J\neq \emptyset$ for which $K$ \emph{has a disjoint image} from $\{x\}$.  I.e., for every $k\in K$,
the image of the corresponding map $S\to X$ avoids $x\in X$. 

\medskip

Set
$$\CF^\dagger_x:=\on{C}^*_c\left((\Ran_{\on{untl,aug}})_{x\notin},\wt\CF|_{(\Ran_{\on{untl,aug}})_{x\notin}}\right).$$

The goal of this section is to prove the following assertion:

\begin{thmconstr}  \label{t:dagger}
There exists a canonical isomorphism 
$$(\CF^\dagger_x)^\vee\simeq \CG_{\{x\}}.$$
\end{thmconstr} 

\ssec{Plan of the proof}

\sssec{} 

Let $\Ran_{\neq \{x\}}$ be the open sub-prestack of $\Ran$, which is the complement
of the image of the (finitary, pseudo-proper) map $\on{pt}\to \Ran$, corresponding to the point $x$. 

\medskip

The following results from the definitions:

\begin{lem}
We have a canonical isomorphism
$$\CG_{\{x\}}\simeq \left(\on{coFib}\left(\on{C}_c^*\left(\Ran_{\neq \{x\}},\CF\right)\to \on{C}^*_c(\Ran,\CF)\right)\right)^\vee.$$
\end{lem}

\sssec{}

Recall now that $\wt\CF$ was constructed as 
$$\on{coFib}\left((\psi_{\on{aug}})_!\circ (\xi_{\on{aug}})^! (\CF) \to \pi^!\circ \psi_!\circ \xi^!(\CF)\right).$$
 
We are going to construct a commutative diagram
$$
\CD
\on{C}_c^*\left(\Ran_{\neq \{x\}},\CF\right) @>>>  
\on{C}^*_c\left((\Ran_{\on{untl,aug}})_{x\notin},
(\psi_{\on{aug}})_!\circ (\xi_{\on{aug}})^! (\CF)|_{(\Ran_{\on{untl,aug}})_{x\notin}}\right) \\
@VVV   @VVV   \\
 \on{C}^*_c(\Ran,\CF) @>>> \on{C}^*_c\left((\Ran_{\on{untl,aug}})_{x\notin},
\pi^!\circ \psi_!\circ \xi^!(\CF) |_{(\Ran_{\on{untl,aug}})_{x\notin}}\right),
\endCD
$$
with the horizontal arrows being isomorphisms. This will prove \thmref{t:dagger}. 

\sssec{}

To construct the isomorphism
$$ \on{C}^*_c(\Ran,\CF)  \simeq 
\on{C}^*_c\left((\Ran_{\on{untl,aug}})_{x\notin},\pi^!\circ \psi_!\circ \xi^!(\CF) |_{(\Ran_{\on{untl,aug}})_{x\notin}}\right)$$
we consider the Cartesian diagram
$$
\CD
(\Ran_{\on{untl,aug}})_{x\notin} \underset{ \Ran_{\on{untl}}} \times \Ran^\to @>>>  \Ran^\to  @>{\xi}>>  \Ran \\
@VVV    @VV{\psi}V    \\
(\Ran_{\on{untl,aug}})_{x\notin} @>{(K\subseteq I)\mapsto I}>>  \Ran_{\on{untl}}.
\endCD
$$

As in the proof of \propref{p:inserting the unit}, it suffices to show that the map
$$(\Ran_{\on{untl,aug}})_{x\notin} \underset{ \Ran_{\on{untl}}} \times \Ran^\to \to \Ran$$
has contractible fibers over any $S$-point of $\Ran$. 

\medskip

For such a point $J\subset \Maps(S,X)$, the category of its lifts to an $S$-point of the lax prestack 
$(\Ran_{\on{untl,aug}})_{x\notin} \underset{ \Ran_{\on{untl}}} \times \Ran^\to$ is that of 
$$(K\subseteq I\supseteq J),\quad K \text{ is disjoint from } x.$$

However, this subcategory has an initial object, namely, one with $K=\emptyset$ and
$I=J$. 

\ssec{Construction}

To construct the isomorphism 
\begin{equation} \label{e:isom away from x}
\on{C}_c^*\left(\Ran_{\neq \{x\}},\CF\right) 
\simeq \on{C}^*_c\left((\Ran_{\on{untl,aug}})_{x\notin},
(\psi_{\on{aug}})_!\circ (\xi_{\on{aug}})^! (\CF)|_{(\Ran_{\on{untl,aug}})_{x\notin}}\right) 
\end{equation} 
we proceed as follows.

\sssec{}

Consider the lax prestack, denoted $\Ran^\sim_{\neq \{x\}}$, whose category of $S$-points is that of pairs
$K\subseteq J$ such that $K\neq \emptyset$ and $K$ is disjoint from $x$. As morphisms we allow
inclusions of the $K$'s and isomorphism of the $J$'s. 

\medskip

We have the evident forgetful map $\kappa_x:\Ran^\sim_{\neq \{x\}}\to \Ran_{\neq \{x\}}$.  We will need the
following assertion: 

\begin{prop}  \label{p:acyclicity}
The map $\kappa_x$ is universally homologically contractible. 
\end{prop}

\begin{proof}

Recall that the map
$$\kappa_{\subseteq}: (\Ran\times \Ran)^\sim_{\not\subseteq}\to (\Ran\times \Ran)_{\not\subseteq}$$
is universally homologically contractible by \propref{p:descr of omega}.

\medskip

Now, the map $\kappa_x$ is obtained from  $\kappa_{\subseteq}$ as a base change by means of
$$\Ran_{\neq \{x\}}\to (\Ran\times \Ran)_{\not\subseteq},\quad J\mapsto (J,\{x\}).$$

\end{proof} 

\sssec{}

Note now that there is a canonically defined map
$$\Ran^\sim_{\neq \{x\}}\overset{\mu}\longrightarrow \Ran^\to_{\on{aug}}\underset{\Ran_{\on{untl,aug}}}\times
(\Ran_{\on{untl,aug}})_{x\notin},\quad (K\subseteq J)\mapsto \left((K\subseteq J\supseteq J),(K\subseteq J)\right)$$
that makes the diagram
$$
\CD
\Ran^\sim_{\neq \{x\}} @>{\mu}>>  \Ran^\to_{\on{aug}}\underset{\Ran_{\on{untl,aug}}}\times (\Ran_{\on{untl,aug}})_{x\notin} 
@>>>  (\Ran_{\on{untl,aug}})_{x\notin}  \\
@V{\kappa_x}VV    @VVV   @VVV   \\
\Ran_{\neq \{x\}} & & \Ran^\to_{\on{aug}}  @>{\psi_{\on{aug}}}>>  \Ran_{\on{untl,aug}}  \\
@VVV   @VV{\xi_{\on{aug}}}V   \\
\Ran @>{\on{id}}>> \Ran
\endCD
$$ 
commutative. 

\medskip

By \propref{p:acyclicity}, in order to construct \eqref{e:isom away from x}, it suffices to construct an isomorphism
\begin{multline}   \label{e:big iso}
\on{C}^*_c\left(\Ran^\sim_{\neq \{x\}},
\mu^!\left(\xi^!_{\on{aug}}(\CF)|_{\Ran^\to_{\on{aug}}
\underset{\Ran_{\on{untl,aug}}}\times (\Ran_{\on{untl,aug}})_{x\notin}}\right)\right)\simeq \\
\simeq 
\on{C}^*_c\left((\Ran_{\on{untl,aug}})_{x\notin},
(\psi_{\on{aug}})_!\circ (\xi_{\on{aug}})^! (\CF)|_{(\Ran_{\on{untl,aug}})_{x\notin}}\right).
\end{multline} 

\sssec{}

Note that we can  identity the right-hand side in \eqref{e:big iso} with
$$\on{C}^*_c\left(\Ran^\to_{\on{aug}}\underset{\Ran_{\on{untl,aug}}}\times (\Ran_{\on{untl,aug}})_{x\notin}, 
\xi^!_{\on{aug}}(\CF)|_{\Ran^\to_{\on{aug}}\underset{\Ran_{\on{untl,aug}}}\times (\Ran_{\on{untl,aug}})_{x\notin}}\right).$$

Now, the required isomorphism follows from the next assertion (see \lemref{l:rel hom type}):

\begin{lem}
The map $\mu$ is value-wise homotopy-type equivalence. 
\end{lem}

\begin{proof}

We need to show that for any $S\in \Sch$, the functor 
$$\Ran^\sim_{\neq \{x\}}(S)\to 
\left(\Ran^\to_{\on{aug}}\underset{\Ran_{\on{untl,aug}}}\times (\Ran_{\on{untl,aug}})_{x\notin}\right)(S)$$
induces an equivalence of homotopy types. We will show that this happens after restriction to the fiber over
any given point $J\in \Maps(S,\Ran)$.

\medskip

The fiber of the left-hand side is the category of all $\emptyset\neq K\subseteq J$ so that $K$ is disjoint from $x$.
The fiber of the right-hand side is the category of
$$K'\subseteq I \supseteq J$$
with $J\cap K'\neq \emptyset$ and where $K'$ is disjoint from $x$.  The functor in question is
$$(K\subseteq J)\mapsto (K\subseteq J\supseteq J).$$

Now, this functor admits a right adjoint, namely
$$(K'\subseteq I \supseteq J)\mapsto (K'\cap J\subseteq J).$$

\end{proof} 

\newpage

\centerline{\bf Part IV: Factorization}  

\section{Verdier duality and factorization}   \label{s:factorize}

As was mentioned in the preamble to \secref{s:pairings aug}, at some point in the proof of the cohomological
product formula, we will need to show that a certain map in $\Shv^!(\Ran)$
$$\CB_{\on{red}}\to \BD_{\Ran}(\CA_{\on{red}})$$
is an isomorphism. A priori, to check that some map in $\Shv^!(\Ran)$ is an isomorphism, we need to check that
its pullback to $X^\CI$ for every finite non-empty set $\CI$ is an isomorphism.  

\medskip

What we will do in this section is show 
that there are certain pieces of structure on $\CA_{\on{red}}$ and $\CB_{\on{red}}$ that allow to consider only the case of $\CI$ being a 
one-element set. The relevant pieces of structure are that of \emph{cocommutative factorization coalgebra} and 
\emph{commutative factorization algebra}, respectively. 

\medskip

The material from this section will be used in \secref{s:pointwise} for the proof of the local duality statement, \thmref{t:local duality}. 
The prerequisites for the present section are Sects. \ref{s:prestacks} and \ref{ss:Verdier on Ran}
(specifically, \thmref{t:products}). 

\ssec{Commutative and cocommutative factorization coalgebras}  \label{ss:alg and coalg}

Following \cite{BD}, \cite{FG} or \cite{Lu2}, one can introduce the notion of \emph{factorization sheaf} on the Ran
space. The reason this notion is relevant for us is that if $\CF_1\to \CF_2$ is a homomorphism of factorization sheaves,
which induces an isomorphism when evaluated on $X^{\{1\}}$, then it is an isomorphism.

\medskip

The general notion of factorization sheaf is indispensable for many purposes, but introducing it requires a longish detour of 
homotopy-theoretic nature. For our purposes it will be sufficient to deal with \emph{commutative} or 
\emph{cocommutative} factorization algebras, and those can be introduced directly, using the structure on $\Ran$ of 
commutative semi-group. This will be done in the present subsection. 

\sssec{}  \label{sss:com alg first}

Note that the prestack $\Ran$ carries a canonically defined commutative semi-group structure,
symbolically denoted by 
$$\Ran\times \Ran \overset{\unn}\longrightarrow \Ran.$$

Thus, pullback endows the category $\Shv^!(\Ran)$ with a symmetric comonoidal structure. 
Therefeore, we can talk about commutative algebras in $\Shv^!(\Ran)$. These are
objects $\CF\in \Shv^!(\Ran)$, equipped with a commutative binary operation
$$\CF\boxtimes \CF\to \unn^!(\CF),$$
which satisfies a homotopy-coherent system of compatibilities.

\sssec{}

We have:

\begin{lem} \label{l:glue union}
For any $n$, the product map
$$\Ran^{\times n}\to \Ran$$
is finitary pseudo-proper.
\end{lem}

\begin{proof}

For simplicity, let us consider the case of $n=2$. The proof is parallel to \secref{sss:diag Ran}. For a finite set $\CI$
let us calculate the fiber product
$$X^\CI \underset{\Ran}\times (\Ran\times \Ran).$$ 

It equals the colimit
$$\underset{\CI\twoheadrightarrow \CK\simeq \CK_1\cup \CK_2}{\on{colim}}\, X^{\CK},$$
where the category of indices has as objects the diagrams
$$\CI\twoheadrightarrow \CK\simeq \CK_1\cup \CK_2,\quad \CK_1,\CK_2\neq \emptyset,$$
and as morphisms commutative diagrams
$$
\CD
\CI  @>>>   \CK'  @<<<  \CK'_1\sqcup \CK'_2 \\
@V{\on{id}}VV   @VVV   @VVV   \\
\CI @>>>   \CK''  @<<<  \CK''_1\sqcup \CK''_2
\endCD
$$
with surjective vertical maps.

\end{proof} 

\sssec{}

In particular, by \corref{c:pseudo-proper}, the functor $\unn^!$ admits a left adjoint, denoted $\unn_!$,
and thus the symmetric comonoidal structure on $\Shv^!(\Ran)$, gives rise by passing to left adjoints, to 
a symmetric \emph{monoidal} structure.  We shall refer to it as the \emph{convolution}
symmetric monoidal structure. 

\medskip

Tautologically, commutative algebras in $\Shv^!(\Ran)$ as defined in \secref{sss:com alg first}
are the same as commutative algebras with respect to the convolution symmetric monoidal structure. 

\sssec{}  \label{sss:factor algebras}

Let 
$$(\Ran\times \Ran)_{\on{disj}}\subset \Ran\times \Ran$$
be the open subfunctor corresponding to the condition that $I_1,I_2\subset \Maps(S,X)$ 
have disjoint images.

\medskip

Let $\CF$ be a commutative algebra in $\Shv^!(\Ran)$ (with respect to the convolution symmetric monoidal structure).
We shall say that $\CF$ is a \emph{commutative factorization algebra}
if the following condition holds: the map
$$\CF\boxtimes \CF\to \unn^!(\CF),$$
has the property that the induced map
$$\CF\boxtimes \CF|_{(\Ran\times \Ran)_{\on{disj}}}\to \unn^!(\CF)|_{(\Ran\times \Ran)_{\on{disj}}}$$
is an isomorphism. 

\medskip

The following (nearly evident) observation will be of crucial importance:

\begin{lem}  \label{l:isom on diag}
Let $\CF_1\to \CF_2$ be a homomorphism of commutative factorization algebras. Suppose that
the induced map $(\CF_1)_X\to (\CF_2)_X$ is an isomorphism in $\Shv(X)$. Then the initial map
is an isomorphism.
\end{lem}

Here for $\CF\in \Shv^!(\Ran)$ we denote by $\CF_X\in \Shv(X)$ the pullback of $\CF$ under the map
$$\on{ins}_{\{1\}}:X=X^{\{1\}} \to \Ran.$$

\sssec{}

By a \emph{cocommutative coalgebra} in $\Shv^!(\Ran)$ we shall mean a cocommutative coalgebra object in $\Shv^!(\Ran)$
with respect to the convolution symmetric monoidal structure. 

\medskip

I.e., this is an object $\CG\in \Shv^!(\Ran)$ equipped with a cocommutative cobinary map
\begin{equation} \label{e:coproduct}
\CG\to \unn_!(\CG\boxtimes \CG),
\end{equation}
which satisfies a homotopy-coherent system of compatibilities.

\sssec{}

Let us now note the following:

\begin{lem}  \label{l:on disj}  \hfill

\smallskip

\noindent{\em(a)}
The restriction of the map $\unn$ to $(\Ran\times \Ran)_{\on{disj}}$ is \'etale. 

\smallskip

\noindent{\em(b)} The diagonal map
$$(\Ran\times \Ran)_{\on{disj}}\to (\Ran\times \Ran)_{\on{disj}}\underset{\Ran}\times (\Ran\times \Ran)$$
is open and closed (i.e., is the inclusion of a union of connected components). 
\end{lem}

\begin{rem}
Let $f:Z\to Y$ be a (separated) map between schemes, and let $\overset{\circ}{Z}\subset Z$  be an open subsheme,
such that $f|_{\overset{\circ}Z}$ is \'etale.  In this case, the diagonal map 
$$\overset{\circ}{Z}\to Z\underset{Y}\times \overset{\circ}{Z}$$
is open and closed. 

\medskip

Point (b) of \lemref{l:on disj} states that the above phenomenon takes place for the map $$\unn:(\Ran\times \Ran)\to \Ran$$ and 
$(\Ran\times \Ran)_{\on{disj}}\subset (\Ran\times \Ran)$, but it requires a proof since we are dealing not with schemes
but prestacks.
\end{rem}

\begin{proof}[Proof of \lemref{l:on disj}]

To prove point (a), it suffices to show that the map
$$X^\CI \underset{\Ran}\times (\Ran\times \Ran)_{\on{disj}}\to X^\CI$$
is \'etale for any finite set $\CI$.

\medskip

However, it is easy to see that the above fiber product identifies with
$$\underset{\CI=\CI_1\sqcup \CI_2}\sqcup\,  (X^{\CI_1}\times X^{\CI_2})_{\on{disj}},$$
where 
$$(X^{\CI_1}\times X^{\CI_2})_{\on{disj}}\subset X^{\CI_1}\times X^{\CI_2}$$
is the open subset corresponding to the condition that for any $i_1\in \CI_1$ and $i_2\in \CI_2$, the
corresponding maps $S\rightrightarrows X$ have non-intersecting images.

\medskip 

For point (b) it suffices to show that for any pair of finite sets $\CI_1\times \CI_2$, the map
\begin{equation} \label{e:Ran Ran}
(X^{\CI_1}\times X^{\CI_2})\underset{\Ran\times \Ran}\times 
(\Ran\times \Ran)_{\on{disj}}\to (X^{\CI_1}\times X^{\CI_2})
\underset{\Ran\times \Ran}\times (\Ran\times \Ran) \underset{\Ran}\times (\Ran\times \Ran)_{\on{disj}}
\end{equation}
is open and closed. 

\medskip

The left-hand side in \eqref{e:Ran Ran} identifies with $(X^{\CI_1}\times X^{\CI_2})_{\on{disj}}$. The right-hand side 
is the disjoint union of the schemes $(X^{\CJ_1}\times X^{\CJ_2})_{\on{disj}}$ over the set of isomorphisms
$$\CI_1\sqcup \CI_2\simeq \CJ_1\sqcup \CJ_2,$$
where $\CJ_1$ and $\CJ_2$ are finite non-empty sets. 

\medskip

The map in \eqref{e:Ran Ran} is the inclusion of the connected component with $\CJ_1=\CI_1$ and $\CJ_2=\CI_2$.

\end{proof}

\sssec{}

From point (b) of \lemref{l:on disj}, by base change (i.e., using the fact that the map $\unn$ is pseudo-proper) we obtain:

\begin{cor}  \label{c:on disj}
For $\CG'\in \Shv^!(\Ran\times \Ran)$ there exists a canonically defined map
$$\unn^!\circ \unn_!(\CG')|_{(\Ran\times \Ran)_{\on{disj}}}\to \CG'|_{(\Ran\times \Ran)_{\on{disj}}}.$$
\end{cor}

\sssec{}   \label{sss:factor coalgebras}
 
Given a \emph{cocommutative coalgebra} $\CG$ in $\Shv^!(\Ran)$, we shall say that it is a 
\emph{cocommutative factorization coalgebra} if the composed map
\begin{equation} \label{e:factor for co}
\unn^!(\CG)|_{(\Ran\times \Ran)_{\on{disj}}}\to \unn^!\circ \unn_!(\CG\boxtimes \CG)|_{(\Ran\times \Ran)_{\on{disj}}}\to
\CG\boxtimes \CG|_{(\Ran\times \Ran)_{\on{disj}}},
\end{equation} 
where the first arrow is induced by \eqref{e:coproduct}, and the second by \corref{c:on disj}, is an isomorphism. 

\ssec{Interaction with duality}  \label{ss:factor and Verdier}

As we shall see shortly, if $\CG$ is a cocommutative coalgebra in $\Shv^!(\Ran)$, then its Verdier dual
$\BD_{\Ran}(\CG)$ has a natural structure of commutative algebra in $\Shv^!(\Ran)$. 

\medskip

Suppose now that
$\CG$ was actually a cocommutative factorization coalgebra. Will it be true that $\BD_{\Ran}(\CG)$ is
a commutative factorization algebra?  The answer turns out to be ``yes", provided that $\CG$ satisfies
a certain connectivity hypothesis; it is here that \thmref{t:products} will play a role. 

\sssec{}  \label{sss:alg structure on dual}

Let $\CG$ be a cocommutative coalgebra in $\Shv^!(\Ran)$. We claim that $$\CF:=\BD_{\Ran}(\CG)$$
acquires a canonically defined structure of commutative algebra in $\Shv^!(\Ran)$.

\medskip

Indeed, starting from the coproduct map
$$\CG\to \unn_!(\CG\boxtimes \CG),$$
and applying the functor $\BD_{\Ran}$, we obtain a map
\begin{equation} \label{e:dual of product}
\BD_{\Ran}\circ \unn_!(\CG\boxtimes \CG)\to \BD_{\Ran}(\CG).
\end{equation} 

\medskip

Precomposing with \eqref{e:duality and direct image}, we obtain a map
$$\unn_!(\BD_{\Ran\times \Ran}(\CG\boxtimes \CG))\to \BD_{\Ran}(\CG).$$

\medskip

Finally, precomposing with \eqref{e:product and Verdier}, we obtain the desired product map
\begin{equation} \label{e:product on dual}
\unn_!\left(\BD_{\Ran}(\CG)\boxtimes \BD_{\Ran}(\CG)\right)\to \BD_{\Ran}(\CG).
\end{equation} 

The higher compatibilities for \eqref{e:product on dual} are constructed similarly. 

\sssec{}

We now claim:

\begin{prop} \label{p:Verdier factorizes}
Let $\CG\in \Shv^!(\Ran)$ be a cocommutative coalgebra, and let $\CF:=\BD_{\Ran}(\CG)$ 
be its Verdier dual commutative algebra. Assume that $\CG$ is a cocommutative factorization
algebra and: 

\begin{itemize}

\item We are working in the context of constructible sheaves; 

\item The ring of coefficients $\Lambda$ is of finite cohomological dimension;

\item $X$ is a curve;

\item The object $\CG_X\in \Shv^!(X)$ lives in (perverse) cohomological degrees $\leq -2$ and all of its cohomologies are compact.

\end{itemize}

Then $\CF$ is a commutative factorization algebra.
\end{prop}

The rest of this subsection is devoted to the proof of \propref{p:Verdier factorizes}. 

\sssec{}

For a finite set $\CI$ consider the corresponding object 
$$\CG_{X^\CI}:=\on{ins}_\CI^!(\CG)\in \Shv(X^\CI).$$ 
The factorization hypothesis on $\CG$ implies that the restriction of $\CG_{X^\CI}$ to 
$\overset{\circ}X{}^\CI\subset X^\CI$ is isomorphic to the restriction of $(\CG_X)^{\boxtimes \CI}$.

\medskip

Hence, the assumption that $\CG_X$ lives cohomological degrees $\leq -2$ implies that
$\CG_{X^\CI}|_{\overset{\circ}X{}^\CI}$ lives in degrees $\leq -2i$. In particular, it satisfies
the assumption of \thmref{t:Verdier on Ran}.

\medskip

Furthermore, we obtain that all of the cohomologies of $\CG_{X^\CI}|_{\overset{\circ}X{}^\CI}$ are
compact. Since we are working in the constructible category and since $X^\CI$ is stratified by 
$\overset{\circ}X{}^\CJ$ for $\CJ\twoheadleftarrow \CI$, we obtain that $\CG_{X^\CI}$ itself is bounded
above and all of its cohomologies are compact.

\medskip

Hence, by \thmref{t:products}, the map
\begin{equation} \label{e:duality and prod}
\BD_{\Ran}(\CG)\boxtimes \BD_{\Ran}(\CG)\to \BD_{\Ran\times \Ran}(\CG\boxtimes \CG)
\end{equation}
is an isomorphism. 

\sssec{}

We claim that there exists a commutative diagram
\begin{equation} \label{e:diag of duals}
\CD
\left(\BD_{\Ran}(\CG)\boxtimes \BD_{\Ran}(\CG)\right)|_{(\Ran\times \Ran)_{\on{disj}}}  @>>>  
\BD_{(\Ran\times \Ran)_{\on{disj}}}\left((\CG\boxtimes \CG)|_{(\Ran\times \Ran)_{\on{disj}}}\right)   \\
@VVV   @VVV  \\
\left(\unn^!(\BD_{\Ran}(\CG))\right)|_{(\Ran\times \Ran)_{\on{disj}}}  @>>>
\BD_{(\Ran\times \Ran)_{\on{disj}}}\left(\unn^!(\CG)|_{(\Ran\times \Ran)_{\on{disj}}}\right).
\endCD
\end{equation}

The top horizontal arrow is the composition
\begin{multline}   \label{e:top arrow}
\left(\BD_{\Ran}(\CG)\boxtimes \BD_{\Ran}(\CG)\right)|_{(\Ran\times \Ran)_{\on{disj}}}  \to \\
\to \left(\BD_{\Ran\times \Ran}(\CG\times \CG)\right)_{(\Ran\times \Ran)_{\on{disj}}}\to 
\BD_{(\Ran\times \Ran)_{\on{disj}}}\left((\CG\boxtimes \CG)|_{(\Ran\times \Ran)_{\on{disj}}}\right),
\end{multline}
where the first arrow is \eqref{e:product and Verdier}, and 
the second arrow comes from the fact that $(\Ran\times \Ran)_{\on{disj}}\hookrightarrow \Ran\times \Ran$
is an open embedding (see \secref{sss:etale and dual}). 

\medskip

The bottom arrow comes from the fact that the map $\unn|_{(\Ran\times \Ran)_{\on{disj}}}$ is \'etale
(again, by \secref{sss:etale and dual}). 

\medskip

The left vertical arrow is the restriction to $(\Ran\times \Ran)_{\on{disj}}$ of the map
$$\BD_{\Ran}(\CG)\boxtimes \BD_{\Ran}(\CG)\to \unn^!(\BD_{\Ran}(\CG)),$$
obtained by the $(\unn_!,\unn^!)$ adjunction from the map 
$$\unn_!\left(\BD_{\Ran}(\CG)\boxtimes \BD_{\Ran}(\CG)\right)\to \BD_{\Ran}(\CG)$$
of \eqref{e:product on dual}. 

\medskip

The right vertical arrow is obtained by applying the (contravariant) functor $\BD_{(\Ran\times \Ran)_{\on{disj}}}$ to
the map
$$\unn^!(\CG)|_{(\Ran\times \Ran)_{\on{disj}}}\to \CG\boxtimes \CG|_{(\Ran\times \Ran)_{\on{disj}}}$$
of \eqref{e:factor for co}.

\medskip

The fact that the above diagram is commutative follows by unwinding the definitions.

\sssec{}

We need to prove that the left vertical arrow in \eqref{e:diag of duals} is an isomorphism. We will show that
all other arrows in this diagram are isomorphisms.

\medskip

The right vertical arrow is an isomorphism due to the assumption that $\CG$ is a cocommutative factorization
algebra. 

\medskip

In the top horizontal arrow, which is given by \eqref{e:top arrow}, the first map is an isomorphism because
\eqref{e:duality and prod} is an isomorphism. 

\medskip

Thus, it remains to show that the second map in \eqref{e:top arrow} and the bottom horizontal arrow in
\eqref{e:diag of duals} are isomorphisms. However, this follows from \lemref{l:etale and dual}.

\section{Factorization for augmented sheaves}    \label{s:fact aug}

In this section we will introduce several more ingredients required for the proof of the local duality statement, 
\thmref{t:local duality}, to be used in \secref{s:pointwise}.

\medskip

Let us recall the set-up of the preamble of \secref{s:pairings aug}. We start with 
$\CA_{\on{untl,aug}},\CB_{\on{untl,aug}}\in \Shv^!(\Ran_{\on{untl,aug}})$,
and we want to establish an isomorphism $\CB_{\on{red}}\simeq\BD_{\Ran}(\CA_{\on{red}})$, where 
$$\CA_{\on{red}}:=\on{TakeOut}(\CA_{\on{untl,aug}}) \text{ and } \CB_{\on{red}}:=\on{TakeOut}(\CB_{\on{untl,aug}}).$$

\medskip

As was explained in \secref{s:pairings aug}, the map in one direction 
\begin{equation} \label{e:B to A prev}
\CB_{\on{red}}\to \BD_{\Ran}(\CA_{\on{red}}) 
\end{equation}
comes from a pairing between $\CA_{\on{untl,aug}}$ and $\CB_{\on{untl,aug}}$. What we do in this section is explain what additional
pieces of structure on $\CA_{\on{untl,aug}}$, $\CB_{\on{untl,aug}}$ and a pairing between them are needed to make \eqref{e:B to A prev} into
a homomorphism of commutative factorization algebras. 

\medskip

The prerequisite for this section are: all of Part I and Sects. \ref{s:pairings aug} and \ref{s:factorize}.

\ssec{(Co)commutative and (co)algebras in unital augmented sheaves}  

In this subsection we will explain what kind of structure on $\wt\CF:=\on{AddUnit}_{\on{aug}}(\CF)\in \Shv^!(\Ran_{\on{untl,aug}})$
corresponds to a structure on $\CF\in \Shv^!(\Ran)$ of (co)commutative (co)algebra with respect to the convolution
symmetric monoidal structure. 

\sssec{}

We consider the category $\Shv^!(\Ran_{\on{untl,aug}})$ as equipped with a symmetric monoidal structure
given by the \emph{pointwise} tensor product
$$\wt\CF,\wt\CG\mapsto \wt\CF\overset{!}\otimes \wt\CG.$$

Thus, we can speak about commutative algebras and cocommutative coalgebras in $\Shv^!(\Ran_{\on{untl,aug}})$.
We will now connect these notions to the notions of commutative algebra and cocommutative coalgebra in $\Shv^!(\Ran)$
with respect to the convolution product, introduced in \secref{ss:alg and coalg}.  This relies on the following assertion:

\begin{thmconstr} \label{t:units and tensor}
The functor 
$$\on{AddUnit}_{\on{aug}}: \Shv^!(\Ran)\to \Shv^!(\Ran_{\on{untl,aug}})$$
has a natural symmetric monoidal structure.
\end{thmconstr}

Combining with \thmref{t:main}, we obtain:

\begin{cor}  \label{c:structure on red}
For an object $\CF\in \Shv^!(\Ran)$, to specify on it a structure on it of (co)commutative (co)algebra is equivalent to a specifying
a structure of (co)commutative (co)algebra on $\on{AddUnit}_{\on{aug}}(\CF)$.
\end{cor}

\sssec{Proof of \thmref{t:units and tensor}, Step 1}

We will construct the data of compatibility of binary operations:
\begin{equation} \label{e:binary compatibility}
\on{AddUnit}_{\on{aug}}(\CF) \overset{!}\otimes  \on{AddUnit}_{\on{aug}}(\CG)\simeq 
\on{AddUnit}_{\on{aug}}(\unn_!(\CF\boxtimes \CG)).
\end{equation}
The datum for  higher compatibility is constructed similarly. 

\medskip

Recall that
$$\on{AddUnit}_{\on{aug}}=\on{coFib}\left((\psi_{\on{aug}})_!\circ \xi_{\on{aug}}^!\to \pi^!\circ \psi_!\circ \xi^!\right).$$

\medskip

We will construct a map
\begin{equation} \label{e:comp tensor}
\on{AddUnit}_{\on{aug}}(\CF) \overset{!}\otimes  \on{AddUnit}_{\on{aug}}(\CG)\to 
\on{AddUnit}_{\on{aug}}(\unn_!(\CF\boxtimes \CG))
\end{equation} 
by constructing the following maps:

\begin{itemize}

\item(a) A map (in fact, an isomorphism) 
$$\on{diag}_{\Ran_{\on{untl,aug}}}^!\circ (\pi\times \pi)^!\circ (\psi\times \psi)_!\circ (\xi\times \xi)^!(\CF\boxtimes \CG) \to
\pi^!\circ \psi_!\circ \xi^! \circ \unn_! (\CF\boxtimes \CG);$$

\item(b) A map
$$\on{diag}_{\Ran_{\on{untl,aug}}}^!\circ (\pi\times \on{id}_{\Ran_{\on{untl,aug}}})^!\circ (\psi\times \psi_{\on{aug}})_!
\circ (\xi\times \xi_{\on{aug}})^!(\CF\boxtimes \CG)\to (\psi_{\on{aug}})_!\circ \xi_{\on{aug}}^! \circ \unn_! (\CF\boxtimes \CG);$$

\item(b') A homotopy between the composition
\begin{multline*}
\on{diag}_{\Ran_{\on{untl,aug}}}^!\circ (\pi\times \on{id}_{\Ran_{\on{untl,aug}}})^!\circ (\psi\times \psi_{\on{aug}})_!
\circ (\xi\times \xi_{\on{aug}})^!(\CF\boxtimes \CG)\to \\
\to  (\psi_{\on{aug}})_!\circ \xi_{\on{aug}}^! \circ \unn_! (\CF\boxtimes \CG) \to 
\pi^!\circ \psi_!\circ \xi^! \circ \unn_! (\CF\boxtimes \CG)
\end{multline*}
and
\begin{multline*}
\on{diag}_{\Ran_{\on{untl,aug}}}^!\circ (\pi\times \on{id}_{\Ran_{\on{untl,aug}}})^!\circ (\psi\times \psi_{\on{aug}})_!
\circ (\xi\times \xi_{\on{aug}})^!(\CF\boxtimes \CG)\to \\
\to \on{diag}_{\Ran_{\on{untl,aug}}}^!\circ (\pi\times \pi)^!\circ (\psi\times \psi)_!\circ (\xi\times \xi)^!(\CF\boxtimes \CG) \to 
\pi^!\circ \psi_!\circ \xi^! \circ \unn_! (\CF\boxtimes \CG); 
\end{multline*}


\item(c) A map
$$\on{diag}_{\Ran_{\on{untl,aug}}}^!\circ (\on{id}_{\Ran_{\on{untl,aug}}}\times \pi)^!\circ (\psi_{\on{aug}}\times \psi)_!
\circ (\xi_{\on{aug}}\times \xi)^!(\CF\boxtimes \CG)\to (\psi_{\on{aug}})_!\circ \xi_{\on{aug}}^! \circ \unn_! (\CF\boxtimes \CG);$$

\item(c') A homotopy between the composition
\begin{multline*}
\on{diag}_{\Ran_{\on{untl,aug}}}^!\circ (\on{id}_{\Ran_{\on{untl,aug}}}\times \pi)^!\circ (\psi_{\on{aug}}\times \psi)_!
\circ (\xi_{\on{aug}}\times \xi)^!(\CF\boxtimes \CG)   \to \\
\to (\psi_{\on{aug}})_!\circ \xi_{\on{aug}}^! \circ \unn_! (\CF\boxtimes \CG) 
\to \pi^!\circ \psi_!\circ \xi^! \circ \unn_! (\CF\boxtimes \CG)
\end{multline*}
and
\begin{multline*}
\on{diag}_{\Ran_{\on{untl,aug}}}^!\circ (\on{id}_{\Ran_{\on{untl,aug}}}\times \pi)^!\circ (\psi_{\on{aug}}\times \psi)_!
\circ (\xi_{\on{aug}}\times \xi)^!(\CF\boxtimes \CG)   \to \\
\to \on{diag}_{\Ran_{\on{untl,aug}}}^!\circ (\pi\times \pi)^!\circ (\psi\times \psi)_!\circ (\xi\times \xi)^!(\CF\boxtimes \CG) \to  
\pi^!\circ \psi_!\circ \xi^! \circ \unn_! (\CF\boxtimes \CG);
\end{multline*}


\item(d) A homotopy between the resulting two maps
$$\on{diag}_{\Ran_{\on{untl,aug}}}^!\circ (\pi\times \pi)^!\circ (\psi_{\on{aug}}\times \psi_{\on{aug}})_!
\circ (\xi_{\on{aug}}\times \xi_{\on{aug}})^!(\CF\boxtimes \CG) \rightrightarrows 
\pi^!\circ \psi_!\circ \xi^! \circ \unn_! (\CF\boxtimes \CG).$$

\end{itemize}

\sssec{Proof of \thmref{t:units and tensor}, Step 2}

The map in (a) is obtained as pullback by means of $\pi$ from a map
\begin{equation} \label{e:compare on unital}
\on{diag}_{\Ran_{\on{untl}}}^!\circ (\psi\times \psi)_!\circ (\xi\times \xi)^!(\CF\boxtimes \CG) \to
\psi_!\circ \xi^! \circ \unn_! (\CF\boxtimes \CG).
\end{equation} 

The latter map is obtained by base change from the isomorphism of lax prestacks 
$$\Ran_{\on{untl}}\underset{\Ran_{\on{untl}}\times \Ran_{\on{untl}}}\times (\Ran^\to\times \Ran^\to)\simeq 
\Ran^\to \underset{\Ran}\times (\Ran\times \Ran).$$

The map in (b) is obtained from a map of lax prestacks
$$\Ran_{\on{untl,aug}}\underset{\Ran_{\on{untl}}\times \Ran_{\on{untl,aug}}}\times 
(\Ran^\to \times \Ran^\to_{\on{aug}})\to 
\Ran^\to_{\on{aug}} \underset{\Ran}\times (\Ran\times \Ran),$$
and the map in (c) is obtained from a map of lax prestacks
$$\Ran_{\on{untl,aug}}\underset{\Ran_{\on{untl,aug}}\times \Ran_{\on{untl}}}\times 
(\Ran^\to_{\on{aug}}\times \Ran^\to)\to 
\Ran^\to_{\on{aug}} \underset{\Ran}\times (\Ran\times \Ran).$$

The data of homotopies in (b'), (c') and (d') follows from the construction.

\sssec{Proof of \thmref{t:units and tensor}, Step 3}

We will now show that the map \eqref{e:comp tensor} is an isomorphism. First, it is easy to see
that the left-hand side satisfies the conditions $(*)$ and $(**)$ in \thmref{t:main}. 

\medskip

Note now that the functor $\iota^!:\Shv^!(\Ran_{\on{untl,aug}})\to \Shv^!(\Ran_{\on{untl}})$ is conservative
on the full subcategory consisting of objects satisfying conditions $(*)$ and $(**)$ in \thmref{t:main}. Hence, it 
remains to show that the induced map
$$\left(\iota^!\circ \on{AddUnit}_{\on{aug}}(\CF)\right) \overset{!}\otimes  
\left(\iota^!\circ \on{AddUnit}_{\on{aug}}(\CG)\right)\to \iota^!\circ \on{AddUnit}_{\on{aug}}(\unn_!(\CF\boxtimes \CG))$$
is an isomorphism.

\medskip

However, by \lemref{l:inserting the aug}, the latter map identifies with the map \eqref{e:compare on unital},
and thus is an isomorphism. 

\ssec{Relation to factorization}

In the previous subsection we proved that for $\wt\CF:=\on{AddUnit}_{\on{aug}}(\CF)$, a structure or 
(co)commutative (co)algebra on it with respect to the pointwise tensor product is equivalent to the structure
on $\CF$ of (co)commutative (co)algebra with respect to the convolution product on $\Shv^!(\Ran)$.

\medskip

In this subsection we will address the following question: what property of the (co)commutative (co)algebra $\wt\CF$
guarantees that that $\CF$ is a (co)commutative factorization algebra. 

\sssec{}

Let $(\Ran_{\on{untl,aug}}\times \Ran_{\on{untl,aug}})_{\on{compl,disj}}$ be the following lax prestack. 

\medskip

For $S\in \Sch$,
the category $(\Ran_{\on{untl,aug}}\times \Ran_{\on{untl,aug}})_{\on{compl,disj}}(S)$ is the full subcategory of
$(\Ran_{\on{untl,aug}}\times \Ran_{\on{untl,aug}})(S)$, consisting of quadruples $(K_1\subseteq I_1),(K_2\subseteq I_2)$,
satisfying the following condition: for any $i_1\in I_1- K_1$ and $i_2\in I_2$ the corresponding two maps $S\rightrightarrows X$
have non-intersecting images, and any $i_2\in I_2- K_2$ and $i_1\in I_1$ the corresponding two maps $S\rightrightarrows X$
have non-intersecting images.

\medskip

Denote also
$$(\Ran_{\on{untl}}\times \Ran_{\on{untl}})_{\on{disj}}:=
(\Ran_{\on{untl}}\times \Ran_{\on{untl}})\underset{\Ran_{\on{untl,aug}}\times \Ran_{\on{untl,aug}}}\times
(\Ran_{\on{untl,aug}}\times \Ran_{\on{untl,aug}})_{\on{compl,disj}}.$$

Note that we have a Cartesian diagram
$$
\CD
(\Ran\times \Ran)_{\on{disj}}   @>>>   (\Ran_{\on{untl}}\times \Ran_{\on{untl}})_{\on{disj}}   \\
@VVV    @VVV    \\
\Ran\times \Ran  @>{\iota\times \iota}>>   \Ran_{\on{untl}}\times \Ran_{\on{untl}}.
\endCD
$$

\sssec{}  \label{sss:factor algebras aug}

Let $\wt\CF$ be a commutative algebra in $\Shv^!(\Ran_{\on{untl,aug}})$ (with respect to the pointwise 
tensor product, i.e., the operation $\overset{!}\otimes$). We shall say that $\wt\CF$
is a \emph{commutative factorization algebra}, if the following condition holds: 

\medskip

For any $S\in \Sch$ and an object $(I_1,I_2)\in (\Ran_{\on{untl}}\times \Ran_{\on{untl}})_{\on{disj}}(S)$, the composed map
\begin{equation}  \label{e:factor alg aug middle}
\wt\CF_{\emptyset\subset I_1}\overset{!}\otimes \wt\CF_{\emptyset\subset I_2}\to 
\wt\CF_{\emptyset\subset I_1\cup I_2}\overset{!}\otimes \wt\CF_{\emptyset\subset I_1\cup I_2}
=(\wt\CF\overset{!}\otimes \wt\CF)_{\emptyset\subset I_1\cup I_2}\to \wt\CF_{\emptyset\subset I_1\cup I_2}
\end{equation} 
and the maps
$$\wt\CF_{\emptyset\subset I_1}\to \wt\CF_{\emptyset\subset I_1\cup I_2} \text{ and }
\wt\CF_{\emptyset\subset I_2}\to \wt\CF_{\emptyset\subset I_1\cup I_2}$$
induce an isomorphism
\begin{equation} \label{e:factor alg aug}
\wt\CF_{\emptyset\subset I_1} \oplus  (\wt\CF_{\emptyset\subset I_1}\overset{!}\otimes \wt\CF_{\emptyset\subset I_2}) \oplus 
\wt\CF_{\emptyset\subset I_2}\to \wt\CF_{\emptyset\subset I_1\cup I_2}.
\end{equation} 

Note that the condition of being a \emph{commutative factorization algebra} only depends on the restriction of $\wt\CF$
to $\Ran_{\on{untl}}$ under the map $\iota:\Ran_{\on{untl}}\to \Ran_{\on{untl,aug}}$.

\begin{rem} \label{r:more factor}
Suppose in addition that $\wt\CF$ satisfies conditions $(*)$ and $(**)$ in \thmref{t:main}. 

\medskip

In this case, it is easy to see that 
for any $S$ and an $S$-point $(K_1\subseteq I_1),(K_2\subseteq I_2)$ of 
$(\Ran_{\on{untl,aug}}\times \Ran_{\on{untl,aug}})_{\on{compl,disj}}$, the composed map
$$\wt\CF_{K_1\subseteq I_1}\overset{!}\otimes \wt\CF_{K_2\subseteq I_2}\to 
\wt\CF_{K_1\cup K_2\subseteq I_1\cup I_2}\overset{!}\otimes \wt\CF_{K_1\cup K_2\subseteq I_1\cup I_2}
=(\wt\CF\overset{!}\otimes \wt\CF)_{K_1\cup K_2\subseteq I_1\cup I_2}\to \wt\CF_{K_1\cup K_2\subseteq I_1\cup I_2}$$
and the maps 
$$\wt\CF_{K_1\subseteq I_1}\to \wt\CF_{K_1\cup K_2\subseteq I_1\cup I_2} \text{ and }
\wt\CF_{K_2\subseteq I_2}\to \wt\CF_{K_1\cup K_2\subseteq I_1\cup I_2}$$
induce an isomorphism
$$\wt\CF_{K_1\subseteq I_1} \oplus  (\wt\CF_{K_1\subseteq I_1}\overset{!}\otimes \wt\CF_{K_2\subseteq I_2}) 
 \oplus \wt\CF_{K_2\subseteq I_2}\to  \wt\CF_{K_1\cup K_2\subseteq I_1\cup I_2}.$$

\end{rem}

\sssec{}  \label{sss:factor coalgebras aug}

Let $\wt\CG$ be a cocommutative coalgebra in $\Shv^!(\Ran_{\on{untl,aug}})$ (with respect to the \emph{pointwise} tensor product). 
We shall say that $\wt\CG$ is a \emph{cocommutative factorization coalgebra}, if the following conditions hold:

\begin{itemize}

\item(1) $\wt\CG$ satisfies conditions $(*)$ and $(**)$ in \thmref{t:main}; 

\item(2) For any $S\in \Sch$ and an $S$-point $(K_1\subseteq I_1),(K_2\subseteq I_2)$ of 
$$(\Ran_{\on{untl,aug}}\times \Ran_{\on{untl,aug}})_{\on{compl,disj}},$$ the composed map
\begin{multline*} 
\wt\CG_{K_1\cup K_2\subseteq I_1\cup I_2}\to (\wt\CG\overset{!}\otimes \wt\CG)_{K_1\cup K_2\subseteq I_1\cup I_2}=
\wt\CG_{K_1\cup K_2\subseteq I_1\cup I_2}\overset{!}\otimes \wt\CG_{K_1\cup K_2\subseteq I_1\cup I_2}\to \\
\to \wt\CG_{K_1\cup I_2\subseteq I_1\cup I_2}\overset{!}\otimes \wt\CG_{I_1\cup K_2\subseteq I_1\cup I_2},
\end{multline*} 
together with the maps
$$\wt\CG_{K_1\cup K_2\subseteq I_1\cup I_2}\to \wt\CG_{K_1\cup I_2\subseteq I_1\cup I_2} \text{ and }
\wt\CG_{K_1\cup K_2\subseteq I_1\cup I_2}\to \wt\CG_{I_1\cup K_2\subseteq I_1\cup I_2}$$
induce an isomorphism
\begin{equation} \label{e:aug factor cond co}
\wt\CG_{K_1\cup K_2\subseteq I_1\cup I_2} \to 
\wt\CG_{K_1\cup I_2\subseteq I_1\cup I_2} \oplus 
(\wt\CG_{K_1\cup I_2\subseteq I_1\cup I_2}\overset{!}\otimes \wt\CG_{I_1\cup K_2\subseteq I_1\cup I_2})\oplus
\wt\CG_{I_1\cup K_2\subseteq I_1\cup I_2}.
\end{equation}

\end{itemize}

\sssec{}

We are going prove the following result:

\begin{prop} \label{p:aug and factor} \hfill

\smallskip

\noindent{\em(a)} Let $\CF\in \Shv^!(\Ran)$ be a commutative algebra, and let
$\wt\CF:=\on{AddUnit}_{\on{aug}}(\CF)$ be the corresponding commutative algebra
in $\Shv^!(\Ran_{\on{untl,aug}})$. Then $\CF$ is a commutative factorization algebra
(in the sense of \secref{sss:factor algebras}) if and only if $\wt\CF$ is 
a commutative factorization algebra (in the sense of \secref{sss:factor algebras aug}). 

\smallskip 

\noindent{\em(b)} Let $\CG\in \Shv^!(\Ran)$ be a cocommutative coalgebra, and let
$\wt\CG:=\on{AddUnit}_{\on{aug}}(\CG)$ be the corresponding cocommutative coalgebra
in $\Shv^!(\Ran_{\on{untl,aug}})$. Then $\CG$ is a cocommutative factorization coalgebra
(in the sense of \secref{sss:factor coalgebras}) if and only if $\wt\CG$ is 
a cocommutative factorization coalgebra (in the sense of \secref{sss:factor coalgebras aug}). 

\end{prop}

The rest of this subsection is devoted to the proof of \propref{p:aug and factor}. 

\sssec{Proof of \propref{p:aug and factor}(a)}

Let $\CF$ be a commutative algebra object in $\Shv^!(\CF)$.

\medskip

For $S\in \Sch$ and an object $(I_1,I_2)\in (\Ran_{\on{untl}}\times \Ran_{\on{untl}})_{\on{disj}}(S)$, 
the right-hand side in \eqref{e:factor alg aug} is obtained as the value on $S$ and $(I_1,I_2)$ of 
the object of $\Shv^!((\Ran_{\on{untl}}\times \Ran_{\on{untl}})_{\on{disj}})$,
given by pull-push along the Cartesian diagram
$$
\CD
(\Ran_{\on{untl}}\times \Ran_{\on{untl}})_{\on{disj}} \underset{\Ran_{\on{untl}}}\times \Ran^\to @>>>  \Ran^\to @>{\xi}>> \Ran \\
@VVV    @VV{\psi}V  \\
(\Ran_{\on{untl}}\times \Ran_{\on{untl}})_{\on{disj}}  @>{\unn_{\on{untl}}}>>   \Ran_{\on{untl}},
\endCD
$$
starting from $\CF\in \Shv^!(\Ran)$. In the above diagram, $\unn_{\on{untl}}$ denotes the naturally defined morphism
$$\Ran_{\on{untl}}\times \Ran_{\on{untl}}\to \Ran_{\on{untl}}, \quad I_1,I_2\mapsto I_1\cup I_2.$$

\medskip 

The three direct summands in the left-hand side of  \eqref{e:factor alg aug} are given, respectively, by evaluating at the same
$S$ and $(I_1,I_2)$ of $(\Ran_{\on{untl}}\times \Ran_{\on{untl}})_{\on{disj}}$ the following objects 
of $(\Ran_{\on{untl}}\times \Ran_{\on{untl}})_{\on{disj}}$:

\medskip

\noindent(1) Pull-push along the Caretsian diagram
$$
\CD
(\Ran_{\on{untl}}\times \Ran_{\on{untl}})_{\on{disj}} \underset{ \Ran_{\on{untl}}}\times \Ran^\to   @>>>  \Ran^\to @>{\xi}>> \Ran \\
@VVV    @VV{\psi}V   \\
(\Ran_{\on{untl}}\times \Ran_{\on{untl}})_{\on{disj}}  @>{\on{pr}_1}>> \Ran_{\on{untl}},
\endCD
$$
starting from $\CF\in \Shv^!(\Ran)$.

\medskip

\noindent(2) Pull-push along the diagram
\begin{equation} \label{e:middle diag}
\CD
(\Ran^\to \times \Ran^\to)_{\on{disj}}  @>{(\xi\times \xi)_{\on{disj}}}>>  (\Ran \times \Ran)_{\on{disj}}   \\
@V{(\psi\times \psi)_{\on{disj}}}VV   \\
(\Ran_{\on{untl}}\times \Ran_{\on{untl}})_{\on{disj}},
\endCD
\end{equation} 
starting from $(\CF\boxtimes \CF)|_{ (\Ran \times \Ran)_{\on{disj}}}$, 
where
$$(\Ran^\to \times \Ran^\to)_{\on{disj}} := (\Ran^\to \times \Ran^\to) \underset{(\Ran_{\on{untl}}\times \Ran_{\on{untl}})}\times 
(\Ran_{\on{untl}}\times \Ran_{\on{untl}})_{\on{disj}}.$$

\medskip

\noindent(3) Pull-push along the Caretsian diagram
$$
\CD
(\Ran_{\on{untl}}\times \Ran_{\on{untl}})_{\on{disj}} \underset{ \Ran_{\on{untl}}}\times \Ran^\to   @>>>  \Ran^\to @>{\xi}>> \Ran \\
@VVV    @VV{\psi}V   \\
(\Ran_{\on{untl}}\times \Ran_{\on{untl}})_{\on{disj}}  @>{\on{pr}_2}>> \Ran_{\on{untl}},
\endCD
$$
starting from $\CF\in \Shv^!(\Ran)$.

\medskip

Now, we note that $(\Ran_{\on{untl}}\times \Ran_{\on{untl}})_{\on{disj}} \underset{\Ran_{\on{untl}}}\times \Ran^\to$ is canonically
isomorphic to the disjoint union of
$$(\Ran_{\on{untl}}\times \Ran_{\on{untl}})_{\on{disj}} \underset{\on{pr}_1,\Ran_{\on{untl}}}\times \Ran^\to,$$
$$(\Ran^\to \times \Ran^\to)_{\on{disj}}, \text{ and }$$
$$(\Ran_{\on{untl}}\times \Ran_{\on{untl}})_{\on{disj}} \underset{\on{pr}_2,\Ran_{\on{untl}}}\times \Ran^\to.$$

\medskip

Now, the map in \eqref{e:factor alg aug} is obtained from this identification, where for the middle direct summand
we note that the composition
$$(\Ran^\to \times \Ran^\to)_{\on{disj}}\hookrightarrow 
(\Ran_{\on{untl}}\times \Ran_{\on{untl}})_{\on{disj}} \underset{\Ran_{\on{untl}}}\times \Ran^\to
\to \Ran^\to \overset{\xi}\longrightarrow \Ran$$
identifies with 
$$(\Ran^\to \times \Ran^\to)_{\on{disj}}  \to   (\Ran \times \Ran)_{\on{disj}} \overset{\unn}\longrightarrow \Ran$$
and the map \eqref{e:factor alg aug middle} is induced by the map 
$$(\CF\boxtimes \CF)|_{(\Ran \times \Ran)_{\on{disj}}}\to 
\unn^!(\CF)|_{(\Ran \times \Ran)_{\on{disj}}},$$
given by the (commutative) algebra structure. 

\medskip

Hence, if $\CF$ is a commutative factorization algebra, then the map \eqref{e:factor alg aug} is an isomorphism. 

\medskip

The converse implication follows from the fact that the functor
$$((\psi\times \psi)_{\on{disj}})_!\circ 
((\xi\times \xi)_{\on{disj}})^!:\Shv^!((\Ran\times \Ran)_{\on{disj}})\to \Shv^!((\Ran_{\on{untl}}\times \Ran_{\on{untl}})_{\on{disj}})$$
is conservative, which in turn follows from \propref{p:inserting the unit} and \lemref{l:testing}. 

\sssec{Proof of \propref{p:aug and factor}(b)}

It is easy to see that condition (1) in \secref{sss:factor coalgebras aug}
implies that condition (2) holds if and only if it holds for $K_1=K_2=\emptyset$. In this case,
\eqref{e:aug factor cond co} comes from a map of objects in $\Shv^!((\Ran_{\on{untl}}\times \Ran_{\on{untl}})_{\on{disj}})$.

\medskip

As in the proof of \propref{p:aug and factor}(a), the left-hand side in \eqref{e:aug factor cond co} is the direct sum
\begin{multline} \label{e:factor co left}
\left(\on{pr}_1^!\circ \on{AddUnit}(\CG)\right)|_{(\Ran_{\on{untl}}\times \Ran_{\on{untl}})_{\on{disj}}}\oplus \\
\oplus \left((\on{AddUnit}\boxtimes \on{AddUnit})^!\circ \unn^!(\CG)\right)|_{(\Ran_{\on{untl}}\times \Ran_{\on{untl}})_{\on{disj}}} \oplus\\
\left(\oplus \on{pr}_2^!\circ \on{AddUnit}(\CG)\right)|_{(\Ran_{\on{untl}}\times \Ran_{\on{untl}})_{\on{disj}}}.
\end{multline}

The right-hand side in \eqref{e:aug factor cond co} is the direct sum
\begin{multline}  \label{e:factor co right}
\left(\on{shift}_1^!\circ \on{AddUnit}_{\on{aug}}(\CG)\right)|_{(\Ran_{\on{untl}}\times \Ran_{\on{untl}})_{\on{disj}}} \oplus \\
\oplus 
\left(\unn_{\on{aug}}^!\circ (\on{AddUnit}_{\on{aug}}\boxtimes \on{AddUnit}_{\on{aug}})^!(\CG\boxtimes \CG)\right)
|_{(\Ran_{\on{untl}}\times \Ran_{\on{untl}})_{\on{disj}}} \oplus \\
\oplus \left(\on{shift}_2^!\circ \on{AddUnit}_{\on{aug}}(\CG)\right)|_{(\Ran_{\on{untl}}\times \Ran_{\on{untl}})_{\on{disj}}},
\end{multline}
where
$$\unn_{\on{aug}}:\Ran_{\on{untl}}\times \Ran_{\on{untl}}\to \Ran_{\on{untl,aug}}\times \Ran_{\on{untl,aug}}$$
is the map $(I_1,I_2)\mapsto (I_2\subseteq I_1\sqcup I_2,I_1\subseteq I_1\sqcup I_2)$,
and $\on{shift}_1$ and $\on{shift}_2$ are the maps
$$\Ran_{\on{untl}}\times \Ran_{\on{untl}}\to \Ran_{\on{untl,aug}},$$
given by 
$$(I_1,I_2)\mapsto (I_2\subseteq I_1\cup I_2) \text{ and } (I_1,I_2)\mapsto (I_1\subseteq I_1\cup I_2),$$
respectively. 

\medskip

Note, however, that condition (1) in \secref{sss:factor coalgebras aug}
on $\wt\CG=\on{AddUnit}_{\on{aug}}(\CG)$ implies that the map \eqref{e:aug factor cond co}
maps the first (resp., third) 
direct summand in \eqref{e:factor co left} isomorphically onto the first (reps., third) direct summand in \eqref{e:factor co right}.

\medskip

In addition, condition (1) on $\wt\CG$ implies that the middle direct summand
in \eqref{e:factor co right} receives an isomorphism from
$$\left((\on{AddUnit}\boxtimes \on{AddUnit})^!\circ (\CG\boxtimes \CG)\right)
|_{(\Ran_{\on{untl}}\times \Ran_{\on{untl}})_{\on{disj}}},$$
and in terms of this identification, the map in \eqref{e:aug factor cond co} comes from a map  
\begin{multline}  \label{e:factor co right middle}
\left((\on{AddUnit}\boxtimes \on{AddUnit})^!\circ \unn^!(\CG)\right)|_{(\Ran_{\on{untl}}\times \Ran_{\on{untl}})_{\on{disj}}}   \to \\
\to 
\left((\on{AddUnit}\boxtimes \on{AddUnit})^!\circ (\CG\boxtimes \CG)\right)|_{(\Ran_{\on{untl}}\times \Ran_{\on{untl}})_{\on{disj}}},
\end{multline}
defined as follows:

\medskip 

Consider the diagram \eqref{e:middle diag} 
and note that
\begin{multline*}
\left((\on{AddUnit}\boxtimes \on{AddUnit})^!\circ \unn^!(\CG)\right)|_{(\Ran_{\on{untl}}\times \Ran_{\on{untl}})_{\on{disj}}}  \simeq \\
\simeq  ((\psi\times \psi)_{\on{disj}})_!\circ ((\xi\times \xi)_{\on{disj}})^!\left(\unn^!(\CG)|_{(\Ran\times \Ran)_{\on{disj}}} \right) 
\end{multline*}
and
\begin{multline*}
\left((\on{AddUnit}\boxtimes \on{AddUnit})^!\circ (\CG\boxtimes \CG)\right)
|_{(\Ran_{\on{untl}}\times \Ran_{\on{untl}})_{\on{disj}}}   \simeq \\
\simeq ((\psi\times \psi)_{\on{disj}})_!\circ ((\xi\times \xi)_{\on{disj}})^!\left(\CG\boxtimes \CG|_{(\Ran\times \Ran)_{\on{disj}}} \right).
\end{multline*} 

Now, the map in \eqref{e:factor co right middle} identifies with the map
$$((\psi\times \psi)_{\on{disj}})_!\circ ((\xi\times \xi)_{\on{disj}})^!\left(\unn^!(\CG)|_{(\Ran\times \Ran)_{\on{disj}}} \right) \to
((\psi\times \psi)_{\on{disj}})_!\circ ((\xi\times \xi)_{\on{disj}})^!\left(\CG\boxtimes \CG|_{(\Ran\times \Ran)_{\on{disj}}} \right),$$
induced by \eqref{e:factor for co}. 

\medskip

Thus, if \eqref{e:factor for co} is an isomorphism (i.e., if $\CG$ is a cocommutative factorization coalgebra),
the map \eqref{e:factor co right middle} is an isomorphism and hence the map  \eqref{e:aug factor cond co} 
is an isomorphism.

\medskip

The converse implication follows as in point (a). 

\ssec{Pairings and augmentation}

Finally, in this subsection we will address the following issue. Let $\CA$ and $\CB$ be a commutative algebra and a
cocommutative coalgebra in $\Shv^!(\Ran)$, respectively. Let us be given a pairing between $\wt\CA$ and $\wt\CB$
as mere objects of $\Shv^!(\Ran_{\on{untl,aug}})$. We shall explain what structure on this pairing guarantees that the induced map
$\CB\to \BD_{\Ran}\CA$ is a homomorphism of commutative algebras.

\sssec{}  \label{sss:compat pairing}

Let $\wt\CF$ and $\wt\CG$ be two objects in $\Shv^!(\Ran_{\on{untl,aug}})$, and let us be given
a pairing 
\begin{equation} \label{e:aug pairing again again}
\wt\CF\boxtimes \wt\CG|_{(\Ran_{\on{untl,aug}}\times \Ran_{\on{untl,aug}})_{\on{compl,disj}}}\to 
\omega_{(\Ran_{\on{untl,aug}}\times \Ran_{\on{untl,aug}})_{\on{compl,disj}}}.
\end{equation} 

\medskip

Suppose now that $\wt\CF$ is endowed with a structure of commutative algebra and $\wt\CG$ is
endowed with a structure of cocommutative coalgebra (with respect to the pointwise symmetric monoidal structure). 
In this case, there is a naturally defined
notion of \emph{structure of compatibility} on \eqref{e:aug pairing again again} with the above pieces
of structure on $\wt\CF$ and $\wt\CG$.

\medskip

The initial data in a structure of compatibility is that of a homotopy between the maps
\begin{multline*} 
(\wt\CF\overset{!}\otimes \wt\CF)\boxtimes \wt\CG|_{(\Ran_{\on{untl,aug}}\times \Ran_{\on{untl,aug}})_{\on{compl,disj}}} \to
\wt\CF\boxtimes \wt\CG|_{(\Ran_{\on{untl,aug}}\times \Ran_{\on{untl,aug}})_{\on{compl,disj}}}  \to  \\
\to \omega_{(\Ran_{\on{untl,aug}}\times \Ran_{\on{untl,aug}})_{\on{compl,disj}}} 
\end{multline*} 
and
\begin{multline*} 
(\wt\CF\overset{!}\otimes \wt\CF)\boxtimes \wt\CG|_{(\Ran_{\on{untl,aug}}\times \Ran_{\on{untl,aug}})_{\on{compl,disj}}} \to 
(\wt\CF\overset{!}\otimes \wt\CF)\boxtimes 
(\wt\CG\overset{!}\otimes \wt\CG) |_{(\Ran_{\on{untl,aug}}\times \Ran_{\on{untl,aug}})_{\on{compl,disj}}} \simeq \\
\simeq \left(\wt\CF\boxtimes \wt\CG|_{(\Ran_{\on{untl,aug}}\times \Ran_{\on{untl,aug}})_{\on{compl,disj}}}\right)
\overset{!}\otimes 
\left(\wt\CF\boxtimes \wt\CG|_{(\Ran_{\on{untl,aug}}\times \Ran_{\on{untl,aug}})_{\on{compl,disj}}}\right)  \to \\
\to \omega_{(\Ran_{\on{untl,aug}}\times \Ran_{\on{untl,aug}})_{\on{compl,disj}}}  \overset{!}\otimes 
\omega_{(\Ran_{\on{untl,aug}}\times \Ran_{\on{untl,aug}})_{\on{compl,disj}}} \simeq \\
\simeq \omega_{(\Ran_{\on{untl,aug}}\times \Ran_{\on{untl,aug}})_{\on{compl,disj}}}.
\end{multline*} 

The higher compatibilities amount to a compatible family of homotopies for each surjection of finite sets $\CI_1\twoheadrightarrow \CI_2$
$$\wt\CF^{\overset{!}\otimes \CI_1}\boxtimes \wt\CG^{\overset{!}\otimes \CI_2}\rightrightarrows 
\omega_{(\Ran_{\on{untl,aug}}\times \Ran_{\on{untl,aug}})_{\on{compl,disj}}}.$$

\sssec{}   \label{sss:compat pairing red}

Let now $\CF$ and $\CG$ be two objects of $\Shv^!(\Ran)$, and let
\begin{equation} \label{e:usual pairing again again}
\CF\boxtimes \CG \to (\on{diag}_{\Ran})_!(\omega_{\Ran})
\end{equation} 
be a datum of pairing. 

\medskip

Let now $\CF$ be endowed with a structure of commutative algebra, and let $\CG$ be endowed with a structure of cocommutative coalgebra
(both with respect to the convolution symmetric monoidal structure).
In this case, there is a naturally defined notion of \emph{structure of compatibility} on \eqref{e:usual pairing again again} 
with the above pieces of structure on $\CF$ and $\CG$.

\medskip

The initial data in a structure of compatibility is that of a homotopy between the maps
$$\unn_!(\CF\boxtimes \CF)\boxtimes \CG \to \CF\boxtimes \CG \to (\on{diag}_{\Ran})_!(\omega_{\Ran})$$
and
\begin{multline*} 
\unn_!(\CF\boxtimes \CF)\boxtimes \CG \to \unn_!(\CF\boxtimes \CF)\boxtimes \unn_!(\CG\boxtimes \CG) \to \\
\to (\on{diag}_{\Ran})_!\circ \unn_!(\omega_{\Ran}\boxtimes \omega_{\Ran})\to 
  (\on{diag}_{\Ran})_!(\omega_{\Ran}).
\end{multline*} 

The higher compatibilities amount to a compatible family of homotopies for each surjection of finite sets $\CI_1\twoheadrightarrow \CI_2$
as in \secref{sss:compat pairing}.

\medskip

The following assertion results from the definition: 

\begin{lem} \label{l:pairings and alg}
For $\CF,\CG\in \Shv^!(\Ran)$ and a pairing \eqref{e:usual pairing again again}, a \emph{structure of compatibility} with a given
commutative algebra structure on $\CF$ and a given cocommutative coalgebra structure on $\CG$
is equivalent to a structure on the resulting map $\CF\to \BD_{\Ran}(\CG)$ of homomorphism
of commutative algebras in $\Shv^!(\Ran)$.
\end{lem} 

\sssec{}

Finally, we have the following assertion, which results by unwinding the constructions: 

\begin{lem}  \label{l:compat pairing aug}
Let $\CF$ and $\CG$ be a commutative algebra and a cocommutative coalgebra in 
$\Shv^!(\Ran)$, respectively. Let us be given a pairing 
\begin{equation} \label{e:usual pairing again again again}
\CF\boxtimes \CG \to (\on{diag}_{\Ran})_!(\omega_{\Ran})
\end{equation}
and consider the corresponding pairing
\begin{equation} \label{e:aug pairing again again again}
\wt\CF\boxtimes \wt\CG|_{(\Ran_{\on{untl,aug}}\times \Ran_{\on{untl,aug}})_{\on{compl,disj}}}\to 
\omega_{(\Ran_{\on{untl,aug}}\times \Ran_{\on{untl,aug}})_{\on{compl,disj}}},
\end{equation} 
where $\wt\CF:=\on{AddUnit}_{\on{aug}}(\CF)$, $\wt\CG:=\on{AddUnit}_{\on{aug}}(\CG)$. 

\medskip

\noindent Then a structure
of compatibility on \eqref{e:usual pairing again again again} with the 
commutative algebra structure on $\CF$ and 
and the cocommutative coalgebra structure on $\CG$ is equivalent to a structure of
compatibility on \eqref{e:aug pairing again again again} with the commutative algebra structure on $\wt\CF$ and 
and the cocommutative coalgebra structure on $\wt\CG$.
\end{lem} 

\newpage 

\centerline{\bf Part V: The cohomological product formula}

\bigskip

In this part we will be working in the context of constructible sheaves, and we will assume that the ring
of coefficients $\Lambda$ has a finite cohomological dimension. 

\medskip

After the preparations in Parts 0-IV, this part we will constitute the core of the paper--the derivation of
the cohomological product formula.  

\section{Reduction to a global duality statement}     \label{s:global duality}

In order to understand the contents of this section one only needs to know what $\Ran$ is and how to makes
sense of sheaves on prestacks. In other words, the prerequisites for the present section are Sects. \ref{s:prestacks} and \ref{ss:Ran}. 

\medskip

In this section we take $X$ to be a smooth, connected and complete curve, and let $G$ be a smooth 
fiber-wise connected group-scheme over $X$, which is simply connected at the generic point of $X$. 

\medskip

We are going to state the \emph{cohomological product formula} for the
cohomology of $\Bun_G$. It says that a naturally defined map
$$\on{C}^*_c(\Ran,\CB)\to \on{C}^*_{\on{red}}(\Bun_G)$$
is an isomorphism, where $\CB\in \Shv^!(\Ran)$ is obtained by considering the cohomology
of the classifying space of $G$.

\medskip

We will reduce the proof of the cohomological product formula to a combination of two statements.  The first statement will be the
\emph{non-abelian Poincar\'e duality} that says that a naturally defined map
$$\on{C}^*_c(\Ran,\CA)\to \on{C}^*_{\on{red}}(\Bun_G)$$
is an isomorphism, where $\CA\in \Shv^!(\Ran)$ is obtained by considering the homology of
the affine Grassmannian of $G$.  The second statement will be the \emph{global duality} statement, which says that 
the (compactly supported) cohomologies of $\CA$ and $\CB$ are related by duality.

\ssec{The cohomological product formula}   \label{ss:product formula}

Let $\Bun_G$ denote the moduli stack of $G$-bundles on $X$. We are interested in the (reduced) cohomology of $\Bun_G$
$$\on{C}^*_{\on{red}}(\Bun_G):=\left(\on{C}^{\on{red}}_*(\Bun_G)\right)^\vee.$$

\medskip

In this subsection we will state the cohomological product formula for $\on{C}^*_{\on{red}}(\Bun_G)$.

\sssec{}

First, we are going to define an object, denoted, $\CB\in \Shv^!(\Ran)$ that encodes the reduced cohomology of
$BG$. 

\medskip

For $S\in \Sch$ and an $S$-point $I$ of $\Ran$, let $D_I\subset S\times X$ be the 
corresponding Cartier divisor. We let $BG_I$ denote the Artin stack over $S$ that classifies $G$-bundles
over $D_I$.  Let $f_I:BG_I\to S$ denote the resulting forgetful map. 

\medskip

Define
$$\CB_{S,I}:=\BD_S\left(\on{Fib}
\left((f_{I})_!\circ (f_{I})^!(\Lambda_S)\to \Lambda_S\right)\right).$$

In the above formula $\Lambda_S$ is the \emph{constant sheaf} on $S$, i.e., $\Lambda_S=\BD_{S}(\omega_S)$,
where $\BD_S$ denotes Verdier duality on $S$. 

\medskip

For example, for a $k$-point $\{x_1,...,x_n\}$ of $\Ran$, the
!-fiber of $\CB$ at this point equals
$$\on{Fib}\left(\underset{i=1,...,n}\bigotimes\, \on{C}^*(BG_{x_i})\to \Lambda\right).$$

\medskip

The assignment $(S,I)\mapsto \CB_{S,I}$ forms an object of $\Shv^!(\Ran)$, which is the sought-for $\CB$.

\begin{rem}
Note that there is a minor subtlety involved in the assertion that the assignment $(S,I)\mapsto \CB_{S,I}$
is compatible under pullbacks $S'\to S$:

\medskip

For a given point $I\in \Maps(S,\Ran)$ and the resulting point $I'\in \Maps(S',\Ran)$,
we have a \emph{map}
$$S'\underset{S}\times BG_I\to BG_{I'},$$
but this map is \emph{not necessarily} an isomorphism. However, this map has contractible fibers, and hence induces
an isomorphism $\CB_{S',I'}\to \CB_{S,I}|_{S'}$. See \cite[Proposition 5.4.3]{Main Text} for more details. 

\end{rem}

\sssec{}

For every $S\in \Sch$ and $I\in \Ran(S)$ we have a map of stacks over $S$
$$\on{ev}_{S,I}:S\times \Bun_G\to BG_{I},$$ 
given by restriction.  From here we obtain a map in $\Shv(S)$
$$\Lambda_S\otimes \on{C}_*(\Bun_G)\to (f_{I})_!\circ (f_{I})^!(\Lambda_S)$$
and hence 
$$\Lambda_S\otimes \on{C}^{\on{red}}_*(\Bun_G)\to \on{Fib}((f_{I})_!\circ (f_{I})^!(\Lambda_S)\to \Lambda_S).$$

Applying $\BD_S$ we obtain a map
\begin{equation} \label{e:map over S prel}
\CB_{S,I} \to \BD_S\left( \Lambda_S\otimes \on{C}^{\on{red}}_*(\Bun_G)\right).
\end{equation} 

\medskip

Now, for any $V\in \Lambda\mod$, the map
$$\omega_S \otimes V^\vee \to \BD_S(\Lambda_S\otimes V)$$ 
is an isomorphism.

\medskip

Hence, from \eqref{e:map over S prel} we obtain a map
\begin{equation} \label{e:map over S}
\on{ev}_{S,I}:\CB_{S,I} \to \omega_S\otimes \on{C}^*_{\on{red}}(\Bun_G).
\end{equation} 

\sssec{}

The maps \eqref{e:map over S} combine into a map
$$\CB\to \omega_{\Ran}\otimes \on{C}^*_{\on{red}}(\Bun_G).$$

Finally, applying the functor $\on{C}^*_c(\Ran,-)$, we obtain a map
\begin{equation} \label{e:product formula}
\on{C}^*_c(\Ran,\CB) \to  \on{C}^*_c\left(\Ran,\omega_{\Ran}\otimes \on{C}^*_{\on{red}}(\Bun_G)\right)\simeq
\on{C}^*_c(\Ran,\omega_{\Ran}) \otimes \on{C}^*_{\on{red}}(\Bun_G) \simeq \on{C}^*_{\on{red}}(\Bun_G),
\end{equation} 
where the last isomorphism is \thmref{t:Ran contr}.

\sssec{}

The cohomological product formula says:

\begin{thm} \label{t:product formula}
The map \eqref{e:product formula}
$$\on{C}^*_c(\Ran,\CB) \to \on{C}^*_{\on{red}}(\Bun_G)$$
is an isomorphism.
\end{thm}

Our goal in the rest of Part V is to prove \thmref{t:product formula}. We emphasize that the map \eqref{e:product formula} makes
sense for any $G$.  However, the assertion of \thmref{t:product formula} only holds if the generic fiber of $G$ is simply connected. 
\footnote{For example, if the generic fiber of $G$ is semi-simple but not simply connected, then $\Bun_G$ is disconnected, and the isomorphism in 
\eqref{e:product formula} cannot hold because the left-hand side is insensitive to the isogeny class of $G$.} 

\begin{rem}
As was noted in \secref{sss:Euler}, one can informally think of $\on{C}^*_c(\Ran,\CB)$ as an ``Euler product"
\begin{equation} \label{e:Euler again}
\underset{x}\bigotimes\, \on{C}^*(BG_x).
\end{equation}

This is why we call the assertion of \thmref{t:product formula} the ``cohomological product formula": it gives an expression for the cohomology
of $\Bun_G$ as the Euler product of the cohomologies of $BG_x$.

\end{rem}

\ssec{The key input: non-abelian Poincar\'e duality}

The assertion of \thmref{t:product formula} is a local-to-global result. We will deduce it from another 
local-to-global result of a \emph{dual nature}, namely, the non-abelian Poincar\'e duality.  The precise
meaning in which the two results are dual to each other will constitute the bulk of the proof
of \thmref{t:product formula}. 

\sssec{}   \label{sss:good G}

In \cite[Lemma 7.1.1 and Proposition A.3.11]{Main Text} it was shown that we can (and from now on will) assume that $G$ has the following properties: 

\begin{itemize}

\item $G$ is semi-simple and simply connected over a non-empty open subset $X'\subset X$; 

\item The fibers of $G$ over points in $X-X'$ are cohomologically contractible.\footnote{In fact, we can arrange that each of these 
fibers of $G$ is isomorphic to a vector group.} 

\end{itemize}

In other words, if $G_1$ and $G_2$ are two group-schemes over $X$ (both smooth with connected fibers) that are isomorphic 
at the generic point of $X$, then \thmref{t:product formula} holds for $G_1$ if and only if it does for $G_2$. And for any given $G_1$ 
we can find a $G_2$ satisfying the above two conditions. 

\sssec{}

Let $\Ran'\subset \Ran$ be the Ran space of $X'$. We will now introduce an object of 
$\Shv(\Ran')$, denoted $\CA'$. It will encode the 
reduced homology of the affine Grassmannians built out of $G|_{X'}$. 

\medskip

For $S\in \Sch$ and an $S$-point $I$ of $\Ran$, let $S\times X-\on{Graph}_I\subset S\times X$ be the open subset
equal to the complement of the union of the graphs of the maps $S\to X$, $i\in I$.

\medskip

Let $\Gr_{\Ran'}$ denote the following prestack over $\Ran'$: for $S\in \Sch$ an $S$-point of $\Gr_{\Ran'}$
is a datum of $I\subset \Maps(S,X')$, a $G$-bundle $\CP_G$ on $S\times X$, and a trivialization of $\CP_G|_{S\times X-\on{Graph}_I}$. 

\medskip

Let $g:\Gr_{\Ran'}\to \Ran'$ denote the natural projection. It is well-known that this map is pseudo-proper
(because $G$ was assumed reductive over $X'$). 

\medskip

Set 
$$\CA':=\on{Fib}\left(g_!(\omega_{\Gr_{\Ran'}})\to \omega_{\Ran'}\right)\in \Shv(\Ran').$$ 

For example, for a point $\{x_1,...,x_n\}$ of $\Ran'$, the !-fiber of $\CA'$ at this point is
$$\on{Fib}\left(\underset{i=1,...,n}\bigotimes \on{C}_*(\Gr_{x_i})\to\Lambda\right).$$

\medskip

Note that we have a tautological isomorphism
\begin{multline}  \label{e:homol Gr}
\on{C}^*_c(\Ran',\CA')\simeq \on{Fib}\left(\on{C}^*_c(\Gr_{\Ran'},\omega_{\Gr_{\Ran'}})\to \on{C}^*_c(\Ran',\omega_{\Ran'})\right)\simeq \\
\simeq \on{Fib}\left(\on{C}^*_c(\Gr_{\Ran'},\omega_{\Gr_{\Ran'}})\to \Lambda\right)=\on{C}^{\on{red}}_*(\Gr_{\Ran'}).
\end{multline}

\sssec{}

We have a canonically defined map
$$\Gr_{\Ran'}\to \Bun_G.$$

According to \cite[Theorem 3.2.13]{Main Text} and the assumption on $G$ in \secref{sss:good G}, we have:

\begin{thm} \label{t:non-ab Poinc}
The map $\Gr_{\Ran'}\to \Bun_G$ is universally homologically contractible, and in particular induces an isomorphism
$\on{C}^{\on{red}}_*(\Gr_{\Ran'})\to \on{C}^{\on{red}}_*(\Bun_G)$. 
\end{thm} 

Combining with \eqref{e:homol Gr}, from \thmref{t:non-ab Poinc} we obtain an isomorphism
isomorphism 
\begin{equation} \label{e:non-Ab Poinc}
\on{C}^*_c(\Ran',\CA')\simeq \on{C}^{\on{red}}_*(\Bun_G).
\end{equation}

\ssec{Reduction to a duality statement on chiral homology}

In this subsection we will explain how the assertion of \thmref{t:product formula} is related to that of \thmref{t:non-ab Poinc}.

\sssec{}

We claim that there is a canonically defined pairing 
\begin{equation} \label{e:prel pairing integral}
\on{C}^*_c(\Ran',\CA') \otimes \on{C}^*_c(\Ran,\CB) \to \Lambda.
\end{equation} 

\medskip

In fact, we claim that there is a canonically defined map in $\Shv^!(\Ran'\times \Ran)$
\begin{equation} \label{e:prel pairing}
\CA'\boxtimes \CB\to \omega_{\Ran'\times \Ran},
\end{equation}
from which the map \eqref{e:prel pairing integral} is obtained by composing with the trace map
$$\on{C}^*_c(\Ran'\times \Ran,\omega_{\Ran'\times \Ran})\to \Lambda.$$

\sssec{}  \label{sss:prel pairing}

The map \eqref{e:prel pairing} is constructed as follows. 

\medskip

For $S'\in \Sch$ and $I'\subset \Maps(S',\Ran')$ denote 
$\Gr_{I'}$ denote the prestack $S'\underset{\Ran'}\times \Gr_{\Ran'}$; let $g_{I'}$ denote the resulting
map $\Gr_{I'}\to S'$.

\medskip

Let $S\in \Sch$ and $I\subset \Maps(S,\Ran)$. We need to construct a map
$$\on{Fib}\left((g_{I'})_!\circ (g_{I'})^!(\omega_{S'})\to \omega_{S'}\right) \boxtimes
\BD_S\left(\on{Fib} \left((f_{I})_!\circ (f_{I})^!(\Lambda_S)\to \Lambda_S\right)\right)\to \omega_{S'\times S},$$
which is equivalent to constructing a map
$$\on{Fib}\left((g_{I'})_!\circ (g_{I'})^!(\omega_{S'})\to \omega_{S'}\right) \boxtimes \Lambda_S\to
\omega_{S'}\boxtimes \on{Fib} \left((f_{I})_!\circ (f_{I})^!(\Lambda_S)\to \Lambda_S\right),$$
or equivalently
\begin{equation} \label{e:local map prel}
(g_{I'}\times \on{id})_!\circ (g_{I'}\times \on{id})^!(\omega_{S'}\boxtimes \Lambda_S)\to
(\on{id}\times f_{I})_!\circ (\on{id}\times f_{I})^!(\omega_{S'}\boxtimes\Lambda_S).
\end{equation} 

\medskip

Note now that we have a canonically defined map of prestacks over $S'\times S$
$$\Gr_{I'}\times S \to \Bun_G\times S'\times S \to S'\times BG_{I},$$
which induces the sought-for map in \eqref{e:local map prel}.

\sssec{}

By construction, the diagram
$$
\CD
\on{C}^*_c(\Ran',\CA')  \otimes   \on{C}^*_c(\Ran,\CB) @>>>  \Lambda   \\
@VVV    @VVV     \\
\on{C}^{\on{red}}_*(\Bun_G)  \otimes \on{C}^*_{\on{red}}(\Bun_G) @>>>  \Lambda
\endCD
$$
commutes. 

\medskip

Hence, \thmref{t:product formula} is equivalent to the following assertion:

\begin{thm} \label{t:global duality}
The map \eqref{e:prel pairing integral} defines an isomorphism
$$\on{C}^*_c(\Ran,\CB) \to \left(\on{C}^*_c(\Ran',\CA') \right)^\vee.$$
\end{thm} 

We refer to \thmref{t:global duality} as the ``global duality" statement. 

\section{The local duality statement}  \label{s:local duality}

In order to understand the contents of this section, one needs to know the contents of Part I (the unital augmented version of the Ran space), 
Part II (only Sects. \ref{ss:Verdier}, \ref{ss:Verdier pushforward}, \ref{ss:Verdier on Ran} and 
\ref{ss:open variant}) and Part III (only \secref{s:pairings aug}).

\medskip

In this section we will reduce the assertion of \thmref{t:global duality} to a \emph{local duality} statement,
namely, \thmref{t:local duality}. This theorem says that a certain map in $\Shv^!(\Ran')$ is an isomorphism. 

\medskip

More precisely, \thmref{t:local duality} says that certain two objects of $\Shv^!(\Ran')$ are related by
Verdier duality. These objects are obtained from $\CB$ and $\CA'$ by
the procedure of taking the units out, indicated in the preamble to \secref{s:pairings aug}. 

\ssec{The unital augmented and reduced versions of $\CB$}

First, we claim that the object $\CB$ is obtained as 
$$\on{OblvUnit}\circ \on{OblvAug}(\CB_{\on{untl,aug}})$$
for a canonically defined object $\CB_{\on{untl,aug}}\in \Shv^!(\Ran_{\on{untl,aug}})$, which is in turn of the form
$$\on{AddUnit}_{\on{aug}}(\CB_{\on{red}})$$ for a 
canonically defined object $\CB_{\on{red}}\in \Shv^!(\Ran)$.

\medskip

To understand the meaning of the above formulas, the reader should be familiar with the contents of \secref{s:aug}.

\sssec{}

To specify $\CB_{\on{untl,aug}}$ we need to construct a compatible family of assignments 
$$S\in \Sch,\,\,K\subseteq I\subset \Maps(S,X) \rightsquigarrow  (\CB_{\on{untl,aug}})_{S,K\subseteq I}\in \Shv(S).$$

We construct $(\CB_{\on{untl,aug}})_{S,K\subseteq I}$ in the same way as $(\CB_{\on{untl,aug}})_{S,I}$, with the difference that we replace
the stack $BG_{I}$ by $BG_{K\subset I}$, that classifies $G$-bundles on $D_{I}$, equipped with a
trivialization on $D_{K}$.

\medskip

For example, for $S=\on{pt}$ and $I$ given by an $I$-tuple of \emph{distinct} $k$-points of $X$, the !-fiber of $\CB$ at $K\subseteq I$ is
$$\on{Fib}\left(\underset{i\in I-K}\bigotimes\, \on{C}^*(BG_{x_i})\to \Lambda\right).$$

\medskip

By construction, $(\CB_{\on{untl,aug}})_{S,\emptyset\subset I}=\CB_{S,I}$, so 
$$\CB\simeq \on{OblvUnit}\circ \on{OblvAug}(\CB_{\on{untl,aug}}),$$
as desired. 

\sssec{}

We claim that $\CB_{\on{untl,aug}}$ satisfies conditions $(*)$ and $(**)$ in \thmref{t:main}. Indeed, this follows from 
the corresponding property of the assignment
$$(S,I)\mapsto (\CB_{\on{red}})_{S,K\subseteq I}.$$
Namely, for $L\subset \Maps(S,X)$ such that $I$ and $L$ have disjoint images, the restriction map
$$BG_{K\cup L\subseteq I\cup L}\to BG_{K\subseteq I}$$
is an isomorphism.

\sssec{}   \label{sss:B red}

Applying \thmref{t:main}, we obtain that there exists a canonically defined object 
$$\CB_{\on{red}}\in \Shv^!(\Ran),$$
so that 
$$\CB_{\on{untl,aug}}=\on{AddUnit}_{\on{aug}}(\CB_{\on{red}}).$$

For example, the !-fiber of $\CB_{\on{red}}$ at a point $\{x_1,...,x_n\}\in \Ran$ is
$$\underset{i=1,...,n}\bigotimes \on{C}^*_{\on{red}}(BG_{x_i}).$$ 
So the effect of replacing $\CB$ be $\CB_{\on{red}}$ consists at the level of !-fibers
of replacing the \emph{augmentation ideal in the tensor product} of $\on{C}^*(BG_{x_i})$ by
the \emph{tensor product of the augmentation ideals} in each $\on{C}^*(BG_{x_i})$, i.e., 
$$\on{Fib}\left( \underset{i=1,...,n}\bigotimes\, \on{C}^*(BG_{x_i})\to \Lambda\right) \rightsquigarrow  
\underset{i=1,...,n}\bigotimes\, \on{Fib}\left(\on{C}^*(BG_{x_i})\to \Lambda\right).$$

\medskip

Denote
$$\CB_{\on{red}}':=\CB_{\on{red}}|_{\Ran'} \text{ and } \CB':=\CB|_{\Ran'}.$$

\ssec{The unital augmented and reduced versions of $\CA$}   \label{sss:unital Gr}

Next, we claim that the object $\CA'$ is obtained as 
$$\on{OblvUnit}\circ \on{OblvAug}(\CA'_{\on{untl,aug}})$$
for a canonically defined object $\CA'_{\on{untl,aug}}\in \Shv^!(\Ran'_{\on{untl,aug}})$, which is in turn of the form
$$\on{AddUnit}_{\on{aug}}(\CA'_{\on{red}})$$ for a 
canonically defined object $\CA'_{\on{red}}\in \Shv^!(\Ran')$. 

\sssec{}

The object $\CA'_{\on{untl,aug}}$ is constructed in a manner parallel to $\CA'$, where instead of
$\Gr_{\Ran'}\to \Ran'$, we use the lax prestack 
\begin{equation} \label{e:aug Gr}
g_{\on{until,aug}}:\Gr_{\Ran'_{\on{until,aug}}}\to \Ran'_{\on{until,aug}},
\end{equation} 
constructed as follows. 

\medskip

For $S\in \Sch$, an $S$-point of $\Gr_{\Ran'_{\on{until,aug}}}$ is the category whose objects are:

\begin{itemize}

\item $K\subseteq I\subset \Maps(S,X')$;

\item a $G$-bundle $\CP_G$ on $S\times X$;

\item a trivialization $\gamma$ of $\CP_G|_{S\times X-\on{Graph}_I}$.

\end{itemize}

Given two such objects 
$$(K^1\subseteq I^1,\CP^1_G,\gamma^1) \text{ and } (K^2\subseteq I^2,\CP^2_G,\gamma^2),$$
a morphism between them is an inclusion $K_1\subseteq K_2$ and $I_1\subseteq I_2$, and an isomorphism
$$\CP^1_G|_{S\times X-\on{Graph}_{K_2}}\simeq \CP^2_G|_{S\times X-\on{Graph}_{K_2}},$$
which is compatible with the trivializations of 
$\CP^1_G|_{S\times X-\on{Graph}_{I_2}}$ and $\CP^2_G|_{S\times X-\on{Graph}_{I_2}}$, given by 
$\gamma^1|_{S\times X-\on{Graph}_{I_2}}$ and 
$\gamma^2$, respectively.

\medskip

It is easy to see that the map \eqref{e:aug Gr} is pseudo-proper, i.e., satisfies the assumption of \corref{c:pseudo-proper},
so that the functor $(g_{\on{untl,aug}})_!$ is well-behaved. 

\medskip

For example, for $S=\on{pt}$ and $I$ given by an $I$-tuple of \emph{distinct} $k$-points of $X'$, the !-fiber of $\CA'_{\on{untl,aug}}$ at $K\subseteq I$ is
$$\on{Fib}\left(\underset{i\in I-K}\bigotimes\, \on{C}_*(\Gr_{x_i})\to \Lambda\right).$$

\sssec{}

By construction,
$$\Ran' \underset{\Ran'_{\on{until,aug}}}\times \Gr_{\Ran'_{\on{until,aug}}}\simeq \Gr_{\Ran'},$$
so 
$$\CA'\simeq \on{OblvUnit}\circ \on{OblvAug}(\CA'_{\on{untl,aug}}),$$
as desired. 

\medskip

\noindent {\it Notation:} For $S\in \Sch$ and an $S$-point $K\subseteq I$ of
$\Ran_{\on{untl,aug}}$, we let $\Gr_{K\subseteq I}$ denote the prestack
$$S\underset{\Ran'_{\on{untl,aug}}}\times \Gr_{\Ran'_{\on{untl,aug}}}.$$

\begin{rem}  \label{r:Gr on open}
Note that in defining $\Gr_{\Ran'_{\on{until,aug}}}$, we can avoid referring to the complete curve $X$. Indeed,
instead of $\CP_G$ being a $G$-bundle on $X$ and $\gamma$ its trivialization over $S\times X-\on{Graph}_I$, we can let $\CP_G$
be a $G$-bundle on $X'$ and $\gamma'$ its trivialization over $S\times X'-I$. And when defining morphisms, we take isomorphisms
$$\CP^1_G|_{S\times X'-\on{Graph}_{K_2}}\simeq \CP^2_G|_{S\times X'-\on{Graph}_{K_2}}.$$
\end{rem} 

\sssec{}

We claim that $\CA'_{\on{untl,aug}}$ satisfies conditions $(*)$ and $(**)$ in \thmref{t:main}. Indeed, this follows
from the corresponding property of the lax prestack
$$\Gr_{\Ran'_{\on{untl,aug}}}\to \Ran'_{\on{untl,aug}}.$$
Namely, for $I,K,L\subset \Maps(S,X)$ such that $I$ and $L$ have disjoint images, the restriction map
$$\Gr_{K\subseteq I}\to \Gr_{K\cup L\subseteq I\cup L}$$
is an isomorphism.

\medskip

Hence, from \thmref{t:main} we obtain that there exists a canonically defined object 
$$\CA_{\on{red}}'\in \Shv^!(\Ran'),$$
so that 
$$\CA'_{\on{untl,aug}}=\on{AddUnit}_{\on{aug}}(\CA_{\on{red}}').$$

For example, the !-fiber of $\CA_{\on{red}}$ at a point $\{x_1,...,x_n\}\in \Ran'$ is
$$\underset{i=1,...,n}\bigotimes\, \on{C}_*^{\on{red}}(\Gr_{x_i}).$$

So the effect of replacing $\CA$ be $\CA_{\on{red}}$ consists at the level of !-fibers of replacing 
$$\on{Fib}\left( \underset{i=1,...,n}\bigotimes\, \on{C}_*(\Gr_{x_i})\to \Lambda\right) \rightsquigarrow  
\underset{i=1,...,n}\bigotimes\, \on{Fib}\left(\on{C}_*(\Gr_{x_i})\to \Lambda\right).$$

\sssec{}

Recall now (see \secref{ss:open variant}) that we denote by $\bj_{\Ran}$ the open embedding 
$\Ran'\to \Ran$. Recall also that the restriction functor $\bj^!$ admits both a left and a right adjoints,
denoted $(\bj_{\Ran})_!$ and $(\bj_{\Ran})_*$, respectively.

\medskip

Set
$$\CA_{\on{red}}:=(\bj_{\Ran})_!(\CA_{\on{red}}').$$

Denote also
$$\CA_{\on{untl,aug}}:=\on{AddUnit}_{\on{aug}}(\CA_{\on{red}}) \text{ and } \CA:=
\on{OblvUnit}\circ \on{OblvAug}(\CA_{\on{untl,aug}}).$$

\ssec{The unital augmented and reduced versions of the pairing}

Recall the notion of pairing between two sheaves on $\Ran_{\on{untl,aug}}$, see \secref{ss:notion of pairing}.

\medskip

We will show that the pairing \eqref{e:prel pairing} extends to a canonically defined map
\begin{equation} \label{e:aug pairing for Gr'}
\CA'_{\on{untl,aug}}\boxtimes \CB_{\on{untl,aug}} \to \omega_{(\Ran'_{\on{untl,aug}}\times \Ran_{\on{untl,aug}})_{\on{sub,disj}}},
\end{equation} 
where
$$(\Ran'_{\on{untl,aug}}\times \Ran_{\on{untl,aug}})_{\on{sub,disj}}:=
(\Ran_{\on{untl,aug}}\times \Ran_{\on{untl,aug}})_{\on{sub,disj}}\cap (\Ran'_{\on{untl,aug}}\times \Ran_{\on{untl,aug}}).$$

We will use the map \eqref{e:aug pairing for Gr'} and \thmref{t:aug pairings} to define a pairing between $\CA'_{\on{red}}$ and $\CB'_{\on{red}}$. 

\sssec{}

To construct the map \eqref{e:aug pairing for Gr'}, given $S_1\in \Sch$ with $K_1\subseteq I_1\subset \Maps(S_1,X')$ 
and $S_2\in \Sch$ with $K_2\subseteq I_2\subset \Maps(S_2,X)$ so that the resulting object of
$(\Ran_{\on{untl,aug}}\times \Ran_{\on{untl,aug}})(S_1\times S_2)$ belongs to
$$(\Ran_{\on{untl,aug}}\times \Ran_{\on{untl,aug}})_{\on{\on{sub,disj}}}(S_1\times S_2)\subset 
(\Ran_{\on{untl,aug}}\times \Ran_{\on{untl,aug}})(S_1\times S_2),$$
we need to construct a map
$$(\CA'_{\on{untl,aug}})_{S_1,K_1\subseteq I_1}\boxtimes 
(\CB_{\on{untl,aug}})_{S_2,K_2\subseteq I_2}\to \omega_{S_1\times S_2}.$$

This map is constructed as in \secref{sss:prel pairing} from the map of lax prestacks
$$\Gr_{K_1\subseteq I_1}\times S_2\to S_1\times BG_{K_2\subseteq I_2}.$$

\sssec{}

From \eqref{e:aug pairing for Gr'} we in particular obtain a pairing
\begin{equation} \label{e:aug pairing for Gr}
\CA'_{\on{untl,aug}}\boxtimes \CB'_{\on{untl,aug}} \to \omega_{(\Ran'_{\on{untl,aug}}\times \Ran'_{\on{untl,aug}})_{\on{sub,disj}}},
\end{equation} 

Applying \corref{c:pairing on reduced}, from \eqref{e:aug pairing for Gr} we obtain a pairing 
\begin{equation} \label{e:pairing on red Gr}
\CA_{\on{red}}'\boxtimes \CB_{\on{red}}'\to (\on{diag}_{\Ran'})_!(\omega_{\Ran'}).
\end{equation}

Hence, we obtain a map
\begin{equation} \label{e:local duality}
\CB_{\on{red}}'\to \BD_{\Ran'}(\CA_{\on{red}}').
\end{equation}

We claim:

\begin{thm} \label{t:local duality}
The map is \eqref{e:local duality} is an isomorphism.
\end{thm}

We refer to \thmref{t:local duality} as the ``local duality" statement.  

\begin{rem}
The assertion of \thmref{t:local duality}, as well as its proof, are valid for any $G$ that is \emph{reductive} over $X'$. I.e., we do not
the ``semi-simple simply connected" hypothesis for the validity of \thmref{t:local duality}. It is for the deduction
$$\text{Local duality}\,\, \Rightarrow \text{Global duality}$$
that such a hypothesis is needed: it will be used in showing that $\CA_{\on{red}}'$ verifies the conditions of \thmref{t:Verdier on Ran}.
\end{rem}

\ssec{Local duality implies global duality}

In this subsection we will show how \thmref{t:local duality} implies \thmref{t:global duality}. Here we will use the material from Part II
in a crucial way, specifically \thmref{t:Verdier Ran open}. 

\sssec{}

The map \eqref{e:pairing on red Gr} gives rise to a map
\begin{equation} \label{e:interm pairing loc}
\CA_{\on{red}}'\boxtimes \CB_{\on{red}}\to (\on{id}_{\Ran}\times \bj_{\Ran})_*\circ (\on{diag}_{\Ran'})_!(\omega_{\Ran'})\simeq 
(\on{Graph}_{\bj_{\Ran}})_!(\omega_{\Ran'}),
\end{equation}
and by applying $\on{C}^*_c(\Ran'\times \Ran,-)$ to a map
\begin{equation} \label{e:interm pairing}
\on{C}^*_c(\Ran',\CA_{\on{red}}')\otimes \on{C}^*_c(\Ran,\CB_{\on{red}})\to \Lambda.
\end{equation} 

\medskip

We claim:

\begin{lem}  \label{l:primes}
The diagram
$$
\CD
\on{C}^*_c(\Ran',\CA')\otimes \on{C}^*_c(\Ran,\CB) @>>>  \Lambda  \\
@A{\text{\eqref{e:id to add}}\otimes \text{\eqref{e:id to add}}}AA    @AAA   \\
\on{C}^*_c(\Ran',\CA_{\on{red}}')\otimes \on{C}^*_c(\Ran,\CB_{\on{red}})   @>>>   \Lambda   \\
\endCD
$$
commutes.
\end{lem} 

\begin{proof}

We will show that the diagram 
\begin{equation} \label{e:primes}
\CD
\CA' \boxtimes \CB  @>>>  \omega_{\Ran'\times \Ran}  \\
@A{\text{\eqref{e:id to add}}\otimes \text{\eqref{e:id to add}}}AA   @AAA    \\
\CA_{\on{red}}' \boxtimes \CB_{\on{red}}   @>{\text{\eqref{e:interm pairing loc}}}>>  (\on{Graph}_{\bj_{\Ran}})_!(\omega_{\Ran'})
\endCD
\end{equation} 
commutes. This would imply the assertion of the lemma. 

\medskip

Consider the map
$$\bj_{\Ran_{\on{untl,aug}}}:\Ran'_{\on{untl,aug}}\to \Ran_{\on{untl,aug}},$$
and consider the object 
$$(\CA_{\on{untl,aug}})_{\on{bad}}:=(\bj_{\Ran_{\on{untl,aug}}})_!(\CA'_{\on{untl,aug}})\in \Shv^!(\Ran_{\on{untl,aug}}).$$

We have
$$(\bj_{\Ran_{\on{untl,aug}}})^!((\CA_{\on{untl,aug}})_{\on{bad}})\simeq \CA'_{\on{untl,aug}}.$$

(We remark that the object $(\CA_{\on{untl,aug}})_{\on{bad}}$ does not have a clear
geometric meaning; for example it does not satisfy conditions $(*)$ and $(**)$ of \thmref{t:main};
in particular, it is \emph{not} isomorphic to $\CA_{\on{untl,aug}}$.) 

\medskip

By adjunction, the map \eqref{e:aug pairing for Gr'} gives rise to a pairing
$$(\CA_{\on{untl,aug}})_{\on{bad}}\boxtimes \CB_{\on{untl,aug}}\to 
\omega_{(\Ran_{\on{untl,aug}}\times \Ran_{\on{untl,aug}})_{\on{sub,disj}}}.$$
 
Applying \thmref{t:aug pairings}(ii and iv), we obtain a pairing
$$\on{TakeOut}((\CA_{\on{untl,aug}})_{\on{bad}})\boxtimes \CB_{\on{red}}\to (\on{diag}_{\Ran})_!(\omega_{\Ran}),$$
so that the diagram
$$
\CD
\on{OblvUnit}\circ \on{OblvAug}((\CA_{\on{untl,aug}})_{\on{bad}}) \boxtimes \CB_{\on{untl,aug}} @>>>  \omega_{\Ran\times \Ran}  \\
@A{\text{\eqref{e:id to add}}\otimes \text{\eqref{e:id to add}}}AA    @AAA    \\
\on{TakeOut}((\CA_{\on{untl,aug}})_{\on{bad}})\boxtimes \CB_{\on{red}} @>>>  (\on{diag}_{\Ran})_!(\omega_{\Ran})
\endCD
$$
commutes.

\medskip

Restricting to $\Ran'\times \Ran$, we obtain a commutative diagram that identifies with \eqref{e:primes}. 

\end{proof} 

\sssec{}

Note that the left vertical arrow in the diagram in \lemref{l:primes} is an isomorphism by \corref{c:id to add}.
Hence, in order to prove \thmref{t:global duality}, it suffices to show that the map
\begin{equation} \label{e:interm map}
\on{C}^*_c(\Ran,\CB_{\on{red}}) \to \left(\on{C}^*_c(\Ran',\CA_{\on{red}}')\right)^\vee,
\end{equation}
defined by the pairing \eqref{e:interm pairing}, is an isomorphism. 

\sssec{}

Recall that by \lemref{l:j}, we have a canonical isomorphism
$$(\bj_{\Ran})_*\circ \BD_{\Ran'}(\CA'_{\on{red}})\simeq \BD_{\Ran}\circ (\bj_{\Ran})_!(\CA_{\on{red}}').$$

\medskip

We have a commutative diagram
$$
\CD
\on{C}^*_c(\Ran,\CB_{\on{red}}) @>{\text{\eqref{e:interm map}}}>>   \left(\on{C}^*_c(\Ran',\CA_{\on{red}}')\right)^\vee    \\
@VVV    @AA{\sim}A   \\
\on{C}^*_c(\Ran,(\bj_{\Ran})_*(\CB_{\on{red}}'))   & &  \left(\on{C}^*_c(\Ran,(\bj_{\Ran})_!(\CA_{\on{red}}'))\right)^\vee  \\  
@V{\text{\eqref{e:local duality}}}VV     @AA{\text{\eqref{e:duality and cohomology}}}A   \\
\on{C}^*_c(\Ran,(\bj_{\Ran})_*\circ \BD_{\Ran'}(\CA_{\on{red}}')) @>{\sim}>> \on{C}^*_c(\Ran,\BD_{\Ran}\circ (\bj_{\Ran})_!(\CA_{\on{red}}')).
\endCD
$$

Hence, in order to prove that \eqref{e:interm map} is an isomorphism (and thus finish the proof of \thmref{t:product formula}), 
we need to show that the maps
$$\on{C}^*_c(\Ran,(\bj_{\Ran})_*(\CB_{\on{red}}'))\to \on{C}^*_c(\Ran,(\bj_{\Ran})_*\circ \BD_{\Ran'}(\CA_{\on{red}}'))$$
and
$$\on{C}^*_c(\Ran,\CB_{\on{red}}) \to \on{C}^*_c(\Ran,(\bj_{\Ran})_*(\CB_{\on{red}}'))$$
and
$$\on{C}^*_c(\Ran,\BD_{\Ran}\circ (\bj_{\Ran})_!(\CA_{\on{red}}'))\to \left(\on{C}^*_c(\Ran,(\bj_{\Ran})_!(\CA_{\on{red}}'))\right)^\vee,$$
appearing in the above commutative diagram, are isomorphisms.

\medskip

Now, \thmref{t:local duality} implies that the first of these three maps is an isomorphism. The fact that the second map is an isomorphism
follows from the next assertion (proved right below): 

\begin{prop}  \label{p:B is extended}
For $G$ satisfying the assumption of \secref{sss:good G}, the map
$$\CB_{\on{red}} \to (\bj_{\Ran})_*(\CB_{\on{red}}')$$
is an isomorphism.
\end{prop}

Finally, the fact that the third of the above maps is an isomorphism follows from \thmref{t:Verdier Ran open} using the
next assertion (also proved right below): 

\begin{prop} \label{p:estimate for extensions}
The object $\CA_{\on{red}}'\in \Shv(\Ran')$ satisfies the cohomological estimate of \thmref{t:Verdier on Ran} 
\footnote{We repeat that this estimate says that 
every integer $k\geq 0$ there exists an integer $n_k\geq 0$, such that 
the object $\on{ins}_\CI^!(\CA_{\on{red}}')|_{\overset{\circ}X{}'{}^\CI}$ is concentrated in \emph{perverse} cohomological degrees 
$\leq -k-|\CI|$ whenever $|\CI|>n_k$.} over the curve $X'$.
\end{prop}

\begin{proof}[Proof of \propref{p:B is extended}]

As we shall see in \secref{sss:B factorizes}, the object $\CB_{\on{red}}\in \Shv^!(\Ran)$ has a structure of commutative
factorization algebra.  We have the following general assertion:

\begin{lem}
Let $\CF\in  \Shv^!(\Ran)$ have a structure of commutative factorization (co)algebra. Then the map
$\CF\to (\bj_{\Ran})_*(\CF|_{\Ran'})$
is an isomorphism if and only if the map
$\CF_X\to \bj_*(\CF_X|_{X'})$
is.
\end{lem}

The proof of the lemma is immediate from the description of the functor $(\bj_{\Ran})_*$ given by 
\eqref{e:descr j*}. 

\medskip

Hence, in order to prove the proposition, it suffices to show that for any $x\in X-X'$, we have
$(\CB_{\on{red}})_{\{x\}}=0$.  Note that for a singleton set we have
$$(\CB_{\on{red}})_{\{x\}}\simeq (\on{OblvUnit}\circ \on{OblvAug}(\CB_{\on{untl,aug}}))_{\{x\}}=
\CB_{\{x\}},$$
and the latter is by definition $\on{C}^*_{\on{red}}(BG_x)$. 

\medskip

Now, the assumption that $G_x$ is contractible implies that $BG_x$ is universally homologically contractible. Hence,
$\on{C}^*_{\on{red}}(BG_x)=0$ as required.

\end{proof} 

\begin{proof}[Proof of \propref{p:estimate for extensions}]

We will show that the restriction of $\CA_{\on{red}}'$ to $\overset{\circ}X{}^n$ lives in (perverse) cohomological degrees 
$\leq -3n$. 

\medskip

As we shall see in \secref{sss:A factorizes}, the object $\CA_{\on{red}}'\in \Shv^!(\Ran')$ has a structure of cocommutative
factorization algebra. 
\footnote{We say ``factorization algebra" because we have not introduced the general notion of \emph{factorization sheaf}.
For our purposes here it is the ``factorization" and not the ``cocommutative" part that is important.} 
Hence, in order to prove the required cohomological estimate, it suffices to do so for
$n=1$. I.e., we are considering the affine Grassmannian 
$$\Gr_{X'}\overset{g_X}\longrightarrow X',$$
and we need to show that
$$\on{Fib}\left((g_{X'})_!(\omega_{\Gr_{X'}})\to \omega_{X'}\right)$$
lives in (perverse) cohomological degrees $\leq -3$.

\medskip

By passing to the \'etale cover of $X'$, we can assume that $G|_{X'}$ is a constant group-scheme. Hence, it is
enough to show that for some/any point $x\in X$, we have that $\on{C}^{\on{red}}_*(\Gr_x)$ lives in cohomological
degrees $\leq -2$.  

\medskip

However, the latter is a well-known property of the affine Grassmannian of
\emph{semi-simple simply connected} groups. (Note that if $G$ is reductive but not semi-simple simply connected, then $\Gr_x$
is \emph{disconnected} and the estimate of \propref{p:estimate for extensions} does not hold.)

\end{proof} 

\section{Reduction to a pointwise duality statement}  \label{s:pointwise}

Our goal of the rest of Part V is to prove the local duality statement, \thmref{t:local duality}.
The material in this section will use Part IV of the paper. 

\medskip 

In this section we shall take $X$
to be a smooth (but not necessarily complete) curve. We will take $G$ to be a smooth group-scheme
over $X$, whose fibers are semi-simple and simply connected. We will reduce \thmref{t:local duality} to \thmref{t:pointwise duality} 
that says that the map 
$$\CB_{\on{red}}\to \BD_{\Ran}(\CA_{\on{red}})$$
induces an isomorphism
$$(\CB_{\on{red}})_{\{x\}}\to (\BD_{\Ran}(\CA_{\on{red}}))_{\{x\}}$$
for \emph{some} curve $X$ and \emph{some} point $x\in X$. 

\ssec{The structure on $\CA$ of cocommutative factorization coalgebra}

We wish to show that the map $\CB_{\on{red}}\to \BD_{\Ran}(\CA_{\on{red}})$, given by \eqref{e:local duality} is an isomorphism.
As was explained in the preamble to \secref{ss:alg and coalg}, a convenient tool for this would be to first endow $\CA_{\on{red}}$
(resp., $\CB_{\on{red}}$) with a structure of cocommutative (reps., commutative) factorization coalgebra (resp., algebra).

\medskip

In this subsection we will define the relevant structure on $\CA_{\on{red}}$, by deducing it from the corresponding structure
on $\CA_{\on{untl,aug}}$. 

\sssec{}   \label{sss:A factorizes}

Let us note that the diagonal map
$$\Gr_{\Ran_{\on{untl,aug}}}\to \Gr_{\Ran_{\on{untl,aug}}}\underset{\Ran_{\on{untl,aug}}}\times \Gr_{\Ran_{\on{untl,aug}}}$$
defines on $\CA_{\on{untl,aug}}$ a structure of cocommutative coalgebra on $\Shv^!(\Ran_{\on{untl,aug}})$
(with respect to the pointwise symmetric monoidal structure).

\medskip

By \corref{c:structure on red} we obtain that $\CA_{\on{red}}$ acquires a structure of cocommutative coalgebra in
$\Shv^!(\Ran)$ (with respect to the convolution symmetric monoidal structure). 

\sssec{}

We have the following crucial observation:

\begin{prop} \label{p:Gr factorizes}
The cocommutative coalgebra $\CA_{\on{untl,aug}}\in \Shv^!(\Ran_{\on{untl,aug}})$ is a cocommutative factorization coalgebra
(in the sense of \secref{sss:factor coalgebras aug}).
\end{prop}

\begin{proof}

The assertion of the proposition follows from the corresponding property of the lax prestack $\Gr_{\Ran_{\on{untl,aug}}}$. 
Namely, we claim that for any $S\in \Sch$ and an $S$-point $(K_1\subseteq I_1),(K_2\subseteq I_2)$ of 
$$(\Ran_{\on{untl,aug}}\times \Ran_{\on{untl,aug}})_{\on{compl,disj}},$$ the map
$$\Gr_{K_1\cup K_2\subseteq I_1\cup I_2}\overset{\on{diag}}\longrightarrow 
\Gr_{K_1\cup K_2\subseteq I_1\cup I_2}\times \Gr_{K_1\cup K_2\subseteq I_1\cup I_2}\to
\Gr_{K_1\cup I_2\subseteq I_1\cup I_2}\times \Gr_{I_1\cup K_2\subseteq I_1\cup I_2}$$
is an isomorphism.

\end{proof} 

By \propref{p:aug and factor}(b), from \propref{p:Gr factorizes} we obtain:

\begin{cor} \label{c:G factorizes}
The cocommutative coalgebra $\CA_{\on{red}}\in \Shv^!(\Ran)$ admits a canonical structure of cocommutative factorization coalgebra
(in the sense of \secref{sss:factor coalgebras}). 
\end{cor} 

\sssec{}

Consider the object $\BD_{\Ran}(\CA_{\on{red}})\in \Shv^!(\Ran)$. By \secref{sss:alg structure on dual}, we obtain that
$\BD_{\Ran}(\CA_{\on{red}})$ acquires a structure of commutative algebra in $\Shv^!(\Ran)$ (with respect to the
convolution symmetric monoidal structure).

\medskip

Applying \propref{p:Verdier factorizes}, and using \propref{p:estimate for extensions}, we obtain:

\begin{cor} \label{c:A factor}
The commutative algebra $\BD_{\Ran}(\CA_{\on{red}})$ admits a canonical structure of commutative factorization algebra
(in the sense of \secref{sss:factor algebras}). 
\end{cor} 

\ssec{The structure on $\CB$ of commutative factorization algebra}  \label{ss:B factorizes}

In this subsection we will define a structure of commutative factorization algebra on $\CB_{\on{red}}$, by deducing it from the
corresponding structure on $\CB_{\on{untl,aug}}$. 

\sssec{}   \label{sss:B factorizes}

Next, we note that the diagonal maps
$$BG_{K\subset I}\to BG_{K\subset I}\times BG_{K\subset I}$$
define on $\CB_{\on{untl,aug}}$ a structure of commutative algebra on $\Shv^!(\Ran_{\on{untl,aug}})$
(with respect to the pointwise symmetric monoidal structure).

\medskip

Hence, by \corref{c:structure on red} we obtain that $\CB_{\on{red}}$ acquires a structure of commutative algebra in
$\Shv^!(\Ran)$ (with respect to the convolution product). 

\sssec{}

We note:

\begin{prop}  \label{p:B factor}
The commutative algebra $\CB_{\on{untl,aug}}$ is a commutative factorization algebra (in the sense of \secref{sss:factor algebras aug}). 
\end{prop}

\begin{proof}

This follows from the corresponding property of the assignment
$$(S,K\subseteq I)\rightsquigarrow BG_{K\subset I}.$$

Namely, for $S\in \Sch$ and an $S$-point $(I_1,I_2)$ of 
$$(\Ran_{\on{untl}}\times \Ran_{\on{untl}})_{\on{disj}},$$ the map
$$BG_{x_{I_1\cup I_2}}\overset{\on{diag}}\longrightarrow BG_{x_{I_1\cup I_2}}\times BG_{x_{I_1\cup I_2}}\to
BG_{x_{I_1}}\times BG_{x_{I_2}}$$
is an isomorphism, and the fact that $BG_{x_{I}}$ are quasi-compact Artin stacks over $S$, so that the K\"unneth formula
holds (see \cite[Proposition A.5.19]{Main Text}).

\end{proof}

By \propref{p:aug and factor}(a), from \propref{p:B factor}, we obtain:

\begin{cor} \label{c:B factor}
The commutative algebra $\CB_{\on{red}}\in \Shv^!(\Ran)$ has a canonical structure of commutative factorization algebra
(in the sense of \secref{sss:factor algebras}). 
\end{cor} 

\ssec{Compatibility of the pairing with the algebra structure}

We have endowed both $\BD_{\Ran}(\CA)$ and $\CB$ with a structure of commutative algebra in $\Shv^!(\Ran)$. In this subsection we will show
that the map $\CB\to \BD_{\Ran}(\CA)$ has a natural structure of homomorphism of commutative algebras. 

\sssec{}

Observe that the pairing 
$$\CB_{\on{untl,aug}}\boxtimes \CA_{\on{untl,aug}} \to \omega_{(\Ran_{\on{untl,aug}}\times \Ran_{\on{untl,aug}})_{\on{sub,disj}}}$$
of \eqref{e:aug pairing for Gr} has a structure of compatibility with the algebra and coalgebra
structure on $\CB_{\on{untl,aug}}$ and $\CA_{\on{untl,aug}}$ (see \secref{sss:compat pairing} for what this means). 

\medskip

Hence, by \lemref{l:compat pairing aug}, we obtain that the induced pairing 
$$\CB_{\on{red}}\boxtimes \CA_{\on{red}}\to (\on{diag}_{\Ran})_!(\omega_{\Ran})$$
of \eqref{e:pairing on red Gr} is has a structure of compatibility with the algebra and coalgebra
structure on $\CB_{\on{red}}$ and $\CA_{\on{red}}$ (see \secref{sss:compat pairing red}) for what this means). 

\medskip

Therefore, by \lemref{l:pairings and alg}, the map
$$\CB_{\on{red}}\to \BD_{\Ran}(\CA_{\on{red}}),$$
appearing in \thmref{t:local duality} has a structure of homomorphism of commutative algebras.

\sssec{}  \label{sss:reduce to diag}

Combining with Corollaries \ref{c:A factor} and \ref{c:B factor}, and \lemref{l:isom on diag} we obtain that 
in order to prove \thmref{t:local duality}, it suffices to show that the map 
\begin{equation} \label{e:assertion on diag}
(\CB_{\on{red}})_X\to (\BD_{\Ran}(\CA_{\on{red}}))_X
\end{equation} 
is an isomorphism.

\ssec{Reduction to the constant group-scheme case}

In this subsection we will exploit some properties of the assertion of \thmref{t:local duality} that embody its
locality property.

\medskip

We will show that for the validity of the isomorphism \eqref{e:assertion on diag} is insensitive to changes
of the curve or forms of $G$.

\sssec{}

Let $f:X_1\to X_2$ be an \'etale map between (not necessarily complete) curves. Let $f_{\Ran}$ denote 
the resulting map $\Ran_1\to \Ran_2$. The map $f_{\Ran}$ is not \'etale, but we shall now single out a locus
of $\Ran_1$ over which it is well-behaved. 

\medskip

Let $(X_1\times X_1)_{\on{rel.disj}}\subset (X_1\times X_1)$ be the open subset obtained by removing the
closed subset
$$X_1\underset{X_2}\times X_1-\on{diag}_{X_1},$$
(where $\on{diag}_{X_1}$ is a connected component in $X_1\underset{X_2}\times X_1$ due to the assumption that
$X_1$ be \'etale over $X_2$). 

\medskip

For a finite set $\CI$, let $(X_1)^\CI_{\on{rel.disj}}\subset (X_1)^\CI$ be the open subset
$$\underset{i_1\neq i_2}\cap\,  (X_1)^{\CI-\{i_1,i_2\}}\times (X_1\times X_1)_{\on{rel.disj}}.$$

For example $k$-points of $(X_1)^\CI_{\on{rel.disj}}$ are maps
$$\CI\to X_1(k)$$
with the following property: if for a pair of indices $i_1$ and $i_2$ we have 
$f(x_{i_1})=f(x_{i_2})$ then $x_{i_1}=x_{i_2}$. 

\medskip

Set
$$(\Ran_1)_{\on{rel.disj}}:=\underset{\CI\in (\on{Fin}^s)^{\on{op}}}{\on{colim}}\, (X_1)^\CI_{\on{rel.disj}}.$$

We claim:

\begin{prop} \label{p:rel disj}  \hfill 

\smallskip

\noindent{\em(a)}  The forgetful map $(\Ran_1)_{\on{rel.disj}}\to \Ran_1$ is an open embedding.  

\smallskip

\noindent{\em(b)} The composed map
$$(\Ran_1)_{\on{rel.disj}}\to \Ran_1 \overset{f_{\Ran}}\longrightarrow \Ran_2$$
is \'etale, and is surjective if $f$ is.

\end{prop}

The proof of \propref{p:rel disj} relies on the following observation:

\begin{lem}
For a surjection of finite sets $\CI\twoheadrightarrow \CK$, we have
$$X_1^\CK\underset{X_1^\CI}\times (X_1)^\CI_{\on{rel.disj}}=(X_1)^\CK_{\on{rel.disj}},$$
and
$$X_2^\CK\underset{X_2^\CI}\times (X_1)^\CI_{\on{rel.disj}}=(X_1)^\CK_{\on{rel.disj}},$$
as open subsets of $X_1^\CK$ 
\end{lem}

\begin{proof}[Proof of \propref{p:rel disj}]

To prove point (a) we need to show that for a finite non-empty set $\CJ$, the fiber product
$$X_1^\CJ\underset{\Ran_1}\times (\Ran_1)_{\on{rel.disj}}$$
is an open subscheme of $X_1^\CJ$. In fact, we will show that the above fiber product identifies with 
$(X_1)^\CJ_{\on{rel.disj}}$.

\medskip

We have:
\begin{multline*} 
X_1^\CJ\underset{\Ran_1}\times (\Ran_1)_{\on{rel.disj}}\simeq 
\underset{\CI}{\on{colim}}\, X_1^\CJ \underset{\Ran_1}\times (X_1)^\CI_{\on{rel.disj}} 
\simeq \underset{\CI}{\on{colim}}\, (X_1^\CJ \underset{\Ran_1}\times X_1^\CI) \underset{X_1^\CI}\times (X_1)^\CI_{\on{rel.disj}} \simeq \\
\simeq \underset{\CI}{\on{colim}}\, \underset{\CI\twoheadrightarrow \CK\twoheadleftarrow \CJ}{\on{colim}}\, 
X_1^\CK \underset{X_1^\CI}\times (X_1)^\CI_{\on{rel.disj}}\simeq 
\underset{\CI}{\on{colim}}\, \underset{\CI\twoheadrightarrow \CK\twoheadleftarrow \CJ}{\on{colim}}\,  (X_1)^\CK_{\on{rel.disj}}
\simeq \underset{\CK \twoheadleftarrow \CJ}{\on{colim}}\, (X_1)^\CK_{\on{rel.disj}}\simeq (X_1)^\CJ_{\on{rel.disj}}.
\end{multline*}  

\medskip

To prove point (b), we need to show that for a finite non-empty set $\CJ$, the fiber product
$$X_2^\CJ\underset{\Ran_2}\times (\Ran_1)_{\on{rel.disj}}$$
is \'etale over $X_2^\CJ$. We will show that the above fiber product identifies with $(X_1)^\CJ_{\on{rel.disj}}$.

\medskip

We have:
\begin{multline*} 
X_2^\CJ\underset{\Ran_2}\times (\Ran_1)_{\on{rel.disj}}\simeq 
\underset{\CI}{\on{colim}}\, X_2^\CJ \underset{\Ran_2}\times (X_1)^\CI_{\on{rel.disj}} 
\simeq \underset{\CI}{\on{colim}}\, (X_2^\CJ \underset{\Ran_2}\times X_2^\CI) \underset{X_2^\CI}\times (X_1)^\CI_{\on{rel.disj}} \simeq \\
\simeq \underset{\CI}{\on{colim}}\, \underset{\CI\twoheadrightarrow \CK\twoheadleftarrow \CJ}{\on{colim}}\, 
X_2^\CK \underset{X_2^\CI}\times (X_1)^\CI_{\on{rel.disj}}\simeq 
\underset{\CI}{\on{colim}}\, \underset{\CI\twoheadrightarrow \CK\twoheadleftarrow \CJ}{\on{colim}}\,  (X_1)^\CK_{\on{rel.disj}}
\simeq (X_1)^\CJ_{\on{rel.disj}}.
\end{multline*}  

\end{proof}

\sssec{}

Let $G_2$ be a group-scheme on $X_2$, and let $G_1$ be its pullback to $X_1$.

\medskip

According to Remark \ref{r:Gr on open}, we have well-defined objects $\CA_{\on{red},i}\in \Shv(\Ran_i)$, $i=1,2$.

\medskip

It follows from the definitions that we have a canonical isomorphism
$$(f_{\Ran})^!(\CA_{\on{red},2})|_{(\Ran_1)_{\on{rel.disj}}}\simeq \CA_{\on{red},1}|_{(\Ran_1)_{\on{rel.disj}}}.$$

\medskip

Note now that we also have well-defined objects $\CB_{\on{red},i}\in \Shv(\Ran_i)$, $i=1,2$, and a canonical isomorphism
$$(f_{\Ran})^!(\CB_{\on{red},2})|_{(\Ran_1)_{\on{rel.disj}}}\simeq \CB_{\on{red},1}|_{(\Ran_1)_{\on{rel.disj}}}.$$

\medskip

By \lemref{l:etale and dual} and \propref{p:rel disj}, we have canonical isomorphisms
$$\left((f_{\Ran})^!\circ \BD_{\Ran_2}(\CA_{\on{red},2})\right)|_{(\Ran_1)_{\on{rel.disj}}}\simeq
\BD_{(\Ran_1)_{\on{rel.disj}}}\left((f_{\Ran})^!(\CA_{\on{red},2})|_{(\Ran_1)_{\on{rel.disj}}}\right)$$
and
$$\BD_{\Ran_1}(\CA_{\on{red},1})|_{(\Ran_1)_{\on{rel.disj}}} \simeq 
\BD_{(\Ran_1)_{\on{rel.disj}}}(\CA_{\on{red},1}|_{(\Ran_1)_{\on{rel.disj}}}).$$

\medskip

Furthermore, by unwinding the constructions, we obtain that the diagram
$$
\CD
\CB_{\on{red},1}|_{(\Ran_1)_{\on{rel.disj}}}  @>>>  \BD_{\Ran_1}(\CA_{\on{red},1})|_{(\Ran_1)_{\on{rel.disj}}}  \\
@V{\sim}VV  @VV{\sim}V   \\
(f_{\Ran})^!(\CB_{\on{red},2})|_{(\Ran_1)_{\on{rel.disj}}} & &   \BD_{(\Ran_1)_{\on{rel.disj}}}(\CA_{\on{red},1}|_{(\Ran_1)_{\on{rel.disj}}})  \\
@VVV     @VV{\sim}V    \\
\left((f_{\Ran})^! \circ \BD_{\Ran_2}(\CA_{\on{red},2})\right)|_{(\Ran_1)_{\on{rel.disj}}} @>{\sim}>> 
\BD_{(\Ran_1)_{\on{rel.disj}}}\left((f_{\Ran})^!(\CA_{\on{red},2})|_{(\Ran_1)_{\on{rel.disj}}}\right)
\endCD
$$
commutes. 

\sssec{}  \label{sss:deductions}

Thus, we obtain that if $f:X_1\to X_2$ is \'etale and surjective, then so is the map
$$(\Ran_1)_{\on{rel.disj}}\to \Ran_2,$$
and hence the assertion of \thmref{t:local duality} for $G_1$ 
implies that for $G_2$.   

\medskip

Conversely, if the assertion of \thmref{t:local duality} holds for $G_2$, then the map
$$\CB_{\on{red},1}|_{(\Ran_1)_{\on{rel.disj}}}  \to  \BD_{\Ran_1}(\CA_{\on{red},1})|_{(\Ran_1)_{\on{rel.disj}}}$$
is an isomorphism. I.e., the assertion of \thmref{t:local duality} holds for $G_1$ \emph{over}
$(\Ran_1)_{\on{rel.disj}}$.   

\medskip

However, since $X_1\subset (\Ran_1)_{\on{rel.disj}}$ and using \secref{sss:reduce to diag}, we 
obtain that the assertion of \thmref{t:local duality} for $G_2$ implies the
assertion of \thmref{t:local duality} for $G_1$. 

\sssec{}

For a reductive group scheme $G$ on $X$,  let $X_1\to X$ be an \'etale cover 
such that $G_1:=G|_{X_1}$ is a constant group-scheme. Thus, we obtain that it is enough 
to prove \thmref{t:local duality} in the case of constant group schemes.

\medskip

Now, if $G$ is a constant group scheme on $X$, let $X'\to X$ be an open cover such that $X'$ admits an
\'etale map to $\BA^1$. Hence, we obtain that it is enough to prove \thmref{t:local duality} 
(or, equivalently, that \eqref{e:assertion on diag} is an isomorphism) in the case when $G$ is a constant
group scheme and $X=\BA^1$.

\medskip

Note that for $X=\BA^1$ and the constant group-scheme, both sides in \eqref{e:assertion on diag}
are \emph{constant}. This is due to the $\BA^1$-equivariance structure on all the objects involved.
Hence, the map \eqref{e:assertion on diag} is an isomorphism in this case if and only if the map
\begin{equation} \label{e:pointwise A1}
(\CB_{\on{red}})_{\{x\}}\to (\BD_{\Ran}(\CA_{\on{red}}))_{\{x\}}
\end{equation} 
is an isomorphism for some/any $x\in \BA^1$.

\medskip

By the same logic as in \secref{sss:deductions}, the assertion that \eqref{e:pointwise A1} if and only if
it is an isomorphism for some/any curve $X$ and some $x\in X$.

\medskip

Hence, we obtain that \thmref{t:local duality} follows from the next assertion:

\begin{thm} \label{t:pointwise duality}
The map 
$$(\CB_{\on{red}})_{\{x\}}\to (\BD_{\Ran}(\CA_{\on{red}}))_{\{x\}}$$ 
of \eqref{e:pointwise A1}
is an isomorphism for some/any $X$ and some/any $x\in X$ and the constant group-scheme $G$. 
\end{thm} 

We shall refer to \thmref{t:pointwise duality} as the \emph{pointwise duality} statement. 

\section{First proof of the pointwise duality statement}  \label{s:dagger}

The goal of this section is to prove \thmref{t:pointwise duality}, and thereby finish the proof
of \thmref{t:product formula}. A prerequisite for the present section is \secref{s:expl} and the
notion for a map between lax prestacks to be \emph{universally homologically contractible} from \secref{ss:uhc lax}.

\medskip

We let $X$ be a smooth curve and $G$ a constant reductive group scheme over $X$. 

\medskip

We will essentially reproduce the proof of \cite[Theorem 7.2.10]{Main Text} in a simplified context 
of the constant group-scheme. 

\ssec{Local non-abelian Poincar\'e duality}

\sssec{}

Recall the lax prestack 
$(\Ran_{\on{untl,aug}})_{x\notin}$, see \secref{sss:not cont x}. Denote
$$\Gr_{(\Ran_{\on{untl,aug}})_{x\notin}}:=(\Ran_{\on{untl,aug}})_{x\notin}\underset{\Ran_{\on{untl,aug}}}\times \Gr_{\Ran_{\on{untl,aug}}}.$$

\medskip

Explicitly, the above lax prestack attaches to a test scheme $S$ the following category. Its objects are triples:

\begin{itemize}

\item $K\subseteq I\subset \Maps(S,X)$, where the images of the maps $S\to X$ corresponding to elements of $K$ are disjoint from $x$;

\item a $G$-bundle $\CP_G$ on $S\times X$;

\item a trivialization $\gamma$ of $\CP_G|_{S\times X-\on{Graph}_I}$.

\end{itemize}

Given two such objects 
$$(K^1\subseteq I^1,\CP^1_G,\gamma^1) \text{ and } (K^2\subseteq I^2,\CP^2_G,\gamma^2),$$
a morphism between them is an inclusion $K_1\subseteq K_2$ and $I_1\subseteq I_2$, and an isomorphism
$$\CP^1_G|_{S\times X-\on{Graph}_{K_2}}\simeq \CP^2_G|_{S\times X-\on{Graph}_{K_2}},$$
which is compatible with the trivializations of 
$\CP^1_G|_{S\times X-\on{Graph}_{I_2}}$ and $\CP^2_G|_{S\times X-\on{Graph}_{I_2}}$, given by 
$\gamma^1|_{S\times X-\on{Graph}_{I_2}}$ and 
$\gamma^2$, respectively.

\sssec{}

Note that by taking the fiber at $x\in X$, we obtain a map
\begin{equation} \label{e:from dagger to BG}
\Gr_{(\Ran_{\on{untl,aug}})_{x\notin}}\to BG_x.
\end{equation}

We will prove:

\begin{thm}   \label{t:local non-ab Poinc}
The map \eqref{e:from dagger to BG} is a universally homologically contractible\footnote{See \secref{ss:uhc prestack} for 
what it means to be universally homologically contractible for a map from a  lax prestack to a prestack.}. In particular, 
it induces an isomorphism on homology. 
\end{thm} 

One can view \thmref{t:local non-ab Poinc} as a local version of non-abelian Poincar\'e duality.

\medskip

We will now show how \thmref{t:local non-ab Poinc} implies  \thmref{t:pointwise duality}. 

\sssec{Proof of \thmref{t:pointwise duality}} 

Recall that in \thmref{t:dagger}, we constructed a canonical isomorphism
$$((\CA_{\on{red}}^\dagger)_x)^\vee\simeq (\BD_{\Ran}(\CA_{\on{red}}))_{\{x\}}.$$

Consider the resulting pairing
$$(\CA_{\on{red}}^\dagger)_x \otimes (\BD_{\Ran}(\CA_{\on{red}}))_{\{x\}}\to k.$$

Note that by construction
$$(\CA_{\on{red}}^\dagger)_x\simeq \on{C}^{\on{red}}_*(\Gr_{(\Ran_{\on{untl,aug}})_{x\notin}}).$$

\medskip

By unwinding the construction in \thmref{t:pointwise duality}, we obtain that the following diagram commutes
$$
\CD
(\CA_{\on{red}}^\dagger)_x \otimes (\BD_{\Ran}(\CA_{\on{red}}))_{\{x\}}    @>>>  k  \\
@A{\sim}AA     @AAA  \\
\on{C}^{\on{red}}_*(\Gr_{(\Ran_{\on{untl,aug}})_{x\notin}}) \otimes (\BD_{\Ran}(\CA_{\on{red}}))_{\{x\}} & & 
\on{C}^{\on{red}}_*(BG_x) \otimes \on{C}^*_{\on{red}}(BG_x)   \\
@AAA    @AA{\sim}A  \\
\on{C}^{\on{red}}_*(\Gr_{(\Ran_{\on{untl,aug}})_{x\notin}}) \otimes (\CB_{\on{red}})_{\{x\}}  @>>>
\on{C}^{\on{red}}_*(BG_x) \otimes (\CB_{\on{red}})_{\{x\}},
\endCD
$$
where the lower left vertical arrow is induced by the map of \thmref{t:pointwise duality}, and the bottom
horizontal arrow is induced by the map of \eqref{e:from dagger to BG}.

\medskip

Applying \thmref{t:dagger}, we deduce that the map
$$(\CB_{\on{red}})_{\{x\}}\to (\BD_{\Ran}(\CA_{\on{red}}))_{\{x\}}$$
is an isomorphism, as required. 

\qed 

\sssec{}

Thus, our remaining goal is to prove \thmref{t:local non-ab Poinc}.  Note that the above derivation of \thmref{t:pointwise duality} from 
\thmref{t:dagger} shows that these two theorems are logically equivalent.  In particular, since \thmref{t:pointwise duality} is a \emph{local}
statement (i.e., it makes sense for a not necessarily complete curve), we obtain that so is \thmref{t:dagger}.

\medskip

However, in order to prove \thmref{t:dagger}, we will use global methods. In particular, for the proof we will assume that $X$ is complete
(which we could have from the start). The proof of \thmref{t:dagger}
will amount to a combination of two results, namely, Theorems
\ref{t:germs to BG} and \ref{t:from dagger to germs}, both global in nature. It is reasonable to think that the two global 
phenomena used in the proof of \thmref{t:local non-ab Poinc} cancel each other out. 

\ssec{The prestack of germs of bundles}

\sssec{}

Let $\Bun_G^{\on{around}_x}$ denote the following prestack. For $S\in \Sch$, the groupoid of $S$-points of 
$\Bun_G^{\on{around}_x}$ has as objects $G$-bundles $\CP_G$ on $S\times X$, and as morphisms isomorphisms
between $G$-bundles defined on an open subset of $S\times X$ of the form $S\times X-\on{Graph}_K$ for \emph{some} finite
set $K\subset \Maps(S,X-x)$. 

\begin{rem}
The prestack $\Bun_G^{\on{around}_x}$ should be thought of as classifying germs of $G$-bundles on $X$ defined
in an (unspecified, but non-empty) neighborhood of $x$.
\end{rem}

\sssec{}

Restriction to $x\in X$ defines a map
\begin{equation} \label{e:from germs to fiber}
\Bun_G^{\on{around}_x}\to BG_x.
\end{equation} 

In \secref{ss:proof of contractibility} we will prove:

\begin{thm}   \label{t:germs to BG}
The map \eqref{e:from germs to fiber} is a universally homologically contractible. In particular, it 
induces an isomorphism on homology. 
\end{thm} 

We note that \thmref{t:germs to BG} is essentially the same as \cite[Proposition 7.3.17]{Main Text}.

\sssec{}

Note now that we also have a tautologically defined forgetful map
\begin{equation} \label{e:from dagger to germs}
\Gr_{(\Ran_{\on{untl,aug}})_{x\notin}}\to \Bun_G^{\on{around}_x}.
\end{equation}

In \secref{ss:proof from dagger to germs} we will prove:

\begin{thm} \label{t:from dagger to germs} 
The map \eqref{e:from dagger to germs} 
is universally homologically contractible. In particular, it 
induces an isomorphism on homology.
\end{thm} 

We note that \thmref{t:from dagger to germs} is essentially the same as \cite[Proposition 7.3.16]{Main Text}
(in the simplified context of a constant group-scheme). 

\sssec{}

Note that the map \eqref{e:from dagger to BG} is the composition of the maps 
\eqref{e:from dagger to germs} and \eqref{e:from germs to fiber}. So, \thmref{t:local non-ab Poinc}
follows from the combination of Theorems \ref{t:germs to BG} and \ref{t:from dagger to germs}. 

\ssec{Proof of \thmref{t:germs to BG}}  \label{ss:proof of contractibility}

\sssec{}

Since the map $\on{pt}\to BG_x$ is a smooth cover, it suffices to show that the prestack
$$\Bun_G^{\on{around}_x}\underset{BG_x}\times \on{pt}$$
is universally homologically contractible.

\medskip

Let $\bMaps((X;x),(G;1))^{\on{around}_x} $ be the group-prestack whose $S$-points are maps from $S\times X$ to $G$, defined on an open 
subset of $S\times X$ of the form $S\times X-\on{Graph}_I$ 
for \emph{some} finite set $I\subset \Maps(S,X-x)$ and equal to the constant map with value $1\in G$ when restricted
to $S\times x\subset S\times X$. 

\medskip

By \cite[Theorem 3.3.6]{Main Text} (applied in the case of a constant group-scheme), 
for any $G$-bundle $\CP_G$ on $S\times X$ equipped with a trivialization along $S\times x$, there exists an 
\'etale cover of $\wt{S}\to S$ and $K\subset \Maps(\wt{S},X-x)$, such that $\CP_G|_{\wt{S}\times X-\on{Graph}_K}$ can be trivialized in a 
way compatible with the given trivialization on $\wt{S}\times x$. 

\medskip

This implies that the prestack $\Bun_G^{\on{around}_x}\underset{BG_x}\times \on{pt}$ 
identifies with the \'etale sheafification of the prestack $$B(\bMaps((X;x),(G;1))^{\on{around}_x}).$$ 
Thus, it remains to show that the prestack 
$B(\bMaps((X;x),(G;1))^{\on{around}_x})$ is universally homologically contractible.

\medskip

It is easy to see that if $H$ is a group-prestack, which is universally homologically contractible (as a mere prestack), then so is $BH$. 

\medskip

Hence, we obtain that in order to prove \thmref{t:germs to BG}, it suffices to show that 
the prestack $\bMaps((X;x),(G;1))^{\on{around}_x}$ is universally homologically contractible. 

\sssec{}   \label{sss:parameters G}

Let $\bMaps((X;x),(G;1))^{\on{around}_x}_{\on{param}}$ denote the following lax prestack. For $S\in \Sch$, the category of 
$S$-points of $\bMaps((X;x),(G;1))^{\on{around}_x}_{\on{param}}$ has as objects pairs $K \subset \Maps(S,X-x)$ and a map
$$z:S\times X-\on{Graph}_K\to G,$$
which takes value $1\in G$ over $S\times x\subset S\times X$. 
Morphisms between $(K^1,z^1)$ and $(K^2,z^2)$ are inclusions $K^1\subseteq K^2$ such that 
$$z^1|_{S\times X-\on{Graph}_{K^2}}= z^2.$$

\medskip

We have the forgetful map
\begin{equation}  \label{e:forget parameters}
\bMaps((X;x),(G;1))^{\on{around}_x}_{\on{param}}\to \bMaps((X;x),(G;1))^{\on{around}_x} .
\end{equation}

We claim:

\begin{lem}  \label{l:forget param}
The map \eqref{e:forget parameters} is universally homologically contractible.
\end{lem} 

\begin{proof}

By \lemref{l:contr contr}, it is enough to show that for any $S$-point of $\bMaps((X;x),(G;1))^{\on{around}_x} $,
the category of its lifts to an $S$-point of $\bMaps((X;x),(G;1))^{\on{around}_x}_{\on{param}}$ is contractible.

\medskip

Let an $S$-point of $\bMaps((X;x),(G;1))^{\on{around}_x} $ be given by $(K,z:S\times X-\on{Graph}_K\to Z)$. The category
of its lifts to an $S$-point of $\bMaps((X;x),(G;1))^{\on{around}_x}_{\on{param}}$ is that of
$$(K',z':S\times X-\on{Graph}_{K'}\to Z,z|_{S\times X-\on{Graph}_{K\cup K'}}=z'|_{S\times X-\on{Graph}_{K\cup K'}}).$$

\medskip

Now, left cofinal in this category is the subcategory consisting of those $K'$ for which $K\subset K'$. Finally,
this subcategory is contractible since it has an initial object.

\end{proof}

\sssec{}

Hence, by \lemref{l:forget param}, it suffices to show that the lax prestack $\bMaps((X;x),(G;1))^{\on{around}_x}_{\on{param}}$
is universally homologically contractible. However, this is the assertion of \cite[Theorem 3.3.2]{Main Text} (applied in the case of
a constant group-scheme).    (Note that the theorem from \cite{Main Text} quoted above is a global result; this is what makes
\thmref{t:germs to BG} a global assertion.) 

\ssec{Proof of \thmref{t:from dagger to germs}}  \label{ss:proof from dagger to germs}

\sssec{}

Let $(\Bun_G^{\on{around}_x})_{\on{param}}$ be the lax prestack defined as follows. For $S\in \Sch$, the category of 
$S$-points of $(\Bun_G^{\on{around}_x})_{\on{param}}$ has as objects pairs $K \subset \Maps(S,X-x)$ and $G$-bundle
$\CP_G$ on $S\times X$.

\medskip

For two such objects $(K^1,\CP^1_G)$ and $(K^2,\CP^2_G)$, a datum of a morphism between them is an inclusion 
$K_1\subseteq K_2$ and an isomorphism $$\CP^1_G|_{S\times X-\on{Graph}_{K_2}}\simeq \CP^2_G|_{S\times X-\on{Graph}_{K_2}}.$$

The map \eqref{e:from dagger to germs} clearly factors as
\begin{equation} \label{e:from dagger to germs factor}
\Gr_{(\Ran_{\on{untl,aug}})_{x\notin}}\to (\Bun_G^{\on{around}_x})_{\on{param}} \to \Bun_G^{\on{around}_x}.
\end{equation}

We will prove that both arrows in \eqref{e:from dagger to germs factor} are universally homologically contractible.\footnote{Here we are
using the notion of being universally homologically contractible for a map between lax prestacks, see \secref{ss:uhc lax} for what this means.}
We note, however, that the fact
that the second arrow in \eqref{e:from dagger to germs factor} 
is universally homologically contractible is proved by repeating the proof of \lemref{l:forget param}.

\sssec{}

By \propref{p:Cart uhc}, in order to prove that
\begin{equation} \label{e:Gr to param}
\Gr_{(\Ran_{\on{untl,aug}})_{x\notin}}\to (\Bun_G^{\on{around}_x})_{\on{param}}
\end{equation} 
is universally homologically contractible, it suffices to show that for any $S\in \Sch$, the functor
$$\Gr_{(\Ran_{\on{untl,aug}})_{x\notin}}(S)\to (\Bun_G^{\on{around}_x})_{\on{param}}(S)$$
is a coCartesian fibration
and that for any given object of $(\Bun_G^{\on{around}_x})_{\on{param}}(S)$ the map 
$$S\underset{(\Bun_G^{\on{around}_x})_{\on{param}}}\times \Gr_{(\Ran_{\on{untl,aug}})_{x\notin}}\to S$$
is universally homologically contractible. 

\sssec{}

The coCartesian property is established as follows. For a map
$$(K^1,\CP^1_G)\to (K^2,\CP^2_G)$$
in  $(\Bun_G^{\on{around}_x})_{\on{param}}(S)$ and an object $\Gr_{(\Ran_{\on{untl,aug}})_{x\notin}}(S)$,
given by $I^1\supseteq K^1$ and a trivialization $\gamma^1$ of $\CP^1_G$ over $S\times X-\on{Graph}_{I^1}$, we produce a new
object of $(\Bun_G^{\on{around}_x})_{\on{param}}(S)$ by setting $I^2:=I^1\cup K^2$ with $\gamma^2$, given by restriction. 

\sssec{}

Let us now be given an object $(K,\CP_G)$ of  $(\Bun_G^{\on{around}_x})_{\on{param}}(S)$. The fiber of \eqref{e:Gr to param}
over it is the lax prestack over $S$ that assigns to $S'\to S$ the category of
$$\left((K'\subseteq I'\subset \Maps(S',X)),\gamma'\right),$$
where $K'$ is the image of $K$ in $\Maps(S',X)$, and 
where $\gamma'$ is a trivialization of $\CP_G|_{S'\times X}$ over $$S'\times X-\on{Graph}_{I'}.$$ 

\medskip

Consider now the lax prestack 
$$S \underset{\Bun_G} \times \Gr_{\Ran_{\on{untl}}}.$$
By definition, it assigns to $S'\to S$ the category of 
$$(I'\subset \Maps(S',X),\gamma'),$$
where $\gamma'$ is a trivialization of $\CP_G|_{S'\times X}$ over $S'\times X-\on{Graph}_{I'}$.

\medskip

We have a tautological functor
$$S\underset{(\Bun_G^{\on{around}_x})_{\on{param}}}\times \Gr_{(\Ran_{\on{untl,aug}})_{x\notin}}\to 
S \underset{\Bun_G} \times \Gr_{\Ran_{\on{untl}}},$$
which admits a left adjoint (take the union of $I'$ with $K'$). 

\medskip

Hence, by \lemref{l:rel hom type}, it is enough to show that the map 
$$S \underset{\Bun_G} \times \Gr_{\Ran_{\on{untl}}}\to S$$
is universally homologically contractible. 

\sssec{}

By \cite[Theorem 3.3.6]{Main Text} (applied in the case of a constant group-scheme),
we can assume that the given $G$-bundle on $S\times X$ admits a rational trivialization.
In this case, the fiber product $$S \underset{\Bun_G} \times \Gr_{\Ran_{\on{untl}}}$$ identifies with 
$$S\times \bMaps(X,G)_{\on{param}},$$
where $\bMaps(X,G)_{\on{param}}$ is the prestack that assigns to $S'\to S$ the groupoid of 
$$(I'\subset \Maps(S',X), \gamma':S'\times X-\on{Graph}_{I'}\to G).$$

\medskip

Thus, it remains to show that the prestack $\bMaps(X,G)_{\on{param}}$ is universally homologically contractible.
However, the latter is given by \cite[Theorem 3.3.2]{Main Text} (applied in the case of a constant group-scheme). 
(We note again that the theorem from \cite{Main Text} quoted above is a global result, making \thmref{t:from dagger to germs} 
into a global assertion). 

\section{Second proof of the pointwise duality statement}  \label{s:P1}
 
In this section we will give another proof of \thmref{t:pointwise duality}, also using global methods. One component of this proof relies
on the notion of \emph{universal homological left cofinality}, see \secref{ss:l cofinal}. 

\medskip

In this section we let $X$ be a proper curve with a marked point $x_\infty\in X$. We denote $X':=X-x_\infty$,
and we let $x_0$ be a marked point on $X'$.  In practice, we will take $X=\BP^1$ with points $\infty$ and $0$,
respectively.

\medskip

We let $G$ be a constant reductive group-scheme on $X$.

\ssec{Local non-abelian Poincar\'e duality, $\BP^1$-version}

In this subsection we will formulate a theorem, \thmref{t:P1 non-ab Poinc}, which would be a semi-global 
$\BP^1$-analog of \thmref{t:local non-ab Poinc}, and deduce from it \thmref{t:pointwise duality}.  

\sssec{}   \label{sss:shifted Gr}

Let $\unn_{x_\infty,\on{untl,aug}}$ denote the map
$$\Ran_{\on{untl,aug}}\to \Ran_{\on{untl,aug}},\quad (K\subseteq I)\mapsto ((K\cup x_\infty)\subseteq (I\cup x_\infty)).$$

Consider the map
$$\unn_{x_\infty,\on{untl,aug}}\circ \iota\circ \phi: \Ran\to \Ran_{\on{untl,aug}},\quad I\mapsto (\{x_\infty\}\subset (I\cup x_\infty)).$$

\medskip

Consider the fiber product
$$\Ran \underset{\Ran_{\on{untl,aug}}}\times \Gr_{\Ran_{\on{untl,aug}}}.$$

This is a prestack that assigns to a test scheme $S$ the category, whose objects are triples:

\begin{itemize}

\item $I\subset \Maps(S,X)$;

\item a $G$-bundle $\CP_G$ on $S\times X$;

\item a trivialization $\gamma$ of the restriction of $\CP_G$ to $S\times X-(\on{Graph}_I\cup (S\times x_\infty))$. 

\end{itemize}

Morphisms between $(I^1,\CP^1_G,\gamma^1)$ and $(I^2,\CP^2_G,\gamma^2)$ are non-empty only when $I^1=I^2$,
and in the latter case consist of isomorphisms
$$\CP^1_G|_{S\times (X-x_\infty)}\simeq \CP^2_G|_{S\times (X-x_\infty)},$$
compatible with the data of $\gamma^1$ and $\gamma^2$, respectively. 

\sssec{}

We have a naturally defined map
\begin{equation} \label{e:Gr to point}
\Ran \underset{\Ran_{\on{untl,aug}}}\times \Gr_{\Ran_{\on{untl,aug}}}\to  BG,
\end{equation}
given by restricting the bundle to $x_0\in X'$.

\medskip

We will prove:

\begin{thm} \label{t:P1 non-ab Poinc}
Let $(X,x_\infty,x_0)=(\BP^1,\infty,0)$.
Then the map \eqref{e:Gr to point} is universally homologically contractible. 
\end{thm}

We regard \thmref{t:P1 non-ab Poinc} as a semi-global version of \thmref{t:non-ab Poinc}. 

\medskip

We shall now show how \thmref{t:P1 non-ab Poinc} implies \thmref{t:pointwise duality}.

\sssec{}

Consider the object 
$$(\bj_{\Ran})_*(\CA_{\on{red}}')\in \Shv^!(\Ran(X)).$$

We have a tautologically defined map
\begin{equation}  \label{e:contr Ran}
(\BD_{\Ran'}(\CA_{\on{red}}'))_{\{x_0\}}\to \left(\on{C}_c^*(\Ran, (\bj_{\Ran})_*(\CA_{\on{red}}'))\right)^\vee.
\end{equation}

We claim:

\begin{lem}  \label{l:contr Ran}
Assume that $(X,x_\infty,x_0)=(\BP^1,\infty,0)$. Then the map \eqref{e:contr Ran} is an isomorphism.
\end{lem} 

\begin{proof}

We will prove the lemma for $\CA_{\on{red}}'$ replaced by any $\CF\in \Shv^!(\Ran')$, which is equivariant with respect to the action of
$\BG_m$ by dilations.

\medskip

By \corref{c:Verdier on Ran}, we can assume that
$\CF=(\on{ins}'_\CI)_!(\CF_\CI)$ for some finite set $\CI$ and $\CF_\CI\in \Shv(X'{}^\CI)$ with $\CF_\CI$ being $\BG_m$-equivariant. 

\medskip

Then it suffices to show that the map
\begin{equation}  \label{e:contr I}
(\BD_{X'{}^\CI}(\CF_\CI))_{x_0^\CI}\to \left(\on{C}_c^*(X^\CI,(\bj_{X^\CI})_*(\CF_\CI))\right)
\end{equation}
is an isomorphism, where $x_0^\CI$ denotes the corresponding point of $X^\CI$. 

\medskip

Let $(\CF_\CI)^*_{x_0^\CI}$ denote the *-fiber of $\CF_\CI$ at the point $x_0^\CI\in X'{}^\CI$. Then the map
\eqref{e:contr I} is the dual of the canonical map
$$\on{C}^*(X'{}^\CI,\CF_\CI)\to (\CF_\CI)^*_{x_0^\CI}.$$

Now, the assertion follows from the contraction principle for $\BG_m$-equivariant sheaves on $(\BA^1)^\CI$, see
\cite[Theorem C.5.3]{DrGa}. 

\end{proof} 

\sssec{}

Recall the map $\unn_{x_\infty,\on{untl,aug}}$ introduced above. It follows from \thmref{t:main} that we have a canonical isomorphism
$$\CA_{\on{red}}'\simeq \on{TakeOut}\circ \unn_{x_\infty,\on{untl,aug}}^!(\CA_{\on{untl,aug}})|_{\Ran'}.$$

From here we obtain a map
\begin{equation} \label{e:takeout point}
\on{TakeOut}\circ \unn_{x_\infty,\on{untl,aug}}^!(\CA_{\on{untl,aug}})\to (\bj_{\Ran})_*(\CA_{\on{red}}').
\end{equation}

We claim:

\begin{lem}  \label{l:takeout point}
The map \eqref{e:takeout point} is an isomorphism.
\end{lem}

\begin{proof}

Follows easily from \thmref{t:main}: it is true for $\CA_{\on{red}}$ replaced by an arbitrary 
$\CF\in \Shv^!(\Ran)$, $\CA'_{\on{red}}$ by $\CF|_{\Ran'}$ and $\CA_{\on{untl,aug}}$ by 
$\wt\CF:=\on{AddUnit}_{\on{aug}}(\CF)\in \Shv^!(\Ran_{\on{untl,aug}})$. 

\end{proof} 

\sssec{}

By construction, the following diagram commutes:

$$
\CD
(\CB'_{\on{red}})_{\{x_0\}}    @>{\text{\eqref{e:pointwise A1}}}>>   \left(\BD_{\Ran'}(\CA_{\on{red}}')\right)_{\{x_0\}}  \\
@V{=}VV  @VV{\text{\eqref{e:contr Ran}}}V    \\
\on{C}^*_{\on{red}}(BG)  & &   \left(\on{C}_c^*(\Ran, (\bj_{\Ran})_*(\CA_{\on{red}}'))\right)^\vee   \\
@V{\text{\eqref{e:Gr to point}}}VV     @VV{\text{\eqref{e:takeout point}}}V    \\
\on{C}^*_{\on{red}}\left(\Ran \underset{\Ran_{\on{untl,aug}}}\times \Gr_{\Ran_{\on{untl,aug}}}\right)    & &  
\on{C}_c^*(\Ran, \on{TakeOut}\circ \unn_{x_\infty,\on{untl,aug}}^!(\CA_{\on{untl,aug}}))^\vee   \\
& &      @VV{\text{\eqref{e:nat trans 3}}}V  \\
@V{\sim}VV  
 \on{C}_c^*(\Ran, \on{OblvUnit}\circ \on{OblvAug}\circ \unn_{x_\infty,\on{untl,aug}}^!(\CA_{\on{untl,aug}}))^\vee   \\
& &    @VV{\sim}V    \\
\on{C}^{\on{red}}_*\left(\Ran \underset{\Ran_{\on{untl,aug}}}\times \Gr_{\Ran_{\on{untl,aug}}}\right)^\vee  
@>{\sim}>>
\left(\on{C}^*_c(\Ran,\phi^!\circ  \iota^! \circ \unn_{x_\infty}^!(\CA))\right)^\vee.  
\endCD
$$

We need to show that the top horizontal arrow in this diagram is an isomorphism. For this we 
can take $(X,x_\infty,x_0)=(\BP^1,\infty,0)$. We claim that in this case all the other arrows in
this diagram are isomorphisms. 

\medskip

Indeed, for $(X,x_\infty,x_0)=(\BP^1,\infty,0)$, Theorem \ref{t:P1 non-ab Poinc} says
that the second from the top left vertical arrow is an isomorphism, and \lemref{l:contr Ran}
says that the top right vertical arrow is an isomorphism. 

\medskip

The other arrows are isomorphism for any $(X,x_\infty,x_0)$, which follows from 
\lemref{l:takeout point} and \corref{c:main}.

\medskip

This proves \thmref{t:pointwise duality}.

\ssec{The stack of bundles on the punctured curve}

In this subsection we shall reduce \thmref{t:P1 non-ab Poinc} to a combination of another two theorems. These theorems are parallel to 
Theorems \ref{t:germs to BG} and \ref{t:from dagger to germs}, respectively. 

\sssec{}

Let $\Bun_G'$ be the prestack that assigns to $S\in \Sch$ the groupoid whose objects are $G$-bundles on
$S\times X$, and whose morphisms are isomorphisms of $G$-bundles on $S\times X'$.  (In other words, this is the subgroupoid
of $G$-bundles on $S\times X'$ whose objects are those $G$-bundles that can be extended to $S\times X$.) 

\medskip

Restriction to $x_0\in X'$ defines a map 
\begin{equation} \label{e:open to point}
\Bun_G'\to BG.
\end{equation}

We will prove:

\begin{thm} \label{t:open to point}
Assume that $(X,x_\infty,x_0)=(\BP^1,\infty,0)$. Then the map \eqref{e:open to point}
is universally homologically contractible. In particular, it induces an isomorphism on
homology.
\end{thm}

\sssec{}

Recall the lax prestack 
$$\Ran \underset{\Ran_{\on{untl,aug}}}\times \Gr_{\Ran_{\on{untl,aug}}},$$
see \secref{sss:shifted Gr} (in particular, its definition involved the distinguished point $x_\infty$). 

\medskip

We have a naturally defined map
\begin{equation} \label{e:Gr to open}
\Ran \underset{\Ran_{\on{untl,aug}}}\times \Gr_{\Ran_{\on{untl,aug}}}\to \Bun'_G
\end{equation}

We will prove:

\begin{thm} \label{t:Gr to open}
The map \eqref{e:Gr to open} is universally homologically contractible. In particular, it 
induces an isomorphism on homology.
\end{thm}

Note that \thmref{t:Gr to open} is stated for any $(X,x_\infty)$, i.e., it is not linked to the $\BP^1$ situation. 

\sssec{}

Clearly, the composition of the maps \eqref{e:Gr to open} and \eqref{e:open to point} equals the map \eqref{e:Gr to point}.
Hence, Theorems \ref{t:open to point} and \ref{t:Gr to open} imply \thmref{t:P1 non-ab Poinc}.

\ssec{Proof of \thmref{t:open to point}}

\sssec{}

Let $\Bun_{G,\on{level}_{x_0}}$ be the moduli stack of $G$-bundles with structure of level $1$ at $x_0$ (i.e., $G$-bundles
equipped with a trivialization of the fiber at $x_0$), and let $\Bun'_{G,\on{level}_{x_0}}$ be the corresponding variant, i.e.,
$$\Bun'_{G,\on{level}_{x_0}}:=\Bun'_G\underset{BG}\times \on{pt}.$$

It suffices to show that $\Bun'_{G,\on{level}_{x_0}}$ is universally homologically contractible.

\medskip

Consider the simplicial object $(\Bun_{G,\on{level}_{x_0}}/\Bun'_{G,\on{level}_{x_0}})^\bullet$
of $\on{PreStk}$ equal to the \v{C}ech nerve of the map 
$$\Bun_{G,\on{level}_{x_0}}\to \Bun'_{G,\on{level}_{x_0}}.$$

It suffices to show that the map
\begin{equation} \label{e:modulo infty}
|\on{C}_*(\Bun_{G,\on{level}_{x_0}}/\Bun'_{G,\on{level}_{x_0}})^\bullet|\to \Lambda
\end{equation}
is an isomorphism. 

\sssec{}

Consider the simplicial object $(\Gr_{x_\infty})^\bullet$,of $\on{PreStk}$ given by taking the powers of the
indscheme $\Gr_{x_\infty}$.  Clearly, 
$$|\on{C}_*((\Gr_{x_\infty})^\bullet)|\to \Lambda$$
is an isomorphism.

\medskip

Thus, to prove that \eqref{e:modulo infty}, it suffices to construct a simplicial map
\begin{equation} \label{e:Gr infty}
(\Gr_{x_\infty})^\bullet\to (\Bun_{G,\on{level}_{x_0}}/\Bun'_{G,\on{level}_{x_0}})^\bullet,
\end{equation}
so that it induces a term-wise isomorphism on homology.

\sssec{}

The map \eqref{e:Gr infty} is constructed as follows: it sends an $n$-tuple of $G$-bundles
$(\CP^1,...,\CP^n)$ on $S\times X$ each equipped with a trivialization $\gamma^i$ over $S\times X'$ to
the same $n$-tuple, where the trivializations of each $\CP^i$ at $S\times x_0$ is given by 
$\gamma^i|_{S\times x_0}$, and where the identifications 
$$\CP^i|_{S\times X'}\simeq \CP^j|_{S\times X'}$$ are given by $\gamma_j\circ \gamma_i^{-1}$. 

\medskip

Now, each of the above maps 
$$(\Gr_{x_\infty})^n\to (\Bun_{G,\on{level}_{x_0}}/\Bun'_{G,\on{level}_{x_0}})^n,$$
is a fibration, locally trivial in the Zariski topology, with the typical fiber being
$$\bMaps((X',x_0),(G,1)).$$

Thus, it remains to show that the trace map
$$\on{C}_*(\bMaps((X',x_0),(G,1)))\to \Lambda$$
is an isomorphism.

\sssec{}

We shall now use the fact that $(X',x_0)$ equals $(\BA^1,0)$. In fact, we claim that for any affine scheme $Z$ with a point
$z$, the map
$$\on{C}_*(\bMaps((\BA^1,0),(Z,z)))\to \Lambda$$
is an isomorphism.

\medskip

Indeed, let us embed $Z$ into an affine space $\BA^n$, so that $z=0\in \BA^n$. We can represent 
$\bMaps((\BA^1,0),(Z,z))$ as the colimit of affine schemes each equal to
$$\bMaps((\BA^1,0),(Z,z)) \cap \bMaps((\BA^1,0),(\BA^n,0))^{\leq d},$$
where $\bMaps((\BA^1,0),(\BA^n,0))^{\leq d}$ is the vector space of polynomial maps of degree $\leq d$.

\medskip

Hence, it suffices to show that each $\bMaps((\BA^1,0),(Z,z)) \cap \bMaps((\BA^1,0),(\BA^n,0))^{\leq d}$
has a trivial homology.

\medskip

Now, the $\BG_m$-action on $\BA^1$ defines an action of $\BG_m$ on $\bMaps((\BA^1,0),(\BA^n,0))$,
which preserves each $\bMaps((\BA^1,0),(Z,z)) \cap \bMaps((\BA^1,0),(\BA^n,0))^{\leq d}$, and contracts
it to the constant map $\BA^1\to 0\in Z$. 

\ssec{Proof of \thmref{t:Gr to open}}

We emphasize again that \thmref{t:Gr to open} is stated for any $(X,x_\infty)$, i.e., it is not linked to the $(\BP^1,\infty)$
situation. 

\sssec{}

Since the map 
$$\Bun_G\to \Bun'_G$$ is surjective on isomorphism classes of $S$-points, it suffices to show that the base-changed map
$$\Bun_G \underset{\Bun'_G}\times (\Ran \underset{\Ran_{\on{untl,aug}}}\times \Gr_{\Ran_{\on{untl,aug}}})\to \Bun_G$$
is universally homologically contractible.

\medskip

Note that
$$\Bun_G \underset{\Bun'_G}\times (\Ran \underset{\Ran_{\on{untl,aug}}}\times \Gr_{\Ran_{\on{untl,aug}}})
\simeq \Gr_{\Ran}\underset{\Ran}\times \Ran,$$
where $\Ran\to \Ran$ is the map 
$$\unn_{x_\infty},\quad I\mapsto I\cup x_\infty.$$

\medskip

Thus, it remains to show that the composed map
\begin{equation} \label{e:add point uhc}
\Gr_{\Ran}\underset{\Ran}\times \Ran
\to \Gr_{\Ran}\to \Bun_G
\end{equation} 
is universally homologically contractible. 

\medskip
 
We will deduce this from the fact that the usual uniformization map 
\begin{equation} \label{e:unif}
\Gr_{\Ran}\to \Bun_G
\end{equation} 
is universally homologically contractible (the latter due to \thmref{t:non-ab Poinc}).

\sssec{}

Note that the map $\Gr_{\Ran}  \to  \Bun_G$ factors as
$$\Gr_{\Ran}  \to \Gr_{\Ran_{\on{untl}}} \to \Bun_G,$$
where
$$\Gr_{\Ran_{\on{untl}}}:=\Ran_{\on{untl}}\underset{\Ran_{\on{untl,aug}}}\times \Gr_{\Ran_{\on{untl,aug}}}$$
for $\Gr_{\Ran_{\on{untl,aug}}}$ introduced in \secref{sss:unital Gr}. I.e., we have a commutative diagram
\begin{equation} \label{e:Gr add point}
\CD
\Gr_{\Ran}\underset{\Ran}\times \Ran  @>>>  \Gr_{\Ran}  @>>>   \Bun_G    \\
@VVV    @VVV  @VV{\on{id}_S}V   \\
\Gr_{\Ran_{\on{untl}}}\underset{\Ran_{\on{untl}}}\times \Ran_{\on{untl}}   
@>>>  \Gr_{\Ran_{\on{untl}}}  @>>>   \Bun_G,
\endCD
\end{equation}
where $\Ran_{\on{untl}}\to \Ran_{\on{untl}}$ is the map
$$\unn_{x_\infty,\on{untl}},\quad I\mapsto I\cup x_\infty.$$

\sssec{}

We claim that the maps
$$\Gr_{\Ran}\to \Gr_{\Ran_{\on{untl}}}, \,\, \Gr_{\Ran_{\on{untl}}}\underset{\Ran_{\on{untl}}}\times \Ran_{\on{untl}}  \to \Gr_{\Ran_{\on{untl}}}
\text{ and } 
\Gr_{\Ran}\underset{\Ran}\times \Ran\to \Gr_{\Ran_{\on{untl}}}\underset{\Ran_{\on{untl}}}\times \Ran_{\on{untl}},$$
appearing in the diagram \eqref{e:Gr add point}, are universally homologically left cofinal.

\medskip

Indeed, consider the Cartesian square
$$
\CD
\Gr_{\Ran} @>>>  \Gr_{\Ran_{\on{untl}}}   \\
@VVV    @VVV   \\
\Ran @>{\phi}>> \Ran_{\on{untl}},
\endCD
$$
and the assertion follows from the combination of the following facts: 

\smallskip

\noindent(i) \thmref{t:Ran left cofinal} and \propref{p:ins point};

\smallskip

\noindent(ii) The map $\Gr_{\Ran_{\on{untl}}}\to \Ran_{\on{untl}}$
is a value-wise coCartesian fibration;

\smallskip

\noindent(iii) \propref{p:uhlc}. 

\sssec{}

Since $\Bun_G$ is a prestack (as opposed to a lax prestack), since the map
$$\Gr_{\Ran}\underset{\Ran}\times \Ran\to \Gr_{\Ran_{\on{untl}}}\underset{\Ran_{\on{untl}}}\times \Ran_{\on{untl}},$$
is universally homologically left cofinal, by \propref{p:l cofinal ff}, it suffices to show that the composed map 
$$\Gr_{\Ran_{\on{untl}}}\underset{\Ran_{\on{untl}}}\times \Ran_{\on{untl}}  \to   \Gr_{\Ran_{\on{untl}}} \to   \Bun_G$$
is universally homologically contractible. 

\medskip

Since 
$$\Gr_{\Ran_{\on{untl}}}\underset{\Ran_{\on{untl}}}\times \Ran_{\on{untl}}  \to   \Gr_{\Ran_{\on{untl}}}$$
is universally homologically left cofinal, it suffices to show that the map 
$$\Gr_{\Ran_{\on{untl}}}\to \Bun_G$$
is universally homologically contractible. 

\medskip

However, the latter follows from the fact that $\Gr_{\Ran}\to \Gr_{\Ran_{\on{untl}}}$
is universally homologically left cofinal and the 
fact that the map $\Gr_{\Ran}\to \Bun_G$ is universally homologically contractible. 

\newpage 

\centerline{\bf Part VI: The Atiyah-Bott formula and the numerical product formula}

\bigskip

\section{The Atiyah-Bott formula}   \label{s:AB}

In this section we let $X$ be a smooth connected complete curve. Let $G$ be a group-scheme over $X$ satisfying the assumptions 
of \secref{sss:good G}.  We will assume that our sheaf theory is such that the ring of coefficients $\Lambda$ is a field of characteristic $0$. 

\medskip

In this section we will apply \thmref{t:product formula} to deduce a much more explicit expression for the cohomology 
of $\Bun_G$, the Atiyah-Bott formula.

\medskip

The prerequisites for this section are Sects. \ref{ss:product formula} and \ref{ss:B factorizes} and \cite[Theorem 5.6.4]{Main Text}.

\ssec{Statement of the Atiyah-Bott formula}

\sssec{}

Let $G_0$ be the split form of $G$. Consider the commutative algebra in $\Lambda\mod$:
$$\on{C}^*(BG_0)=:B_0.$$

Consider the corresponding $\BZ$-graded (classical) commutative algebra $\on{H}^*(B_0)$. 
Let $M_0$ denoted the $\BZ$-graded $\Lambda$-vector space $\fm/\fm^2$, where $\fm\subset \on{H}^*(B_0)$ 
is the augmentation ideal of $B_0$ (the set of elements of strictly positive degree). 

\medskip

We shall regard $M_0$ as a semi-simple object of $\Lambda\mod$ (turning the $\BZ$-grading into a cohomological one). 

\sssec{}  \label{sss:splitting}

The (graded) vector space $M_0$ is the sum 
$$\underset{e}\oplus\, \Lambda[-2e],$$
where the $e$'s are the exponents of $G_0$.  Note that because $G_0$ was assumed semi-simple, we have $e\geq 2$. 
When $k=\ol\BF_q$ (and our sheaf theory is that of $\BQ_\ell$-adic sheaves), $M_0$ carries a canonical action of the 
geometric Frobenius, where the action on the $e$-th direct summand is given by $q^{e}$. 

\medskip

It is well-known that $\on{H}^*(B_0)$ is isomorphic to $\on{Sym}_\Lambda(M_0)$. When $k=\ol\BF_q$, this is isomorphism
can be chosen compatible with an action of the Frobenius. 

\medskip

Lifting the generators we can therefore choose an isomorphism of commutative algebras in $\Lambda\mod$
$$\on{Sym}_\Lambda(M_0)\simeq \on{C}^*(B_0).$$
For $k=\ol\BF_q$ the latter can also be chosen compatible with an action of the Frobenius. 

\sssec{}

The assignment $G_0\mapsto M_0$ is functorial with respect to automorphisms of $G_0$. 
Hence over the open curve $X'$ (see \secref{sss:good G} for the notation) we have a well-defined lisse sheaf $M\in \Shv^!(X')$,
whose !-fiber $M_x$ at $x\in X$ is identified with $M_0$ for every choice of identification of $G_x$ with $G_0$. 

\medskip

The Atiyah-Bott formula reads:

\begin{thm} \label{t:AB}  \hfill

\smallskip

\noindent{\em(a)} There exists a (non-canonical) isomorphism
$$\on{C}^*(\Bun_G)\simeq \on{Sym}_\Lambda(\on{C}^*(X',M)).$$

\smallskip

\noindent{\em(b)} When $k=\ol\BF_q$ and $X$ and $G$ are defined over $\BF_q$, the above isomorphism
can be chosen compatible with the Frobenius-equivariance structure. 

\end{thm} 

\ssec{Proof of the Atiyah-Bott formula}

\sssec{}

The starting point for the proof is the isomorphism
$$\on{C}^*_{\on{red}}(\Bun_G)\simeq \on{C}^*_c(\Ran,\CB),$$
provided by \thmref{t:product formula}. So, from now on our goal will be to construct
a (non-canonical) isomorphism 
\begin{equation} \label{e:calc prod}
\on{C}^*_c(\Ran,\CB)\simeq \on{Sym}^+_\Lambda(\on{C}^*(X',M)),
\end{equation}
where $\on{Sym}^+$ denotes the augmentation ideal of $\on{Sym}$.

\sssec{}

Recall the object $\CB_{\on{red}}\in \Shv^!(\Ran)$ (see \secref{sss:B red}), and recall also that by \corref{c:main} we have a canonical isomorphism 
\begin{equation} \label{e:red B vs B}
\on{C}^*_c(\Ran,\CB)\simeq \on{C}^*_c(\Ran,\CB_{\on{red}}).
\end{equation}

\medskip

Recall also that by \secref{sss:B factorizes}, $\CB_{\on{red}}$ has a structure of \emph{commutative algebra} in the symmetric monoidal category
$\Shv^!(\Ran)$ (endowed with the convolution product). We claim:

\begin{prop}  \label{p:identify alg}
There is a (non-canonical) isomorphism between 
$\CB_{\on{red}}$ and the (non-unital) free commutative algebra on the object $(\on{ins}_X)_!(\bj_*(M))$ in 
the symmetric monoidal category $\Shv^!(\Ran)$.
\end{prop} 

In the above proposition, $\on{ins}_X$ denotes the map $X\to \Ran$ and $\bj$ is the open embedding $X'\hookrightarrow X$. 

\sssec{}

Before proving \propref{p:identify alg}, let us show that it implies the existence of a non-canonical isomorphism 
$$\on{C}^*_c(\Ran,\CB_{\on{red}})\simeq \on{Sym}^+_\Lambda(\on{C}^*(X',M))$$
(while the latter implies \eqref{e:calc prod} by \eqref{e:red B vs B}).

\medskip

Indeed, the functor
$$\CF\mapsto \on{C}^*_c(\Ran,\CF), \quad \Shv^!(\Ran)\to \Lambda\mod$$
has a natural symmetric monoidal structure.  Hence, 
$$\on{C}^*_c\left(\Ran,\on{FreeCom}_{\Shv^!(\Ran)}((\on{ins}_X)_!(\bj_*(M)))\right)\simeq 
\on{FreeCom}_{\Lambda\mod}\left(\on{C}^*_c(\Ran,(\on{ins}_X)_!(\bj_*(M)))\right),$$
while
$$\on{C}^*_c(\Ran,(\on{ins}_X)_!(\bj_*(M)))\simeq \on{C}^*_c(X,\bj_*(M))\simeq \on{C}^*(X',M)$$
and
$$\on{FreeCom}_{\Lambda\mod}(-)\simeq \on{Sym}^+_\Lambda(-).$$

\ssec{Proof of \propref{p:identify alg}}

\sssec{}

First, it is easy to see that for any $\CF\in \Shv(X)$, the commutative algebra 
$$\on{FreeCom}_{\Shv^!(\Ran)}((\on{ins}_X)_!(\CF))$$
in $\Shv^!(\Ran)$ is actually a \emph{commutative factorization algebra} (see \cite[Lemma 5.6.15]{Main Text}).

\medskip

Second, by \corref{c:B factor}, $\CB_{\on{red}}$ is also a commutative factorization algebra. 

\medskip

Therefore, by \cite[Theorem 5.6.4]{Main Text}, the existence of an isomorphism stated in the proposition is equivalent to the existence 
of an isomorphism in the category of commutative algebras in $\Shv(X)$ (with respect to the pointwise tensor product):
$$(\CB_{\on{red}})_X\simeq \on{FreeCom}_{\Shv(X)}(\bj_*(M)).$$

\sssec{}

Note that by the assumption on $G$, we have 
$$(\CB_{\on{red}})_X\simeq \bj_*((\CB'_{\on{red}})_{X'}).$$

Hence, it is enough to establish the existence of an isomorphism in the category of commutative algebras in $\Shv(X')$:

\begin{equation} \label{e:B as sym}
(\CB'_{\on{red}})_{X'} \simeq \on{FreeCom}_{\Shv(X')}(M).
\end{equation} 

\sssec{}

Let $\wt{X}\to X'$ be an \'etale Galois cover such that $G|_{\wt{X}}$ is the constant group-scheme with fiber $G_0$. 
Let $\Gamma$ denote the Galois group of $\wt{X}$ over $X'$. 

\medskip

We have:
\begin{equation} \label{e:B upstairs}
(\CB'_{\on{red}})_{X'}|_{\wt{X}}\simeq (B_0)_{\on{red}}\otimes \omega_{\wt{X}} \text{ and }
M|_{\wt{X}} \simeq M_0\otimes \omega_{\wt{X}}.
\end{equation} 

Thus, the datum of an isomorphism \eqref{e:B as sym} is equivalent to that of a $\Gamma$-equivariant isomorphism
$$(B_0)_{\on{red}}\otimes \omega_{\wt{X}} \simeq \on{Sym}^+_\Lambda(M_0)\otimes \omega_{\wt{X}},$$
where the $\Gamma$-equivariant structure on each side comes from the isomorphisms in \eqref{e:B upstairs}.

\sssec{}

The datum of a $\Gamma$-equivariant map of commutative algebras
\begin{equation} \label{e:B upstairs bis}
\on{Sym}^+_\Lambda(M_0)\otimes \omega_{\wt{X}}\to (B_0)_{\on{red}}\otimes \omega_{\wt{X}}
\end{equation}
is equivalent to that of a $\Gamma$-equivariant map in $\Shv(\wt{X})$:
\begin{equation} \label{e:from M}
M_0\otimes \Lambda_{\wt{X}}\to (B_0)_{\on{red}}\otimes \Lambda_{\wt{X}}
\end{equation}

We will find a map in \eqref{e:from M} so that at the level of cohomology sheaves, it induces a map
that fits into a commutative diagram
\begin{equation} \label{e:com diag M}
\CD
\on{H}^*(M_0) \otimes \Lambda_{\wt{X}} @>>> \on{H}^*((B_0)_{\on{red}})  \otimes \Lambda_{\wt{X}}   \\
@V{\sim}VV     @VV{\sim}V    \\ 
M_0  \otimes \Lambda_{\wt{X}}    @<<<  \fm \otimes \Lambda_{\wt{X}},  
\endCD
\end{equation} 
where we recall that $M_0$ was defined as $\fm/\fm^2$, where $\fm$ is the augmentation ideal in $\on{H}^*(B_0)$.

\medskip

The map in \eqref{e:B upstairs bis}, corresponding to a map in \eqref{e:from M} with the property of making
\eqref{e:com diag M} commutative, is automatically an isomorphism. 

\sssec{}

Consider the space $\CV$ of \emph{all}, i.e., not necessarily $\Gamma$-equivariant, (resp., in the situation (b) of
\thmref{t:AB}, Frobenius-equivariant) maps in \eqref{e:from M} that make the diagram \eqref{e:com diag M} 
commute. 

\medskip

The space $\CV$ is non-empty since $(B_0)_{\on{red}}$ is non-canonically (resp., in the situation (b) of \thmref{t:AB}, Frobenius-equivariantly) 
isomorphic to $\on{Sym}^+_\Lambda(M_0)$, see \secref{sss:splitting}.  

\medskip

The $\Gamma$-equivariant structure on both sides of \eqref{e:from M} defines an action of $\Gamma$
on $\CV$, and our task it show that $\CV$ contains a $\Gamma$-fixed point. 

\sssec{}

Note, however, that $\CV$ is naturally a torsor for a canonically defined object $V\in \Lambda\mod^{\leq 0}$. Furthermore,
$V$ is also endowed with an action of $\Gamma$, compatible with the action of $V$ on $\CV$. 

\medskip

Now, the required assertion follows from the fact that $H^1(\Gamma,V)=0$ (the latter because $\Gamma$ is a finite group and
$\Lambda$ is a field of characteristic $0$). 

\section{The numerical product formula}  \label{s:num}

The only prerequisite for this section is the statement of \thmref{t:AB}. We will give a derivation of the numerical product formula
\eqref{e:num prod prev} from the cohomological product formula \eqref{e:product formula prev},
slightly different from the one in \cite[Sect. 6]{Main Text}.

\ssec{The setting}

In this section we take the ground field $k$ to be $\ol\BF_q$, but we assume that $X$ and $G$ come by extension of scalars from
$X^0$ and $G^0$, respectively, defined over $\BF_q$. We take our sheaf theory to be that of $\BQ_\ell$-adic sheaves. 
We fix an embedding 
$$\tau:\BQ_\ell\hookrightarrow \BC.$$

\sssec{}

On the one hand, we consider the cohomology
$$\on{H}^*(\Bun_G),$$
where each $\on{H}^*(\Bun_G)$ is endowed with an (invertible) action of the geometric Frobenius, denoted $\on{Frob}$.  

\medskip

On the other hand, for each closed point $x$ of $X^0$, we consider the finite group $G(k_x)$ (i.e., the group of rational points of the fiber $G_x$
of $G^0$ over the residue $k_x$ field at $x$), and the number 
$$\frac{|k_x|^{\dim(G_x)}}{|G(k_x)|}\in \BR\subset \BC.$$

\sssec{}

The numerical product formula says:

\begin{thm} \label{t:num prod} \hfill

\smallskip

\noindent{\em(a)} Each cohomology $\on{H}^i(\Bun_G)$ is finite-dimensional and the sum of complex numbers
$$\underset{i}\Sigma\, (-1)^i\cdot \tau\left(\on{Tr}(\on{Frob}^{-1},\on{H}^i(\Bun_G))\right)$$
converges absolutely. 

\smallskip

\noindent{\em(b)} The product (of real numbers) 
$$\underset{x\in |X^0|}\Pi\, \frac{|k_x|^{\dim(G_x)}}{|G(k_x)|}$$
converges absolutely. 

\smallskip

\noindent{\em(c)} The expressions in (a) and (b) are equal to one another (as complex numbers). 

\end{thm} 

The goal of this section is to prove \thmref{t:num prod}.

\begin{rem}
In the course of the proof we will show that $\tau\left(\on{Tr}(\on{Frob}^{-1},\on{C}^*_c(\Ran,\CB))\right)$ equals 
$$\underset{x\in |X^0|}\Pi\, \tau\left(\on{Tr}(\on{Frob}_x^{-1},\on{H}^*(BG_x))\right).$$

\medskip

As was noted in \secref{sss:Euler}, this justifies why we want to think of $\on{C}^*_c(\Ran,\CB)$ as the Euler product
\eqref{e:Euler again}.
\end{rem}

\ssec{Proof of the numerical product formula}

\sssec{}

First, it is not difficult to see that the validity of the assertion of \thmref{t:num prod} only depends on the generic fiber of $G$. Hence, 
we can (and will) assume that $G$ satisfies the assumptions of \secref{sss:good G}, where the the open subset $X'\subset X$ is 
also defined over $\BF_q$.  

\medskip

Under this assumption, we can apply \thmref{t:AB} and rewrite $\on{H}^*(\Bun_G)$ as the cohomology of $\on{Sym}_{\BQ_\ell}(\on{C}^*(X',M))$,
which is the same as 
$$\on{Sym}_{\BQ_\ell}(\on{H}^*(X',M)).$$  
Note now that the finite-dimensionality assertion in point (a) of the theorem readily follows. 

\sssec{}

For every closed point $x$ of $X^0$, the  
$\tau$-weights of $\on{Frob}^{-1}_x$ on the !-fiber $M_x$ of $M$ at $x$ are $\leq -4$ (see \secref{sss:splitting}). 
Hence, by \cite{Del}, the $\tau$-weights of $\on{Frob}^{-1}$
on $\on{H}^*(X',M)$ are $\leq -2$. I.e., if we view the pair $(\on{Frob}^{-1},\on{H}^i(X',M))$  
as a complex vector space equipped with an endomorphism via $\tau$, and if $\lambda$ is an eigenvalue, then $|\lambda|<2$. 

\medskip

This implies the convergence assertion in point (a) of the theorem. 

\medskip

Consider the formal series
\begin{equation} \label{e:ser1}
\underset{n\geq 0}\Sigma\, \tau\left(\on{Tr}\left(\on{Frob}^{-1},\on{Sym}^n_{\BQ_\ell}(\on{H}^*(X',M))\right)\right)\cdot t^n.
\end{equation} 

We obtain that this series converges absolutely for $|t|< 2$ and its value at $t=1$ equals the sum  in point (a) of the theorem. 

\begin{rem}
We obtain that the sum in point (a) of the theorem equals
$$\tau\left(\frac{1}{\on{det}(1-\on{Frob}^{-1},\on{H}^*(X',M'))}\right).$$
\end{rem}

\sssec{}

Next, note that for a closed point $x$ of $X^0$
the quantity $\frac{1}{|G(k_x)|}$ equals the number of $k_x$-points of the stack $BG_x$. Hence, applying the Grothendieck-Lefschetz
formula for the stack $BG_x$, we obtain
$$\frac{|k_x|^{\dim(G_x)}}{|G(k_x)|}=\underset{i}\Sigma\, (-1)^i\cdot \tau\left(\on{Tr}(\on{Frob}_x^{-1},\on{H}^i(BG_x))\right).$$

\medskip

We have an isomorphism
$$\on{H}^*(BG_x)\simeq 
\begin{cases} 
&\on{Sym}_{\BQ_\ell}(M_x) \text{ for } x\in X' \\
&\BQ_\ell \text{ for } x\notin X'.
\end{cases}$$

Consider the formal series
\begin{equation} \label{e:ser2}
\underset{x\in |X^0|,x\in X'}\Pi\, \left(\underset{n\geq 0}\Sigma\, \tau\left(\on{Tr}(\on{Frob}_x^{-1},\on{Sym}^n_{\BQ_\ell}(M_x))\right)
\cdot t^{n\cdot\deg(x)}\right).
\end{equation} 

The Grothendieck-Lefschetz formula for $X'$ implies that the series \eqref{e:ser1} equals the series \eqref{e:ser2}. In particular, we obtain
that the series \eqref{e:ser2} converges absolutely for $|t|<2$ and its value at $t=1$ equals the sum in point (a) of the theorem. 

\medskip

The absolute convergence of \eqref{e:ser2} at $t=1$ is equivalent to the statement of point (b) of the theorem, and the product in 
point (b) of the theorem equals the value of \eqref{e:ser2} at $t=1$. This implies the equality in point (c) of the theorem. 

\qed

\end{document}